\documentclass[10pt]{report}

\usepackage{indentfirst, times, graphicx, subfigure,
  setspace, amsmath, multirow, epstopdf, amsmath, amsthm,
  amssymb, algorithmic, mathrsfs}
  
\usepackage[chapter]{algorithm} 
\usepackage{color}
  
\usepackage{standalone}
\usepackage{lmodern}            

\newtheorem{prop}{Proposition}

\def\blfootnote{\xdef\@thefnmark{}\@footnotetext}
\newcommand{\lcolor}[1]{\textcolor{black}
	{#1}}  

\newcommand{\xcolor}[1]{\textcolor{black}{#1}}

\newcommand{\expect}[1]{\mathbb{E}{\l(#1\r)}}
\newcommand{\mf}[1]{\mathbf{#1}}

\newcommand{\dotp}[2]{\left\langle#1,#2\right\rangle}

\def\r{\right}
\def\l{\left}

\usepackage{titlesec}
 \titleformat{\chapter}[display]
     {\normalfont\Large\bfseries}{\chaptertitlename\ \thechapter}{20pt}{\Large}

\newcommand{\beas}{\begin{eqnarray*}}
\newcommand{\enas}{\end{eqnarray*}}
\newcommand{\bea}{\begin{eqnarray}}
\newcommand{\ena}{\end{eqnarray}}
\newcommand{\bms}{\begin{multline*}}
\newcommand{\ems}{\end{multline*}}
\newcommand{\bels}{\begin{align*}}
\newcommand{\enls}{\end{align*}}
\newcommand{\bel}{\begin{align}}
\newcommand{\enl}{\end{align}}

\newcommand{\ignore}[1]{}

\newtheorem{theorem}{Theorem}[section]
\newtheorem{corollary}{Corollary}[section]

\newtheorem{proposition}{Proposition}[section]
\newtheorem{remark}{Remark}[section]
\newtheorem{lemma}{Lemma}[section]
\newtheorem{definition}{Definition}[section]

\newtheorem{assumption}{Assumption}[section]
\newtheorem{example}{Example}
\newtheorem{Alg}{Algorithm}

\begin{document}

\pagenumbering{roman}   

\begin{center}%
\doublespacing
\null
\vfill\vfill
\MakeUppercase{I. Asynchronous Optimization over weakly Coupled Renewal Systems}\par
 \vskip 0.16 true in
      by%
      \vskip 0.16 true in
      {\begin{tabular}[t]{c}Xiaohan Wei \\*[0.6 true in]%
       \rule{4in}{1.5pt}
       \end{tabular}%
       \par}

      \vskip 1.0 true in
      \singlespacing
      Presented to the     \\
      FACULTY OF THE USC GRADUATE SCHOOL      \\
      UNIVERSITY OF SOUTHERN CALIFORNIA   \\
      In Partial Fulfillment of the       \\
      Requirements for the Degree         \\
      DOCTOR OF PHILOSOPHY                \\
      \MakeUppercase{(Electrical Engineering)}           \\*[1.0 true cm]%
      \vfill
      {\small December \ \ 2019\par}
    \end{center}%
\par
\vfill
\begin{center}%
{\normalsize\ Copyright~ 2019 \ \hfill ~Xiaohan Wei}%
\end{center}%
           
\thispagestyle{empty} 

\newpage
\chapter*{}
\noindent  Approved by\\

\noindent Professor Michael Neely,\\
Committee Chair,\\
Department of Electrical Engineering,\\
\textit{University of Southern California}.\\

\noindent Professor Stanislav Minsker,\\
Committee Chair,\\
Department of Mathematics,\\
\textit{University of Southern California}.\\

\noindent Professor Larry Goldstein,\\
Department of Mathematics,\\
\textit{University of Southern California}.\\

\noindent Professor Mihailo Jovanovic,\\
Department of Electrical Engineering,\\
\textit{University of Southern California}.\\

\noindent Professor Ashutosh Nayyar,\\
Department of Electrical Engineering,\\
\textit{University of Southern California}.\\

\newpage


\chapter*{Dedication}
\addcontentsline{toc}{chapter}{Dedication}

To my parents and my wife, Yuhong, who supported me both mentally and financially over the years.

\newpage

\doublespacing
\chapter*{Acknowledgements}
\addcontentsline{toc}{chapter}{Acknowledgements}

First, I would like to thank my advisor professor Michael J. Neely for guiding me throughout the PhD journey since Summer 2013. He is a man of accuracy and rigorousness, always passionate about discussing concrete research problems, and willing to roll up the sleeves and grind through technical details with me. His way of treating research topics significantly impacts me. Rather than blindly following existing works and doing incremental works when trying to get into a new area, I learned to ask fundamental mathematical questions, making connections to the tools and theories we already familiar with and be not afraid of getting my hands dirty. His blazing new ideas are my morale boost when grasping in the dark.

Next, I would like to thank professor Stanislav Minsker, who is the advisor on my high-dimensional statistics research. I got to know him during the Math-547 statistical learning course Fall 2015. Though not much senior than me, he is already extremely knowledgable on the statistical learning area and has been widely recognized for his works on robust high-dimensional statistics. He is a quick thinker and can always point out meaningful new directions hiding rather deeply which eventually lead to high-quality publications. I would have published no paper on this area should I never met with him. Along the way, he also teaches me how to sell my works and helps me practicing my seminar talks, which lead to impressive presentations and Ming-Hsieh scholarships.    

Also, I would like to thank professor Larry Goldstein, whom I met during a small paper reading group Spring 2016. He is an expert on Stein's method and, as a senior professor, surprisingly accessible to PhD students and active on various research areas. Together with Prof. Minsker, we had quite a few fruitful discussions and made some nice progress on robust statistics.  

I would also like to thank professor Mihailo Jovanovic, Ashutosh Nayyar for discussing research problems with me and siting on my qualifying exam committee. I appreciate them for their valuable comments and suggestions. 

Moreover, I thank my senior lab mates Hao Yu and  Sucha Supittayapornpong who were always accessible to discussing problems with me and came up with new research ideas. Also, Ruda Zhang, Lang Wang, and Jie Ruan studied various math courses and interesting math problems with me and helped me clear up the hurdles on different stages, for which I really appreciate. Special thanks to professor Qing Ling, who was my undergraduate advisor, but continuously influences me on various aspects of my academic career. 

Last but not least, I would like to take the chance to express my gratitude for folks who made contribution on various stages of my research. In particular, I thank Zhuoran Yang, for lighting up new areas and expanding my research horizon, Dongsheng Ding, who brings idea from control perspective and is always passionate to try out research ideas with me, Sheng Chen for sharing with me his perspective on robust LASSO problems, professor Jason D. Lee for working on the geometric median problem with me, and Jianshu Chen from Tencent AI who introduced me to the area of reinforcement learning.

\singlespacing
\renewcommand{\contentsname}{Table of Contents}  
\tableofcontents  


\chapter*{Abstract}
\addcontentsline{toc}{chapter}{Abstract}

A renewal system divides the slotted timeline into back to back time periods called ``renewal frames''. At the beginning of each frame, it chooses a policy from a set of options for that frame. The policy determines the duration of the frame, the penalty incurred during the frame (such as energy expenditure), and a vector of performance metrics (such as instantaneous number of jobs served). The starting points of this line of research are Chapter 7 of the book \cite{neely2010stochastic}, the seminal work \cite{neely2013dynamic}, and Chapter 5 of the PhD thesis of Chih-ping Li \cite{li2011stochastic}, who graduated before I came to USC. These works consider stochastic optimization over a single renewal system. By way of contrast, this thesis considers optimization over multiple parallel renewal systems, which is computationally more challenging and yields much more applications. The goal is to minimize the time average overall penalty subject to time average overall constraints on the corresponding performance metrics. The main difficulty, which is not present in earlier works, is that these systems act asynchronously due to the fact that the renewal frames of different renewal systems are not aligned. The goal of the thesis is to resolve this difficulty head-on via a new asynchronous algorithm and a novel supermartingale stopping time analysis which shows our algorithms not only converge to the optimal solution but also enjoy fast convergence rates.
Based on this general theory, we further develop novel algorithms for data center server provision problems with performance guarantees as well as new heuristics for the multi-user file downloading problems.

We start by reviewing existing works on the optimization over a single renewal system in Chapter 1. Then, in Chapter 2, we propose a new algorithm for the asynchronous renewal optimization so that each system can make its own decision after observing a global multiplier that is updated every slot.  We show that this algorithm satisfies the desired constraints and achieves $O(\epsilon)$ near optimality with $O(1/\epsilon^2)$ convergence time.  Based on the new algorithm, we formulate the data center server provision problem as an asynchronous renewal optimization in Chapter 3 and develop a corresponding algorithm which exceeds the state-of-the-art. In Chapter 4, we look at another application, namely, the multi-user file downloading, which can be formulated as a constrained multi-armed bandit problem. We show that our proposed algorithm leads to a useful heuristic approximately solving the problem with experimentally near optimal performance.

In Chapter 5, we consider the constrained optimization over a renewal system with observed random events at the beginning of each renewal frame. We propose an online algorithm which does not need the knowledge of the distributions of random events. We prove that this proposed algorithm is feasible and achieves $O(\varepsilon)$ near optimality by constructing an exponential supermartingale. Simulation experiments demonstrates the near optimal performance of the proposed algorithm.

Finally, in Chapter 6, we consider online learning over weakly coupled Markov decision processes.
We develop a new distributed online algorithm where each
MDP makes its own decision each slot after observing a multiplier computed from past
information. While the scenario is significantly more challenging than the classical online
learning context, the algorithm is shown to have a tight $\mathcal{O}(\sqrt{T})$ regret and constraint
violations simultaneously over a time horizon $T$.


\clearpage
\pagenumbering{arabic}  

\doublespacing



\chapter{Introduction to Renewal Systems}

\section{Optimization over a single renewal system: A review}

\begin{figure}[htbp]
\centering
   \minipage{0.5\textwidth}
   \includegraphics[width=\linewidth]{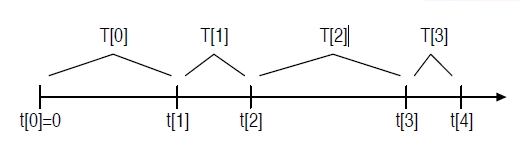} 
   \caption{The sample timeline of a renewal system.}
   \label{fig:renewal0}
   \endminipage
\end{figure}

Renewal systems are generalizations of renewal processes studied in probability and random processes courses. Parallel to Markov decision processes versus Markov chains, renewal systems are controlled renewal processes. Since this is not a widely used term, to set the tone of the thesis, 
we start with a review of optimization over a single renewal system.

Consider a dynamical system operating over a discrete slotted timeline $t \in \{0, 1, 2, \ldots\}$. The timeline is segmented into back-to-back intervals of time slots called \emph{renewal frames}. The start of each renewal frame for a system is called a \emph{renewal time} or simply a \emph{renewal} for that system.
The duration of each renewal frame is a random positive integer whose distribution depends on a control action chosen at the start of the frame. We use $k=0, 1,2,\cdots$ to index the renewals. Let $t_k$ be the time slot corresponding to the $k$-th renewal with the convention that $t_0=0$. Let $\mathcal{T}_{k}$ be the set of all slots from $t_k$ to $t_{k+1}-1$. See Fig. \ref{fig:renewal0} for a graphical illustration.

At time $t_k$, the decision maker chooses a possibly random decision $\alpha_k$ in a set $\mathcal{A}$. This action determines the distributions of the following random variables:
\begin{itemize}
\item The duration of the $k$-th renewal frame $T_k:=t_{k+1}-t_k$, which is a positive integer.
\item A vector of performance metrics at each slot of that frame 
$\mathbf{z}[t]:=\left(z_1[t],~z_2[t],~\cdots,~z_L[t]\right)$,\\
$t\in\mathcal{T}_{k}$, where $L$ is a fixed positive integer.
\item A penalty incurred at each slot of the frame $y[t]$, $t\in\mathcal{T}_{k}$.
\end{itemize}
In the special case where $T_k = 1,~\forall k$, this reduces to the classical slotted stochastic system, which has been relatively well-understood. Let $\mathcal{F}_k$ be the system history up to $t_k-1$, which includes $\{y[t]\}_{j=0}^{t_k-1}$, $\{\mathbf z[t]\}_{j=0}^{t_k-1}$ and $\{T_j\}_{j=0}^{k-1}$. 
The key property we rely on is as follows.  

\begin{definition}[Renewal property]\label{def:renewal}
A system is said to satisfy the renewal property if the random $T_k$, 
$\mathbf{z}[t]$ and $y[t]$,~$t\in\mathcal{T}_k$ are conditionally independent of the history $\mathcal{F}_k$ given $\alpha_k=\alpha\in\mathcal{A}$.
\end{definition}

The goal is to minimize the time average penalty subject to $L$ time average constraints on the performance metrics, i.e. we aim to solve the following optimization problem:
\begin{align}
\min~~&\limsup_{T\rightarrow\infty}\frac{1}{T}\sum_{t=0}^{T-1}\expect{y[t]}\label{problem-1}\\
\textrm{s.t.}~~ &\limsup_{T\rightarrow\infty}\frac{1}{T}\sum_{t=0}^{T-1}\expect{z_l[t]}\leq d_l,~~l\in\{1,2,\cdots,L\},\label{problem-2}
\end{align} 
where $\{d_l\}_{l=1}^L$ are known constants. Let
\[
 y(\alpha_k) := \sum_{t\in \mathcal{T}_{k}}y[t],~~~z_{l}(\alpha_k) = \sum_{t\in \mathcal{T}_{k}}z_l[t],~~~T(\alpha_k) = T_k
\]
be realizations during the $k$-th frame using an action $\alpha_k$. 
Under mild technical conditions (e.g. existence of second moments, see Section \ref{sec:chap-2-pre} for details), the problem \eqref{problem-1}-\eqref{problem-2} can also be written as a fractional program form:
\begin{align}
\min~~&\limsup_{K\rightarrow\infty}\frac{\expect{\sum_{k=0}^{K-1} y(\alpha_k)}}{\expect{\sum_{k=0}^{K-1}T(\alpha_k)}}\label{problem-3}\\
\textrm{s.t.}~~ &\limsup_{T\rightarrow\infty}\frac{\expect{\sum_{k=0}^{K-1}z_l(\alpha_k)}}{\expect{\sum_{k=0}^{K-1}T(\alpha_k)}}\leq d_l,~~l\in\{1,2,\cdots,L\},\label{problem-4}\\
&\alpha_k\in\mathcal{A},~~\forall k \nonumber
\end{align}

\subsection{Optimization over i.i.d. actions}
Suppose the system adopts an i.i.d. sequence of random actions $\{\alpha_k^*\}_{k=0}^{\infty}$, where the decision $\alpha^*_k\in\mathcal A$ made on frame $k$ independent of the past.  
Then, by the renewal property, it is easy to see that 
$\{ y(\alpha_k^*), \mathbf z(\alpha_k^*), T(\alpha_k^*)\}$ are i.i.d. random variables. We have
\begin{align*}
&\limsup_{K\rightarrow\infty}\frac{\expect{\sum_{k=0}^{K-1}y(\alpha_k^*)}}{\expect{\sum_{k=0}^{K-1}T(\alpha_k^*)}} 
= \frac{\lim_{K\rightarrow\infty}\frac1K \expect{\sum_{k=0}^{K-1}y(\alpha_k^*)}}{\lim_{K\rightarrow\infty}\frac1K\expect{\sum_{k=0}^{K-1}T(\alpha_k^*)}} = \frac{\expect{ y(\alpha_k^*)}}{\expect{T(\alpha_k^*)}}\\
&\limsup_{K\rightarrow\infty}\frac{\expect{\sum_{k=0}^{K-1}z_l(\alpha_k^*)}}{\expect{\sum_{k=0}^{K-1}T(\alpha_k^*)}} 
= \frac{\lim_{K\rightarrow\infty}\frac1K \expect{\sum_{k=0}^{K-1}z_l(\alpha_k^*)}}{\lim_{K\rightarrow\infty}\frac1K\expect{\sum_{k=0}^{K-1}T(\alpha_k^*)}} = \frac{\expect{ z_l(\alpha_k^*)}}{\expect{T(\alpha_k^*)}}
\end{align*}
As a consequence, if we consider solving \eqref{problem-3}-\eqref{problem-4} over the set of i.i.d. random actions, then?
\begin{align}
\min~~&\frac{\expect{ y(\alpha_k^*)}}{\expect{T(\alpha_k^*)}}\label{problem-5}\\
\textrm{s.t.}~~ &\frac{\expect{ z_l(\alpha_k^*)}}{\expect{T(\alpha_k^*)}}\leq d_l,~~l\in\{1,2,\cdots,L\},\label{problem-6}
\end{align} 
\begin{assumption}
The problem \eqref{problem-5}-\eqref{problem-6} is feasible, i.e. there exists $\alpha_k^*$ such that \eqref{problem-6} are satisfied. Furthermore, we assume the set of all feasible performance vectors
 $(\frac{\expect{ y(\alpha_k^*)}}{\expect{T(\alpha_k^*)}}, ~\frac{\expect{ \mathbf z(\alpha_k^*)}}{\expect{T(\alpha_k^*)}})$
 over all i.i.d. actions $\alpha_k^*$ is compact.
\end{assumption}
The compactness assumption is adopted so that there exists at least one i.i.d. action which solves \eqref{problem-5}-\eqref{problem-6}. In fact, one can show that under proper technical conditions the minimum achieved by \eqref{problem-3}-\eqref{problem-4} is the same as that of \eqref{problem-5}-\eqref{problem-6} (see, for example, Lemma \ref{stationary-lemma} in the next section).


\subsection{Ergodic Markov decision process (MDP): An example}
As one of the main motivations for this line of research, 
in this section, we show that the well-known MDP is a special case of the renewal system.
Consider a discrete time MDP over an infinite horizon.  It
consists of a finite state space $\mathcal{S}$, and an action space $\mathcal{U}$ at each state $s\in\mathcal{S}.$\footnote{To simplify the notation, we assume each state has the same action space $\mathcal{A}$. All our analysis generalizes trivially to states with different action spaces.} For each state $s\in\mathcal{S}$, we use $P_u(s,s')$ to denote the transition probability from $s\in\mathcal{S}$ to $s'\in\mathcal{S}$ when taking action $u\in\mathcal{U}$, i.e. 
\[
P_u(s,s')  = Pr(s[t+1]=s'~|~s[t] = s,~u[t]=u),
\]
where $s[t]$ and $u[t]$ are state and action at time slot $t$. 

At time slot $t$, after observing the state $s[t]\in\mathcal{S}$ and choosing the action $u[t]\in\mathcal{U}$, the MDP receives a penalty $y(u[t],s[t])$ and $L$ types of resource costs $z_{1}(u[t],s[t]),\cdots,z_{L}(u[t],s[t])$, where these functions are all bounded mappings from $\mathcal{S}\times\mathcal{U}$ to $\mathbb{R}$. For simplicity we write 
$y[t] = y(u[t],s[t])$ and $z_l[t] = z_l(u[t],s[t])$. The goal is to minimize the time average penalty with constraints on time average overall costs. This problem 
can be written in the form \eqref{problem-1}-\eqref{problem-2}.

In order to define the renewal frame, we need one more assumption on the MDP. We assume the MDP is \textit{ergodic}, i.e. there exists a state which is recurrent and the corresponding Markov chain is aperiodic under any randomized stationary policy\footnote{A \textit{randomized stationary policy} $\pi$ is an algorithm which chooses actions at state $s\in\mathcal{S}$ according to a fixed conditional distribution $\pi(u | s),~u\in\mathcal{U}$ and is independent of all other past information, i.e. $Pr(u[t] |\mathcal{F}_t) = \pi(u[t] |s[t])$,~$u[t]\in\mathcal{U}$, $s[t]\in\mathcal{S}$ and $\mathcal{F}_t$ is the past information up to time 
$t$.}, 
with bounded expected recurrence time. Under this assumption, the renewals for the MDP can be defined as successive revisitations to the recurrent state, and the action set $\mathcal{A}$ in such scenario is defined as the set of all randomized stationary policies that can be implemented in one renewal frame. 
Thus, our renewal system formulation includes ergodic MDPs. We refer to \cite{Al99}, \cite{Be01}, and \cite{Ro02} for more details on MDP theory and related topics. We also refer readers to Chapter 5 for more MDP specific algorithms and analysis.

\subsection{The Drift-plus-penalty(DPP) ratio algorithm}
In this section, we introduce the classical DPP ratio algorithm solving \eqref{problem-3}-\eqref{problem-4}
( \cite{neely2010stochastic}, \cite{neely2013dynamic}). It is a frame-based algorithm which updates parameters at the beginning of each frame. 
We start by defining the ``virtual queues'' $Q_l[k]$ for each constraint with $Q_l[0] = 0$ and 
\[
Q_l[k+1] = \max\{Q_l[k] + z_l(\alpha_k) - d_lT(\alpha_k), 0\},
\]
which is updated per frame. 
Let $\mathbf Q[t]$ be the vector of virtual queues. Define the \textit{drift} as follows:
\[
\Delta[k] : = \frac12(\|\mathbf Q[k+1]\|_2^2 - \|\mathbf Q[k]\|_2^2),
\]
Let $\mathcal{F}_k$ be the system history up to $t_k-1$, which includes $\{y(\alpha_j)\}_{j=0}^{t-1}$, $\{\mathbf z(\alpha_j)\}_{j=0}^{t-1}$ 
Then, it is easy to show that
\[
\expect{\Delta[k]|\mathcal{F}_k} \leq B + \sum_{l=1}^LQ_l[k]\expect{z_l(\alpha_k)-d_lT(\alpha_k)|\mathcal{F}_k}.
\]
Assuming that the second moment of $z_l(\alpha_k)-d_lT(\alpha_k)$ exists, then, there exists a constant $B$ such that 
$$
B\geq \frac12\expect{(z_l(\alpha_k)-d_lT(\alpha_k))^2|\mathcal{F}_k}.
$$
We define the DPP expression as $\Delta[k] + Vy(\alpha_k)$, where $V>0$ is a trade-off parameter, which has the following bound:
\begin{align}
&\expect{\Delta[k] + Vy(\alpha_k)~|\mathcal{F}_k}\nonumber\\
\leq& B + \sum_{l=1}^LQ_l[k]\expect{z_l(\alpha_k)-d_lT(\alpha_k)|\mathcal{F}_k} + V\expect{y(\alpha_k)~|\mathcal{F}_k} \label{eq:simple-dpp-ub-pre}\\
=& B + \expect{T(\alpha_k)|\mathcal F_k}
\underbrace{\frac{V\expect{y(\alpha_k)~|\mathcal{F}_k} + \sum_{l=1}^LQ_l[k]\expect{z_l(\alpha_k)-d_lT(\alpha_k)|\mathcal{F}_k}}{\expect{T(\alpha_k)|\mathcal F_k}}}_{\text{minimize this}}.  \label{eq:simple-dpp-ub}
\end{align}
Then, the algorithm (Algorithm \ref{dpp-ratio-algorithm}) aims at minimizing the ratio on the right hand side.
\begin{algorithm}
\begin{Alg}\label{dpp-ratio-algorithm}
DPP ratio algorithm: Fix a trade-off parameter $V>0$.
\begin{itemize}
\item 
At the beginning of each frame, the proposed algorithm takes action $\alpha_k$ in order to minimize the ratio
\begin{equation}\label{eq:simple-ratio}
\frac{V\expect{y(\alpha_k)~|\mathcal{F}_k} + \sum_{l=1}^LQ_l[k]\expect{z_l(\alpha_k)|\mathcal{F}_k}}{\expect{T(\alpha_k)|\mathcal F_k}}.
\end{equation}
\item Update the virtual queue $\mathbf Q[k]$ via 
\begin{equation}\label{eq:Q-update}
Q_l[k+1] = \max\{Q_l[k] + z_l(\alpha_k) - d_lT(\alpha_k), 0\}.
\end{equation}
\end{itemize}
\end{Alg}
\end{algorithm}
Note that due to the renewal property of the system, maximizing the above ratio is the same as maximizing the ratio:
\[
\frac{V\expect{y(\alpha_k)} + \sum_{l=1}^LQ_l[k]\expect{z_l(\alpha_k)}}{\expect{T(\alpha_k)}}.
\]

\subsection{A (somewhat) simple illustrative performance analysis}\label{sec:simple-analysis}
The performance of this algorithm has been shown in a number of works (\cite{neely2010stochastic,neely2013dynamic}). We reproduce the proof here but from a somewhat different perspective compared to previous works since it is more illustrative for our purpose and serves as the foundation of our new analysis later.

The key step, which is repeatedly used throughout the thesis is as follows: Since our proposed algorithm minimizes \eqref{eq:simple-ratio}, it must satisfy:
\begin{equation}\label{eq:key-simple}
\frac{V\expect{y(\alpha_k)~|\mathcal{F}_k} + \sum_{l=1}^LQ_l[k]\expect{z_l(\alpha_k)|\mathcal{F}_k}}{\expect{T(\alpha_k)|\mathcal F_k}}\leq 
\frac{V\expect{y(\alpha_k^*)} + \sum_{l=1}^LQ_l[k]\expect{z_l(\alpha_k^*)}}{\expect{T(\alpha_k^*)}}
\end{equation}
for any i.i.d. decisions $\alpha_k^*$, where we use the fact that the $\alpha_k^*$ is independent of history $\mathcal F_{k}$ and thus the conditioning on the right hand side can be omitted. In particular, we can choose $\alpha_k^*$ to be the solution to \eqref{problem-5}-\eqref{problem-6} and let 
$[f^*,\mf g^*] = [\expect{y(\alpha_k^*)}/\expect{T(\alpha_k^*)}, \expect{\mf z(\alpha_k^*)}/\expect{T(\alpha_k^*)}]$ be the optimal performance vector. 
Rearranging terms in above inequality gives
\[
\expect{\l.V(y(\alpha_k)- f^*T(\alpha_k)) +\sum_{l=1}^LQ_l[k](z_l(\alpha_k) - g_l^*T(\alpha_k)) ~\r|\mathcal{F}_k}\leq 0.
\]
This implies that the expression inside the expectation is a \textit{supermartingale difference sequence}, a fact not necessarily needed here but is the key to our new analysis later. Now, taking expectation from both sides and sum up from $k= 0$ to $K-1$ give
\[
\sum_{k=0}^{K-1}\expect{V(y(\alpha_k)- f^*T(\alpha^k)) +\sum_{l=1}^LQ_l[k](z_l(\pi_k) - g_l^*T(\pi^k))}\leq 0.
\]
Substituting $g_l^*\leq d_l$ gives
\begin{equation}\label{eq:dpp-pre-0}
\sum_{k=0}^{K-1}\expect{V(y(\alpha_k)- f^*T(\alpha_k)) +\sum_{l=1}^LQ_l[k](z_l(\alpha_k) - d_lT(\alpha_k))}\leq 0.
\end{equation}
On the other hand, taking expectation from both sides of the inequality \eqref{eq:simple-dpp-ub-pre} and sum up from $k= 0$ to $K-1$ gives
\[
\frac{\expect{\|\mf Q[k+1]\|_2^2}}{2} +\sum_{k=0}^{K-1}\expect{Vy(\alpha_k)}
\leq \sum_{k=0}^{K-1}\expect{Vy(\alpha_k) +\sum_{l=1}^LQ_l[k](z_l(\alpha_k) - d_l T(\alpha_k))} +BK.
\]
Sum the above inequality and \eqref{eq:dpp-pre-0} gives
\begin{equation}\label{eq:simple-final}
\frac{\expect{\|\mf Q[k+1]\|_2^2}}{2} +\sum_{k=0}^{K-1}\expect{Vy(\alpha_k)}\leq V\sum_{k=0}^{K-1}f^*\expect{T(\alpha_k)} +BK.
\end{equation}
This bound ``kills two birds in the same cage'', allowing us to get objective bound and constraint violations at the same time immediately. On one hand, since $\expect{\|\mf Q[k+1]\|_2^2}\geq0$, we have
\begin{equation}\label{eq:simple-obj}
\frac{\sum_{k=0}^{K-1}\expect{Vy(\alpha_k)}}{\sum_{k=0}^{K-1}\expect{T(\alpha_k)}}\leq f^* + \frac{B}{V},
\end{equation}
On the other hand, Let $C$ be a constant such that $C\geq |\expect{y(\pi)}|,~T_{\max}\geq|\expect{T(\pi^k)}|~\forall \pi\in\Pi$,
\begin{equation}\label{eq:simple-constraint}
\expect{\|\mf Q[k+1]\|_2}\leq \sqrt{2BK + 4VK(C+T_{\max})}~~\Rightarrow
\frac{\sum_{k=0}^{K-1}\expect{z_l(\alpha_k)}}{\sum_{k=0}^{K-1}\expect{T(\alpha_k)}}\leq \sqrt{\frac{ 2B + 4V(C+T_{\max})}{K}},
\end{equation}
which follows from the virtual queue updating rule \eqref{eq:Q-update} that 
$\expect{\|\mf Q[k+1]\|_2}\geq \sum_{k=0}^{K-1}\expect{z_l(\alpha_k)}$ and $\sum_{k=0}^{K-1}\expect{z_l(\alpha_k)}\geq K$.

\begin{remark}
The bounds \eqref{eq:simple-obj}, \eqref{eq:simple-constraint} are not the tightest possible bounds, but (I believe) simple enough to highlight the key steps.
\end{remark}

\section{The coupled renewal systems}

So far readers have gain some understanding on the renewal systems we will talk about throughout the thesis. In this section, we introduce our coupled renewal system model. Many of the notations are the same as those of the last section except we add a superscript $n$ to index the renewal systems.
Consider $N$ renewal systems that operate over a slotted timeline ($t \in \{0, 1, 2, \ldots\}$).  
The timeline for each system $n \in \{1, \ldots, N\}$ is segmented into back-to-back intervals, which are renewal frames. 
The duration of each renewal frame is a random positive integer with distribution that depends on a control action chosen by the system at the start of the frame.  The decision at each renewal frame also determines the penalty and a vector of performance metrics during this frame.  The systems are coupled by time average constraints placed on these metrics over all systems.  The goal is to design a decision strategy for each system so that overall time average penalty is minimized subject to time average constraints. 

We use $k=0, 1,2,\cdots$ to index the renewals. Let $t^n_k$ be the time slot corresponding to the $k$-th renewal of the $n$-th system with the convention that $t^n_0=0$. Let $\mathcal{T}^{n}_{k}$ be the set of all slots from $t^n_k$ to $t^n_{k+1}-1$.
At time $t^n_k$, the $n$-th system chooses a possibly random decision $\alpha^n_k$ in a set $\mathcal{A}^n$. This action determines the distributions of the following random variables:
\begin{itemize}
\item The duration of the $k$-th renewal frame $T^n_k:=t^n_{k+1}-t^n_k$, which is a positive integer.
\item A vector of performance metrics at each slot of that frame 
$\mathbf{z}^n[t]:=\left(z^n_1[t],~z^n_2[t],~\cdots,~z^n_L[t]\right)$,\\
$t\in\mathcal{T}^{n}_{k}$.
\item A penalty incurred at each slot of the frame $y^n[t]$, $t\in\mathcal{T}^{n}_{k}$.
\end{itemize}
We assume each system has the \textit{renewal property} as defined in Definition \ref{def:renewal} that given $\alpha^n_k=\alpha^n\in\mathcal{A}^n$, the random variables $T^n_k$, $\mathbf{z}^n[t]$ and $y^n[t]$,~$t\in\mathcal{T}^n_k$ are independent of the information of all systems from the slots before $t^n_k$ with the following \textit{known} conditional expectations $\expect{\left.T^n_k\right|\alpha^n_k=\alpha^n}$, $\expect{\left.\sum_{t\in\mathcal{T}^n_k}y^n[t]\right|\alpha^n_k=\alpha^n}$ and $\expect{\left.\sum_{t\in\mathcal{T}^n_k}\mathbf{z}^n[t]\right|\alpha^n_k=\alpha^n}$.  Fig. \ref{fig:simple-renewal} plots a sample timeline of three parallel renewal systems.

\begin{figure}[htbp]
\centering
   \minipage{0.6\textwidth}
   \includegraphics[width=\linewidth]{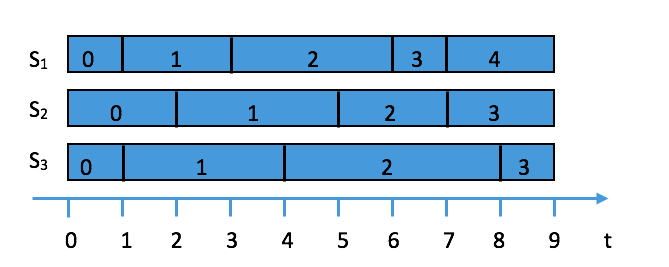} 
   \caption{The sample timelines of three asynchronous parallel renewal systems, where the numbers underneath the figure index time slots and the numbers inside the blocks index the renewals of each system.}
   \label{fig:simple-renewal}
   \endminipage
\end{figure}

To make the framework a little bit more general, we introduce an uncontrollable external i.i.d. random process $\{\mathbf{d}[t]\}_{t=0}^{\infty}\subseteq\mathbb{R}^L$ which can be observed during each time slot. Let $d_l:=\expect{d_l[t]}$. 
The expectation of $\mathbf{d}[t]$ often serves as the constraints of corresponding performance metrics. As we shall see in the example application on an energy-aware scheduling problem, $\mathbf{z}^n[t]$ and $\mathbf{d}[t]$ could represent vectors of job services and arrivals for difference classes, respectively, and the constraints are that the time average service is no less than the time average of arrivals for all classes of jobs.   
 The goal is to minimize the total time average penalty of these $N$ renewal systems subject to $L$ total time average constraints on the performance metrics related to the external i.i.d. process, i.e. we aim to solve the following optimization problem:
\begin{align}
\min~~&\limsup_{T\rightarrow\infty}\frac{1}{T}\sum_{t=0}^{T-1}\sum_{n=1}^N\expect{y^n[t]}\label{prob-1}\\
\textrm{s.t.}~~ &\limsup_{T\rightarrow\infty}\frac{1}{T}\sum_{t=0}^{T-1}\sum_{n=1}^N\expect{z^n_l[t]}\leq d_l,~~l\in\{1,2,\cdots,L\}.\label{prob-2}
\end{align}

\section{Example Applications and previous works}\label{sec:application}

\subsection{Multi-server energy-aware scheduling}\label{sec:server-app}
Consider a slotted time system with $L$ classes of jobs and $N$ servers. Job arrivals are Poisson distributed with rates $\lambda_1,~\cdots,~\lambda_L$, respectively. These jobs are stored in separate queues denoted as $Q_1[t],~\cdots,~Q_L[t]$ in a router waiting to be served. Assume the system is empty at time $t=0$ so that $Q_l[0]=0,~\forall l\in\{1,2,\cdots,L\}$. Let $\lambda_l[t]$ be the precise number of class $l$ job arrivals at slot $t$, then, we have $\expect{\lambda_l[t]}=\lambda_l,~\forall l\in\{1,2,\cdots,L\}$. 
Let $\mu^n_l[t]$ and $e^n[t]$ be the number of class $l$ jobs served and the energy consumption for server $n$ at time slot $t$, respectively. Fig. \ref{fig:multi-server} sketches an example architecture of the system with 3 classes of jobs and 10 servers.

Each server makes decisions over renewal frames and the first frame starts at time slot $t=0$. Successive renewals can happen at different slots for different servers. For the $n$-th server, at the beginning of the $k$-th frame ($k\in\mathbb{N}$), it chooses a processing mode $m^n_k$ within the set of all modes $\mathcal{M}^n$. 
The processing mode $m_k^n$ determines distributions on the number of jobs served, the service time, and the energy expenditure, with conditional expectations: 
\begin{itemize}
\item $\widehat{T}^n(m^n_k):=\expect{\left.T^n_k\right|~m^n_k}$. The expected frame size. 
\item
$\widehat{\mu}^n_l(m^n_k)= \expect{\left.\sum_{t\in\mathcal{T}^n_k}\mu^n_l[t]\right|~m^n_k}$. The expected number of class $l$ jobs served. 
\item $\widehat{e}^n(m^n_k)= \expect{\left.\sum_{t\in\mathcal{T}^n_k}e^n[t]\right|~m^n_k}$. The expected energy consumption.
\end{itemize}

The goal is to minimize the time average energy consumption, subject to the queue stability constraints, i.e.
\begin{align}
\min~~&\limsup_{T\rightarrow\infty}\frac1T\sum_{t=0}^{T-1}\sum_{n=1}^N\expect{e^n[t]}\label{ews-1}\\
s.t.~~&\liminf_{T\rightarrow\infty}\frac1T\sum_{t=0}^{T-1}\sum_{n=1}^N\expect{\mu^n_l[t]}
\geq\lambda_l,~\forall l\in\{1,2,\cdots,L\}.\label{ews-2}
\end{align}
Thus, we have formulated the problem into the form \eqref{prob-1}-\eqref{prob-2}.
Note that the external process in this example is the arrival process of $L$ classes of jobs with potentially unknown arrival rates $\lambda_l$.

\begin{figure}[htbp]
\centering
   \minipage{0.6\textwidth}
   \includegraphics[width=\linewidth]{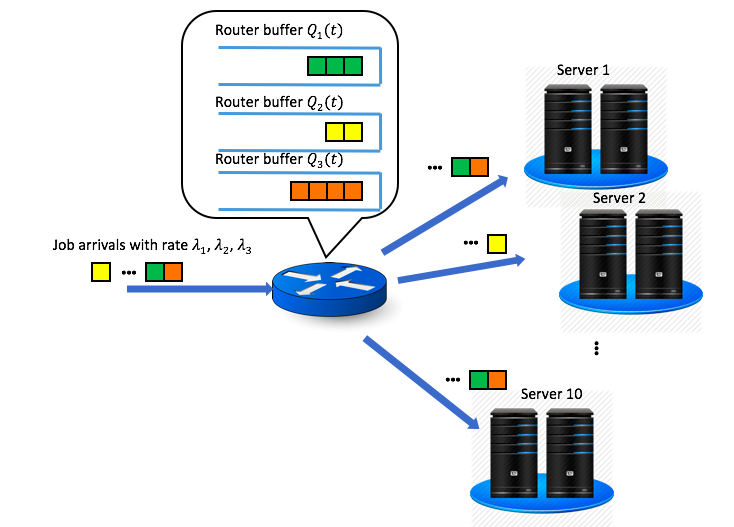} 
   \caption{Illustration of an energy-aware scheduling system with 3 classes of jobs and 10 parallel servers.}
   \label{fig:multi-server}
   \endminipage
\end{figure}

\subsection{Coupled ergodic MDPs}\label{sec:MDP}
Consider $N$ discrete time Markov decision processes (MDPs) over an infinite horizon.  Each MDP
consists of a finite state space $\mathcal{S}^n$, and an action space $\mathcal{U}^n$ at each state $s\in\mathcal{S}^n.$\footnote{To simplify the notation, we assume each state has the same action space $\mathcal{A}^n$. All our analysis generalizes trivially to states with different action spaces.} For each state $s\in\mathcal{S}$, we use $P_u^n(s,s')$ to denote the transition probability from $s\in\mathcal{S}^n$ to $s'\in\mathcal{S}^n$ when taking action $u\in\mathcal{U}^n$, i.e. 
\[
P_u^n(s,s')  = Pr(s[t+1]=s'~|~s[t] = s,~u[t]=u),
\]
where $s[t]$ and $u[t]$ are state and action at time slot $t$. 

At time slot $t$, after observing the state $s[t]\in\mathcal{S}^n$ and choosing the action $u[t]\in\mathcal{U}^n$, the n-th MDP receives a penalty $y^n(u[t],s[t])$ and $L$ types of resource costs $z_{1}^n(u[t],s[t]),\cdots,z_{L}^n(u[t],s[t])$, where these functions are all bounded mappings from $\mathcal{S}^n\times\mathcal{U}^n$ to $\mathbb{R}$. For simplicity we write $y^n[t] = y^n(u[t],s[t])$ and $z_l^n[t] = z_l^n(u[t],s[t])$. The goal is to minimize the time average overall penalty with constraints on time average overall costs, where these MDPs are weakly coupled through the time average constraints. This problem 
can be written in the form \eqref{prob-1}-\eqref{prob-2}.

In order to define the renewal frame, we need one more assumption on the MDPs. We assume each of the MDPs is \textit{ergodic}, i.e. there exists a state which is recurrent and the corresponding Markov chain is aperiodic under any randomized stationary policy, 
with bounded expected recurrence time. Under this assumption, the renewals for each MDP can be defined as successive revisitations to the recurrent state, and the action set $\mathcal{A}^n$ in such scenario is defined as the set of all randomized stationary policies that can be implemented in one renewal frame. 
Thus, our formulation includes coupled ergodic MDPs. We refer to \cite{Al99}, \cite{Be01}, and \cite{Ro02} for more details on MDP theory and related topics.

As a side remark, this multi-MDP problem can be viewed as a single MDP on an enlarged state space. Constrained MDPs are discussed previously in \cite{Al99}.
One can show that under the previous ergodic assumption, the minimum of \eqref{prob-1}-\eqref{prob-2} is achieved by a randomized stationary policy, and furthermore, such a policy can be obtained via solving a linear program reformulated from \eqref{prob-1}-\eqref{prob-2} offline. However, formulating such LP requires the knowledge of all the parameters in the problem, including the statistics of the external process $\{\mathbf{d}[t]\}_{t=0}^\infty$, and the resulting LP is often computationally intractable when the number of MDPs is very large. 

\subsection{Why this problem is difficult}
Compared to \eqref{problem-1}-\eqref{problem-2}, this problem is much more challenging because these $N$ systems are weakly coupled by the time average constraints \eqref{prob-2}, yet each of them operates over its own renewal frames. The renewals of different systems do not have to be synchronized and they do not have to occur at the same rate (e.g. see Fig. \ref{fig:simple-renewal}). Our goal is to develop an algorithm that does not need the knowledge of $d_l=\expect{d_l[t]}$ with a provable performance guarantee.

Note that due to the asynchronicity, the DPP ratio algorithm (Algorithm \ref{dpp-ratio-algorithm}) does not apply. More specifically, in order to cope with the time average constraints, Algorithm \ref{dpp-ratio-algorithm} introduces virtual queues to penalize constraint violations. These virtual queues are then updated frame-wise and the analysis is also on the per frame scale of this particular system. For parallel renewal systems, it is however not clear what is a proper scale to update the virtual queues. 

Naturally, one would think of introducing a virtual queue for each constraint and update the queue whenever at least one of the systems starts its new renewal frame. However, this means for those systems who have yet to reach the renewal, we are updating algorithm parameters in the middle of the renewal of these systems. This creates grave difficulties on how to piece together the analysis from each individual systems. On the other hand, since time is slotted, one could also think of 
``giving up'' the notion of renewals, synchronizing all systems on the slot scale and designing a slot-based algorithm. However, this does not make the problem any simpler since by doing so, the algorithm can still update algorithm parameters in the middle of renewals.

Prior approaches treat this challenge only in special cases. 
The works \cite{Ne12} and \cite{Neely12} consider a special case where all quantities introduced above are deterministic functions of the actions.  The work in \cite{neely2011online} develops a two stage algorithm for stochastic multi-renewal systems, but the first stage must be solved offline.

On the other hand, for the special case where the system is a coupled Markov decision processes (MDPs). Classical methods for MDPs,  such as dynamic programming and linear programming \cite{bertsekas1995dynamic}\cite{puterman2014markov}\cite{Ro02},  can be used to solve this problem. However, it can be
impractical for two reasons:
First, the state space has dimension that depends on the number of renewal systems, making solutions difficult  when the number of renewal systems is large.  Second, some statistics of the system, such as the average $\mf d[t]$ process governing the resource constraints, can be unknown.

\subsection{Other works related to renewal and asynchronous optimization}

The problem considered in the current paper is a generalization of optimization over a single renewal system. It is shown in \cite{Ne09} that for the single renewal system with finite action set, the problem can be solved (offline) via a linear fractional program. Methods for solving linear fractional programs can be found in \cite{BV04} and \cite{Sc83}. The \textit{drift-plus-penalty ratio} approach is also developed in  \cite{Neely2010} and \cite{neely2013dynamic} for the single renewal system.

Note that there are also many other algorithms which consider ``asynchronous optimization'' in a different sense compare to ours. More specifically, the works \cite{BT97}\cite{BGPS06}\cite{SN11} \cite{PXYY15} consider the scenario where the asynchronicity shown in Fig. \ref{fig:simple-renewal} results from uncontrollable delays due to environmental uncertainties. These delays are of fixed distributions independent of the actions or even deterministic. Thus, the delays do not appear in the optimization objectives.

On the other hand, our problem is also related to the multi-server scheduling as is shown in one of the example applications. When assuming proper statistics of the arrivals and/or services, energy optimization problems in multi-server systems can also be treated via queueing theory. Specifically, by assuming both arrivals and services are Poisson distributed, \cite{GDHS13} treats the multi-server system as an M/M/k/setup queue and explicitly computes several performance metrics via the renewal reward theorem. By assuming arrivals are Poisson and only one server, \cite{LN14} and \cite{Yao02} treat the system as a multi-class M/G/1 queue and optimize the average energy consumption via polymatroid optimization.

\section{Outline and our contributions}
The rest of the thesis is organized as follows:
\begin{itemize}
\item \textbf{Chapter 2:}(published in \cite{wei2018asynchronous}) We develop a new algorithm for the general asynchronous renewal optimization, where each system operates on its own renewal frame. It is fully analyzed with convergence as well as convergence time results. 
 As a first technical contribution, we fully characterize the fundamental performance region of the problem \eqref{prob-1}-\eqref{prob-2}.
 We then construct a supermartingale along with a stopping-time to ``synchronize'' all systems on a slot basis, by which we could piece together analysis of each individual system to prove the convergence of the proposed algorithm. Furthermore, encapsulating this new idea into convex analysis tools, we
 prove the $\mathcal{O}(1/\varepsilon^2)$ convergence time of the proposed algorithm to reach $\mathcal{O}(\varepsilon)$ near optimality under a mild assumption on the existence of a Lagrange multiplier. Specifically, we show that for any accuracy $\epsilon>0$ and any time
$T\geq1/\varepsilon^2$, the sequence $\{y^n[t]\}$ and $\{\mathbf{z}^n[t]\}$ produced by our algorithm satisfies,
\begin{align*}
&\frac1T\sum_{t=0}^{T-1}\sum_{n=1}^N\expect{y^n[t]}\leq f_*+\mathcal{O}(\varepsilon),\\
&\frac1T\sum_{t=0}^{T-1}\sum_{n=1}^N\expect{z^n_l[t]}\leq d_l+\mathcal{O}(\varepsilon),
l\in\{1,2,\cdots,L\},
\end{align*}
where $f_*$ denotes the optimal objective value of \eqref{prob-1}-\eqref{prob-2}.
Simulation experiments on the aforementioned multi-server energy-aware scheduling problem also demonstrate the effectiveness of the proposed algorithm.

\item \textbf{Chapter 3 Data center server provision:} (published in \cite{wei2017data})
We consider a cost minimization problem
for data centers with N servers and randomly arriving service
requests. A central router decides which server to use for each
new request.
We formulate this problem as an asynchronous renewal optimization, and develop a
distributed control algorithm so that each server makes its own
decisions, the request queues are bounded and the overall time
average cost is near optimal with probability 1. The algorithm
does not need probability information for the arrival rate or
job sizes. Next, an improved algorithm that uses a single queue
is developed via a ``virtualization'' technique which is shown to
provide the same (near optimal) costs. Simulation experiments
on a real data center traffic trace demonstrate the efficiency of
our algorithm compared to other existing algorithms.

\item \textbf{Chapter 4 Multi-user file downloading:} (published in \cite{wei2015power})
We treat power-aware throughput maximization
in a multi-user file downloading system. Each user can
receive a new file only after its previous file is finished. The
file state processes for each user act as coupled Markov chains
that form a generalized restless bandit system. First, an optimal
algorithm is derived for the case of one user. The algorithm
maximizes throughput subject to an average power constraint.
Next, the one-user algorithm is extended to a low complexity
heuristic for the multi-user problem. The heuristic uses a simple
online index policy. In a special case with no power-constraint,
the multi-user heuristic is shown to be throughput optimal.
Simulations are used to demonstrate effectiveness of the heuristic
in the general case. For simple cases where the optimal solution
can be computed offline, the heuristic is shown to be near-optimal
for a wide range of parameters.

\item \textbf{Chapter 5 Opportunistic Scheduling over Renewal Systems:} (published in \cite{wei2016online})
In this chapter, we consider an opportunistic scheduling problem over a single renewal system. Different from previous chapters, we consider teh scenario where at the beginning of each renewal frame, the  controller observes a random event  and then chooses an action in response to the event, which affects the duration of the frame, the amount of resources used, and a penalty metric. The goal is to make frame-wise decisions so as to minimize the time average penalty subject to time average resource constraints. This problem has applications to task processing and communication in data networks, as well as to certain classes of Markov decision problems. 
We formulate the problem as a dynamic fractional program and propose an adaptive algorithm which uses an empirical accumulation as a feedback parameter. A key feature of the proposed algorithm is that it does not require knowledge of the random event statistics and potentially allows (uncountably) infinite event sets. We prove the algorithm satisfies all desired constraints and achieves $O(\epsilon)$ near optimality with probability 1.

\item \textbf{Chapter 6 Online Learning in Weakly Coupled Markov Decision
Processes:} (published in \cite{wei2018online})
In this chapter, we consider a special case of the multiple parallel renewal systems,
namely, the parallel Markov decision processes coupled by
global constraints, where the time varying objective and constraint functions can only be
observed after the decision is made.
Special attention is given to how well the decision maker can perform in $T$ slots, starting from any state, compared to the best feasible randomized stationary policy in hindsight. We develop a new distributed online algorithm where each MDP makes its own decision each slot after observing a multiplier computed from past information. 
While the scenario is significantly more challenging than the classical online learning context, the algorithm is shown to have a tight $O(\sqrt{T})$ regret and constraint violations simultaneously.  To obtain such a bound, we combine several new ingredients including ergodicity and mixing time bound in weakly coupled MDPs, a new regret analysis for online constrained optimization, a drift analysis for queue processes, and a perturbation analysis based on Farkas' Lemma.

\end{itemize}



\chapter{Asynchronous Optimization over Weakly Coupled Renewal Systems}

In this chapter, we present our asynchronous algorithm along with the new analysis. Along the way, we try to provide some intuitions and high level ideas of the analysis. 

Consider $N$ renewal systems that operate over a slotted timeline ($t \in \{0, 1, 2, \ldots\}$).  
The timeline for each system $n \in \{1, \ldots, N\}$ is segmented into back-to-back intervals, which are renewal frames. 
The duration of each renewal frame is a random positive integer with distribution that depends on a control action chosen by the system at the start of the frame.  The decision at each renewal frame also determines the penalty and a vector of performance metrics during this frame.  The systems are coupled by time average constraints placed on these metrics over all systems.  The goal is to design a decision strategy for each system so that overall time average penalty is minimized subject to time average constraints. 

Recall that we use $k=0, 1,2,\cdots$ to index the renewals. Let $t^n_k$ be the time slot corresponding to the $k$-th renewal of the $n$-th system with the convention that $t^n_0=0$. Let $\mathcal{T}^{n}_{k}$ be the set of all slots from $t^n_k$ to $t^n_{k+1}-1$.
At time $t^n_k$, the $n$-th system chooses a possibly random decision $\alpha^n_k$ in a set $\mathcal{A}^n$. This action determines the distributions of the following random variables:
\begin{itemize}
\item The duration of the $k$-th renewal frame $T^n_k:=t^n_{k+1}-t^n_k$, which is a positive integer.
\item A vector of performance metrics at each slot of that frame 
$\mathbf{z}^n[t]:=\left(z^n_1[t],~z^n_2[t],~\cdots,~z^n_L[t]\right)$,\\
$t\in\mathcal{T}^{n}_{k}$.
\item A penalty incurred at each slot of the frame $y^n[t]$, $t\in\mathcal{T}^{n}_{k}$.
\end{itemize}

We assume each system has the \textit{renewal property} as defined in Definition \ref{def:renewal} that given $\alpha^n_k=\alpha^n\in\mathcal{A}^n$, the random variables $T^n_k$, $\mathbf{z}^n[t]$ and $y^n[t]$,~$t\in\mathcal{T}^n_k$ are independent of the information of all systems from the slots before $t^n_k$ with the following \textit{known} conditional expectations $\expect{\left.T^n_k\right|\alpha^n_k=\alpha^n}$, $\expect{\left.\sum_{t\in\mathcal{T}^n_k}y^n[t]\right|\alpha^n_k=\alpha^n}$ and $\expect{\left.\sum_{t\in\mathcal{T}^n_k}\mathbf{z}^n[t]\right|\alpha^n_k=\alpha^n}$.


\section{Technical preliminaries}\label{sec:chap-2-pre}
Throughout the chapter, we make the following basic assumptions.
\begin{assumption}\label{feasible-assumption}
The problem \eqref{prob-1}-\eqref{prob-2} is feasible, i.e. there are action sequences 
$\{\alpha_k^n\}_{k=0}^\infty$ for all $n\in\{1,2,\cdots,N\}$ so that the corresponding process $\{\mathbf{z}^n[t]\}_{t=0}^\infty$ satisfies the constraints \eqref{prob-2}.
\end{assumption}

Following this assumption, we define $f_*$ as the infimum objective value for \eqref{prob-1}-\eqref{prob-2} over all decision sequences that satisfy the constraints.

\begin{assumption}[Boundedness]\label{bounded-assumption}
For any $k\in\mathbb{N}$ and any $n\in\{1,2,\cdots,N\}$, there exist absolute constants $y_{\max}$, $z_{\max}$ and $d_{\max}$ such that
\begin{align*}
|y^n[t]|\leq y_{\max},~~|z^n_l[t]|\leq z_{\max},~~|d_l[t]|\leq d_{\max},~~\forall t\in\mathcal{T}^n_k,
~\forall l\in\{1,2,\cdots,L\}.
\end{align*}
Furthermore, 
there exists an absolute constant $B\geq1$ such that for every fixed $\alpha^n\in\mathcal{A}^n$ and every $s\in\mathbb{N}$ for which $Pr(T^n_k\geq s|\alpha^n_k=\alpha^n)>0$,
\begin{equation}\label{residual-life-bound}
\expect{\left.(T^n_k-s)^2\right|~\alpha^n_k=\alpha^n,T^n_k\geq s}\leq B.
\end{equation}
\end{assumption}
\begin{remark}
The quantity $T^n_k-s$ is usually referred to as the residual lifetime. In the special case where $s=0$, \eqref{residual-life-bound} gives the uniform second moment bound of the renewal frames as
\[\expect{\left.(T^n_k)^2\right|~\alpha^n_k=\alpha^n}\leq B.\]
Note that \eqref{residual-life-bound} is satisfied for a large class of problems. In particular, it can be shown to hold in the following three cases:
\begin{enumerate}
\item If the inter-renewal $T^n_k$ is deterministically bounded.
\item If the inter-renewal $T^n_k$ is geometrically distributed.
\item If each system is a finite state ergodic MDP with a finite action set.
\end{enumerate}
\end{remark}

\begin{definition}\label{PV-def}
For any $\alpha^n\in\mathcal{A}^n$, let
\[\widehat{y}^n(\alpha^n):=\expect{\left.\sum_{t\in\mathcal{T}^n_k}y^n[t]\right|\alpha^n_k=\alpha^n},~
~\widehat{z}^n_l(\alpha^n):=\expect{\left.\sum_{t\in\mathcal{T}^n_k}z^n_l[t]\right|\alpha^n_k=\alpha^n},\]
and $\widehat{T}^n(\alpha^n):=\expect{T^n_k|\alpha^n_k=\alpha^n}$. Define
\begin{align*}
&\widehat{f}^n(\alpha^n):=\widehat{y}^n(\alpha^n)/\widehat{T}^n(\alpha^n),\\
&\widehat{g}^n_l(\alpha^n):=\widehat{z}^n_l(\alpha^n)/\widehat{T}^n(\alpha^n),~\forall l\in\{1,2,\cdots,L\},
\end{align*}
and let $\left(\widehat{f}^n(\alpha^n),~\widehat{\mathbf{g}}^n(\alpha^n)\right)$ be a performance vector under the action $\alpha^n$.
\end{definition}

Note that by Assumption \ref{bounded-assumption}, $\widehat{y}^n(\alpha^n)$ and $\widehat{\mathbf{z}}^n(\alpha^n)$ in Definition \ref{PV-def} are both bounded, and $T^n_k\geq1,~\forall k\in\mathbb{N}$, thus, the set $\left\{\left(\widehat{f}^n(\alpha^n),~\widehat{\mathbf{g}}^n(\alpha^n)\right),~\alpha^n\in\mathcal{A}^n\right\}$ is also bounded. The following mild assumption 
states that this set is also closed.

\begin{assumption}\label{compact-assumption}
 The set $\left\{\left(\widehat{f}^n(\alpha^n),~\widehat{\mathbf{g}}^n(\alpha^n)\right),~\alpha^n\in\mathcal{A}^n\right\}$ is compact.
\end{assumption}
The motivation of this assumption is to guarantee that there always exists at least one solution to each subproblem in our algorithm. Finally, we define the performance region of each individual system as follows.

\begin{definition}\label{PR-def}
Let $\mathcal{S}^n$ be the convex hull of $\left\{\left(\widehat{y}^n(\alpha^n),~\widehat{\mathbf{z}}^n(\alpha^n),~\widehat{T}^n(\alpha^n)\right):~\alpha^n\in\mathcal{A}^n\right\}\subseteq\mathbb{R}^{L+2}$. Define
\[\mathcal{P}^n:=\left\{\left(y/T,~\mathbf{z}/T\right):~(y,\mathbf{z},T)\in\mathcal{S}^n\right\}\subseteq\mathbb{R}^{L+1}\]
as the performance region of system $n$.
\end{definition}

\section{Algorithm}\label{section:algorithm}

\subsection{Proposed algorithm}
In this section, we propose an algorithm where each system can make its own decision after observing a global vector of multipliers which is updated using the global information from all systems. We start by defining a vector of virtual queues $\mathbf{Q}[t]:=\left(Q_1[t],~Q_2[t],~\cdots,~Q_L[t]\right)$, which are 0 at $t=0$ and updated as follows,
\begin{align}
Q_l[t+1]=\max\left\{Q_l[t]+\sum_{n=1}^Nz_l^n[t]-d_l[t],~0\right\},~
l\in\{1,2,\cdots,L\}.\label{queue-update}
\end{align}
These virtual queues will serve as global multipliers to control the growth of corresponding resource consumptions.
Then, the proposed algorithm is presented in Algorithm \ref{proposed-algorithm}.
\begin{algorithm}
\begin{Alg}\label{proposed-algorithm}
Fix a trade-off parameter $V>0$:
\begin{itemize}
\item At the beginning of $k$-th frame of system $n$, the system observes the vector of virtual queues $\mathbf{Q}[t^n_k]$ and makes a decision $\alpha^n_k\in\mathcal{A}^n$ so as to solve the following subproblem:
\begin{align}\label{DPP-ratio}
D^n_k:=\min_{\alpha^n\in\mathcal{A}^n}
\frac{\expect{\left.\sum_{t\in\mathcal{T}^n_k}\left(Vy^n[t]+\dotp{\mathbf{Q}[t^n_k]}{\mathbf{z}^n[t]}\right)\right|\alpha^n_k=\alpha^n,\mathbf{Q}[t^n_k]}}{\expect{\left.T^n_k\right|\alpha^n_k=\alpha^n,\mathbf{Q}[t^n_k]}}.
\end{align}
\item Update the virtual queue after each slot:
\begin{align}
Q_l[t+1]=\max\left\{Q_l[t]+\sum_{n=1}^Nz_l^n[t]-d_l[t],~0\right\},~
l\in\{1,2,\cdots,L\}. \label{eq:virtual-queue-2}
\end{align}
\end{itemize}
\end{Alg}
\end{algorithm}

Note that using the notation specified in Definition \ref{PV-def}, we can rewrite \eqref{DPP-ratio} in a more concise way as follows:
\begin{align}\label{DPP-ratio-simple}
\min_{\alpha^n\in\mathcal{A}^n} \left\{V\widehat{f}^n(\alpha^n)+\dotp{\mathbf{Q}[t^n_k]}{\widehat{\mathbf{g}}^n(\alpha^n)}\right\},
\end{align}
which is a deterministic optimization problem. Then,
by the compactness assumption (Assumption \ref{compact-assumption}), there always exists a solution to this subproblem. 
\begin{remark}
We would like to compare this algorithm to the DPP ratio algorithm (Algorithm \ref{dpp-ratio-algorithm}). For each renewal system, both algorithms update the decision variable frame-wise based on the virtual queue value at the beginning of each frame. The major difference is that the proposed algorithm updates virtual queue slot-wise while Algorithm \ref{dpp-ratio-algorithm} updates virtual queues per frame. Such a seemingly small change, somewhat surprisingly, requires significant generalizations of the analysis on Algorithm \ref{dpp-ratio-algorithm}.
\end{remark}

This algorithm requires knowledge of the conditional expectations associated with the performance vectors $\left(\widehat{f}^n(\alpha^n),~\widehat{\mathbf{g}}^n(\alpha^n)\right),~\alpha^n\in\mathcal{A}^n$, but only requires individual systems $n$ to know their own $\left(\widehat{f}^n(\alpha^n),~\widehat{\mathbf{g}}^n(\alpha^n)\right),~\alpha^n\in\mathcal{A}^n$, and therefore decouples these systems. Furthermore, the virtual queue update uses observed $d_l[t]$ and does not require knowledge of distribution or mean of $d_l[t]$.

In addition, we introduce $\mathbf{Q}[t]$ as ``virtual queues'' for the following two reasons: First, it can be mapped to real queues in applications (such as the server scheduling problem mentioned in Section \ref{sec:server-app}), where $\mathbf{d}[t]$ stands for the arrival process and $\mathbf{z}[t]$ is the service process.
Second, stabilizing these virtual queues implies the constraints \eqref{prob-2} are satisfied, as is illustrated in the following lemma.

\begin{lemma}\label{lemma:queue-bound}
If $Q_l[0]=0$ and $\lim_{T\rightarrow\infty}\frac{1}{T}{\expect{Q_l[T]}}=0$, then, $\limsup_{T\rightarrow\infty}\frac{1}{T}\sum_{t=0}^{T-1}\sum_{n=1}^N\expect{z^n_l[t]}\leq d_l$.
\end{lemma}
\begin{proof}[Proof of Lemma \ref{lemma:queue-bound}]
Fix $l\in\{1,2,\cdots,L\}$. For any fixed $T$, $Q_l[T]=\sum_{t=0}^{T-1}(Q_l[t+1]-Q_l[t])$. For each summand, by queue updating rule \eqref{queue-update},
\begin{align*}
Q_l[t+1]-Q_l[t]=&\max\left\{Q_l[t]+\sum_{n=1}^Nz^n_l[t]-d_l[t],~0\right\}-Q_l[t]\\
\geq&Q_l[t]+\sum_{n=1}^Nz^n_l[t]-d_l[t]-Q_l[t]=\sum_{n=1}^Nz^n_l[t]-d_l[t].
\end{align*}
Thus, by the assumption $Q_l[0]=0$,
$$Q_l[T]\geq\sum_{t=0}^{T-1}\left(\sum_{n=1}^Nz^n_l[t]-d_l[t]\right).$$
Taking expectations of both sides with $\expect{d_l[t]}=d_l,~\forall l$, gives
$$\expect{Q_l[T]}\geq\sum_{t=0}^{T-1}\left(\sum_{n=1}^N\expect{z^n_l[t]}-d_l\right).$$
Dividing both sides by $T$ and passing to the limit gives
\[\limsup_{T\rightarrow\infty}\frac1T\sum_{t=0}^{T-1}\left(\sum_{n=1}^N\expect{z^n_l[t]}-d_l\right)
\leq\lim_{T\rightarrow\infty}\frac{1}{T}{\expect{Q_l[T]}}=0,\]
finishing the proof.
\end{proof}


\subsection{Computing subproblems}

Since a key step in the algorithm is to solve the optimization problem \eqref{DPP-ratio-simple}, we make several comments on the computation of the ratio minimization \eqref{DPP-ratio-simple}.
In general, one can solve the ratio optimization problem \eqref{DPP-ratio} (therefore \eqref{DPP-ratio-simple}) via a bisection search algorithm. For more details, see section 7 of \cite{Neely2010}. However, more often than not, bisection search is not the most efficient one. We will discuss two special cases arising from applications where we can find a simpler way of solving the subproblem.

First of all, when there are only a finite number of actions in the set $\mathcal{A}^n$, one can solve \eqref{DPP-ratio-simple} simply via enumerating. This is a typical scenario in energy-aware scheduling where a finite action set consists of different processing modes that can be chosen by servers.

 Second, when the set $\left\{\left(\widehat{y}^n(\alpha^n),~\widehat{\mathbf{z}}^n(\alpha^n),~\widehat{T}^n(\alpha^n)\right):~\alpha^n\in\mathcal{A}^n\right\}$ specified in Definition \ref{PR-def} is itself a convex hull of a finite sequence 
 $\{(y_j,\mathbf{z}_j,T_j)\}_{j=1}^m$, then, \eqref{DPP-ratio-simple} can be rewritten as a simple enumeration:
 \[
\min_{i\in\{1,2,\cdots,m\}}~~\left\{V\frac{y_i}{T_i}+\dotp{\mathbf{Q}[t^n_k]}{\frac{\mathbf{z}_i}{T_i}}\right\}.
 \]
 To see this, note that by definition of convex hull, for any $\alpha^n\in\mathcal{A}^n$,
 $\left(\widehat{y}^n(\alpha^n),~\widehat{\mathbf{z}}^n(\alpha^n),~\widehat{T}^n(\alpha^n) \right)
 = \sum_{j=1}^mp_j\cdot(y_j,z_j,T_j)$ for some $\{p_j\}_{j=1}^m$, $p_j\geq0$ and $\sum_{j=1}^mp_j=1$. Thus, 
  \begin{align*}
V\widehat{f}^n(\alpha^n)+\dotp{\mathbf{Q}[t^n_k]}{\widehat{\mathbf{g}}^n(\alpha^n)}
 =& V\frac{\sum_{j=1}^mp_jy_j}{\sum_{j=1}^mp_jT_j}
 +\dotp{\mathbf{Q}[t^n_k]}{\frac{\sum_{j=1}^mp_j\mathbf{z}_j}{\sum_{j=1}^mp_jT_j}}\\
 =& \sum_{i=1}^m\frac{p_iT_i}{\sum_{j=1}^mp_jT_j}\left(V\frac{y_i}{T_i}+\dotp{\mathbf{Q}[t^n_k]}{\frac{\mathbf{z}_i}{T_i}}\right)\\
 =:& \sum_{i=1}^m q_i\left(V\frac{y_i}{T_i}+\dotp{\mathbf{Q}[t^n_k]}{\frac{\mathbf{z}_i}{T_i}}\right),
 \end{align*}
 where we let $q_i=\frac{p_iT_i}{\sum_{j=1}^mp_jT_j}$. 
Note that
 $q_i\geq0$ and $\sum_{i=1}^mq_i = 1$ because $T_i\geq1$. Hence, solving \eqref{DPP-ratio-simple} is equivalent to choosing $\{q_i\}_{i=1}^m$ to minimize the above expression, which boils down to choosing a single 
 $(y_i,\mathbf{z}_i,T_i)$ among $\{(y_j,\mathbf{z}_j,T_j)\}_{j=1}^m$ which achieves the minimum.

Note that such a convex hull case stands out not only because it yields a simple solution, but also because of the fact that ergodic coupled MDPs discussed in 
Section \ref{sec:MDP} have the region  $\left\{\left(\widehat{y}^n(\alpha^n),~\widehat{\mathbf{z}}^n(\alpha^n),~\widehat{T}^n(\alpha^n)\right):~\alpha^n\in\mathcal{A}^n\right\}$ 
being the convex hull of a finite sequence of points $\{(y_j,\mathbf{z}_j,T_j)\}_{j=1}^m$, where each point $(y_j,\mathbf{z}_j,T_j)$ results from a
 pure stationary policy (\cite{Al99}).
\footnote{A pure stationary policy is an algorithm where the decision to be taken at any time $t$ is a deterministic function of the state at time $t$, and independent of all other past information.} Thus, solving \eqref{DPP-ratio-simple} for the ergodic coupled MDPs reduces to choosing a pure policy among a finite number of pure policies.

\section{Limiting Performance}\label{section:limiting}
In this section, we provide the performance analysis of Algorithm \ref{proposed-algorithm}. Let  $f_*$ be the optimal objective value for problem \eqref{prob-1}-\eqref{prob-2}.
 The goal is to show the following 
 bound similar to that of Algorithm \ref{dpp-ratio-algorithm}:
 \begin{align*}
 &\frac1T\sum_{t=0}^{T-1}\sum_{n=1}^N\expect{y^n[t]}\leq f_*+\frac{C}{V},\\
 &\expect{\|Q[T]\|}\leq C'\sqrt{VT},
 \end{align*}
 for some constant $C,C'>0$. Then, by Lemma \ref{lemma:queue-bound}, one readily obtains the constraint satisfaction result.

For the rest of the chapter, the underlying probability space is denoted as the tuple $(\Omega,~\mathcal{F},~P)$.
 Let $\mathcal{F}[t]$ be the system history up until time slot $t$. Formally, $\{\mathcal{F}[t]\}_{t=0}^\infty$ is a filtration with $\mathcal{F}[0]=\{\emptyset,\Omega\}$ and each $\mathcal{F}[t],~t\geq1$ is the $\sigma$-algebra generated by all random variables from slot 0 to $t-1$.
 
 For the rest of the chapter, we always assume Assumptions \ref{feasible-assumption}-\ref{compact-assumption} hold without explicitly mentioning them.

 \subsection{Convexity, performance region and other properties}
In this section, we present several lemmas on the fundamental properties of the optimization problem \eqref{prob-1}-\eqref{prob-2}.

 The following lemma demonstrates the convexity of $\mathcal{P}^n$ in Definition \ref{PR-def}.

\begin{lemma}\label{convex-lemma}
The performance region $\mathcal{P}^n$ specified in Definition \ref{PR-def} is convex for any $n\in\{1,2,\cdots,N\}$.
Furthermore, it is the convex hull of the set
$\left\{ \left(\widehat{f}^n(\alpha^n),~\widehat{\mathbf{g}}^n(\alpha^n)\right) : \alpha^n\in\mathcal{A}^n \right\}$ and thus compact, where $\left(\widehat{f}^n(\alpha^n),~\widehat{\mathbf{g}}^n(\alpha^n)\right)$ is specified Definition \ref{PV-def}.
\end{lemma}

First of all, we have the following fundamental performance lemma which states that the optimality of \eqref{prob-1}-\eqref{prob-2} is achievable within $\mathcal{P}^n$ specified in Definition \ref{PR-def}.
\begin{lemma}\label{stationary-lemma}
For each $n\in\{1,2,\cdots,N\}$, there exists a pair 
$\left(\overline{f}^n_*,~\overline{\mathbf{g}}^n_*\right)\in\mathcal{P}^n$ such that the following hold:
\begin{align*}
&\sum_{n=1}^N\overline{f}^n_*=f_*\\
&\sum_{n=1}^N\overline{g}^n_{l,*}\leq d_l,~l\in\{1,2,\cdots,L\},
\end{align*}
where $f^*$ is the optimal objective value for problem \eqref{prob-1}-\eqref{prob-2}, i.e. the optimality is achievable within $\otimes_{n=1}^N\mathcal{P}^n$, the Cartesian product of $\mathcal{P}^n$. 

Furthermore, for any $\left(\overline{f}^n,~\overline{\mathbf{g}}^n\right)\in\mathcal{P}^n,~n\in\{1,2,\cdots, N\}$, satisfying 
$\sum_{n=1}^N\overline{g}^n_{l}\leq d_l,~l\in\{1,2,\cdots,L\}$, we have $\sum_{n=1}^N\overline{f}^n\geq f_*$, i.e. one cannot achieve better performance than \eqref{prob-1}-\eqref{prob-2} in $\otimes_{n=1}^N\mathcal{P}^n$.
\end{lemma}
The proof of this Lemma is delayed to Section \ref{appendix-proof}. In particular, the proof uses the following lemma, which also plays an important role in several lemmas later.

\begin{lemma}\label{bound-lemma-1}
Suppose $\{y^n[t]\}_{t=0}^\infty$, $\{\mathbf{z}^n[t]\}_{t=0}^\infty$ and $\{T^n_k\}_{k=0}^\infty$ are processes resulting from any algorithm,\footnote{Note that this algorithm might make decisions using the past information.} then, $\forall T\in\mathbb{N}$,
\begin{align}
&\frac1T\sum_{t=0}^{T-1}\expect{f^n[t]-y^n[t]}\leq\frac{B_1}{T},\label{bound-1}\\
&\frac1T\sum_{t=0}^{T-1}\expect{g^n_l[t]-z^n_l[t]}\leq\frac{B_2}{T},~l\in\{1,2,\cdots,L\},
\label{bound-2}
\end{align}
where $B_1=2y_{\max}\sqrt{B}$, $B_2=2z_{\max}\sqrt{B}$ and $f^n[t]$, $\mathbf{g}^n[t]$ are constant over each renewal frame for system $n$ defined by
\begin{align*}
f^n[t]=\widehat{f}^n(\alpha^n),~~\textrm{if}~t\in\mathcal{T}^n_k,\alpha^n_k=\alpha^n\\
\mathbf{g}^n[t]=\widehat{\mathbf{g}}^n(\alpha^n),~~\textrm{if}~t\in\mathcal{T}^n_k,
\alpha^n_k=\alpha^n,
\end{align*}
and $\left(\widehat{f}^n(\alpha^n),\widehat{\mathbf{g}}^n(\alpha^n)\right)$ are defined in Definition \ref{PV-def}.
\end{lemma}
The proof of this lemma is delayed to Section \ref{appendix-proof}.

\begin{remark}
Note that directly computing $\overline{f}^n_*$ and $\overline{g}^n_{l,*}$ indicated by Lemma \ref{stationary-lemma} would be difficult because of the fractional nature of $\mathcal{P}^n$, the coupling between different systems through time average constraints and the fact that $d_l=\expect{d_l[t]}$ might be unknown. However, Lemma \ref{stationary-lemma} can be used to prove important performance theorems regarding our proposed algorithm as is indicated by the following lemma.
\end{remark}

 \subsection{Main result and near optimality analysis}\label{sec-4.4}
The following theorem gives the performance bound of our proposed algorithm.

\begin{theorem}\label{thm:main}
The sequences $\{y^n[t]\}_{t=0}^\infty$ and $\{\mathbf{z}^n[t]\}_{t=0}^\infty$ produced by the proposed algorithm satisfy all the constraints in \eqref{prob-2} and achieves $\mathcal{O}(1/V)$ near optimality, i.e.
\[\limsup_{T\rightarrow\infty}\frac1T\sum_{t=0}^{T-1}\sum_{n=1}^N\expect{y^n[t]}\leq f_*+\frac{NC_1+C_3}{V},\]
where $f_*$ is the optimal objective of \eqref{prob-1}-\eqref{prob-2}, $C_1= 6Lz_{\max}(Nz_{\max}+d_{\max})B$ i and $C_3:=(Nz_{\max}+d_{\max})^2L/2$.
\end{theorem}
\begin{proof}[Proof of Theorem \ref{thm:main}]
Define the drift-plus-penalty (DPP) expression at time slot $t$ as 
\begin{equation}\label{compound-dpp}
P[t]:=\expect{\sum_{n=1}^NVy^n[t]+\frac12\left(\|\mathbf{Q}[t+1]\|^2-\|\mathbf{Q}[t]\|^2\right)}.
\end{equation}
By the queue updating rule \eqref{queue-update}, we have
\begin{align*}
P[t]\leq&\expect{\sum_{n=1}^NVy^n[t]+\frac12\sum_{l=1}^L\left(\sum_{n=1}^Nz^n_l[t]-d_l[t]\right)^2
+\sum_{l=1}^LQ_l[t]\left(\sum_{n=1}^Nz^n_l[t]-d_l[t]\right)}\\
\leq&\frac12(Nz_{\max}+d_{\max})^2L+\expect{\sum_{n=1}^NVy^n[t]
+\sum_{l=1}^LQ_l[t]\left(\sum_{n=1}^Nz^n_l[t]-d_l[t]\right)}\\
=&\frac12(Nz_{\max}+d_{\max})^2L+\expect{\sum_{n=1}^NVy^n[t]
+\sum_{l=1}^LQ_l[t]\left(\sum_{n=1}^Nz^n_l[t]-d_l\right)}
\end{align*}
where the second inequality follows from the boundedness assumption (Assumption \ref{bounded-assumption}) that $\sum_{l=1}^L\left(\sum_{n=1}^Nz^n_l[t]-d_l[t]\right)^2\leq (Nz_{\max}+d_{\max})^2L$, and the equality follows from the fact that $d_l[t]$ is i.i.d. and independent of $Q_l[t]$, thus,
\[\expect{Q_l[t]d_l[t]}=\expect{Q_l[t]\cdot\expect{d_l[t]|Q_l[t]}}=\expect{Q_l[t]d_l}.\]
For simplicity, define $C_3=\frac12(Nz_{\max}+d_{\max})^2L$. Now, by the achievability of optimality in $\otimes_{n=1}^N\mathcal{P}^n$ (Lemma \ref{stationary-lemma}), we have $\sum_{n=1}^N\overline{g}^n_{l,*}\leq d_l$, thus, substituting this inequality into the above bound for $P[t]$ gives
\begin{align*}
P[t]\leq& C_3+\expect{\sum_{n=1}^NVy^n[t]
+\sum_{n=1}^N\sum_{l=1}^LQ_l[t]\left(z^n_l[t]-\overline{g}^n_{l,*}\right)}\\
=&C_3+\sum_{n=1}^N\expect{Vy^n[t]+\dotp{\mathbf{Q}[t]}{\mathbf{z}^n[t]-\overline{\mathbf{g}}^n_*}}\\
=&C_3+\sum_{n=1}^N\expect{X^n[t]}+V\sum_{n=1}^N\overline{f}^n_*\\
=&C_3+\sum_{n=1}^N\expect{X^n[t]}+Vf_*,
\end{align*}
where we use the definition of $X^n[t]$ in \eqref{def-X} by substituting $(\overline{f}^n,\overline{\mathbf{g}}^n)$ with $(\overline{f}^n_*,\overline{\mathbf{g}}^n_*)$, i.e.
 $X^n[t]=V(y^n[t]-\overline{f}^n_*)+\dotp{\mathbf{Q}[t]}{\mathbf{z}^n[t]-\overline{\mathbf{g}}^n_*}$, in the second from last equality and use the optimality condition (Lemma \ref{stationary-lemma}) in the final equality. Thus, it follows
 \begin{align}
\frac1T\sum_{t=0}^{T-1}P[t]\leq
&C_3+Vf_*+\sum_{n=1}^N\frac1T\sum_{t=0}^{T-1}\expect{X^n[t]}.   \nonumber
\end{align}

By the virtual queue updating rule \eqref{eq:virtual-queue-2} and the trivial bound $Q_l[t]\leq \mathcal{O}(t)$, we readily get 
\[
\sum_{t=0}^{T-1}\expect{X^n[t]}=
\sum_{t=0}^{T-1}\expect{V(y^n[t] -  f^n_*) + \sum_{l=1}^LQ_l[t](z_l^n[t] - g^n_*)}\leq C(T^2 +VT),
\]
for some constant $C>0$. However, this bound is too weak to allow us proving the convergence result. 
The key to this proof is to improve such a bound so that
\begin{equation*}
\sum_{t=0}^{T-1}\expect{X^n[t]} \leq C_1T + C_2V.
\end{equation*}
where $C_1$ and $C_2$ are two constants independent of $V$ or $T$. This is Lemma \ref{sync-lemma}. As a consequence for any $T\in\mathbb{N}$,
\begin{equation}\label{inter-dpp}
\frac1T\sum_{t=0}^{T-1}P[t]\leq (NC_1+C_3) + \frac{NC_2V}{T}.
\end{equation}

On the other hand, by the definition of $P[t]$ in \eqref{compound-dpp} and then telescoping sums with $\mathbf{Q}[0]=0$, we have
\begin{align*}
\frac1T\sum_{t=0}^{T-1}P[t]
=&\frac1T\sum_{t=0}^{T-1}\expect{\sum_{n=1}^NVy^n[t]+\frac12\left(\|\mathbf{Q}[t+1]\|^2-\|\mathbf{Q}[t]\|^2\right)}\\
=&\frac1T\sum_{t=0}^{T-1}\sum_{n=1}^NV\expect{y^n[t]}
+\frac{1}{2T}\expect{\|\mathbf{Q}[T]\|^2}.
\end{align*}
Combining this with inequality \eqref{inter-dpp} gives
\begin{equation}\label{final-dpp}
\frac1T\sum_{t=0}^{T-1}\sum_{n=1}^NV\expect{y^n[t]}
+\frac{1}{2T}\expect{\|\mathbf{Q}[T]\|^2}
\leq NC_1+C_3+Vf_*+\frac{NC_2V}{T}.
\end{equation}
Since $\frac{1}{2T}\expect{\|\mathbf{Q}[T]\|^2}\geq0$, we can throw away the term and the inequality still holds, i.e. 
\begin{equation}\label{final-dpp-2}
\frac1T\sum_{t=0}^{T-1}\sum_{n=1}^N\expect{y^n[t]}
\leq f_*+\frac{NC_1+C_3}{V}+\frac{NC_2}{T}.
\end{equation}
Taking $\limsup_{T\rightarrow\infty}$ from both sides gives the near optimality in the theorem.

To get the constraint violation bound, we use Assumption \ref{bounded-assumption} that $|y^n[t]|\leq y_{\max}$, then, by \eqref{final-dpp} again, we have
\begin{align*}
\frac{1}{T}\expect{\|\mathbf{Q}[T]\|^2}
\leq 2(NC_1+C_3)+4Vy_{\max}+\frac{2NC_2V}{T}.
\end{align*}
By Jensen's inequality $\expect{\|\mathbf{Q}[T]\|^2}\geq\expect{\|\mathbf{Q}[T]\|}^2$.
This implies that 
\[
\expect{\|\mathbf{Q}[T]\|}\leq\sqrt{(2(NC_1+C_3)+4Vy_{\max})T+2NC_2V},
\]
which implies
\begin{equation}\label{key-bound}
\frac{1}{T}\expect{\|\mathbf{Q}[T]\|}\leq\sqrt{\frac{2(NC_1+C_3)+4Vy_{\max}}{T}+\frac{2NC_2V}{T^2}}.
\end{equation}
Sending $T\rightarrow\infty$ gives
\[
\lim_{T\rightarrow\infty}\frac{1}{T}{\expect{Q_l[T]}}=0,~~\forall l\in\{1,2,\cdots,L\}.
\]
Finally, by Lemma \ref{lemma:queue-bound}, all constraints are satisfied.
\end{proof}

Note that the above proof implies a more refined result that illustrates the convergence time. Fix an 
$\varepsilon>0$, let $V=1/\varepsilon$, then, for all $T\geq 1/\varepsilon$, \eqref{final-dpp-2} implies that 
\[\frac1T\sum_{t=0}^{T-1}\sum_{n=1}^N\expect{y^n[t]}
\leq f_*+\mathcal{O}(\varepsilon).\]
However, \eqref{key-bound} suggests a larger convergence time is required for constraint satisfaction!
For $V=1/\varepsilon$, it can be shown that \eqref{key-bound} implies that 
\[\frac{1}{T}\sum_{t=0}^{T-1}\sum_{n=1}^N\expect{z^n_l[t]}\leq d_l + \mathcal{O}(\varepsilon),\]
whenever $T\geq 1/\varepsilon^3$. The next section shows a tighter $1/\varepsilon^2$ convergence time with a mild Lagrange multiplier assumption. The rest of this section is devoted to proving Lemma \ref{sync-lemma}.

\subsection{Key-feature inequality and supermartingale construction}\label{sec-4.2}
In this section and the next section, our goal is to show that the term 
\begin{equation}\label{eq:target}
\sum_{t=0}^{T-1}\expect{V(y^n[t] -  f^n_*) + \sum_{l=1}^LQ_l[t](z_l^n[t] - g^n_*)}\leq C'(V+ T).
\end{equation}

Learning from the single renewal analysis (equation \eqref{eq:key-simple}), we have
the following key-feature inequality connecting our proposed algorithm with the performance vectors inside $\mathcal{P}^n$.

\begin{lemma}\label{key-feature}
Consider the stochastic processes $\{y^n[t]\}_{t=0}^\infty$, $\{\mathbf{z}^n[t]\}_{t=0}^\infty$, and 
$\{T^n_k\}_{k=0}^\infty$ resulting from the proposed algorithm. For any system $n$, the following holds for any $k\in\mathbb{N}$ and any 
$(\overline{f}^n,\overline{\mathbf{g}}^n)\in\mathcal{P}^n$,
\begin{align}\label{key-feature-in}
\frac{\expect{\left.\sum_{t\in\mathcal{T}^n_k}\left(Vy^n[t]+\dotp{\mathbf{Q}[t^n_k]}{\mathbf{z}^n[t]}\right)\right|\mathbf{Q}[t^n_k]}}{\expect{T^n_k|\mathbf{Q}[t^n_k]}}\leq V\overline{f}^n+\dotp{\mathbf{Q}[t^n_k]}{\overline{\mathbf{g}}^n},
\end{align}
\end{lemma}

\begin{proof}[Proof of Lemma \ref{key-feature}]
First of all,
since the proposed algorithm solves \eqref{DPP-ratio} over all possible decisions in $\mathcal{A}^n$, it must achieve value less than or equal to that of any action $\alpha^n\in\mathcal{A}^n$ at the same frame. This gives,
\begin{align*}
D^n_k\leq\frac{\expect{\left.\sum_{t\in\mathcal{T}^n_k}\left(Vy^n[t]+\dotp{\mathbf{Q}[t^n_k]}{\mathbf{z}^n[t]}\right)\right|\mathbf{Q}[t^n_k],\alpha^n_k=\alpha^n}}{\expect{\left.T^n_k\right|\mathbf{Q}[t^n_k],\alpha^n_k=\alpha^n}}
=\frac{V\widehat{y}^n(\alpha^n)+\dotp{\mathbf{Q}[t^n_k]}{\widehat{\mathbf{z}}^n(\alpha^n)}}{\widehat{T}^n(\alpha^n)},
\end{align*}
where $D^n_k$ is defined in \eqref{DPP-ratio} and 
the equality follows from the renewal property of the system that $T^n_k$, $\sum_{t\in\mathcal{T}^n_k}y^n[t]$ and $\sum_{t\in\mathcal{T}^n_k}\mathbf{z}^n[t]$ are conditionally independent of $\mathbf{Q}[t^n_k]$ given $\alpha^n_k=\alpha^n$.

Since $T^n_k\geq1$, this implies
\begin{align*}
\widehat{T}^n(\alpha^n)\cdot D^n_k\leq V\widehat{y}^n(\alpha^n)+\dotp{\mathbf{Q}[t^n_k]}{\widehat{\mathbf{z}}^n(\alpha^n)},
\end{align*}
thus, for any $\alpha^n\in\mathcal{A}^n$,
\[V\widehat{y}^n(\alpha^n)+\dotp{\mathbf{Q}[t^n_k]}{\widehat{\mathbf{z}}^n(\alpha^n)}
-D^n_k\cdot\widehat{T}^n(\alpha^n)\geq0.\]
Since $\mathcal{S}^n$ specified in Definition \ref{PR-def} is the convex hull of 
$\left\{(\widehat{y}^n(\alpha^n),~\widehat{\mathbf{z}}^n(\alpha^n),~\widehat{T}^n(\alpha^n)),~\alpha^n\in\mathcal{A}^n\right\}$,
it follows for any vector $(y,\mathbf{z},T)\in\mathcal{S}^n$, we have
\[Vy+\dotp{\mathbf{Q}[t^n_k]}{\mathbf{z}}
-D^n_k\cdot T\geq0.\]
Dividing both sides by $T$ and using the definition of $\mathcal{P}^n$ in Definition \ref{PR-def} give
\[D^n_k\leq V\overline{f}^n+\dotp{\mathbf{Q}[t^n_k]}{\overline{\mathbf{g}}^n},~\forall(\overline{f}^n,\overline{\mathbf{g}}^n)\in\mathcal{P}^n.\]
Finally, since $\{y^n[t]\}_{t=0}^\infty$, $\{\mathbf{z}^n[t]\}_{t=0}^\infty$, and 
$\{T^n_k\}_{k=0}^\infty$ result from the proposed algorithm and the action chosen is determined by $\mathbf{Q}[t^n_k]$ as in \eqref{DPP-ratio},
\[D^n_k=\frac{\expect{\left.\sum_{t\in\mathcal{T}^n_k}\left(Vy^n[t]+\dotp{\mathbf{Q}[t^n_k]}{\mathbf{z}^n[t]}\right)\right|\mathbf{Q}[t^n_k]}}{\expect{T^n_k|\mathbf{Q}[t^n_k]}}.\]
This finishes the proof.
\end{proof}

Our next step is to give a frame-based analysis for each system by constructing a supermartingale on the per-frame timescale. We start with a definition of supermartingale:
\begin{definition}[Supermartingale]\label{def:sup-MG}
Consider a probability space $(\Omega,\mathcal F, \mathcal P)$ and a 
filtration $\{\mathcal F_i\}_{i=0}^{\infty}$ on this space with $\mathcal F_0 = \{\emptyset,\Omega\}$, $\mathcal F_i\subseteq\mathcal F_{i+1},~\forall i$ and $\mathcal F_i\subseteq\mathcal F,~\forall i$. Consider a process $\{X_i\}_{i=0}^{\infty}\subseteq\mathbb{R}$ adapted to this filtration, i.e. $X_i\in\mathcal F_{i+1},~\forall i$. Then, we have $\{X_i\}_{i=0}^{\infty}$ is a supermartigale if $\expect{|X_i|}<\infty$ and 
$\expect{X_{i+1} | \mathcal F_{i+1}}\leq X_i$. Furthermore, $\{X_{i+1}-X_i\}_{i=0}^{\infty}$ is called a supermartingale difference sequence.
\end{definition}

Note that by definition of supermartigale, we always have$\expect{X_{i+1}-X_i|\mathcal F_{i+1}}\leq 0$.
Along the way, we also have a standard definition of stopping time which will be used later:

 \begin{definition}[Stopping time]
Given a probability space $(\Omega, \mathcal{F}, P)$ and a filtration
$\{\varnothing, \Omega\}=\mathcal{F}_0\subseteq\mathcal{F}_1\subseteq\mathcal{F}_2\cdots$
in $\mathcal{F}$. A stopping time $\tau$ with respect to the filtration $\{\mathcal{F}_i\}_{i=0}^{\infty}$ is a random variable such that for any $i\in\mathbb{N}$,
\[\{\tau=i\}\in\mathcal{F}_i,\]
i.e. the stopping time occurring at time $i$ is contained in the information during slots $0,~1,~2,~\cdots,~i-1$.
\end{definition}

Recall that 
$\{\mathcal{F}[t]\}_{t=0}^{\infty}$ is a filtration (with $\mathcal F[t]$ representing system history during slots $\{0, \cdots, t-1\}$). Fix a system $n$ and  recall that $t_k^n$ is the time slot where the $k$-th renewal occurs for system $n$.  We would like to define a filtration corresponding to the random times 
$t_k^n$. To this end, define the collection of sets $\{\mathcal F_k^n\}_{k=0}^\infty$ such that for each $k$, 
\[
\mathcal F_k^n := \{A \in\mathcal F : A \cap \{t_k^n \leq t\} \in \mathcal F[t], \forall t \in \{0, 1, 2,\cdots\}\} 
\]

For example, the following set $A$ is an element of $\mathcal F_3^n$: 
\[
A = \{t_3^n=5\} \cap \{y[0]=y_0, y[1]=y_1, y[2]=y_2, y[3]=y_3, y[4]=y_4\}
\]
where $y_0,\cdots, y_4$ are specific values. Then $A \in\mathcal F_3^n$ because 
for $i \in \{0, 1, 2, 3, 4\}$ we have $A \cap \{t_3^n \leq i\} = \emptyset \in \mathcal F[i]$, and for $i \in \{5, 6, 7, \cdots\}$ we have $A \cap \{t \leq i\} = A \in \mathcal F[i]$.  The following technical lemma is proved in Section \ref{appendix-proof}.

\begin{lemma}\label{lemma:filtration}
The sequence $\{\mathcal F_k^n \}_{k=0}^\infty$ is a valid filtration, i.e. 
$\mathcal F_k^n \subseteq\mathcal F_{k+1}^n ,~\forall k\geq0$. 
Furthermore, for any real-valued adapted process $\{Z^n[t-1]\}_{t=1}^\infty$ with respect to $\{\mathcal{F}[t]\}_{t=1}^\infty$,
\footnote{Meaning that for each $t$ in $\{1, 2,3, \cdots\}$, the random variable $Z^n[t-1]$ is determined by events in $\mathcal F[t]$.}  
$$\left\{G_{t^n_k}(Z^n[0],~Z^n[1],~\cdots,Z^n[t^n_k-1])\right\}_{k=1}^\infty$$ 
is also adapted to $\{\mathcal F_k^n \}_{k=1}^\infty$, where for any $t\in\mathbb{N}$, $G_t(\cdot)$ is a fixed real-valued measurable mappings.
That is, for any $k$, it holds that the value of any measurable 
function of $(Z^n[0], \cdots, Z[t_k^n-1])$ is determined by events in $\mathcal F_k^n$. 
\end{lemma}

With Lemma \ref{key-feature} and Lemma \ref{lemma:filtration}, we can construct a supermartingale as follows,

\begin{lemma}\label{supMG}
Consider the stochastic processes $\{y^n[t]\}_{t=0}^\infty$, $\{\mathbf{z}^n[t]\}_{t=0}^\infty$, and 
$\{T^n_k\}_{k=0}^\infty$ resulting from the proposed algorithm. For any $(\overline{f}^n,\overline{\mathbf{g}}^n)\in\mathcal{P}^n$,
let
\begin{equation} \label{def-X}
X^n[t]:=V\left(y^n[t]-\overline{f}^n\right)+\dotp{\mathbf{Q}[t]}{\mathbf{z}^n[t]-\overline{\mathbf{g}}^n},
\end{equation}
then,
\[\expect{\left.\sum_{t\in\mathcal{T}^n_k}X^n[t]\right|\mathcal F_k^n }\leq Lz_{\max}(Nz_{\max}+d_{\max})B:=C_0,\]
where $B$, $z_{\max}$ and $d_{\max}$ are as defined in Assumption \ref{bounded-assumption}. Furthermore, define a real-valued process $\{Y^n_K\}_{K=0}^\infty$ on the frame such that $Y^n_0=0$ and
\[Y^n_K=\sum_{k=0}^{K-1}\left(\sum_{t\in\mathcal{T}^n_k}X^n[t]-C_0\right),~K\geq1.\]
Then, $\{Y^n_K\}_{K=0}^\infty$ is a supermartingale adapted to the aforementioned filtration 
$\{\mathcal F_k^n \}_{K=0}^\infty$.
\end{lemma}

\begin{remark}
Note that in the above lemma the quantity $X^n[t]$ is the term we aim to bound in \eqref{eq:target}. Having $\{Y^n_K\}_{K=0}^\infty$ being a supermartingale implies $\expect{Y^n_K}\leq 0,~\forall K$. This implies 
$$
\expect{\sum_{\tau=0}^{t^n_K-1}X^n[\tau]}\leq C_0K\leq C_0t^n_K.
$$
Thus, this lemma proves \eqref{eq:target} is true when $T$ is taken to be the end of any renewal frame of system $n$.
Our goal in the next section is to get rid of this restriction and finish the proof via a stopping time argument.
\end{remark}

\begin{proof}[Proof of Lemma \ref{supMG}]
Consider any $t\in\mathcal{T}^n_k$, then, we can decompose $X^n[t]$ as follows
\begin{align}\label{eq-decompose} 
X^n[t]=&V(y^n[t]-\overline{f}^n)+\dotp{\mathbf{Q}[t^n_k]}{\mathbf{z}^n[t]-\overline{\mathbf{g}}^n}
+\dotp{\mathbf{Q}[t]-\mathbf{Q}[t^n_k]}{\mathbf{z}^n[t]-\overline{\mathbf{g}}^n}.
\end{align}
By the queue updating rule \eqref{queue-update}, we have for any $l\in\{1,2,\cdots,L\}$ and any $t>t^n_k$,
\begin{equation}\label{queue-update-bound}
|Q_l[t]-Q_l[t^n_k]|\leq\sum_{s=t^n_k}^{t-1}\left|\sum_{m=1}^Nz^m_l[s]-d_l[t]\right|
\leq(t-t^n_k)(Nz_{\max}+d_{\max})
\end{equation}
Thus, for the last term in \eqref{eq-decompose}, by H\"{o}lder's inequality, 
\begin{align*}
\dotp{\mathbf{Q}[t]-\mathbf{Q}[t^n_k]}{\mathbf{z}^n[t]-\overline{\mathbf{g}}^n}
\leq&\|\mathbf{Q}[t]-\mathbf{Q}[t^n_k]\|_1\cdot\|\mathbf{z}^n[t]-\overline{\mathbf{g}}^n\|_{\infty}\\
\leq&\sum_{s=t^n_k}^{t-1}\left\|\sum_{m=1}^N\mathbf{z}^n[s]-\mathbf{d}[t]\right\|_1\cdot\|\mathbf{z}^n[t]-\overline{\mathbf{g}}^n\|_{\infty}\\
\leq&(t-t^n_k)L(Nz_{\max}+d_{\max})\cdot2z_{\max},
\end{align*}
where the second inequality follows from  \eqref{queue-update-bound} and the last inequality follows from the boundedness assumption (Assumption \ref{bounded-assumption}) of corresponding quantities. Substituting the above bound into \eqref{eq-decompose} gives a bound on $\expect{\left.\sum_{t\in\mathcal{T}^n_k}X^n[t]\right|\mathcal F_k^n }$ as
\begin{align}
\expect{\left.\sum_{t\in\mathcal{T}^n_k}X^n[t]\right|\mathcal F_k^n }
\leq&\expect{\left.\sum_{t\in\mathcal{T}^n_k}\left(V\left(y^n[t]-\overline{f}^n\right)+\dotp{\mathbf{Q}[t^n_k]}{\mathbf{z}^n[t]-\overline{\mathbf{g}}^n}\right)\right|\mathcal F_k^n }\nonumber\\
&+\expect{\left.\sum_{t\in\mathcal{T}^n_k}(t-t^n_k)\right|\mathcal F_k^n }
\cdot2L(Nz_{\max}+d_{\max})z_{\max}\nonumber\\
\leq&\expect{\left.\sum_{t\in\mathcal{T}^n_k}\left(V\left(y^n[t]-\overline{f}^n\right)+\dotp{\mathbf{Q}[t^n_k]}{\mathbf{z}^n[t]-\overline{\mathbf{g}}^n}\right)\right|\mathcal F_k^n }\nonumber\\
&+\expect{\left.(T^n_k)^2\right|\mathcal F_k^n }
\cdot L(Nz_{\max}+d_{\max})z_{\max},\label{bound-on-X}
\end{align}
where we use the fact that $0+1+\cdots+T^n_k-1 = (T^n_k-1)T^n_k/2\leq (T^n_k)^2$ in the last inequality.

Next, by the queue updating rule \eqref{queue-update}, $Q_l[t^n_k]$ is determined by $z_l^n[0],\cdots,z_l^n[t^n_k-1]$
($n=1,2,\cdots,N$) and $d_l[0],\cdots,d_l[t^n_k-1]$ for any $l\in\{1,2,\cdots,L\}$. Thus, by Lemma \ref{lemma:filtration},
$\mathbf{Q}[t^n_k]$ is determined by $\mathcal F_k^n $. 
For the proposed algorithm, each system makes decisions purely based on the virtual queue state $\mathbf{Q}[t^n_k]$, and
by the renewal property of each system, given the decision at the $k$-th renewal, the random quantities $T^n_k$, $\mathbf{z}^n[t]$ and $y^n[t]$,~$t\in\mathcal{T}^n_k$ are independent of the outcomes from the slots before $t^n_k$. 
This implies the following display,
\begin{align}
&\expect{\left.\sum_{t\in\mathcal{T}^n_k}\left(V\left(y^n[t]-\overline{f}^n\right)+\dotp{\mathbf{Q}[t^n_k]}{\mathbf{z}^n[t]-\overline{\mathbf{g}}^n}\right)\right|\mathcal F_k^n }\nonumber\\
&=\expect{\left.\sum_{t\in\mathcal{T}^n_k}V\left(y^n[t]-\overline{f}^n\right)\right|\mathcal F_k^n }+\dotp{\mathbf{Q}[t^n_k]}{\expect{\left.\sum_{t\in\mathcal{T}^n_k}\left(\mathbf{z}^n[t]-\overline{\mathbf{g}}^n\right)\right|~\mathcal F_k^n }}\nonumber\\
&=\expect{\left.\sum_{t\in\mathcal{T}^n_k}V\left(y^n[t]-\overline{f}^n\right)\right|\mathbf{Q}[t^n_k]}+\dotp{\mathbf{Q}[t^n_k]}{\expect{\left.\sum_{t\in\mathcal{T}^n_k}\left(\mathbf{z}^n[t]-\overline{\mathbf{g}}^n\right)\right|~\mathbf{Q}[t^n_k]}}\nonumber\\
&=\expect{\left.\sum_{t\in\mathcal{T}^n_k}\left(V\left(y^n[t]-\overline{f}^n\right)+\dotp{\mathbf{Q}[t^n_k]}{\mathbf{z}^n[t]-\overline{\mathbf{g}}^n}\right)\right|\mathbf{Q}[t^n_k]},\label{mark-1}
\end{align}
By Lemma \ref{key-feature}, we have the following:
\begin{align*}
\expect{\left.\sum_{t\in\mathcal{T}^n_k}\left(Vy^n[t]+\dotp{\mathbf{Q}[t^n_k]}{\mathbf{z}^n[t]}\right)\right|\mathbf{Q}[t^n_k]}
\leq \left(V\overline{f}^n+\dotp{\mathbf{Q}[t^n_k]}{\overline{\mathbf{g}}^n}\right)\cdot\expect{T^n_k|\mathbf{Q}[t^n_k]}.
\end{align*}
Thus, rearranging terms in above inequality gives
the expectation on the right hand side of \eqref{mark-1} is no greater than 0 and hence the first expectation on the right hand side of \eqref{bound-on-X} is also no greater than 0. For the second expectation in \eqref{bound-on-X}, using \eqref{residual-life-bound} in Assumption \ref{bounded-assumption} gives $\expect{\left.(T^n_k)^2\right|\mathcal F_k^n }\leq B$ and the first part of the lemma is proved. 

For the second part of the lemma, by Lemma \ref{lemma:filtration} and the definition of $Y^n_K$, the process $\{Y^n_K\}_{K=0}^\infty$
is adapted to $\{\mathcal F_k^n \}_{K=0}^{\infty}$.
Moreover, by Assumption \ref{bounded-assumption},
\begin{align*}
\expect{\left|\sum_{t\in\mathcal{T}^n_k}X^n[t]\right|}
\leq\expect{\sum_{t\in\mathcal{T}^n_k}\left|X^n[t]\right|}<\infty,~\forall k.
\end{align*}
Thus, $\expect{|Y^n_K|}<\infty,~\forall K\in\mathbb{N}$, i.e. it is absolutely integrable. Furthermore, by the first part of the lemma,
\begin{align*}
\expect{Y^n_{K+1}~|~\mathcal F_k^n }
=Y^n_K+\expect{\left.\left(\sum_{t\in\mathcal{T}^n_K}X^n[t]-C_0\right)~\right|~\mathcal F_k^n }
\leq Y^n_K,
\end{align*}
finishing the proof.
\end{proof}

\subsection{Synchronization lemma}\label{section:sync}
So far, we have analyzed the processes related to each individual system over its renewal frames. However, due the asynchronous behavior of different systems, the supermartingales of each system cannot be immediately summed.

In order to prove the result \eqref{eq:target} and get a global performance bound, we have to get rid of any index related to individual renewal frames only. In other words, 
we need to look at the system property at any time slot $T$ as opposed to any renewal $t^n_k$.

For any fixed slot $T>0$, let $S^n[T]$ be the number of renewals up to (and including) time slot $T$, with the convention that the first renewal occurs at time $t=0$, so $t_0^n=0$ and $S^n[0]=1$, i.e. $t^n_0=0$.
The next lemma shows $S^n[T]$ is a valid stopping time, whose proof is in the appendix.

\begin{lemma}\label{valid-stopping-time}
For each $n\in\{1,2,\cdots,N\}$, 
the random variable $S^n[T]$ is a stopping time
with respect to the filtration $\{\mathcal F_k^n \}_{k=0}^\infty$, i.e. $\{S^n[T]= k\}\in\mathcal F_k^n ,~\forall k\in\mathbb{N}$.
\end{lemma}

The following theorem tells us a stopping-time truncated supermartingale is still a supermartingale.
\begin{theorem}[Theorem 5.2.6 in \cite{Durrett}]\label{stopping-time}
If $\tau$ is a stopping time and $Z[i]$ is a supermartingale with respect to $\{\mathcal{F}_i\}_{i=0}^\infty$, then $Z[i\wedge \tau]$ is also a supermartingale, where $a\wedge b\triangleq\min\{a,b\}$.
\end{theorem}

With this theorem and the above stopping time construction, we have the following lemma which finishes the argument proving \eqref{eq:target}:

\begin{lemma}\label{sync-lemma}
For each $n\in\{1,2,\cdots,N\}$ and any fixed $T\in\mathbb{N}$, we have
\begin{align*}
\frac1T\sum_{t=0}^{T-1}\expect{X^n[t]}\leq C_1+\frac{C_2V}{T},
\end{align*}
where $X^n[t]$ is defined in \eqref{eq-decompose} and 
\[C_1:=6Lz_{\max}(Nz_{\max}+d_{\max})B,~~C_2:=2y_{\max}\sqrt{B}.\]
\end{lemma}
\begin{proof}
First, note that the renewal index $k$ starts from 0. Thus,
for any fixed $T\in \mathbb{N}$, $ t^n_{S^n[T]-1}\leq T<t^n_{S^n[T]}$, and
\begin{align}
\expect{\sum_{t=0}^{T-1}X^n[t]}=&\expect{\sum_{t=0}^{t^n_{S^n[T]}-1}X^n[t]-\sum_{t=T}^{t^n_{S^n[T]}-1}X^n[t]}\nonumber\\
=&\expect{\sum_{t=1}^{t^n_{S^n[T]}-1}X^n[t]}-\expect{\sum_{t=T}^{t^n_{S^n[T]}-1}X^n[t]}\nonumber\\
=&\expect{Y^n_{S^n[T]}}+C_0\expect{S^n[T]}-\expect{\sum_{t=T}^{t^n_{S^n[T]}-1}X^n[t]}\nonumber\\
\leq&\expect{Y^n_{S^n[T]}}+C_0(T+1)-\expect{\sum_{t=T}^{t^n_{S^n[T]}-1}X^n[t]},\label{decompose-inequality}
\end{align}
where the third equality follows from the definition of $Y^n_K$ in Lemma \ref{supMG} and the last inequality follows from the fact that the number of renewals up to time slot $T$ is no more than the total number of slots, i.e.
$S^n[T]\leq T+1$. For the term $\expect{Y^n_{S^n[T]}}$, we apply Theorem \ref{stopping-time} with 
$\tau = S^n[T]$ and index $K$ to obtain
$\{Y^n_{K\wedge S^n[T]}\}_{K=0}^\infty$ is a supermartingale. This implies
\[\expect{Y^n_{K\wedge S^n[T]}}\leq\expect{Y^n_{0\wedge S^n[T]}}=\expect{Y^n_0}=0,~\forall K\in\mathbb{N}.\]
Since $S^n[T]\leq T+1$, it follows by substituting $K=T+1$,
\[\expect{Y^n_{S^n[T]}}=\expect{Y^n_{(T+1)\wedge S^n[T]}}\leq0.\]
For the last term in \eqref{decompose-inequality}, by queue updating rule \eqref{queue-update}, for any $l\in\{1,2,\cdots,L\}$,
\[|Q_l[t]|\leq\sum_{s=0}^{t-1}\left|\sum_{m=1}^Nz^m_l[s]-d_l[t]\right|
\leq t(Nz_{\max}+d_{\max}),\]
it then follows from H\"{o}lder's inequality again that
\begin{align*}
\expect{\left|\sum_{t=T}^{t^n_{S^n[T]}-1}X^n[t]\right|}
=&\expect{\left|\sum_{t=T}^{t^n_{S^n[T]}-1}\left(V(y^n[t]-\overline{f}^n)+\dotp{\mathbf{Q}[t]}{\mathbf{z}^n[t]-\overline{\mathbf{g}}^n}\right)\right|}\\
\leq&\expect{\sum_{t=T}^{t^n_{S^n[T]}-1}\left(V\left|y^n[t]-\overline{f}^n\right|+\|\mathbf{Q}[t]\|_1\cdot\|\mathbf{z}^n[t]-\overline{\mathbf{g}}^n\|_{\infty}\right)}\\
\leq&\expect{\sum_{t=T}^{t^n_{S^n[T]}-1}\left(2Vy_{\max}+L(Nz_{\max}+d_{\max})t\cdot2z_{\max}\right)}\\
=&2Vy_{\max}\cdot\expect{t^n_{S^n[T]}-T}+Lz_{\max}(Nz_{\max}+d_{\max})\\
&\cdot\left((2T-1)\cdot\expect{t^n_{S^n[T]}-T}+\expect{t^n_{S^n[T]}-T}^2\right)\\
\leq& 2Vy_{\max}\sqrt{B}+2Lz_{\max}(Nz_{\max}+d_{\max})\sqrt{B}T+Lz_{\max}(Nz_{\max}+d_{\max})B\\
\leq& 2Vy_{\max}\sqrt{B}+2Lz_{\max}(Nz_{\max}+d_{\max})B(T+1),
\end{align*} 
where in the second from last inequality we use \eqref{residual-life-bound} of Assumption \ref{bounded-assumption} that the residual life $t^n_{S^n[T]}-T$ satisfies 
$$\expect{(t^n_{S^n[T]}-T)^2}
=\expect{\expect{\left.(t^n_{S^n[T]}-T)^2\right|~t^n_{S^n[T]}-t^n_{S^n[T]-1}\geq T-t^n_{S^n[T]-1}}}\leq B$$ 
and $\expect{t^n_{S^n[T]}-T}\leq\sqrt{B}$, and in the last inequality we use the fact that $B\geq1$, thus, $\sqrt{B}\leq B$.
Substitute the above bound into \eqref{decompose-inequality} gives
 \begin{align*}
 \expect{\sum_{t=0}^{T-1}X^n[t]}\leq& C_0(T+1)+2Vy_{\max}B+2Lz_{\max}(Nz_{\max}+d_{\max})B(T+1)\\
 =&2Vy_{\max}\sqrt{B}+3Lz_{\max}(Nz_{\max}+d_{\max})B(T+1)\\
 \leq&2Vy_{\max}\sqrt{B}+6Lz_{\max}(z_{\max}+d_{\max})BT
 \end{align*}
 where we use the definition $C_0=Lz_{\max}(z_{\max}+d_{\max})B$ from Lemma \ref{supMG} in the equality and use $T+1\leq2T$ in the final equality.
 Dividing both sides by $T$ finishes the proof.
\end{proof}

\section{Convergence Time Analysis}\label{sec-convergence-time}
\subsection{Lagrange Multipliers}
Consider the following optimization problem:
\begin{align}
\min&~~\sum_{n=1}^N\overline{f}^n\label{modi-prob-1}\\
s.t.&~~\sum_{n=1}^N\overline{g}^n_l\leq d_l,~\forall l\in\{1,2,\cdots, L\}, \\
&~~(\overline{f}^n,\overline{\mathbf{g}}^n)\in\mathcal{P}^n,~\forall n \in\{1,2,\cdots,N\}.
\label{modi-prob-3}
\end{align}
Since $\mathcal{P}^n$ is convex, it follows $\mathcal{P}^n$ is convex and $\otimes_{n=1}^N\mathcal{P}^n$ is also convex. Thus, \eqref{modi-prob-1}-\eqref{modi-prob-3} is a convex program. Furthermore, by Lemma \ref{stationary-lemma}, we have \eqref{modi-prob-1}-\eqref{modi-prob-3} is feasible if and only if \eqref{prob-1}-\eqref{prob-2} is feasible, and when assuming feasibility, they have the same optimality $f_*$ as is specified in Lemma \ref{stationary-lemma}.

Since $\mathcal{P}^n$ is convex, one can show (see Proposition 5.1.1 of \cite{Be09}) that there \textit{always} exists a sequence $(\gamma_0, \gamma_1,\cdots,\gamma_L)$ so that $\gamma_i\geq0,~i=0,1,\cdots,L$ and
\[
\sum_{n=1}^N\gamma_0\overline{f}^n+\sum_{l=1}^L\gamma_l\sum_{n=1}^N\overline{g}^n_l
\geq \gamma_0f_*+\sum_{l=1}^L\gamma_ld_l,
~\forall (\overline{f}^n,\overline{\mathbf{g}}^n)\in\mathcal{P}^n,
\]
i.e. there always exists  a hyperplane parametrized by  $(\gamma_0, \gamma_1,\cdots,\gamma_L)$, supported at $(f_*,d_1,\cdots,d_L)$ and containing the set $\left\{\left(\sum_{n=1}^N\overline{f}^n,~\sum_{n=1}^N\overline{\mathbf{g}}^n\right):~(\overline{f}^n,\overline{\mathbf{g}}^n)\in\mathcal{P}^n,~\forall n\in\{1,2,\cdots,N\}\right\}$ on one side. This hyperplane is called ``separating hyperplane''.
The following assumption stems from this property and simply assumes this separating hyperplane to be non-vertical (i.e. $\gamma_0>0$):
\begin{assumption}\label{sep-hype}
There exists non-negative finite constants $\gamma_1,~\gamma_2,~\cdots,~\gamma_L$ such that the following holds,
\begin{align*}
\sum_{n=1}^N\overline{f}^n+\sum_{l=1}^L\gamma_l\sum_{n=1}^N\overline{g}^n_l
\geq f_*+\sum_{l=1}^L\gamma_ld_l,
~\forall (\overline{f}^n,\overline{\mathbf{g}}^n)\in\mathcal{P}^n,
\end{align*}
i.e. there exists a separating hyperplane parametrized by $(1,\gamma_1,\cdots,\gamma_L)$.
\end{assumption} 

\begin{remark}
The parameters $\gamma_1,~\cdots,~\gamma_L$ are called Lagrange multipliers and this assumption is equivalent to the existence of Lagrange multipliers for 
constrained convex program \eqref{modi-prob-1}-\eqref{modi-prob-3}. 
It is known that Lagrange multipliers exist if the Slater's condition holds (\cite{Be09}), which states that there exists a nonempty interior of the feasible region for the convex program. Slater's condition is very common in convex optimization theory and plays an important role in convergence rate analysis, such as the analysis of the interior point algorithm (\cite{BV04}). In the current context, this condition is satisfied, for example, in energy aware server scheduling problems, if the highest possible sum of service rates from all servers is strictly higher than the arrival rate.
\end{remark}

\begin{lemma}\label{bound-lemma-2}
Suppose $\{y^n[t]\}_{t=0}^\infty$, $\{\mathbf{z}^n[t]\}_{t=0}^\infty$ and $\{T^n_k\}_{k=0}^\infty$ are processes resulting from the proposed algorithm.
Under the Assumption \ref{sep-hype}, 
\begin{align*}
\frac1T\sum_{t=0}^{T-1}\left(f_*-\sum_{n=1}^N\expect{y^n[t]}\right)
\leq\frac1T\sum_{t=0}^{T-1}\sum_{l=1}^L\gamma_l\left(\sum_{n=1}^N\expect{z^n_l[t]}-d_l\right)
+\frac{C_4}{T},
\end{align*}
where $C_4=B_1N+B_2N\sum_{l=1}^L\gamma_l$, and $B_1$, $B_2$ are defined in Lemma \ref{bound-lemma-1}.
\end{lemma}

\begin{proof}
First of all, from the statement of Lemma \ref{bound-lemma-1}, for the proposed algorithm, we can define the corresponding processes $(f^n[t],\mathbf{g}^n[t])$ for all $n$ as
\begin{align*}
f^n[t] =& \widehat{f}^n(\alpha^n) = \widehat{y}^n(\alpha^n)/\widehat{T}^n(\alpha^n),
~~\textrm{if}~t\in\mathcal{T}^n_k,\alpha^n_k=\alpha^n\\
\mathbf{g}^n[t] =& \widehat{\mathbf{g}}^n(\alpha^n)= \widehat{\mathbf z}^n(\alpha^n)/\widehat{T}^n(\alpha^n),~~\textrm{if}~t\in\mathcal{T}^n_k,\alpha^n_k=\alpha^n,
\end{align*}
where the last equality follows from the definition of $\widehat{f}^n(\alpha^n)$ and $\widehat{\mathbf{g}}^n(\alpha^n)$ in Definition \ref{PV-def}.
Since
$\left(\widehat{y}^n(\alpha^n),~\widehat{\mathbf z}^n(\alpha^n),~\widehat{T}^n(\alpha^n)\right)
\in\mathcal{S}^n$, by definition of $\mathcal{P}^n$ in Definition \ref{PR-def},
$(f^n[t],\mathbf{g}^n[t])\in\mathcal P^n\subseteq\mathcal{P}^n,~\forall n,~\forall t$.
Since $\mathcal{P}^n$ is a convex set by Lemma \ref{convex-lemma}, it follows
\begin{align*}
\left(\expect{f^n[t]},~\expect{\mathbf{g}^n[t]}\right)\in\mathcal{P}^n,~~\forall t,~\forall n.
\end{align*}
By Assumption \ref{sep-hype}, we have
\begin{align*}
\sum_{n=1}^N\expect{f^n[t]}+\sum_{l=1}^L\gamma_l\sum_{n=1}^N\expect{g^n_l[t]}
\geq f_*+\sum_{l=1}^L\gamma_ld_l,~~\forall t.
\end{align*}
Rearranging terms gives
\begin{align*}
f_*-\sum_{n=1}^N\expect{f^n[t]}\leq \sum_{l=1}^L\gamma_l\left(\sum_{n=1}^N\expect{g^n_l[t]}-d_l\right),~~\forall t.
\end{align*}
Taking the time average from 0 to $T-1$ gives
\begin{align}
\frac1T\sum_{t=0}^{T-1}\left(f_*-\sum_{n=1}^N\expect{f^n[t]}\right)
 \leq \frac1T\sum_{t=0}^{T-1}\sum_{l=1}^L\gamma_l\left(\sum_{n=1}^N\expect{g^n_l[t]}-d_l\right). \label{inter-ave-bound-1}
\end{align}
For the left hand side of \eqref{inter-ave-bound-1}, we have
\begin{align}
l.h.s.&=\frac1T\sum_{t=0}^{T-1}\left(f_*-\sum_{n=1}^N\expect{y^n[t]}\right)
+\frac1T\sum_{t=0}^{T-1}\sum_{n=1}^N\expect{y^n[t]-f^n[t]}\nonumber\\
&\geq\frac1T\sum_{t=0}^{T-1}\left(f_*-\sum_{n=1}^N\expect{y^n[t]}\right)-\frac{B_1N}{T}.\label{inter-ave-bound-2}
\end{align}
where the inequality follows from \eqref{bound-1} in Lemma \ref{bound-lemma-1}.
For the right hand side of \eqref{inter-ave-bound-1}, we have
\begin{align}
r.h.s.&= \frac1T\sum_{t=0}^{T-1}\sum_{l=1}^L\gamma_l\left(\sum_{n=1}^N\expect{z^n_l[t]}-d_l\right)
+\frac1T\sum_{t=0}^{T-1}\sum_{l=1}^L\gamma_l\sum_{n=1}^N\expect{g^n_l[t]-z^n_l[t]}\nonumber\\
&\leq\frac1T\sum_{t=0}^{T-1}\sum_{l=1}^L\gamma_l\left(\sum_{n=1}^N\expect{z^n_l[t]}-d_l\right)+\frac{B_2N\sum_{l=1}^L\gamma_l}{T}, \label{inter-ave-bound-3}
\end{align}
where the inequality follows from the fact that $\gamma_l\geq0,~\forall l$ and \eqref{bound-2} in Lemma \ref{bound-lemma-1}. Substituting \eqref{inter-ave-bound-2} and \eqref{inter-ave-bound-3} into \eqref{inter-ave-bound-1} finishes the proof.
\end{proof}

\subsection{Convergence time theorem}
\begin{theorem}
Fix $\varepsilon\in(0,1)$ and define $V=1/\varepsilon$. If the problem \eqref{prob-1}-\eqref{prob-2} is feasible and the Assumption \ref{sep-hype} holds, then, for all $T\geq1/\varepsilon^2$, 
\begin{align}
&\frac1T\sum_{t=0}^{T-1}\sum_{n=1}^N\expect{y^n[t]}\leq f_*+\mathcal{O}(\varepsilon),\label{ctime-1}\\
&\frac1T\sum_{t=0}^{T-1}\sum_{n=1}^N\expect{z^n_l[t]}\leq d_l+\mathcal{O}(\varepsilon),
l\in\{1,2,\cdots,L\}.\label{ctime-2}
\end{align}
Thus, the algorithm provides $\mathcal{O}(\varepsilon)$ approximation with the convergence time $\mathcal{O}(1/\varepsilon^2)$.
\end{theorem}
\begin{proof}
First of all, by queue updating rule \eqref{queue-update}, 
\begin{equation}\label{queue-bound}
\sum_{t=0}^{T-1}\left(\sum_{n=1}^N\expect{z^n_l[t]}-d_l\right)\leq\expect{Q_l[T]}.
\end{equation}
By Lemma \ref{bound-lemma-2}, we have
\begin{align}
\frac1T\sum_{t=0}^{T-1}\left(f_*-\sum_{n=1}^N\expect{y^n[t]}\right)
\leq&\frac1T\sum_{t=0}^{T-1}\sum_{l=1}^L\gamma_l\left(\sum_{n=1}^N\expect{z^n_l[t]}-d_l\right)
+\frac{C_4}{T},\nonumber\\
\leq&\sum_{l=1}^L\frac{\gamma_l}{T}\expect{Q_l[T]}+\frac{C_4}{T}.\label{inter-ctime-1}
\end{align}
Combining this with \eqref{final-dpp} gives
\begin{align}
\frac{1}{2T}\expect{\|\mathbf{Q}[T]\|^2}
&\leq NC_1+C_3+\frac{V}{T}\sum_{t=0}^{T-1}\left(f_*-\sum_{n=1}^N\expect{y^n[t]}\right)+\frac{NC_2V}{T}\nonumber\\
&\leq NC_1+C_3 +\frac{(NC_2+C_4)V}{T}+V\sum_{l=1}^L\frac{\gamma_l}{T}\expect{Q_l[T]}
\nonumber\\
&\leq NC_1+C_3 +\frac{(NC_2+C_4)V}{T}+\frac{V}{T}\|\gamma\|\cdot\|\expect{\mathbf{Q}[T]}\|,
\label{inter-ctime-2}
\end{align}
where $\gamma:=(\gamma_1,~\cdots,~\gamma_L)$,
the second inequality follows from \eqref{inter-ctime-1} and the final inequality follows from Cauchy-Schwarz. Then, by Jensen's inequality, we have
\[\|\expect{\mathbf{Q}[T]}\|^2\leq\expect{\|\mathbf{Q}[T]\|^2}.\]
Thus, it follows by \eqref{inter-ctime-2} that
\begin{align*}
\|\expect{\mathbf{Q}[T]}\|^2- 2V\|\gamma\|\cdot\|\expect{\mathbf{Q}[T]}\| - 2(NC_1+C_3)T
-2(NC_2+C_4)V\leq 0.
\end{align*}
The left hand side is a quadratic form on $\|\expect{\mathbf{Q}[T]}\|$, and the inequality implies that 
$\|\expect{\mathbf{Q}[T]}\|$ is deterministically upper bounded by the largest root of the equation 
$x^2-bx-c=0$ with $b=2V\|\gamma\|$ and $c=2(NC_1+C_3)T+2(NC_2+C_4)V$. Thus,
\begin{align*}
\|\expect{\mathbf{Q}[T]}\|\leq&\frac{b+\sqrt{b^2+4c}}{2}\\
=& V\|\gamma\|+\sqrt{V^2\|\gamma\|^2+2(NC_1+C_3)T+2(NC_2+C_4)V}\\
\leq& 2V\|\gamma\|+ \sqrt{2(NC_1+C_3)T} + \sqrt{2(NC_2+C_4)V}.
\end{align*}
Thus, for any $l\in\{1,2,\cdots,L\}$,
\begin{align*}
\frac1T\expect{Q_l[T]}\leq \frac{2V\|\gamma\|}{T} 
 + \sqrt{\frac{2(NC_1+C_3)}{T}} + \frac{\sqrt{2(NC_2+C_4)V}}{T}.
\end{align*}
By \eqref{queue-bound} again,
\begin{align*}
\frac1T\sum_{t=0}^{T-1}\sum_{n=1}^N\expect{z^n_l[t]}\leq d_l+\frac{2V\|\gamma\|}{T} 
 + \sqrt{\frac{2(NC_1+C_3)}{T}} + \frac{\sqrt{2(NC_2+C_4)V}}{T}.
\end{align*}
Substituting $V=1/\varepsilon$ and $T\geq1/\varepsilon^2$ into the above inequality gives 
$\forall l\in\{1,2,\cdots,L\}$,
\begin{align*}
\frac1T\sum_{t=0}^{T-1}\sum_{n=1}^N\expect{z^n_l[t]}\leq&
d_l + \left(2\|\gamma\|+\sqrt{2(NC_1+C_3)}\right)\varepsilon + \sqrt{2(NC_2+C4)}\varepsilon^{3/2}\\
=&d_l+\mathcal{O}(\varepsilon).
\end{align*}
Finally, substituting $V=1/\varepsilon$ and $T\geq1/\varepsilon^2$ into \eqref{final-dpp-2} gives
\[\frac1T\sum_{t=0}^{T-1}\sum_{n=1}^N\expect{y^n[t]}\leq f_*+\mathcal{O}(\varepsilon),\]
finishing the proof.
\end{proof}

\section{Simulation Study in Energy-aware Scheduling}\label{section-application}
 Here, we apply the algorithm introduced in Section \ref{section:algorithm} to deal with the energy-aware scheduling problem described in Section \ref{sec:application}. To be specific, we consider a scenario with 5 homogeneous servers and 3 different classes of jobs, i.e. $N=5$ and $L=3$.
We assume that each server can only choose one class of jobs to serve during each frame. So the mode set $\mathcal{M}^n$ contains three actions $\{1,2,3\}$ and the action $i$ stands for serving the $i$-th class of jobs and we count the number of serviced jobs at the end of each service duration. The action $m^n_k$ determines the following quantities:
 \begin{itemize}
\item The uniformly distributed total number of class $l$ jobs that can be served with expectation
$\expect{\left.\sum_{t\in\mathcal{T}^n_k}\mu^n_l[t]\right|~m^n_k}:=\widehat{\mu}^n_l(m^n_k)$.
\item The geometrically distributed service duration $H^n_k$ slots with expectation $\expect{\left.H^n_k\right|~m^n_k}:=\widehat{H}^n(m^n_k)$.
\item The energy consumption $\widehat{e}^n(m^n_k)$ for serving all these jobs.
\item The geometrically distributed idle/setup time $I^n_k$ slots with constant energy consumption $p^n$ per slot and zero job service. The expectation $\expect{\left.I^n_k\right|~m^n_k}:=\widehat{I}^n(m^n_k)$.
\end{itemize}
The idle/setup cost is $p^n=3$ units per slot and the rest of the parameters are listed in Table 1.

Following the algorithm description in Section \ref{section:algorithm}, the proposed algorithm has the queue updating rule 
\[Q_l[t+1]=\max\left\{Q_l[t]+\lambda_l[t]-\sum_{n=1}^N\mu^n_l[t],~0\right\},\]
and each system minimizes
\eqref{DPP-ratio} each frame, which can be written as 
\begin{align*}
\min_{m^n_k\in\mathcal{M}^n}\frac{V\left(\widehat{e}^n_l(m^n_k)+p^n\widehat{I}^n(m^n_k)\right)
-\dotp{\mathbf{Q}[t^n_k]}{\widehat{\mu}^n(m^n_k)}}{\widehat{H}^n(m^n_k)+\widehat{I}^n(m^n_k)}.
\end{align*}

\begin{table}
\begin{center}
\caption{Problem parameters}
\begin{tabular}{c|c|c|c|c|c}
  \hline
   & $\lambda_i$ & $\widehat{H}^n(i)$ & $\widehat{\mu}^n(i)$ & $\widehat{e}^n(i)$ & $\widehat{I}^n(i)$\\
   \hline
Class 1   & 2 & 5.5 & 15 (Uniform $[9,21]\cap\mathbb{N}$) & 16 & 2.5 \\
  \hline
 Class 2  & 3 & 4.6 & 21 (Uniform $[15,27]\cap\mathbb{N}$) & 20 & 4.3\\
 \hline
 Class 3  & 4 & 3.8 & 17 (Uniform $[11,23]\cap\mathbb{N}$) & 13 & 3.7\\
 \hline
\end{tabular}
\end{center}
\end{table}

Each plot for the proposed algorithm is the result of running 1 million slots and taking the time average as the performance of the proposed algorithm. The benchmark is the optimal stationary performance obtained by performing a change of variable and solving a linear program, knowing the arrival rates (see also \cite{Neely12} for details). 

Fig. \ref{fig:Stupendous2} shows as the trade-off parameter $V$ gets larger, the time average energy consumptions under the proposed algorithm approaches the optimal energy consumption. Fig. \ref{fig:Stupendous3} shows as $V$ gets large, the time average number of services also approaches the optimal service rate for each class of jobs. In Fig. \ref{fig:Stupendous4}, we plot the time average queue backlog for each class of jobs verses $V$ parameter. We see that the queue backlog for the first class is always low whereas the rest queue backlogs scale up linearly with $V$. This is because the service rate for the first class is always strictly larger than the arrival rate whereas for the rest classes, as $V$ gets larger, the service rates approach the arrival rates. This plot, together with Fig. \ref{fig:Stupendous2}, also demonstrate that $V$ 
is indeed a trade-off parameter which trades queue backlog for near optimality. 
\begin{figure}[htbp]
   \centering
   \includegraphics[height=3in]{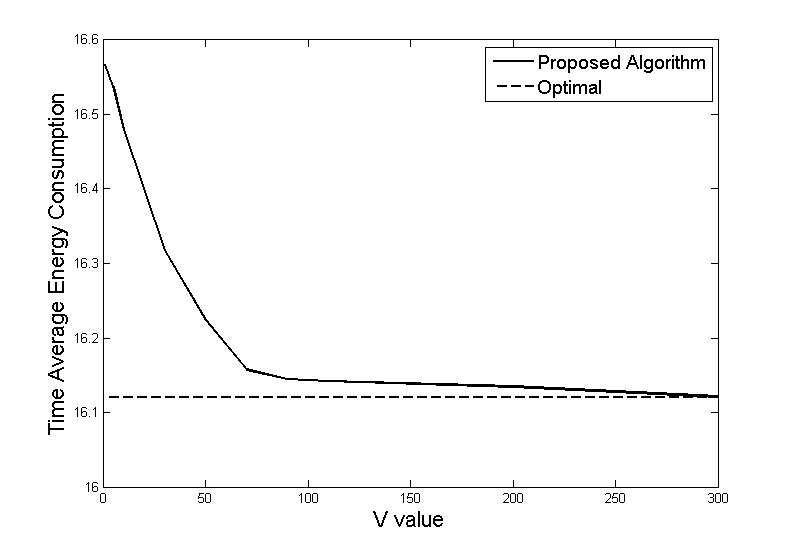} 
   \caption{Time average energy consumption verses $V$ parameter over 1 millon slots.}
   \label{fig:Stupendous2}
\end{figure}

\begin{figure}[htbp]
   \centering
   \includegraphics[height=3in]{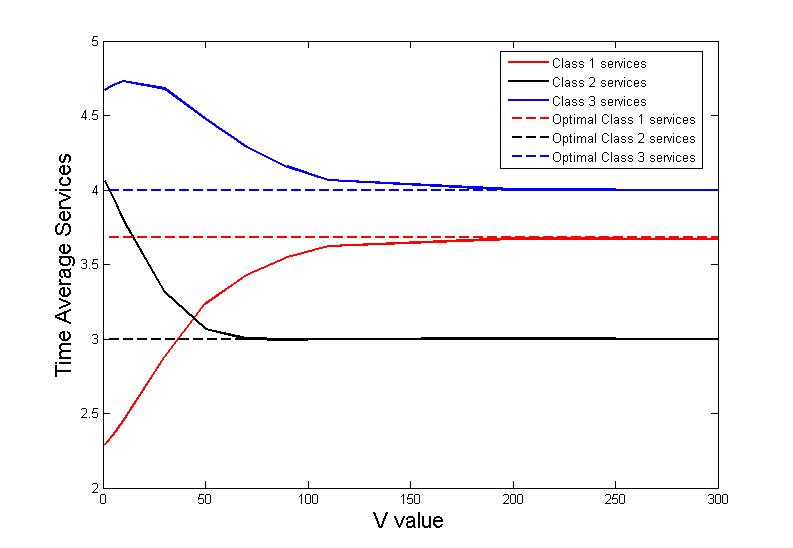} 
   \caption{Time average services verses $V$ parameter over 1 millon slots.}
   \label{fig:Stupendous3}
\end{figure}

\begin{figure}[htbp]
   \centering
   \includegraphics[height=3in]{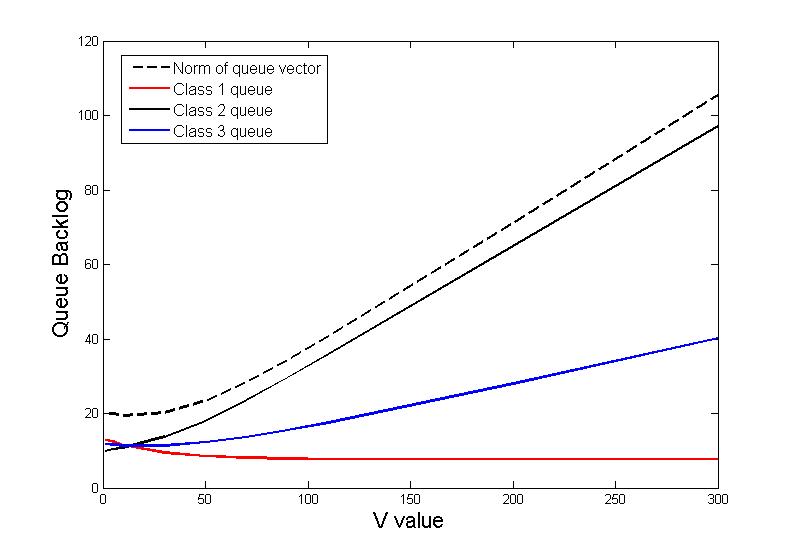} 
   \caption{Time average queue size verses $V$ parameter over 1 million slots.}
   \label{fig:Stupendous4}
\end{figure}

\section{Additional lemmas and proofs.}\label{appendix-proof}

\subsection{Proof of Lemma \ref{convex-lemma}}
\begin{proof}
We first prove the convexity of $\mathcal{P}^n$.
Consider any two points $(f_1,\mathbf{g}_1),~(f_2,\mathbf{g}_2)\in\mathcal{P}^n$. We aim to show that for any $q\in(0,1)$, $(qf_1+(1-q)f_2,q\mathbf{g}_1+(1-q)\mathbf{g}_2)\in\mathcal{P}^n$. Notice that by definition of $\mathcal{P}^n$, there exists $(y_1,\mathbf{z}_1,T_1),~(y_2,\mathbf{z}_2,T_2)\in\mathcal{S}^n$ such that $f_1=y_1/T_1$, $\mathbf{g}_1=\mathbf{z}_1/T_1$, $f_2=y_2/T_2$, and $\mathbf{g}_2=\mathbf{z}_2/T_2$. Thus, it is enough to show 
\begin{equation}\label{convex-combo}
\left(q\frac{y_1}{T_1}+(1-q)\frac{y_2}{T_2},q\frac{\mathbf{z}_1}{T_1}+(1-q)\frac{\mathbf{z}_2}{T_2}\right)\in\mathcal{P}^n.
\end{equation}
To show this, we make a change of variable by letting $p=\frac{qT_2}{(1-q)T_1+qT_2}$. It is obvious that $p\in(0,1)$. Furthermore, $q=\frac{pT_1}{pT_1+(1-p)T_2}$ and 
\begin{align*}
&q\frac{y_1}{T_1}+(1-q)\frac{y_2}{T_2}=\frac{py_1+(1-p)y_2}{pT_1+(1-p)T_2},\\
&q\frac{\mathbf{z}_1}{T_1}+(1-q)\frac{\mathbf{z}_2}{T_2}
=\frac{p\mathbf{z}_1+(1-p)\mathbf{z}_2}{pT_1+(1-p)T_2}.
\end{align*}
Since $\mathcal{S}^n$ is convex, 
$$(py_1+(1-p)y_2,~p\mathbf{z}_1+(1-p)\mathbf{z}_2
,~pT_1+(1-p)T_2)\in\mathcal{S}^n.$$
Thus, by definition of $\mathcal{P}^n$ again, \eqref{convex-combo} holds and the first part of the proof is finished.

To show the second part of the claim, let 
$$\mathcal{Q}^n: = \left\{ \left(\widehat{f}^n(\alpha^n),~\widehat{\mathbf{g}}^n(\alpha^n)\right) : \alpha^n\in\mathcal{A}^n \right\}
= \left\{ \left(\widehat{y}^n(\alpha^n)\left/\widehat{T}^n(\alpha^n)\right.,~\widehat{\mathbf{z}}^n(\alpha^n)\left/\widehat{T}^n(\alpha^n)\right)\right. : \alpha^n\in\mathcal{A}^n \right\}$$ 
and let 
$\text{conv}(\mathcal{Q}^n)$ be the convex hull of $\mathcal{Q}^n$. First of all, By Definition \ref{PR-def},
\[\mathcal{P}^n=\left\{\left(y/T,~\mathbf{z}/T\right):~(y,\mathbf{z},T)\in\mathcal{S}^n\right\}\subseteq\mathbb{R}^{L+1},\]
for $\mathcal{S}^n$ being the convex hull of $\left\{\left(\widehat{y}^n(\alpha^n),~\widehat{\mathbf{z}}^n(\alpha^n),~\widehat{T}^n(\alpha^n)\right):~\alpha^n\in\mathcal{A}^n\right\}$, thus, in view of the definition of $\mathcal{Q}^n$, we have $\mathcal{Q}^n\subseteq\mathcal{P}^n$.
Since both 
$\mathcal{P}^n$ and $\text{conv}(\mathcal{Q}^n)$ are convex, by definition of convex hull (\cite{rockafellar2015convex}) that $\text{conv}(\mathcal{Q}^n)$ is the smallest convex set containing $\mathcal{Q}^n$,
 we have 
$\text{conv}(\mathcal{Q}^n)\subseteq\mathcal{P}^n$.

To show the reverse inclusion $\mathcal{P}^n\subseteq\text{conv}(\mathcal{Q}^n)$, note that any point in 
$\mathcal{P}^n$ can be written in the form $\left( \frac{y}{T},\frac{\mathbf{z}}{T} \right)$, where
$(y,\mathbf{z},T)\in\mathcal{S}^n$. Since $\mathcal{S}^n$ by definition is the convex hull of 
$$\left\{\left(\widehat{y}^n(\alpha^n),~\widehat{\mathbf{z}}^n(\alpha^n),~\widehat{T}^n(\alpha^n)\right):~\alpha^n\in\mathcal{A}^n\right\}\subseteq\mathbb{R}^{L+2},$$
by the definition of convex hull, $(y,\mathbf{z},T)$ can be written as a convex combination of
$m$ points in the above set. Let 
$\left\{\left(\widehat{y}^n(\alpha^n_i),~\widehat{\mathbf{z}}^n(\alpha^n_i),~\widehat{T}^n(\alpha^n_i)\right)\right\}_{i=1}^m$ be these points, so that
\begin{align*}
&(y,\mathbf{z},T) = \sum_{i=1}^m p_i\cdot\left(\widehat{y}^n(\alpha^n_i),~\widehat{\mathbf{z}}^n(\alpha^n_i),~\widehat{T}^n(\alpha^n_i)\right),\\
&p_i\geq0,~~\sum_{i=1}^mp_i = 1.
\end{align*}
As a result, we have
\[
\left( \frac{y}{T},\frac{\mathbf{z}}{T} \right)
=\left( \frac{\sum_{i=1}^m p_iy^n(\alpha^n_i)}{\sum_{i=1}^m p_iT^n(\alpha^n_i)},\frac{\sum_{i=1}^m p_i\mathbf{z}^n(\alpha^n_i)}{\sum_{i=1}^m p_iT^n(\alpha^n_i)} \right).
\]
We make a change of variable by letting $q_j = \frac{p_jT^n(\alpha^n_j)}{\sum_{i=1}^m p_iT^n(\alpha^n_i)},~\forall j=1,2,\cdots,m$, then, 
$$p_j = \frac{q_j}{T^n(\alpha^n_j)}\cdot\sum_{i=1}^m p_iT^n(\alpha^n_i),$$
it follows,
\[
\left( \frac{y}{T},\frac{\mathbf{z}}{T} \right)
=\sum_{i=1}^m q_i\cdot\left( \frac{ y^n(\alpha^n_i)}{T^n(\alpha^n_i)},\frac{\mathbf{z}^n(\alpha^n_i)}{T^n(\alpha^n_i)} \right) = \sum_{i=1}^m q_i\cdot\left(\widehat{f}^n(\alpha^n_i),~\widehat{\mathbf{g}}^n(\alpha^n_i)\right).
\]
Since $\sum_{i=1}^mq_i = 1$ and $q_i\geq0$, it follows any point in $\mathcal{P}^n$ can be written as a convex combination of finite number of points in $ \mathcal{Q}^n$, which implies 
$\mathcal{P}^n\subseteq\text{conv}(\mathcal{Q}^n)$. Overall, we have $\mathcal{P}^n=\text{conv}(\mathcal{Q}^n)$. 

Finally, by Assumption \ref{compact-assumption}, we have $\mathcal{Q}^n=\left\{ \left(\widehat{f}^n(\alpha^n),~\widehat{\mathbf{g}}^n(\alpha^n)\right) : \alpha^n\in\mathcal{A}^n \right\}$ is compact. Thus, $\mathcal{P}^n$, being a convex hull of a compact set, is also compact.
\end{proof}

\subsection{Proof of Lemma \ref{bound-lemma-1}}
\begin{proof}
We prove bound \eqref{bound-1} (\eqref{bound-2} is proved similarly). By definition of $\widehat{f}^n(\alpha^n)$ in Definition \ref{PV-def}, we have for any $\alpha^n\in\mathcal{A}^n$,
\[\widehat{f}^n(\alpha^n)=\frac{\expect{\left.\sum_{t\in\mathcal{T}^n_k}y^n[t]\right|~\alpha^n_k=\alpha^n}}{\expect{T^n_k|~\alpha^n_k=\alpha^n}},\]
thus,
\[\expect{\left.\sum_{t\in\mathcal{T}^n_k}\left(\widehat{f}^n(\alpha^n_k)-y^n[t]\right)\right|~\alpha^n_k=\alpha^n}=0.\]
By the renewal property of the system, given $\alpha^n_k=\alpha^n$, $T^n_k$ and $\sum_{t\in\mathcal{T}^n_k}y^n[t]$ are independent of the past information before $t^n_k$. Thus, the same equality holds if conditioning also on $\mathcal F_k^n $, i.e.
\[\expect{\left.\sum_{t\in\mathcal{T}^n_k}\left(\widehat{f}^n(\alpha^n_k)-y^n[t]\right)\right|~\alpha^n_k=\alpha^n,~\mathcal F_k^n }=0.\]
Hence,
\[\expect{\left.\sum_{t\in\mathcal{T}^n_k}\left(\widehat{f}^n(\alpha^n_k)-y^n[t]\right)\right|~\mathcal F_k^n }=0.\]

By the definition of $f^n[t]$, this further implies that
\[\expect{\left.\sum_{t\in\mathcal{T}^n_k}\left(f^n[t]-y^n[t]\right)\right|~\mathcal F_k^n }=0.\]
Since $|y^n[t]|\leq y_{\max}$ and $\expect{T^n_k}\leq\sqrt{B}$, it follows $\expect{\left|\sum_{t\in\mathcal{T}^n_k}\left(f^n[t]-y^n[t]\right)\right|}<\infty$ and 
the process $\{F^n_K\}_{K=0}^\infty$ defined as
\[F^n_K=\sum_{k=0}^{K-1}\sum_{t\in\mathcal{T}^n_k}\left(f^n[t]-y^n[t]\right),~K\geq1,\]
$F^n_0=0$ is a \textit{martingale}. 

Consider any fixed $T\in\mathbb{N}$ and define $S^n[T]$ as the number of renewals up to $T$. Lemma \ref{valid-stopping-time} shows $S^n[T]$ is a valid stopping time with respect to the filtration $\{\mathcal F_k^n \}_{k=0}^\infty$.
Furthermore, $\{F^n_{K\wedge S^n[T]}\}_{K=0}^\infty$ is a supermartingale by Theorem \ref{stopping-time}, where $a\wedge b:=\min\{a,b\}$.

For this fixed $T$, we have
\begin{align*}
\expect{\sum_{t=0}^{T-1}\left(f^n[t]-y^n[t]\right)}
=&\expect{\sum_{t=0}^{t^n_{S^n[T]}-1}\left(f^n[t]-y^n[t]\right)}
-\expect{\sum_{t=T}^{t^n_{S^n[T]}-1}\left(f^n[t]-y^n[t]\right)}\\
=&\expect{F^n_{S^n[T]}}-\expect{\sum_{t=T}^{t^n_{S^n[T]}-1}\left(f^n[t]-y^n[t]\right)}.
\end{align*}
Since the number of renewals is always bounded by the number of slots at any time, i.e. $S^n[T]\leq T+1$, it follows
\[\expect{F^n_{S^n[T]}}=\expect{F^n_{(T+1)\wedge S^n[T]}}\leq 0.\]
On the other hand, 
\begin{align*}
\left|\expect{\sum_{t=T}^{t^n_{S^n[T]}-1}\left(f^n[t]-y^n[t]\right)}\right|
\leq\expect{t^n_{S^n[T]}-T}\cdot2y_{\max}\leq2y_{\max}\sqrt{B}.
\end{align*}
where the last inequality follows from Assumption \ref{bounded-assumption} for the residual life time. Thus, 
\[\expect{\sum_{t=0}^{T-1}\left(f^n[t]-y^n[t]\right)}\leq2y_{\max}\sqrt{B}.\]
Dividing both sides by $T$ finishes the proof.
\end{proof}

\subsection{Proof of Lemma \ref{lemma:filtration}}
\begin{proof}
Recall that $t^n_k$ is the time slot where the $k$-th renewal occurs ($k=0,1,2,\cdots$), then, it follows from the definition of stopping time (\cite{Durrett}) that $\{t^n_k\}_{k=0}^\infty$ is a sequence of stopping times with respect to $\{\mathcal{F}[t]\}_{t=0}^{\infty}$ satisfying $t^n_k<t^n_{k+1},~\forall k$. Thus, by definition of 
$\mathcal F_k^n $, for any set $A\in\mathcal F_k^n $,
\[A\cap\{t^n_{k+1}\leq t\}=A\cap\{t^n_k\leq t\}\cap\{t^n_{k+1}\leq t\}\in\mathcal{F}[t].\]
Thus, $A\in\mathcal F_{k+1}^n $, which implies $\mathcal F_k^n \subseteq\mathcal F_{k+1}^n ,~\forall k$,
and $\{\mathcal F_k^n \}_{k=0}^\infty$ is indeed a filtration. This finishes the first part of the proof.

Next,
we would like to show that $G_{t^n_k}(Z^n_0,\cdots,Z^n[t^n_k-1])$ is measurable with respect to $\mathcal F_k^n ,~\forall k\geq1$,
i.e. $\left\{G_{t^n_k}(Z^n_0,\cdots,Z^n[t^n_k-1])\in B\right\}\in\mathcal F_k^n $, for any Borel set $B\subseteq\mathbb{R}$. By definition of 
$\mathcal F_k^n $, this is equivalent to showing $\{G_{t^n_k}(Z^n_0,\cdots,Z^n[t^n_k-1])\in B\} \cap \{t^n_k\leq s\}\in\mathcal{F}[s]$ for any slot $s\geq0$. For $s=0$, this is obvious because $ \{t^n_k\leq 0\}=\emptyset,~\forall k\geq1$. Consider any $s\geq1$,
\begin{align*}
&\left\{G_{t^n_k}(Z^n_0,\cdots,Z^n[t^n_k-1])\in B\right\} \cap \{t^n_k\leq s\}\\
&= \bigcup_{i=1}^{s}\left(\left\{G_{i}(Z^n_0,\cdots,Z^n[i-1])  \in B\right\}\bigcap\{t^n_k = i\}\right)\\
&= \bigcup_{i=1}^{s}\left(\left\{(Z^n_0,\cdots,Z^n[i-1])  \in G^{-1}_i(B)\right\}\bigcap\{t^n_k = i\}\right)
 \in\mathcal{F}[s], ~\forall k\geq1,
\end{align*}
where the last step follows from the assumption that
the random variable $Z^n[t-1]$ is measurable with respect to $\mathcal{F}[t]$ for any $t>0$ and $t^n_k$ is a stopping time with respect to $\{\mathcal{F}[t]\}_{t=0}^{\infty}$ for all $k\geq1$. This gives the second part of the claim.
\end{proof}

\subsection{Proof of Lemma \ref{valid-stopping-time}}
\begin{proof}
We aim to prove $\{S^n[T]= k\}\in\mathcal F_k^n ,~\forall k\in\mathbb{N}$.
First of all, recall that the index of the renewal starts from $k=0$ and $t^n_0=0$, thus, 
for any $k\in\mathbb{N}$, 
$\{S^n[T]=k\} = \{t^n_k> T\}\cap\{t^n_{k-1}\leq T\}$,
and any $t\in\mathbb{N}$,
\begin{align}
\{S^n[T]=k\}\cap\{t^n_k\leq t\}
=&\{t^n_k> T\}\cap\{t^n_{k-1}\leq T\}\cap\{t^n_k\leq t\}. \label{the-set}
\end{align}

Consider two cases as follows:
\begin{enumerate}
\item $t\leq T$. In this case, the set \eqref{the-set} is empty and obviously belongs to $\mathcal{F}[t]$.
\item $t>T$. In this case, we have $\{t^n_k> T\}\cap\{t^n_k\leq t\}=\{T<t^n_k\leq t\}\in\mathcal{F}[t]$ as well as $\{t^n_{k-1}\leq T\}\in\mathcal{F}[T]\subseteq\mathcal{F}[t]$. Thus, the set \eqref{the-set} belongs to $\mathcal{F}[t]$.
\end{enumerate}
Overall, we have $\{S^n[T]=k\}\cap\{t^n_k\leq t\}\in \mathcal{F}[t],~\forall t\in\mathbb{N}$. Thus,
$\{S^n[T]=k\}\in\mathcal F_k^n $ and $S^n[T]$ is indeed a valid stopping time with respect to the filtration $\{\mathcal F_k^n \}_{k=0}^\infty$.
\end{proof}

\subsection{Proof of Lemma \ref{stationary-lemma}}

\begin{proof}
To prove the first part of the claim, we define the following notation:
\[\bigoplus_{n=1}^N\mathcal{P}^n:=\left\{\sum_{n=1}^N\mathbf{p}_n,~\mathbf{p}_n\in\mathcal{P}^n,~\forall n
\right\}\]
is the Minkowski sum of sets $\mathcal{P}_n,~n\in\{1,2,\cdots,N\}$, and for any sequence $\{\mathbf{x}[t]\}_{t=0}^\infty$ taking values in 
$\mathbb{R}^d$, define
$$\limsup_{T\rightarrow\infty}\mathbf{x}[T]:=
\left(\limsup_{T\rightarrow\infty}x_1[T],~\cdots,\limsup_{T\rightarrow\infty}x_d[T]\right)$$ 
is a vector of $\limsup$s. By definition, any vector in $\oplus_{n=1}^N\mathcal{P}^n$ can be constructed from $\otimes_{n=1}^N\mathcal{P}^n$, thus, it is enough to show that there exists a vector $\mathbf{r}^*\in\oplus_{n=1}^N\mathcal{P}^n$ such that $r_0^*=f^*$ and the rest of the entries $r^*_l\leq d_l,~l=1,2,\cdots,L$.

By the feasibility assumption for \eqref{prob-1}-\eqref{prob-2}, we can
consider \textit{any algorithm that achieves the optimality} of \eqref{prob-1}-\eqref{prob-2} and the corresponding process $\{(f^n[t],\mathbf{g}^n[t])\}_{t=0}^\infty$ defined in Lemma \ref{bound-lemma-1} for any system $n$. Notice that $(f^n[t],\mathbf{g}^n[t])\in\mathcal{P}^n,~\forall n,~\forall t$. This follows from the definition of $\widehat{f}^n(\alpha^n)$ and $\widehat{\mathbf{g}}^n(\alpha^n)$ in Definition \ref{PV-def} that
\begin{align*}
f^n[t] =& \widehat{f}^n(\alpha^n) = \widehat{y}^n(\alpha^n)/\widehat{T}^n(\alpha^n),
~~\textrm{if}~t\in\mathcal{T}^n_k,\alpha^n_k=\alpha^n\\
\mathbf{g}^n[t] =& \widehat{\mathbf{g}}^n(\alpha^n)= \widehat{\mathbf z}^n(\alpha^n)/\widehat{T}^n(\alpha^n),~~\textrm{if}~t\in\mathcal{T}^n_k,\alpha^n_k=\alpha^n,
\end{align*}
and 
$\left(\widehat{y}^n(\alpha^n),~\widehat{\mathbf z}^n(\alpha^n),~\widehat{T}^n(\alpha^n)\right)
\in\mathcal{S}^n$. By definition of $\mathcal{P}^n$ in Definition \ref{PR-def},
$(f^n[t],\mathbf{g}^n[t])\in\mathcal P^n,~\forall n,~\forall t$.

Since $\mathcal{P}^n$ is convex by Lemma \ref{convex-lemma}, it follows that 
$\left(\expect{f^n[t]},\expect{\mathbf{g}^n[t]}\right)\in\mathcal{P}^n,~\forall n,~\forall t$. Hence,
\[\left(\frac1T\sum_{t=1}^{T-1}\expect{f^n[t]},~\frac1T\sum_{t=1}^{T-1}\expect{\mathbf{g}^n[t]}\right)\in\mathcal{P}^n,~\forall T, \forall n.\]
This further implies that
\[
\mathbf{r}(T):=\left(\frac1T\sum_{t=1}^{T-1}\sum_{n=1}^N\expect{f^n[t]},~\frac1T\sum_{t=1}^{T-1}\sum_{n=1}^N\expect{\mathbf{g}^n[t]}\right)\in\bigoplus_{n=1}^N\mathcal{P}^n.
\]
By Lemma \ref{convex-lemma}, $\mathcal P^n$ is compact in $\mathbb{R}^{L+1}$. Thus,
$\oplus_{n=1}^N\mathcal{P}^n$ is also compact.
This implies that the sequence $\{\mathbf{r}(T)\}_{T=1}^\infty$ has at least one limit point, and any such limit point is contained in $\oplus_{n=1}^N\mathcal{P}^n$. 

We consider a specific limit point of $\{\mathbf{r}(T)\}_{T=1}^\infty$ denoted as $\mathbf{r}^*\in\oplus_{n=1}^N\mathcal{P}^n$, with the first entry denoted as $r_0^*$ satisfying 
$$r_0^* = \limsup_{T\rightarrow\infty}\frac1T\sum_{t=0}^{T-1}\sum_{n=1}^N\expect{f^n[t]}.$$
Then, we have the rest of the entries of $\mathbf{r}^*$ must satisfy
\[r_l^*\leq\limsup_{T\rightarrow\infty}\frac1T\sum_{t=0}^{T-1}\sum_{n=1}^N\expect{\mathbf{g}^n[t]},
~\forall l\in\{1,2,\cdots,L\}.\]
Now, 
by Lemma \ref{bound-lemma-1}, we can connect the $\limsup$ with respect to $f^n[t]$ and 
$\mathbf{g}^n[t]$ to that of $y^n[t]$ and $\mathbf{z}^n[t]$ as follows:
\begin{align*}
&\limsup_{T\rightarrow\infty}\frac1T\sum_{t=0}^{T-1}\sum_{n=1}^N\expect{y^n[t]}\\
=&\limsup_{T\rightarrow\infty}\frac1T\sum_{t=0}^{T-1}\sum_{n=1}^N\left(\expect{y^n[t]-f^n[t]}+\expect{f^n[t]}\right)\\
=&\lim_{T\rightarrow\infty}\frac1T\sum_{t=0}^{T-1}\sum_{n=1}^N\expect{y^n[t]-f^n[t]}
+\limsup_{T\rightarrow\infty}\frac1T\sum_{t=0}^{T-1}\sum_{n=1}^N\expect{f^n[t]}\\
=&\limsup_{T\rightarrow\infty}\frac1T\sum_{t=0}^{T-1}\sum_{n=1}^N\expect{f^n[t]}.
\end{align*}
Similarly, we can show that
\[
\limsup_{T\rightarrow\infty}\frac1T\sum_{t=0}^{T-1}\sum_{n=1}^N\expect{\mathbf{z}^n[t]}
=\limsup_{T\rightarrow\infty}\frac1T\sum_{t=0}^{T-1}\sum_{n=1}^N\expect{\mathbf{g}^n[t]}.\]
Thus, by our preceeding assumption that the algorithm under consideration achieves the optimality of \eqref{prob-1}-\eqref{prob-2}, we have 
\begin{align*}
&r_0^*=\limsup_{T\rightarrow\infty}\frac1T\sum_{t=0}^{T-1}\sum_{n=1}^N\expect{y^n[t]}=f^*\\
&r_l^*\leq\limsup_{T\rightarrow\infty}\frac1T\sum_{t=0}^{T-1}\sum_{n=1}^N\expect{z_l^n[t]}
\leq d_l,~\forall i\in\{1,2,\cdots,L\}.
\end{align*}
Overall, we have shown that $\mathbf{r}^*\in\oplus_{n=1}^N\mathcal{P}^n$ achieves the optimality of \eqref{prob-1}-\eqref{prob-2}, and the first part of the lemma is proved.

To prove the second part of the lemma, we show that any point in $\otimes_{n=1}^N\mathcal{P}^n$ is achievable by the corresponding time averages of some algorithm. Specifically, consider the following class of \textit{randomized stationary algorithms}: For each system $n$, at the beginning of $k$-th frame, the controller independently chooses an action $\alpha^n_k$ from the set $\mathcal{A}^n$ with a fixed probability distribution. 

Thus, the actions $\{\alpha^n_k\}_{k=0}^{\infty}$ result from any randomized stationary algorithm is i.i.d.. By the renewal property of each system, we have 
$$\left\{\left(\sum_{t\in\mathcal{T}^n_k}y^n[t],~\sum_{t\in\mathcal{T}^n_k}\mathbf{z}^n[t],~T^n_k\right)\right\}_{k=0}^\infty,$$
is also an i.i.d. process for each system $n$.

Next, we would like to show that any point in $\mathcal{S}^n$ can be achieved by the corresponding expectations of some randomized stationary algorithm.
Recall that $\mathcal{S}^n$ defined in Definition \ref{PR-def} is the convex hull of 
$$\mathcal{G}^n:=\left\{\left(\widehat{y}^n(\alpha^n),~\widehat{\mathbf z}^n(\alpha^n),~\widehat{T}^n(\alpha^n)\right),~\alpha^n\in\mathcal{A}^n\right\} \subseteq\mathbb{R}^{L+2},$$ 
By definition of convex hull, 
 any point $(y,\mathbf{z},T)\in\mathcal{S}^n$, can be written as a convex combination of a finite number of points from the set $\mathcal{G}^n$. Let $\left\{\left(\widehat{y}^n(\alpha^n_i),~\widehat{\mathbf z}^n(\alpha^n_i),~\widehat{T}^n(\alpha^n_i)\right)\right\}_{i=1}^m$ be these points,
then, we have there exists a finite sequence $\{p_i\}_{i=1}^m$, such that
\begin{align*}
&(y,\mathbf{z},T) = \sum_{i=1}^m p_i\cdot\left(\widehat{y}^n(\alpha^n_i),~\widehat{\mathbf z}^n(\alpha^n_i),~\widehat{T}^n(\alpha^n_i)\right),\\
&p_i\geq0,~\sum_{i=1}^mp_i=1.
\end{align*}
We can then use $\{p_i\}_{i=1}^m$ to construct the following randomized stationary algorithm: At the start of each frame $k$, the controller independently chooses action $\alpha_i\in\mathcal{A}^n$ with probability $p_i$ defined above for $i=1,2,\cdots,m$. Then, the one-shot expectation of this particular randomized stationary algorithm on system $n$ satisfies
\[
\left(\expect{\sum_{t\in\mathcal{T}^n_k}y^n[t]},~\expect{\sum_{t\in\mathcal{T}^n_k}\mathbf{z}^n[t]},~\expect{T^n_k}\right)=
 \sum_{i=1}^m p_i\cdot\left(\widehat{y}^n(\alpha^n_i),~\widehat{\mathbf z}^n(\alpha^n_i),~\widehat{T}^n(\alpha^n_i)\right)=(y,\mathbf{z},T),
\]
which implies any point in $\mathcal{S}^n$ can be achieved by the corresponding expectations of a randomized stationary algorithm.

Next, by definition of $\mathcal P^n$ in Definition \ref{PR-def}, any $(\overline{f}^n,\overline{\mathbf{g}}^n)\in\mathcal P^n$ can be written as $(\overline{f}^n,\overline{\mathbf{g}}^n)=(y/T,\mathbf{z}/T)$, where $(y,\mathbf{z},T)\in\mathcal{S}^n$. Thus, it is achievable by the ratio of one-shot expectations from a randomized stationary algorithm, i.e.
\[
\frac{\expect{\sum_{t\in\mathcal{T}^n_k}y^n[t]}}{\expect{T^n_k}}=\frac{y}{T}=\overline{f}^n,~~
\frac{\expect{\sum_{t\in\mathcal{T}^n_k}\mathbf{z}^n[t]}}{\expect{T^n_k}}=\frac{\mathbf{z}}{T}
=\overline{\mathbf{g}}^n.
\]
Now we claim that for $y^n[t]$, $\mathbf{z}^n[t]$ and $T^n_k$ result from the randomized stationary algorithm,
\begin{align}
&\lim_{T\rightarrow\infty}\frac1T\sum_{t=0}^{T-1}\expect{y^n[t]}=\frac{\expect{\sum_{t\in\mathcal{T}^n_k}y^n[t]}}{\expect{T^n_k}},\label{lln-1}\\
&\lim_{T\rightarrow\infty}\frac1T\sum_{t=0}^{T-1}\expect{\mathbf{z}^n[t]}=\frac{\expect{\sum_{t\in\mathcal{T}^n_k}\mathbf{z}^n[t]}}{\expect{T^n_k}}.\label{lln-2}
\end{align}
We prove \eqref{lln-1} and \eqref{lln-2} is shown in a similar way. Consider any fixed $T$, and let $S^n[T]$ be the number of renewals up to (and including) time $T$. Then, from Lemma \ref{valid-stopping-time} in Section \ref{section:limiting}, $S^n[T]$ is a valid stopping time with respect to the filtration 
$\{\mathcal F_k^n \}_{k=0}^\infty$. We write
\begin{equation}\label{split-stop}
\frac1T\sum_{t=0}^{T-1}\expect{y^n[t]}=\frac{1}{T}
\expect{\sum_{k=0}^{S^n[T]}\sum_{t\in\mathcal{T}^n_k}y^n[t]}-\frac1T\expect{\sum_{t=T}^{t^n_{S^n[T]}-1}y^n[t]}.
\end{equation}
For the first part on the right hand side of \eqref{split-stop}, since $\left\{\sum_{t\in\mathcal{T}^n_k}y^n[t]\right\}_{k=0}^{\infty}$ is an i.i.d. process, by Wald's equality (Theorem 4.1.5 of \cite{Durrett}),
\[
\frac{1}{T}
\expect{\sum_{k=0}^{S^n[T]}\sum_{t\in\mathcal{T}^n_k}y^n[t]}=\expect{\sum_{t\in\mathcal{T}^n_k}y^n[t]}
\cdot\frac{\expect{S^n[T]}}{T}.
\] 
By renewal reward theorem (Theorem 4.4.2 of \cite{Durrett}),
\[
\lim_{T\rightarrow\infty}\frac{\expect{S^n[T]}}{T}=\frac{1}{\expect{T^n_k}}.
\]
Thus,
\[
\lim_{T\rightarrow\infty}\frac{1}{T}
\expect{\sum_{k=0}^{S^n[T]}\sum_{t\in\mathcal{T}^n_k}y^n[t]}
=\frac{\expect{\sum_{t\in\mathcal{T}^n_k}y^n[t]}}{\expect{T^n_k}}.
\]
For the second part on the right hand side of \eqref{split-stop}, by Assumption \ref{bounded-assumption},
\[
\left|\expect{\sum_{t=T}^{t^n_{S^n[T]}-1}y^n[t]}\right|\leq y_{\max}\cdot\expect{t^n_{S^n[T]}-T}
\leq\sqrt{B}y_{\max},
\]
which implies $\lim_{T\rightarrow\infty}\frac1T\expect{\sum_{t=T}^{t^n_{S^n[T]}-1}y^n[t]}=0$. Overall, we have \eqref{lln-1} holds.

To this point, we have shown that for any $(\overline{f}^n,\overline{\mathbf{g}}^n)\in\mathcal P^n$, $n\in\{1,2,\cdots,N\}$, there exists a randomized stationary algorithm so that 
\begin{align*}
\lim_{T\rightarrow\infty}\frac1T\sum_{t=0}^{T-1}\expect{y^n[t]}=\overline{f}^n,~~
\lim_{T\rightarrow\infty}\frac1T\sum_{t=0}^{T-1}\expect{\mathbf{z}^n[t]}=\overline{\mathbf{g}}^n,
\end{align*}
for any $n\in\{1,2,\cdots,N\}$. Since $f^*$ is the optimal solution to \eqref{prob-1}-\eqref{prob-2} over all algorithms, it follows for any $(\overline{f}^n,\overline{\mathbf{g}}^n)\in\mathcal P^n$, $n\in\{1,2,\cdots,N\}$ satisfying $\sum_{n=1}^N\overline{g}^n_l\leq d_l,~\forall l\in\{1,2,\cdots,L\}$, we have
$\sum_{n=1}^N\overline{f}^n\geq f^*$, and the second part of the lemma is proved.
\end{proof}



\chapter{Data Center Server Provision via Theory of Coupled Renewal Systems}

The previous chapter introduces a new algorithm and analysis framework for coupled parallel renewal systems. In this chapter, we show that the previous algorithm can be applied (extended) to solve a data center power minimization problem 
consisting of a central controller who makes load balancing decisions per slot and parallel servers
having multiple states making decisions per renewal frame. In particular, the analysis in this chapter, which is customized to the data center application, is stronger than that of previous general algorithm in the sense that we obtain a probability 1 convergence of the algorithm rather than an expected convergence.

\section{System model and problem formulation}

Consider a data center that consists of a central controller and $N$ servers that serve randomly arriving requests. The system operates in slotted time with time
slots $t \in \{0, 1, 2, \ldots\}$.  Each server $n \in \{1, \ldots, N\}$ has three basic states:
\begin{itemize}
  \item Active: The server is available to serve requests. Server $n$ incurs a cost of
  $e_n \geq 0$ on every active slot, regardless of whether or not requests are available to serve. In data center applications, such cost often represents the power consumption of each individual server.

  \item Idle: A low cost sleep state where no requests can be served. The idle state is actually
  comprised of a choice of multiple sleep modes with different per-slot costs. The specific sleep mode also affects the setup time required to transition from the idle state to the active state. For the rest of the paper, we use ``idle'' and ``sleep'' exchangeably.
  \item Setup: A transition period from idle to active during which no requests can be served.
  The setup cost and duration depend on the preceding sleep mode.   The setup duration is typically more than one slot, and can be a random variable that depends on the server $n$ and on the
  preceding sleep mode.
 \end{itemize}

 An active server can choose to transition to the idle state at any time.  When it does
so, it chooses the specific sleep mode to use and the amount of time to sleep.
For example, deeper sleep modes can shut down more electronics and thereby save
on per-slot idling costs.  However, a deeper sleep incurs a longer setup time when
transitioning back to the active state.  Each server makes separate
decisions about when to transition and what sleep mode to use.  The resulting
transition times for each server are asynchronous.   On top of this, a central controller
makes slot-wise decisions for routing requests to servers.  It can also reject requests
(with a certain amount of cost) if it decides they cannot be
supported. The goal is to minimize the overall time average cost. 

This  problem is challenging mainly for two reasons: First, since each setup state generates cost but serves no request, it is not clear whether or not transitioning to idle from the
active state indeed saves power.
It is also not clear which sleep mode the server should switch to. Second, if one server is currently in a setup state, it cannot make another decision until it reaches the active state (which typically takes more than one slot), whereas other active servers can make decisions during this time.  Thus, this problem  can be viewed as a system with coupled Markov decision processes (MDPs) making decisions asynchronously.

\subsection{Related works}
Experimental work on power and delay minimization in data centers is treated in  \cite{gandhi2013dynamic}, which
proposes to turn each server ON and OFF according to the rule of an $M/M/k/setup$ queue. The work in \cite{urgaonkar2010dynamic} applies Lyapunov optimization to optimize power in
virtualized data centers.  However, it assumes each server has negligible setup time and that
ON/OFF decisions are made synchronously at each server.
The works  \cite{yao2012data}, \cite{lin2013dynamic}
focus on power-aware provisioning over a time scale large enough so that the whole data center can adjust its service capacity. Specifically,
\cite{yao2012data} considers load balancing across geographically distributed data centers,
and \cite{lin2013dynamic} considers provisioning over a finite time interval and introduces an online 3-approximation algorithm.

Prior works \cite{horvath2008multi, meisner2009powernap, meisner2011power}
consider servers with multiple hypothetical sleep states with different levels of power consumption and setup times. Although empirical evaluations in these works show significant power saving by introducing sleep states, they are restricted to the scenario where the setup time from sleep to active is on the order of milliseconds, which is not realistic for today's data center. 
Realistic sleep states with setup time on the order of seconds are considered in \cite{gandhi2012sleep}, where effective heuristic algorithms are proposed and evaluated via extensive testbed simulations.
However, little is known about the theoretical performance bound regarding these algorithms.

\subsection{Front-end load balancing}
At each time slot $t \in \{0, 1, 2, \ldots\}$, $\lambda(t)$ new requests arrive at the system (see Fig. \ref{fig:Stupendous0}). We assume $\lambda(t)$ takes values in a finite set $\Lambda$.
Let  $R_n(t),~n\in\mathcal{N}$ denote the number of requests routed into server $n$ at time $t$.
In addition, the system is allowed to reject requests. Let $r(t)$ be the number of requests that are
rejected on slot $t$, and let $c(t)$ be the corresponding per-request cost for such rejection.
Assume $c(t)$ takes values in a finite state space $\mathcal{C}$. The $R_n(t)$ and $r(t)$ decision variables on slot $t$ must be nonnegative integers that satisfy:
 \begin{align*}
 &\sum_{n=1}^NR_n(t)+r(t)=\lambda(t)\\
 &\sum_{n=1}^N R_n(t)\leq R_{\max}
  \end{align*}
 for a given  integer $R_{max}>0$.  The vector process $(\lambda(t), c(t))$ takes values in $\Lambda \times \mathcal{C}$ and is assumed to be an independent and identically distributed (i.i.d.) vector over slots $t \in \{0, 1, 2, \ldots\}$ with an unknown probability mass function.

 \begin{figure}[htbp]
   \centering
   \includegraphics[height=2in]{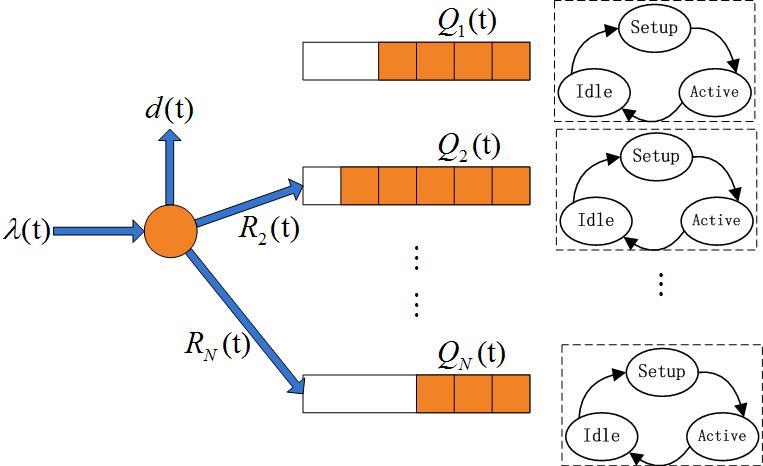} 
   \caption{Illustration of a data center structure which contains a front-end load balancer, $N$ application servers with $N$ request queues and a backend database (omitted here for brevity).}
   \label{fig:Stupendous0}
\end{figure}

 Each server $n$ maintains a request queue $Q_n(t)$ that stores the requests that are
 routed to it.  Requests are served in  a FIFO manner with queueing dynamics as follows:
\begin{equation}\label{queue_update}
Q_n(t+1)=\max\left\{Q_n(t)+R_n(t)-\mu_n(t)H_n(t),~0\right\}.
\end{equation}
where $H_n(t)$ is an indicator variable that is 1 if server $n$ is active on slot $t$, and $0$ else, and $\mu_n(t)$ is a random variable that represents the number of requests can be
served on slot $t$.  Each queue is initialized to $Q_n(0)=0$.
Assume that, every slot in which server $n$ is active, $\mu_n(t)$ is independent and identically distributed with a known mean $\mu_n$. This randomness can model variation in job sizes.

\begin{assumption}\label{observability}
The process $\{(\lambda(t), c(t))\}_{t=0}^{\infty}$ is observable, i.e. the router
can observe the $(\lambda(t),c(t))$  realization each time slot $t$ before making decisions. In contrast, the process $\{\mu_n(t)\}_{t=0}^{\infty}$ is not observable, i.e. given that $H_n(t)=1$, the server $n$ cannot observe the
realization of $\mu_n(t)$ until the end of slot $t$.  {Moreover, $\lambda(t),~c(t)$ and $\mu_n(t)$ are all bounded by 
$\lambda_{\max}$, $c_{\max}$ and $\mu_{\max}$ respectively.}
\end{assumption}

\subsection{Server model}
Each server $n\in \mathcal{N}$ has three types of states: active, idle, and setup (see Fig. \ref{fig:single-system}).
The idle state of each server $n$ is further decomposed into a collection of distinct sleep modes.  Each server $n \in \mathcal{N}$ makes decisions over its own \emph{renewal frames}. Define the renewal frame for server $n$ as the time period between successive visits to active state (with each renewal period ending in an active state). Let $T_n[f]$ denote the frame size of the $f$-th renewal frame for server $n$, for  $f \in \{0, 1, 2, \ldots\}$. Let $t^{n}_f$ denote the start of frame $f$, so that
 $T_n[f]=t^{n}_{f+1}-t^{n}_f$.  Assume that $t^{n}_0=0$ for all $n \in \mathcal{N}$, so that time slot $0$ is the start of the first renewal frame (labeled frame $f=0$) for all servers.  For simplicity, assume all servers are ``active'' on slot $t=-1$. Thus, the slot just before each renewal frame is an active slot.

\begin{figure}[htbp]
   \centering
   \includegraphics[height=1in]{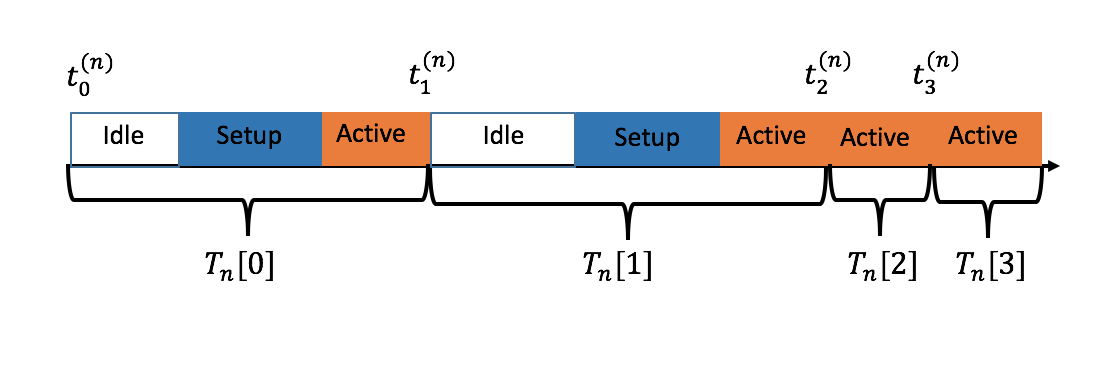} 
   \caption{Illustration of a typical renewal frame construction, where $T_n[i]$ is the length of frame $i$ and $t^{(n)}_i$ is the start slot of frame $i$.}
   \label{fig:single-system}
\end{figure}

 Fix a server $n \in \mathcal{N}$ and a frame index $f \in \{0, 1,2, \ldots\}$.  Time $t_f^{n}$ marks the start of renewal frame $f$. At this time, server $n$  must decide whether to remain active or to
go idle. If it remains active then the renewal frame lasts for one slot, so that $T_n[f]=1$.   If it goes idle, it chooses an idle mode from a finite set $\mathcal{L}_n$, representing the set of idle mode options. Let $\alpha_n[f]$ represent this initial decision for server $n$ at the start of frame $f$, so that:
\[ \alpha_n[f] \in \{active\}\cup\mathcal{L}_n \]
where $\alpha_n[f] =active$ means the server chooses to remain active.  If the server chooses to go idle, so that $\alpha_n[f] \in \mathcal{L}_n$, it then chooses a variable $I_n[f]$ that represents \emph{how much time it remains idle}. The
decision variable $I_n[f]$ is chosen as an integer
in the set $\{1, \ldots, I_{max}\}$ for some given integer $I_{max}>0$.  The consequences of these decisions are described below.

\begin{itemize}
\item Case $\alpha_n[f]=active$.  The frame starts at time $t_f^{n}$ and
has size $T_n[f]=1$. The active variable becomes $H_n(t_f^{n})=1$ and an activation cost of $e_n$ is incurred on this slot $t_f^{n}$. A random service variable $\mu_n(t_f^{n})$ is generated and requests are served according to the queue update \eqref{queue_update}. Recall that, under Assumption \ref{observability}, the  value of $\mu_n(t)$ is not known until the end of the slot.

\item Case $\alpha_n[f] \in \mathcal{L}_n$.  In this case, the server chooses to go idle and $\alpha_n[f]$ represents the specific sleep mode chosen.  The idle duration $I_n[f]$ is also chosen as an integer in the set $[1, I_{max}]$.  After the idle duration completes, the setup duration starts and has an independent and  random
duration $\tau_n[f] = \hat{\tau}(\alpha_n[f])$, where $\hat{\tau}(\alpha_n[f])$ is an integer random variable with a  known mean and variance that depends on the sleep mode $\alpha_n[f]$.
At the end of the setup time the system goes active and serves with a random $\mu_n(t)$ as before.
The active variable is $H_n(t)=0$ for all slots $t$ in the idle and setup times, and is $1$ at the very last slot of the frame.
Further:
\begin{itemize}
\item Idle cost:  Every slot $t$ of the idle time of frame $f$, an idle cost of $g_n(t) = \hat{g}_n(\alpha_n[f])$ is incurred (so that the idle cost depends on the sleep mode).  We have $g_n(t)=0$ if server $n$ is not idle on slot $t$.  The idle cost can be zero, but can also be a small but positive value if some electronics are still running in the sleep mode chosen.

\item Setup cost: Every slot $t$ of the setup time of frame $f$,
a cost of $W_n(t)=\hat{W}_n(\alpha_n[f])$ is incurred. We have $W_n(t)=0$ if server $n$ is not in a setup duration on slot $t$.
\end{itemize}
\end{itemize}

Thus, the length of frame $f$ for server $n$ is:
\begin{equation}\label{frame_length}
T_n[f]=\left\{
         \begin{array}{ll}
           1, & \hbox{if $\alpha_n[f]=active$;}\\
           I_n[f]+\tau_n[f]+1, & \hbox{if $\alpha_n[f]\in\mathcal{L}_n$.}
         \end{array}
       \right.
\end{equation}
In summary, the costs $\hat{g}_n(\alpha_n)$, $\hat{W}_n(\alpha_n)$ and the setup time $\hat{\tau}_n(\alpha_n)$ are functions of $\alpha_n\in\mathcal{L}_n$. We further make the following assumption regarding $\hat{\tau}_n(\alpha_n)$:

\begin{assumption}\label{bounded_moment_assumption}
For any $\alpha_n\in\mathcal{L}_n$, the function $\hat{\tau}_n(\alpha_n)$ is an integer random variable with known mean and variance, as well as bounded first four moments. Denote $\expect{\tau_n(\alpha_n)}=m_{\alpha_n}$ and $\textrm{Var}[\tau_n(\alpha_n)]=\sigma_{\alpha_n}^2$.
\end{assumption}
 {Note that this is a very mild assumption in view of the fact that the setup time of a real server is always bounded. The motivation behind emphasizing the fourth moment here instead of simply proceeding with boundedness assumption is more of theoretical interest than practical importance.
}

 {
Table I summarizes the parameters introduced in this section. The data center architecture is shown is Fig. \ref{fig:Stupendous0}. Since different servers might make different decisions, the renewal frames are not necessarily aligned. }

\begin{table}
\begin{center}
\caption{Parameters}
\begin{tabular}{|l|l|}
  \hline
  Control parameters & Control objectives \\
  \hline
  $R_n(t)$ & Requests routed to server $n$ at slot $t$ \\ 
  $r(t)$ & Requests rejected at slot $t$ \\
  $\alpha_n[f]$ & The option (active/idle) server $n$ takes in frame $f$ \\
  $I_n[f]$ & Number of slots server $n$ stays idle in frame $f$ \\
  \hline
  Other parameters & Meaning\\
  \hline
  $\lambda(t)$ & Number of arrivals at time $t$\\
  $c(t)$ & Per request rejection cost at time $t$\\
  $e_n$ & Per slot active service cost for server $n$\\
  $T_n[f]$ & The length of frame $f$ for server $n$\\
  $t^{(n)}[f]$ & Starting slot of frame $f$ for server $n$\\
  $\tau_n[f]$ & Setup duration in frame $f$\\
  $\mu_n(t)$ & Number of requests served on server $n$ at time $t$\\
  $H_n(t)$ & Server active indicator (equal to 1 if active, 0 if not)\\
  $g_n(t)$ & Idle cost of server $n$ at time $t$\\
  $W_n(t)$ & Setup cost of server $n$ at time $t$\\
  \hline
\end{tabular}
\end{center}
\end{table}
%


\subsection{Performance Objective}

For each $n \in \mathcal{N}$,
let $\overline{C}$, $\overline{W}_n$, $\overline{E}_n$, $\overline{G}_n$ be the time average costs resulting from rejection, setup, service and idle, respectively. They are defined as follows:
$\overline{C}=\lim_{T\rightarrow\infty}\frac{1}{T}\sum_{t=0}^{T-1}\expect{r(t)c(t)}$,
$\overline{W}_n=\lim_{T\rightarrow\infty}\frac{1}{T}\sum_{t=0}^{T-1}\expect{W_n(t)}$,
$\overline{E}_n=\lim_{T\rightarrow\infty}\frac{1}{T}\sum_{t=0}^{T-1}\expect{e_nH_n(t)}$,
$\overline{G}_n=\lim_{T\rightarrow\infty}\frac{1}{T}\sum_{t=0}^{T-1}\expect{g_n(t)}$.

The goal is to design a joint routing and service policy so that the time average overall cost is minimized and all queues are stable, i.e.
\begin{align}\label{obj_1}
\min~\overline{C}+\sum_{n=1}^N\left(\overline{W}_n+\overline{E}_n+\overline{G}_n\right),~
\textrm{s.t.}~Q_n(t)~\textrm{stable } \forall n.
\end{align}
Notice that the constraint in \eqref{obj_1} is not easy to work with. In order to get an optimization problem one can deal with, we further define the time average request rate, rejection rate, routing rate and service rate as $\overline{\lambda}$, $\overline{d}$, $\overline{R}_n$, and $\overline{\mu}_n$ respectively:
$\overline{\lambda}=\lim_{T\rightarrow\infty}\frac{1}{T}\sum_{t=0}^{T-1}\lambda(t) = \expect{\lambda(t)}$,
$\overline{r}=\lim_{T\rightarrow\infty}\frac{1}{T}\sum_{t=0}^{T-1}\expect{r(t)}$,
$\overline{R}_n=\lim_{T\rightarrow\infty}\frac{1}{T}\sum_{t=0}^{T-1}\expect{R_n(t)}$,
$\overline{\mu}_n=\lim_{T\rightarrow\infty}\frac{1}{T}\sum_{t=0}^{T-1}\expect{\mu_n(t)H_n(t)}$.

Then, rewrite the problem \eqref{obj_1} as follows
\begin{align}
\min~~&\overline{C}+\sum_{n=1}^N\left(\overline{W}_n+\overline{E}_n+\overline{G}_n\right) \label{obj_3}\\
\textrm{s.t.}~~&\overline{R}_n\leq\overline{\mu}_n,~\forall n\in\mathcal{N}\label{obj_4}\\
&\sum_{n=1}^NR_n(t)\leq R_{\max},~\sum_{n=1}^NR_n(t)+r(t)=\lambda(t)~\forall t \label{obj_5}
\end{align}
Constraint \eqref{obj_4} requires the time average arrival rate to server $n$ to be less than the time average service rate.
We aim to develop an algorithm so that each server can make its own decision (without looking at the workload or service decision of any other server) and prove its near optimality.

\section{Coupled renewal optimization}
In this section, we show one can apply the algorithm introduced in the previous section to solve \eqref{obj_3}-\eqref{obj_5}. But before jumping into details, we would like to discuss some intuitions behind solving this problem. As a side remark, this data center work is written and published before the general algorithm introduced in the last section, so this intuition is the origin of thesis.

\subsection{Prelude: The original intuition}
First of all, from the queueing model described in the last section and Fig. \ref{fig:Stupendous0}, it is intuitive that an efficient algorithm would have
each server make decisions regarding its own queue state $Q_n(t)$, whereas the front-end load-balancer make routing and rejection decisions slot-wise based on the global information $(\lambda(t), c(t), \mathbf{Q}(t))$.

Next, to get an idea on what exactly the decision should be, 
by virtue of Lyapunov optimization, one would introduce a trade-off parameter $V>0$ and penalize the time average constraint \eqref{obj_4} via $\mathbf{Q}(t)$ to solve the following slotwise optimization problem
\begin{align}
\min~~&V\left(c(t)r(t)+\sum_{n=1}^N\left(W_n(t)+e_nH_n(t)+g_n(t)\right)\right) \label{obj-temp}\\
&+\sum_{n=1}^NQ_n(t)(R_n(t)-\mu_n(t))
~~\textrm{s.t.}~~\textrm{constraint}~\eqref{obj_5},\nonumber
\end{align}
which is naturally separable regarding the load-balancing decision ($r(t)$, $R_n(t)$), and the service decision ($W_n(t),~ H_n(t),~ g_n(t),~ \mu_n(t)$).
However, because of the existence of a setup state (on which no decision could be made), the server does not have an identical decision set every slot and furthermore, the decision set itself depends on previous decisions. This poses a significant difficulty analyzing the above optimization \eqref{obj-temp}.

In order to resolve this difficulty, we try to find the smallest ``identical time unit'' for each individual server in lieu of slots. This motivates the notion of renewal frame in the previous section (see Fig. \ref{fig:single-system}). Specifically, from Fig. \ref{fig:single-system} and the related renewal frame construction, at the starting slot of each renewal, the server faces the identical decision set (remain active or go to idle with certain slots) regardless of previous decisions. Following this idea, we modify \eqref{obj-temp} as follows:
\begin{itemize}
\item For the front-end load balancer, we observe $(\lambda(t), c(t), \mathbf{Q}(t))$ and solve $min~Vc(t)r(t)+\sum_{n=1}^NQ_n(t)R_n(t),~s.t. ~\eqref{obj_5}$, which is detailed in Section \ref{front-section}.
\item For each server, instead of per slot optimization $\min~V(W_n(t)+e_nH_n(t)+g_n(t))-Q_n(t)\mu_n(t)$, we propose to minimize the time average of this quantity per renewal frame $T_n[f]$.
\end{itemize}

\subsection{Coupled renewal optimization}
In order to apply Algorithm \ref{proposed-algorithm} to this scenario, 
we can view the admission control (which chooses $r(t)$ and $R_n(t)$) as one another system besides $N$ servers. Thus, this problem is equivalent to an asynchronous optimization over $N+1$ parallel renewal systems where one of them is just a slotted system. This falls into the form of \eqref{prob-1}-\eqref{prob-2} when setting $l=N$,
\begin{align*}
y^n[t] =& r(t)c(t) + W_n(t) + e_nH_n(t) + g_n(t),\\
z^l[t] =& R_l(t)-\mu_l(t) ,~l\in\{1,2,\cdots,N\}\\
d^l[t] =& 0,
\end{align*}
 and the control variable $r(t)$, $R_n(t)$ are non-negative, and must satisfy the following instant constraints:
 \[
 \sum_{n=1}^NR_n(t)\leq R_{\max},~\sum_{n=1}^NR_n(t)+r(t)=\lambda(t).
 \]
 The only difference compared to \eqref{obj_3}-\eqref{obj_5} is that here the decision variables $r(t)$ and $R_n(t)$ must take values from time-varying ranges per slot and they must be chosen after observing the random variable $c(t)$. However, since $r(t)$ and $R_n(t)$ are updated slot-wise, this minor difference is easy to handle via our renewal optimization framework and we have the following Algorithm \ref{alg:temp-data-center}.

 \begin{algorithm}
  \begin{Alg}\label{alg:temp-data-center}
 Fix a trade-off parameter $V>0$, and at each time slot $t$:
 \begin{itemize}
 \item The admission controller chooses $r(t)$ and $R_n(t)$ according to 
 \begin{equation}\label{eq:lb}
 \min Vc(t)r(t) + \sum_{n=1}^NQ_n(t)R_n(t)~~\text{s.t.} ~~\sum_{n=1}^NR_n(t)\leq R_{\max},~\sum_{n=1}^NR_n(t)+r(t)=\lambda(t). 
 \end{equation}
 \item Each server chooses service options $\alpha_n[f]$ and $I_n[f]$ via the following:
 \begin{equation}\label{eq:para-ratio-1}
\min
\frac{\mathbb{E}\left[\left.\sum_{t=t_f^{n}}^{t=t_{f+1}^{n}-1}\left(VW_n(t)+Ve_nH_n(t)+Vg_n(t)-Q_n(t_f^{n})\mu_n(t)H_n(t)\right)\right|~Q_n(t_f^{n})\right]}{\expect{T_n[f]~\left|~Q_n(t_f^{n})\right.}}
 \end{equation}
 \item Update $Q_n(t)$:
 \[
Q_n(t+1)=\max\left\{Q_n(t)+R_n(t)-\mu_n(t)H_n(t),~0\right\}.
 \]
 \end{itemize}
 \end{Alg}
 \end{algorithm}

 \subsection{Solving \eqref{eq:lb} and \eqref{eq:para-ratio-1}}\label{front-section}
 Note first that in Algorithm \ref{alg:temp-data-center}, the solution to problem \eqref{eq:lb} admits a simple thresholding rule (with shortest queue ties broken arbitrarily):
\begin{equation}\label{reject_decision}
r(t)=\left\{
       \begin{array}{lll}
         \max\{\lambda(t)-R_{\max},~0\},&\hbox{if $\exists n\in\mathcal{N}$ s.t.} \\
         &\hbox{ $Q_n(t)\leq Vc(t)$;}\\
         \lambda(t),&\hbox{otherwise.}
       \end{array}
     \right.
\end{equation}
\begin{equation}\label{routing_decision}
R_n(t)=\left\{
         \begin{array}{lll}
           \min\{\lambda(t), R_{\max}\}, & \hbox{if $Q_n(t)$ is the shortest } \\
                                         &\hbox{queue and $Q_n(t)\leq Vc(t)$;}\\
           0, & \hbox{otherwise.}
         \end{array}
       \right.
\end{equation}

Next, for the problem \eqref{eq:para-ratio-1}, 
recall the definition of $T_n[f]$ and $\alpha_n[f]\in\{active\}\cup\mathcal{L}_n$. If the server chooses to remain active, then the frame length is exactly 1, otherwise, the server is allowed to choose how long it stays in idle with $\expect{T_n[f] | Q(t_f^{n})} = I_{n}[f]+ m_{\alpha_n[f]}+1$, where $I_n[f]\in\left\{1,\cdots,I_{\max}\right\}$.
It can be easily shown that over all randomized decisions between staying active and going to different idle states, it is optimal to make a pure decision which either stays active or goes to one of the idle states with probability 1.

More specifically, let 
\begin{equation}\label{eq:overline-Dn}
\overline D_n[f] = 
\frac{\mathbb{E}\left[\left.\sum_{t=t_f^{n}}^{t=t_{f+1}^{n}-1}\left(VW_n(t)+Ve_nH_n(t)+Vg_n(t)-Q_n(t_f^{n})\mu_n(t)H_n(t)\right)\right|~Q_n(t_f^{n})\right]}{\expect{T_n[f]~\left|~Q_n(t_f^{n})\right.}}.
\end{equation}
We have when the server $n$ chooses to be active, then
\begin{equation}\label{eq:temp-active-alg}
\overline D_n[f]  = Ve_n-Q_n(t_f^{n})\mu_n.
\end{equation}
Otherwise, choosing a specific idle option $\alpha_n[f]\in\mathcal{L}_n$ gives
\begin{equation}\label{eq:temp-idel-alg}
\overline D_n[f] = 
\frac{V\hat{W}_n(\alpha_n[f])m_{\alpha_n[f]}+Ve_n-Q_n(t_f^{n})\mu_n+\frac{B_0}{2}\sigma_{\alpha_n[f]}^2+V\hat{g}(\alpha_n[f])I_n[f]}
{I_n[f]+m_{\alpha_n[f]}+1},
\end{equation}
which follows from the fact that if the server goes idle, then, $H_n(t)$ are all zero during the frame except for the last slot. 
Then, solving \eqref{eq:para-ratio-1} is equivalent to choosing one option which achieves a smaller value of 
$\overline D_n[f] $ between \eqref{eq:temp-active-alg} and \eqref{eq:temp-idel-alg}.

A closer look at the optimization problem \eqref{eq:temp-idel-alg} indicates that the best idle period $I_n[f]$ solving \eqref{eq:temp-idel-alg} is either 1 or $I_{\max}$. This is unfortunately problematic for the application of data center since it means the server is either not idle at all or going to idle for a very long time. When the arrival task stream is of high volatility, this could cause significant delay. In the next section, we will introduce our proposed algorithm for the servers which makes relatively ``smooth'' decisions.

\subsection{The proposed online control algorithm}\label{section_proposed_algorithm}


Our main idea pushing the server away from the binary decision is to add a term in the ratio \eqref{eq:overline-Dn} which is quadratic on the renewal frame length. 
Specifically, for server $n$, at the beginning of its $f$-th renewal frame $t_f^{n}$, it observes its current queue state $Q(t_f^{n})$ and makes decisions on $\alpha_n[f]\in\{active\}\cup\mathcal{L}_n$ and $I_n[f]$ so as to solve the minimization of ratio of expectations in \eqref{server_decision} as follows:
\begin{multline}\label{server_decision}
D_n[f]\triangleq \\
\frac{\mathbb{E}\left[\left.\sum_{t=t_f^{n}}^{t=t_{f+1}^{n}-1}\left(VW_n(t)+Ve_nH_n(t)+Vg_n(t)-Q_n(t_f^{n})\mu_n(t)H_n(t)\right)
+\left(t-t_f^{n}\right)B_0~\right|~Q_n(t_f^{n})\right]}{\expect{T_n[f]~\left|~Q_n(t_f^{n})\right.}}.
\end{multline}
where $B_0=\frac{1}{2}(R_{\max}+\mu_{\max})\mu_{\max}$.
Compared to the objective \eqref{eq:overline-Dn}, the quantity $D_n[f]$ has an extra term
$\sum_{t=t_f^{n}}^{t=t_{f+1}^{n}-1}\left(t-t_f^{n}\right)B_0 = \frac{T_n[f](T_n[f]-1)}{2}B_0$  on the numerate that is quadratic in $T_n[f]$.

Similar to the last section, we are then able to simplify the problem by computing $D_n[f]$ for active and idle options separately.
\begin{itemize}
  \item If the server chooses to go active, i.e. $\alpha_n[f]=active$, then,
\begin{equation}\label{DPP_active}
D_n[f]=Ve_n-Q_n(t_f^{n})\mu_n.
\end{equation}
  \item If the server chooses to go idle, i.e. $\alpha_n[f]\in\mathcal{L}_n$, then, 
 \begin{multline}\label{DPP_idle_origin}
D_n[f]=\\
\frac{V\hat{W}_n(\alpha_n[f])m_{\alpha_n[f]}+Ve_n-Q_n(t_f^{n})\mu_n+\expect{V\hat{g}(\alpha_n[f])I_n[f]+
  \frac{B_0}{2}T_n[f](T_n[f]-1)\left|~Q_n(t_f^{n})\right.}}
{\expect{T_n[f]~\left|~Q_n(t_f^{n})\right.}}
\end{multline}

which follows from the fact that if the server goes idle, then, $H_n(t)$ are all zero during the frame except for the last slot. Now we try to compute the optimal idle option $\alpha_n[f]\in\mathcal{L}_n$ and idle time length $I_n[f]$ given the server chooses to go idle. The following lemma illustrates that the decision on $I_n[f]$ can also be reduced to pure decision.
\begin{lemma}\label{compute_idle}
The best decision minimizing \eqref{DPP_idle_origin} is a pure decision which takes one $\alpha_n[f]\in\mathcal{L}_n$ and one integer value $I_n[f]\in\left\{1,\cdots,I_{\max}\right\}$ minimizing the deterministic function:
\begin{multline}\label{DPP_idle}
D_n[f] = \frac{V\hat{W}_n(\alpha_n[f])m_{\alpha_n[f]}+Ve_n-Q_n(t_f^{n})\mu_n+\frac{B_0}{2}\sigma_{\alpha_n[f]}^2+V\hat{g}(\alpha_n[f])I_n[f]}
{I_n[f]+m_{\alpha_n[f]}+1}\\
+\frac{B_0}{2}(I_n[f]+m_{\alpha_n[f]}+1).
\end{multline}
\end{lemma}
The proof of above lemma is given in appendix A.
\end{itemize}
Then, the server computes the minimum of \eqref{DPP_idle}, which is nothing but a deterministic optimization problem. It
 goes in the following two steps:
\begin{enumerate}
  \item For each $\alpha_n\in\mathcal{L}_n$, first differentiating \eqref{DPP_idle} with respect to $I[f]$ to get a real minimizer. Then, choosing $I[f]$ as one of the two integer values bracketing the real minimizer which achieves a smaller value on \eqref{DPP_idle}.

%

\item Compare \eqref{DPP_idle} for different $\alpha_n\in\mathcal{L}_n$ and choose the one achieving the minimum.
\end{enumerate}
Thus, the server compares \eqref{DPP_active} with the minimum of \eqref{DPP_idle}. If \eqref{DPP_active} is less than the minimum of \eqref{DPP_idle}, then, the server chooses to go active. Otherwise, the server chooses to go idle and stay idle for $I_n[f]$ time slots.

Overall, our final algorithm is summarized in Algorithm \ref{alg:data-center-alg}.
\begin{algorithm}
\begin{Alg}~\label{alg:data-center-alg}
\begin{itemize}
  \item At each time slot $t$, the data center observes $\lambda(t)$, $c(t)$, and $\mathbf{Q}(t)$ chooses rejection decision $r(t)$ according to \eqref{reject_decision} and chooses routing decision $R_n(t)$ according to \eqref{routing_decision}.
  \item For each server $n\in\mathcal{N}$, at the beginning of its $f$-th frame $t_f^{n}$, observe its queue state $Q_n(t_f^{n})$ and compute \eqref{DPP_active} and the minimum of \eqref{DPP_idle}. If \eqref{DPP_active} is less than the minimum of \eqref{DPP_idle}, then the server still stays active. Otherwise, the server switches to the idle state minimizing \eqref{DPP_idle} and stays idle for $I_n[f]$ achieving the minimum of \eqref{DPP_idle}.
  \item Update $Q_n(t),~\forall n\in\mathcal{N}$ according to
   \[
Q_n(t+1)=\max\left\{Q_n(t)+R_n(t)-\mu_n(t)H_n(t),~0\right\}.
 \]
\end{itemize}
\end{Alg}
\end{algorithm}

\section{Probability 1 Performance Analysis of Algorithm \ref{alg:data-center-alg}}\label{section_performance_analysis}
In this section, we prove a probability 1 convergence result for the proposed algorithm (Algorithm \ref{alg:data-center-alg}).
More specifically, we prove the online algorithm introduced in the last section makes all request queues $Q_n(t)$ bounded (on the order of $V$) and achieves the near optimality with sub-optimality gap on the order of $1/V$ with probability 1.

\subsection{Bounded request queues}
In this section, we show that the request queues are deterministically bounded due to the special thresholding nature of the admission control. Such a result is stronger (yet simpler) than the expected virtual queue analysis presented in the last section.
\begin{lemma}\label{bounded_delay}
If $Q_n(0)=0,~\forall n\in\mathcal{N}$, then, each request queue $Q_n(t)$ is deterministically bounded with bound:
$Q_n(t)\leq Vc_{\max}+R_{\max},~\forall t,~\forall n\in\mathcal{N}$,
where $c_{\max}\triangleq \max_{c\in\mathcal{C}}c$.
\end{lemma}
\begin{proof}
We use induction to prove the claim. Base case is trivial since $Q_n(0)=0\leq Vc_{\max}+R_{\max}$. Suppose the claim holds at the beginning of $t=i$ for $i>0$, so that
$Q_n(i)\leq Vc_{\max}+R_{\max}.$
Then,
\begin{enumerate}
  \item If $Q_n(i)\leq Vc_{\max}$, then, it is possible for the queue to increase during slot $i$. However, the increase of the queue within one slot is bounded by $R_{\max}$. which implies at the beginning of slot $i+1$,
      $Q_n(i+1)\leq Vc_{\max}+R_{\max}.$
  \item If $Vc_{\max} < Q_n(i)\leq Vc_{\max}+R_{\max}$, then, according to \eqref{routing_decision}, it is impossible to route any request to server $n$ during slot $i$, and $R_n(i)=0$ which results in
      $Q_n(i+1)\leq Vc_{\max}+R_{\max}.$
\end{enumerate}
Above all, we finished the proof of lemma.
\end{proof}

\begin{lemma}\label{constraint_satisfy}
The proposed algorithm meets the constraint \eqref{obj_4} with probability 1.
\end{lemma}
\begin{proof}
From the queue update rule \eqref{queue_update}, it follows,
$Q_n(t+1)\geq Q_n(t)+R_n(t)-\mu_nH_n(t)$.
Taking telescoping sums from 0 to $T-1$ gives
$Q_n(T)\geq Q_n(0)+\sum_{t=0}^{T-1}R_n(t)-\sum_{t=0}^{T-1}\mu_nH_n(t)$.
Since $Q_n(0)=0$, dividing both sides by $T$ gives
$\frac{Q_n(T)}{T}\geq \frac{1}{T}\sum_{t=0}^{T-1}R_n(t)-\frac{1}{T}\sum_{t=0}^{T-1}\mu_nH_n(t)$.
Substitute the bound $Q_n(T)\leq Vc_{\max}+R_{\max}$ from lemma \ref{bounded_delay} into above inequality and take limit as $T\rightarrow\infty$ give the desired result.
\end{proof}
\subsection{Optimal randomized stationary policy}\label{op_stat_algorithm}
In this section, we introduce a class of algorithms which are theoretically helpful for doing analysis, but practically impossible to implement.

Since servers are coupled only through time average constraint \eqref{obj_4}, each server $n$ can be viewed as a separate renewal system, thus, it can be shown that any possible time average service rate $\overline{\mu}_n$ can be achieved through a frame based stationary randomized service decision, meaning that the decisions are i.i.d. over frames.
Furthermore, it can be shown that the optimality of \eqref{obj_3}-\eqref{obj_5} can be achieved over the following randomized stationary algorithms: At the beginning of each time slot $t$, the data center observes the incoming requests $\lambda(t)$ and rejecting cost $c(t)$, then routes $R_n^*(t)$ incoming requests to server $n$ and rejects $d^*(t)$ requests, both of which are random functions of $(\lambda(t),c(t))$. They satisfy the same instantaneous relation as \eqref{obj_5}.
Meanwhile, server $n$ chooses a frame based stationary randomized service decision $(\alpha_n^*[f],I_n^*[f])$, so that the optimal service rate is achieved.

If one knows the stationary distribution for $(\lambda(t),c(t))$, then, this optimal control algorithm can be computed using dynamic programming or linear programming. Moreover, the optimal setup cost $W_n^*(t)$, idle cost $g_n^*(t)$, and the active state indicator $H^*(t)$ can also be deduced.
Since the algorithm is stationary, these three cost processes are all ergodic Markov processes. Let $T_n^*[f]$ be the frame length process under this algorithm. Thus, it follows from the renewal reward theorem that
$\left\{\sum_{t=t^{n}_f}^{t^{n}_{f+1}-1}W_n^*(t)\right\}_{f=0}^{+\infty}$,
$\left\{\sum_{t=t^{n}_f}^{t^{n}_{f+1}-1}g_n^*(t)\right\}_{f=0}^{+\infty}$,
$\left\{\sum_{t=t^{n}_f}^{t^{n}_{f+1}-1}e_nH_n^*(t)\right\}_{f=0}^{+\infty}$, $\left\{\sum_{t=t^{n}_f}^{t^{n}_{f+1}-1}\mu_n(t)H_n^*(t)\right\}_{f=0}^{+\infty}$
and
$\left\{T_n^*[f]\right\}_{f=0}^{+\infty}$
are all \emph{i.i.d. random variables over frames}. Let $\overline{C}^*$, $\overline{W}_n^*$, $\overline{G}_n^*$ and $\overline{E}_n^*$ be the optimal time average costs. Let $\overline{R}_n^*$, $\overline{\mu}_n^*$ and $\overline{d}^*$ be the optimal time average routing rate, service rate and rejection rate respectively. Then, by the strong law of large numbers,
\begin{equation}
\overline{W}_n^*
=\frac{\expect{\sum_{t=t^{n}_f}^{t^{(n)}_{f}+T_n^*[f]-1}W_n^*(t)}}{\expect{T_n^*[f]}} \label{iid_W}
\end{equation}
\begin{equation}
\overline{E}_n^*
=\frac{\expect{\sum_{t=t^{n}_f}^{t^{(n)}_{f}+T_n^*[f]-1}e_nH_n^*(t)}}{\expect{T_n^*[f]}} \label{iid_E}
\end{equation}
\begin{equation}
\overline{G}_n^*
=\frac{\expect{\sum_{t=t^{n}_f}^{t^{(n)}_{f}+T_n^*[f]-1}g_n^*(t)}}{\expect{T_n^*[f]}} \label{iid_G}
\end{equation}
\begin{align}
\overline{\mu}_n^*
=\frac{\expect{\sum_{t=t^{n}_f}^{t^{(n)}_{f}+T_n^*[f]-1}\mu_n(t)H_n^*(t)}}{\expect{T_n^*[f]}}, \label{iid_mu}
\end{align}
Also, notice that $R_n^*(t)$ and $d^*(t)$ depend only on the random variables $\lambda(t)$ and $c(t)$, which is i.i.d. over slots. Thus, $R_n^*(t)$ and $d^*(t)$ are also \emph{i.i.d. random variables over slots}. By the law of large numbers,
\begin{align}
\overline{R}_n^*=&\expect{R_n^*(t)}, \label{iid_R}\\
\overline{C}^*=&\expect{c(t)d^*(t)}. \label{iid_d}
\end{align}

\begin{remark}
Since the idle time $I_n^*[f]\in[1,I_{\max}]$ and the first two moments of the setup time are bounded, it follows the first two moments of $T_n^*[f]$ are bounded.
\end{remark}


\subsection{Key features of thresholding algorithm}
In this part, we compare the algorithm deduced from the two optimization problems \eqref{eq:lb} and \eqref{server_decision} to that of the best stationary algorithm in section \ref{op_stat_algorithm}, illustrating the key features of the proposed online algorithm.
Define $\mathcal{F}(t)$ as the system history up till slot $t$, which includes all the decisions taken and all the random events before slot $t$. We first consider \eqref{eq:lb}. For simplicity of notations, define two random processes $\{X_n[f]\}_{f=0}^{\infty}$ and $\{Z[t]\}_{t=0}^{\infty}$ as follows
\begin{align*}
X_n[f]=&\sum_{t=t_f^{n}}^{t=t_{f+1}^{n}-1}\left(V\left(W_n(t)-\overline{W}_n^*\right)+V\left(e_nH_n(t)-\overline{E}_n^*\right)\right.\\
       &\left.+V\left(g_n(t)-\overline{G}_n^*\right)-Q_n(t_f^{n})\left(\mu_nH_n(t)-\overline{\mu}^*\right)+\left(t-t_f^{n}\right)B_0-\Psi_n\right),\\
Z[t]=&V\left(c(t)r(t)-\overline{C}^*\right)+\sum_{n=1}^NQ_n(t)\left(R_n(t)-\overline{R}_n^*\right),
\end{align*}
where $\Psi_n=\frac{B_0}{2}\frac{\expect{T^*_n[f](T^*_n[f]-1)}}{\expect{T^*_n[f]}}$ and $B_0=\frac{1}{2}(R_{\max}+\mu_{\max})\mu_{\max}$.

Given the system information $\mathcal{F}(t)$, the random events $c(t)$ and $\lambda(t)$, the solutions \eqref{reject_decision} and \eqref{routing_decision} take rejecting and routing decisions so as to minimize \eqref{eq:lb} over all possible routing and rejecting decisions at time slot $t$. Thus, the proposed algorithm achieves smaller value on \eqref{eq:lb} compared to that of the best stationary algorithm in section \ref{op_stat_algorithm}. Formally, this idea can be stated as the following inequality:
$\expect{\left.Vc(t)r(t)+\sum_{n=1}^NQ_n(t)R_n(t)~\right|~c(t),\lambda(t),\mathcal{F}(t)}$
$\leq \expect{\left.Vc(t)d^*(t)+\sum_{n=1}^NQ_n(t)R_n^*(t)~\right|~c(t),\lambda(t),\mathcal{F}(t)}$.
Taking expectation regarding $c(t)$ and $\lambda(t)$
using the fact that the best stationary algorithm on $R^*_n(t)$ and $d^*(t)$ are i.i.d. over slots (independent of $\mathcal{F}(t)$), together with \eqref{iid_R} and \eqref{iid_d}, we get
\begin{equation}\label{front_end_feature}
\expect{\left.Z(t)~\right|~\mathcal{F}(t)}\leq 0.
\end{equation}
Similarly, for \eqref{server_decision}, the proposed service decisions within frame $f$ minimize $D_n[f]$ in \eqref{server_decision}, thus, compared to the best stationary policy, the inequality \eqref{op_stat_interim} holds.
\begin{figure*}
\begin{align}\label{op_stat_interim}
&\frac{\mathbb{E}\left[\left.\sum_{t=t_f^{n}}^{t=t_{f+1}^{n}-1}\left(V(W_n(t)+e_nH_n(t)+g_n(t))
        -Q_n(t_f^{n})\mu_n(t)H_n(t)\right)+\left(t-t_f^{n}\right)B_0~\right|~\mathcal{F}(t_f^{n})\right]}
{\expect{T_n[f]~\left|~\mathcal{F}(t_f^{n})\right.}}\nonumber\\
\leq&
\frac{\expect{\left.\sum_{t=t_f^{n}}^{t=t_{f}^{(n)}+T_n^*[f]-1}\left(V\left(W_n^*(t)+e_nH_n^*(t)+g_n^*(t)\right)-Q_n(t_f^{n})\mu_nH_n^*(t)\right)
+\frac{B_0}{2}T^*_n[f](T^*_n[f]-1)~\right|~\mathcal{F}(t_f^{n})}}
{\expect{T^*_n[f]~\left|~\mathcal{F}(t_f^{n})\right.}}
\end{align}
\end{figure*}
Again, using the fact that the optimal stationary algorithm gives i.i.d. $W_n^*(t)$, $g_n^*(t)$, $H_n^*(t)$ and $T_n^*[f]$ over frames (independent of $\mathcal{F}(t_f^{n})$),  as well as \eqref{iid_W}, \eqref{iid_E} and \eqref{iid_mu}, we get
\begin{align}
\expect{X_n[f]~\left|~\mathcal{F}(t_f^{n})\right.}
\left/\expect{T_n[f]~\left|~\mathcal{F}(t_f^{n})\right.}\right.\leq0 \label{server_feature}
\end{align}

\subsection{Bounded average of supermartingale difference sequeces}
The key feature inequalities \eqref{front_end_feature} and \eqref{server_feature} provide us with bounds on the expectations. The following lemma serves as a stepping stone passing from expectation bounds to probability 1 bounds.  Recall the basic definition of supermartingale in Definition \ref{def:sup-MG}.
We have the following strong law of large numbers for supermartingale difference sequences:
\begin{lemma}[Corollary 4.2 of \cite{neely2012stability}]\label{prob-1-converge}
Let $\{X_t\}_{t=0}^\infty$ be a supermartingale difference sequence. If 
$$\sum_{t=1}^\infty \left.\expect{X_t^2}\right/t^2<\infty,$$
then,
\[\limsup_{T\rightarrow\infty}\frac1T\sum_{t=0}^{T-1}X_t\leq 0,\]
with probability 1.
\end{lemma}

With this lemma, we are ready to prove the following result:
\begin{lemma}\label{bounded_supMG}
Under the proposed algorithm, the following hold with probability 1,
\begin{align}
&\limsup_{F\rightarrow\infty}\frac{1}{F}\sum_{f=0}^{F-1}X_n[f]\leq0, \label{prob_1_server}\\
&\limsup_{T\rightarrow\infty}\frac{1}{T}\sum_{t=0}^{T-1}Z[t]\leq0. \label{prob_1_front_end}
\end{align}
\end{lemma}
\begin{proof}
The key to the proof is treating these two sequences as supermartingale difference sequences and applying law of large numbers for supermartingale difference sequences (theorem 4.1 and corollary 4.2 in \eqref{prob_1_server}).

We first look at the sequence $\{X_n[f]\}_{f=0}^{\infty}$. Let
$Y_n[F]=\sum_{f=0}^{F-1}X_n[f]$.
We first prove that $Y_n[F]$ is a supermartingale. Notice that $Y_n[F]\in\mathcal{F}\left(t_n^{(F)}\right)$, i.e. it is measurable given all the information before frame $t_n^{(F)}$, and $|Y_n[F]|<\infty,~\forall F<\infty$. Furthermore,
$\expect{Y_n[F+1]-Y_n[F]~\left|~\mathcal{F}\left(t_f^{n}\right)\right.}$
$=\expect{X_n[F]~\left|~\mathcal{F}\left(t_f^{n}\right)\right.}$
$\leq 0\cdot\expect{T_n[F]~\left|~\mathcal{F}\left(t_f^{n}\right)\right.}=0$,
where the only inequality follows from \eqref{server_feature}. Thus, it follows $Y_n[F]$ is a supermartingale. Next, we show that the second moment of supermartingale differences, i.e. $\expect{X_n[f]^2}$, is deterministically bounded by a fixed constant for any $f$. This part of proof is given in Appendix B. Thus, the following holds:
$\sum_{f=1}^{\infty}\expect{X_n[f]^2}\left/f^2\right.<\infty$.
Now, applying Lemma \ref{prob-1-converge} immediately gives \eqref{prob_1_server}.

Similarly, we can prove \eqref{prob_1_front_end} by proving $M[t]=\sum_{t=0}^{T-1}Z[t]$ is a supermartingale with bounded second moment on differences using \eqref{iid_R}, \eqref{iid_d} and \eqref{front_end_feature}. The procedure is almost the same as above and we omitted the details here for brevity.
\end{proof}

\begin{corollary}\label{corollary_ratio_time_average}
The following ratio of time averages is upper bounded with probability 1, \\
$\limsup_{F\rightarrow\infty}\left.\sum_{f=0}^{F-1}X_n[f]\right/\sum_{f=0}^{F-1}T_n[f]\leq 0$.
\end{corollary}
\begin{proof}
From \eqref{prob_1_server}, it follows for any $\epsilon>0$, there exists an $F_0(\epsilon)$ such that $F\geq F_0(\epsilon)$ implies
$\sum_{f=0}^{F-1}X_n[f]\left/\sum_{f=0}^{F-1}T_n[f]\right.\leq\epsilon\left/\frac{1}{F}\sum_{f=0}^{F-1}T_n[f]\right.\leq\epsilon$.
Thus,
$\limsup_{F\rightarrow\infty}\sum_{f=0}^{F-1}X_n[f]\left/\sum_{f=0}^{F-1}T_n[f]\right.\leq\epsilon$.
Since $\epsilon$ is arbitrary, take $\epsilon\rightarrow0$ gives the result.
\end{proof}

\subsection{Near optimal time average cost}
The ratio of time averages in corollary \ref{corollary_ratio_time_average} and the true time average share the same bound, which is proved by the following lemma:
\begin{lemma}\label{true_time_average}
The following time average is bounded with probability 1,
\begin{multline}
\limsup_{T\rightarrow\infty}\frac{1}{T}\sum_{t=0}^{T-1}\left(V\left(W_n(t)+e_nH_n(t)+g_n(t)\right)-Q_n(t_f^{n})(\mu_nH_n(t)-\overline{\mu}_n^*)+\left(t-t_f^{n}\right)B_0\right) \\
\leq V(\overline{W}_n^*+\overline{E}_n^*+\overline{G}_n^*)+\Psi_n,  \label{true_time_average_equation}
\end{multline}
where $\Psi_n=\frac{B_0}{2}\frac{\expect{T^*_n[f](T^*_n[f]-1)}}{\expect{T^*_n[f]}}$ and $B_0=\frac{1}{2}(R_{\max}+\mu_{\max})\mu_{\max}$.
\end{lemma}

The idea of the proof is similar to that of basic renewal theory, which derives upper and lower bounds for each $T$ within any frame $F$ using corollary \ref{corollary_ratio_time_average}, thereby showing that as $T\rightarrow\infty$, the upper and lower bounds meet. See appendix C for details. With the help of this lemma, we are able to prove the following near optimal performance theorem:

\begin{theorem}\label{theorem_near_optimal_perform}
If $Q_n(0)=0,\forall n\in\mathcal{N}$, then the time average total cost under the algorithm is near optimal on the order of $\mathcal{O}(1/V)$, i.e. with probability 1,
\begin{multline}\label{near_optimal_perform}
\limsup_{T\rightarrow\infty}\frac{1}{T}\sum_{t=0}^{T-1}\left(c(t)r(t)+\sum_{n=1}^{N}\left(W_n(t)+e_nH_n(t)+g_n(t)\right)\right)
\\
\leq\underbrace{\overline{C}^*+\sum_{n=1}^N\left(\overline{W}_n^*+\overline{E}_n^*+\overline{G}_n^*\right)}_\text{Optimal cost}+\frac{\sum_{n=1}^N\Psi_n+B_3}{V},
\end{multline}
where $B_3\triangleq\frac{1}{2}\sum_{n=1}^N(R_{\max}+\mu_n)^2$, $\Psi_n=\frac{B_0}{2}\frac{\expect{T^*_n[f](T^*_n[f]-1)}}{\expect{T^*_n[f]}}$ and $B_0=\frac{1}{2}(R_{\max}+\mu_{\max})\mu_{\max}$.
\end{theorem}
See appendix D for details of proof.

\section{Delay improvement via virtualization}\label{section_improve}
\subsection{Delay improvement}
The algorithm in previous sections optimizes time average cost. However, it can route requests to idle queues, which increases system delay. This section considers an improvement in the algorithm that maintains the same average cost guarantees, but reduces delay. This is done by a ``virtualization'' technique that reduces from $N$ server request queues to only one request queue $Q(t)$. Specifically, the same Algorithm 1 is run, with queue updates \eqref{queue_update} for each of the $N$ queues $Q_n(t)$. However, the $Q_n(t)$ processes are now virtual queues rather than actual queues: Their values are only kept in software. Every slot $t$, the data center observes the incoming requests $\lambda(t)$, rejection cost $c(t)$ and virtual queue values, making rejection decision according to \eqref{reject_decision} as before. The admitted requests are queued in $Q(t)$. Meanwhile, each server $n$ makes active/idle decisions observing its own virtual queue $Q_n(t)$ same as before. Whenever a server is active, it grabs the requests from request queue $Q(t)$ and serves them.
This results in an actual queue updating for the system:
\begin{equation}\label{actual_queue_update}
Q(t+1)=\max\left\{Q(t)+\lambda(t)-r(t)-\sum_{n=1}^N\mu_n(t)H_n(t),~0\right\}.
\end{equation}

Fig. \ref{fig:Stupendous1} shows this data center architecture.
\begin{figure}[htbp]
   \centering
   \includegraphics[height=2in]{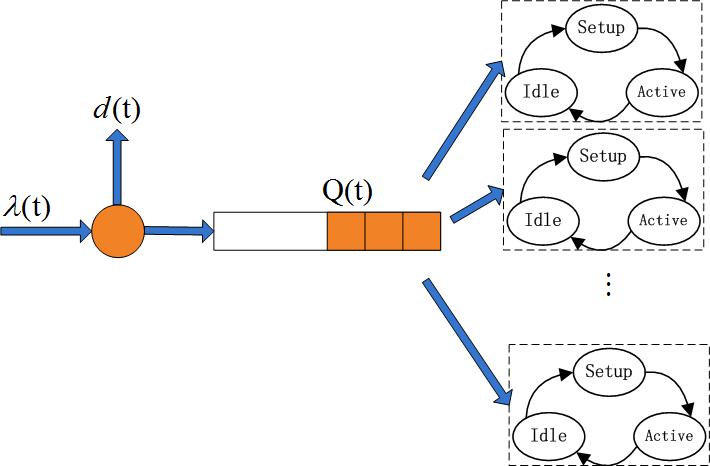} 
   \caption{Illustration of basic data center architecture.}
   \label{fig:Stupendous1}
\end{figure}


\subsection{Performance guarantee}
Since this algorithm does not look at the actual queue $Q(t)$, it is not clear whether or not the actual request queue would be stabilized under the proposed algorithm. The following lemma answers the question. For simplicity, we call the system with $N$ queues, where our algorithm applies, the virtual system, and call the system with only one queue the actual system.
\begin{lemma}\label{actual_queue_bound}
If $Q(0)=0$ and $Q_n(0)=0,~\forall n\in\mathcal{N}$, then the virtualization technique stabilizes the queue $Q(t)$ with the bound:
$Q(t)\leq N(Vc_{\max}+R_{\max})$.
\end{lemma}
\begin{proof}
Notice that this bound is $N$ times the individual queue bound in lemma \ref{bounded_delay}, we prove the lemma by showing that the sum-up weights $\sum_{n=1}^NQ_n(t)$ in the virtual system always dominates the queue length $Q(t)$. We prove this by induction. The base case is obvious since $Q(0)=\sum_{n=1}^NQ_n(0)=0$. Suppose at the beginning of time $t$,
$Q(t)\leq\sum_{n=1}^NQ_n(t)$,
then, during time $t$, we distinguish between the following two cases:
\begin{enumerate}
  \item Not all active servers in actual system have requests to serve. This case happens if and only if there are not enough requests in $Q(t)$ to be served, i.e.
      $\lambda(t)-r(t)+Q(t) < \sum_{n=1}^N\mu_n(t)H_n(t)$.
      Thus, according to queue updating rule \eqref{actual_queue_update}, at the beginning of time slot $t+1$, there will be no request sitting in the actual queue, i.e. $Q(t+1)=0$. Hence, it is guaranteed that
      $Q(t+1)\leq\sum_{n=1}^NQ_n(t+1)$.
  \item All active servers in actual system have requests to serve. Notice that the virtual system and the actual system have exactly the same arrivals, rejections and server active/idle states. Thus, the following holds,
      $Q(t+1)=Q(t)+\lambda(t)-r(t)-\sum_{n=1}^N\mu_n(t)H_n(t)$
            $\leq \sum_{n=1}^NQ_n(t)+\sum_{n=1}^NR_n(t)-\sum_{n=1}^N\mu_n(t)H_n(t)$
            $\leq \sum_{n=1}^N\max\{Q_n(t)+R_n(t)-\mu_n(t)H_n(t),~0\}$
            $= \sum_{n=1}^NQ_n(t+1)$,
      where the first inequality follows from induction hypothesis as well as the fact that $\sum_{n=1}^NR_n(t)=\lambda(t)-r(t)$.
\end{enumerate}
Above all, we proved $Q(t)\leq\sum_{n=1}^NQ_n(t),~\forall t$. Since each $Q_n(t)\leq Vc_{\max}+R_{\max},~\forall t$, the lemma follows.
\end{proof}

Since the virtual system and the actual system have exactly the same cost, and it can be shown that the optimal cost in one queue system is lower bounded by the optimal cost in $N$ queue system, thus, the near optimal performance is still guaranteed.

\section{Simulation}\label{section_simulation}
{
In this section, we demonstrate the performance of our proposed algorithm via extensive simulations. The first simulation runs over i.i.d. traffic. We show that our algorithm indeed achieves $\mathcal{O}(1/V)$ near optimality with $\mathcal{O}(V)$ delay ($[\mathcal{O}(1/V), \mathcal{O}(V)]$ trade-off), which is predicted by Lemma \ref{bounded_delay} and Theorem \ref{theorem_near_optimal_perform}. We then apply our algorithm to a real data center traffic trace with realistic scale, setup time and cost being the power consumption. We compare the performance of the proposed algorithm with several other heuristic algorithms and show that our algorithm indeed delivers lower delay and saves power.
}

\subsection{Near optimality in $N$ queues system}
In the first simulation, we consider a relative small scale problem with i.i.d. generated traffic.
We set the number of servers $N=5$. The incoming requests $\lambda(t)$ are integers following a uniform distribution in $[10,30]$. The request rejecting cost $c(t)$ are also integers following a uniform distribution in $[1,6]$. The maximum admission amount $R_{\max}=40$ and the maximum idle time $I_{\max}=1000$.
There is only one idle option $\alpha_n$ for each server where the idle cost $\hat{g}(\alpha_n)=0$.
The setup time follows a geometric distribution with mean $\expect{\hat{\tau}(\alpha_n)}$, setup cost $\hat{W}_n(\alpha_n)$ per slot, service cost $e_n$ per slot, and the service amount $\mu_n$ follows a uniform distribution over integers. The values $1/\expect{\hat{\tau}(\alpha_n)}$ are generated uniform at random within $[0,1]$ and specified in table II.

The algorithm is run for 1 million slots in each trial and each plot takes the average of these 1 million slots. We compare our algorithm to the optimal stationary algorithm. The optimal stationary algorithm is computed using linear program \cite{fox1966markov} with the full knowledge of the statistics of requests and rejecting costs.

\begin{table}
\begin{center}
\caption{Problem parameters}
\begin{tabular}{|l|l|l|l|l|}
  \hline
   Server & $\mu_n$ & $e_n$ & $\hat{W}_n(\alpha_n)$ & $\expect{\hat{\tau}(\alpha_n)}$ \\
  1 & $\{2,3,4,5,6\}$ & 4 & 2 & 5.893 \\
  2 & $\{2,3,4\}$ & 2 & 3 & 4.342 \\
  3 & $\{2,3,4\}$ & 3 & 3 & 27.397 \\
  4 & $\{1,2,3\}$ & 4 & 2 & 5.817 \\
  5 & $\{2,3,4\}$ & 2 & 4 & 6.211 \\
  \hline
\end{tabular}
\end{center}
\end{table}

In Fig. \ref{fig:Stupendous2}, we show that as our tradeoff parameter $V$ gets
larger, the average cost approaches the optimal value and
achieves a near optimal performance. {Furthermore, the cost curve drops rapidly when $V$ is small and becomes relatively flat when $V$ gets large, thereby demonstrating our $\mathcal{O}(1/V)$ optimality gap in Theorem \ref{theorem_near_optimal_perform}.} 
Fig. \ref{fig:Stupendous3} plots the average sum-up queue size $\sum_{n=1}^5Q_n(t)$
and shows as $V$ gets larger, the average sum-up queue size becomes larger. We also plot the sum of individual queue bound
from Lemma \ref{bounded_delay} for comparison.  {We can see that the real queue size grows linearly with $V$ (although the constant in Lemma \ref{bounded_delay} is not tight due to the much better delay we obtain here), which demonstrates the $\mathcal{O}(V)$ delay bound.}

\begin{figure}[htbp]
   \centering
   \includegraphics[height=2in]{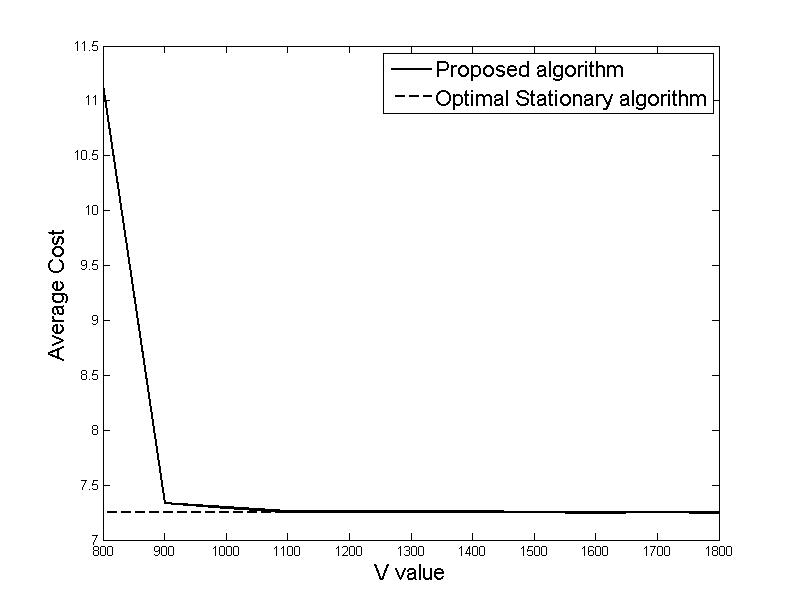} 
   \caption{Time average cost verses $V$ parameter over 1 millon slots.}
   \label{fig:Stupendous2}
\end{figure}

\begin{figure}[htbp]
   \centering
   \includegraphics[height=2in]{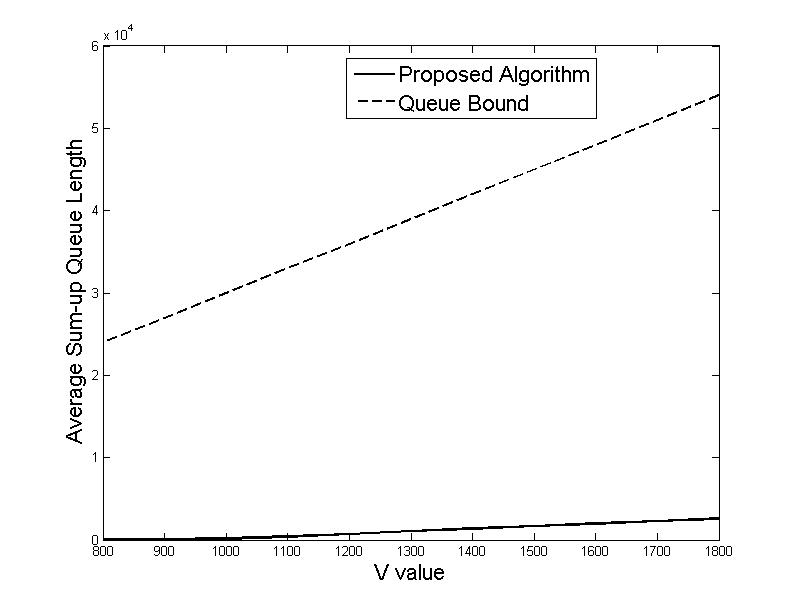} 
   \caption{Time average sum-up request queue length verses $V$ parameter over 1 millon slots.}
   \label{fig:Stupendous3}
\end{figure}

We then tune the requests $\lambda(t)$ to be uniform in $[20,40]$ and keep other parameters unchanged. In Fig. \ref{fig:Stupendous4}, we see that since the request rate gets larger, we need $V$ to be larger in order to obtain the near optimality, but still, the near optimality gap scales roughly $\mathcal{O}(1/V)$. Fig. \ref{fig:Stupendous5} gives the sum-up average queue length in this case. The average queue length is larger than that of Fig. \ref{fig:Stupendous3} with linear growth with respect to $V$.

\begin{figure}[htbp]
   \centering
   \includegraphics[height=2in]{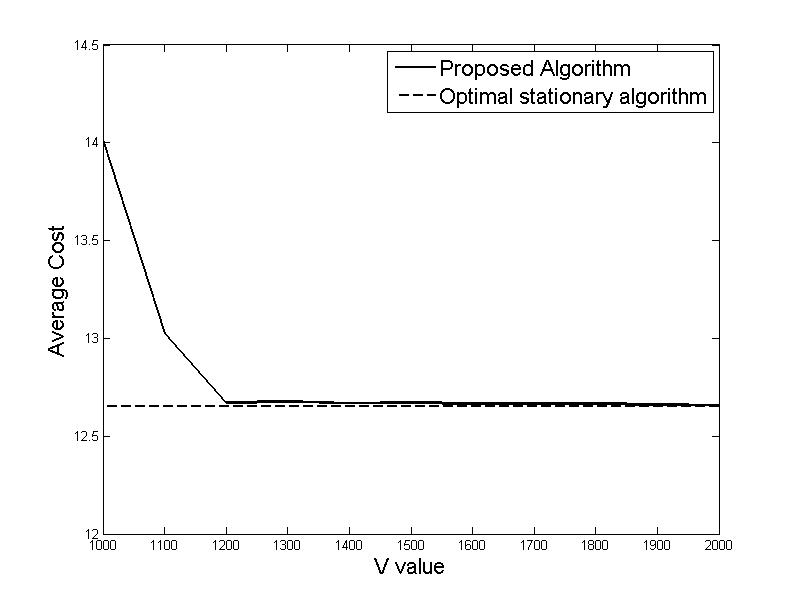} 
   \caption{Time average cost verses $V$ parameter over 1 millon slots.}
   \label{fig:Stupendous4}
\end{figure}

\begin{figure}[htbp]
   \centering
   \includegraphics[height=2in]{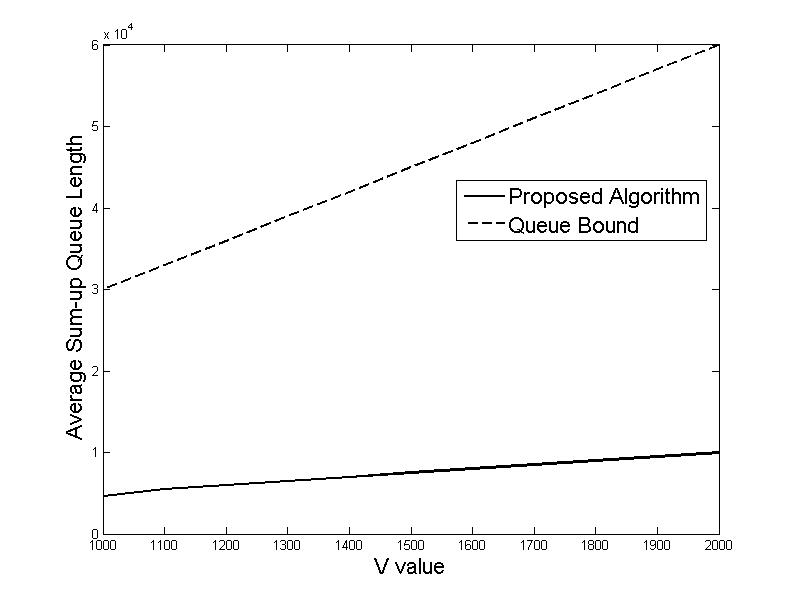} 
   \caption{Time average cost verses $V$ parameter over 1 millon slots.}
   \label{fig:Stupendous5}
\end{figure}

\subsection{Real data center traffic trace and performance evaluation}
This second considers a simulation on a real data center traffic obtained from the open source data sets of the paper \cite{benson2010network}. The trace is plotted in Fig. \ref{fig:trace}. We synthesize different data chunks from the source so that the trace contains both the steady phase and increasing phase. The total time duration is 2800 seconds with each slot equal to 20ms. The peak traffic is 2120 requests per 20 ms, and the time average traffic over this whole time interval is 654 requests per 20 ms.

We consider a data center consisting of 1060 homogeneous servers. We assume each server has only one sleep state and the service quantity of each server at each slot follows a Zipf's law\footnote{The pdf of Zipf's law with parameter $K,p$ is defined as:
$f(n; K,p)=\frac{1/n^p}{\sum_{i=1}^K1/i^p},~n=1,2,\cdots, K$. Thus, the mean of the distribution is $\frac{\sum_{i=1}^K1/i^{p-1}}{\sum_{i=1}^K1/i^p}$.} with parameter $K=10$ and $p=1.9$. This gives the service rate of each server equal to $1.9933\approx2$ requests per 20ms. So the full capacity of the data center is able to support the peak traffic.
Zipf's law is previously introduced to model a wide scope of physics, biology, computer science and social science phenomenon (\cite{newman2005power}), and is adopted in various literatures to simulate the empirical data center service rate (\cite{gandhi2013dynamic, gandhi2012sleep}). 
The setup time of each server is geometrically distributed with success probability equal to $0.001$. This gives the mean setup time 1000 slots (20 seconds). This setup time is previously shown in \cite{gandhi2012sleep} to be a typical time duration for a desktop to recover from the suspend or hibernate state. 

Furthermore, to make a fair comparison with several existing algorithms, we enforce the front end balancer to accept all requests at each time slot (so the rejection rate is always 0). The only cost in the system is then the power consumption. We assume that a server consumes 10 W each slot when active and 0 W each slot when idle. The setup cost is also 10 W per slot. Moreover, we apply the one queue model described in Section \ref{section_improve} for all the rest of the simulations. Following the problem formulation, the maximum idle time of a server for the proposed algorithm is $I_{\max}=5000$, while no such limit is imposed for any other benchmark algorithms.

We first run our proposed algorithm over the trace with virtualization (in Section \ref{section_improve}) for different $V$ values. We set the initial virtual queue backlog $Q_n(0)=2000~\forall n$, and keep 20 servers always on.
 Fig. \ref{fig:costV} and Fig. \ref{fig:queueV} plots the running average power consumption and corresponding queue length for $V=400,~600,~800$ and 1200, respectively. It can be seen that as $V$ gets large, the average power consumption does not improve too much but the queue length changes drastically. This phenomenon results from the $[\mathcal{O}(1/V), \mathcal{O}(V)]$ trade-off of our proposed algorithm. In view of this fact, we choose $V=600$ which gives a reasonable delay performance in Fig. \ref{fig:queueV}.

\begin{figure}[htbp]
   \centering
   \includegraphics[height=2in]{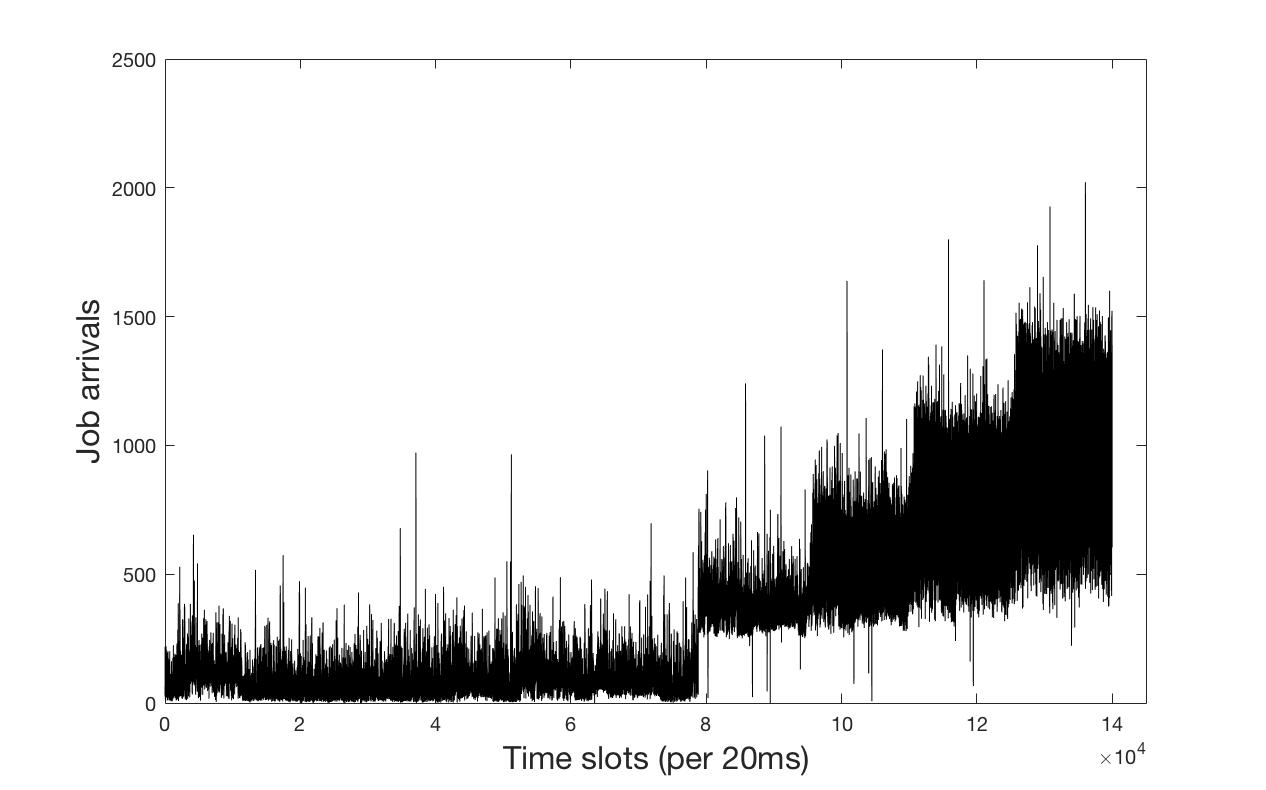} 
   \caption{Synthesized traffic trace from \cite{benson2010network}.}
   \label{fig:trace}
\end{figure}

\begin{figure}[htbp]
   \centering
   \includegraphics[height=2in]{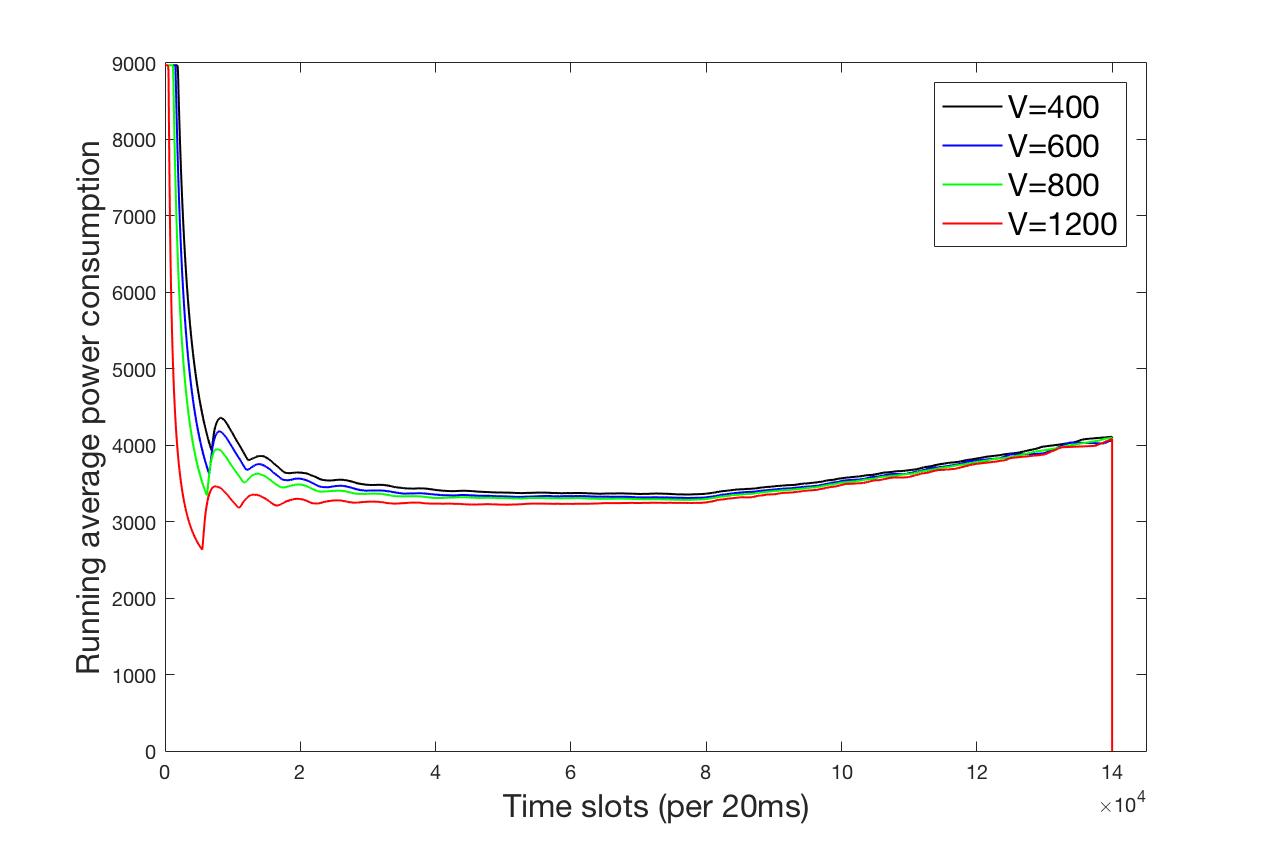} 
   \caption{Running average power consumption from slot 1 to the current slot for different $V$ value.}
   \label{fig:costV}
\end{figure}

\begin{figure}[htbp]
   \centering
   \includegraphics[height=2in]{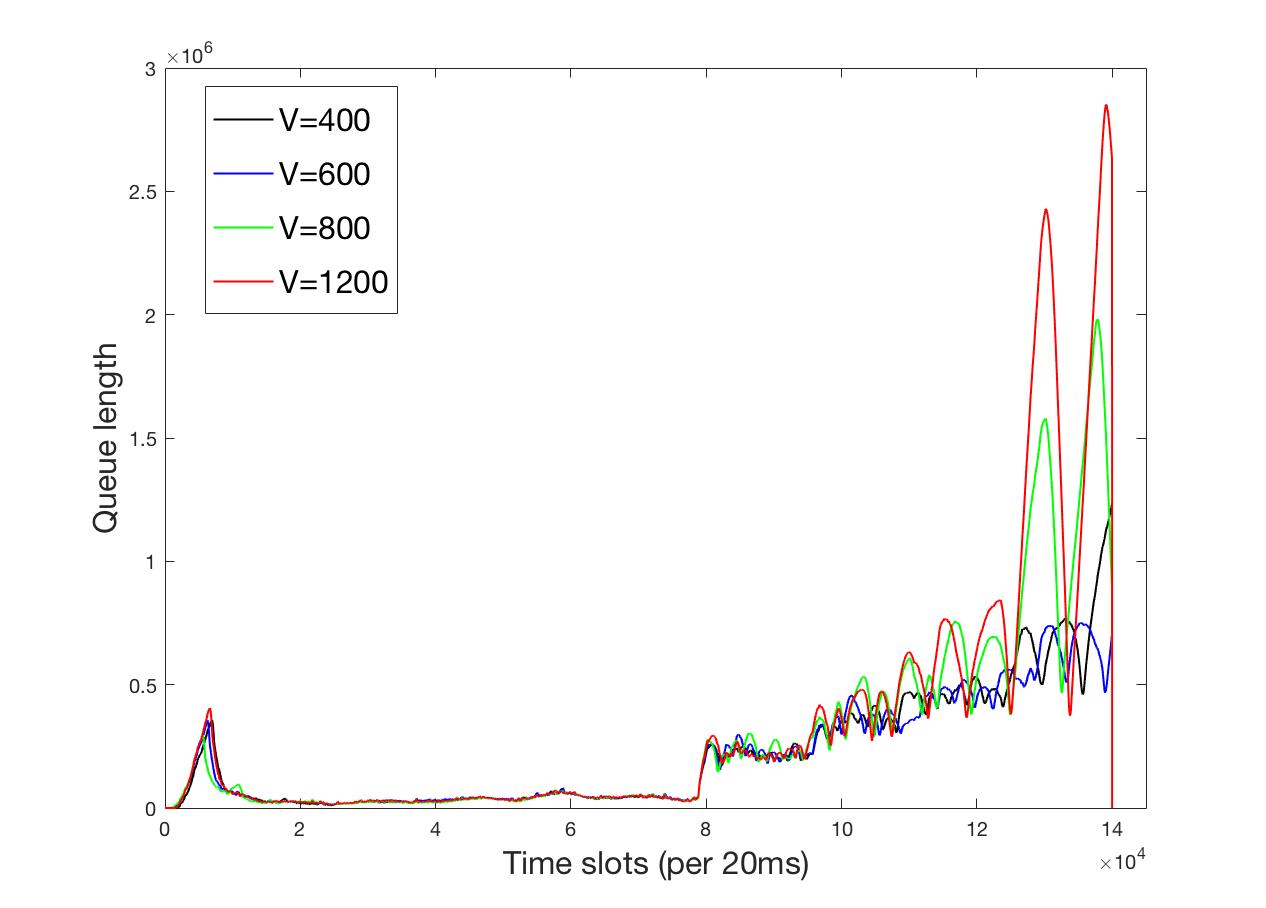} 
   \caption{Instantaneous queue length for different $V$ value.}
   \label{fig:queueV}
\end{figure}

Next, we compare our proposed algorithm with the same initial setup and $V=600$ to the following algorithms:
\begin{itemize}
\item Always-on with $N=327$ active servers and the rest servers staying on the sleep mode. Note that 327 servers can support the average traffic over the whole interval which is 654 requests per 20 ms. 
\item Always-on with full capacity. This corresponds to keeping all 1060 servers on at every slot.
\item Reactive. This algorithm is developed in \cite{gandhi2012sleep} which reacts to the current traffic     $\overline{\lambda}(t)$ and maintains $k_{react}(t)=\left\lceil\overline{\lambda}(t)/2\right\rceil$ servers on. In the simulation, we choose $\overline{\lambda}(t)$ to be the average of the traffic from the latest 10 slots. If the current active server $k(t)>k_{react}(t)$, then, we turn $k(t)-k_{react}(t)$ servers off, otherwise, we turn $k_{react}(t)-k(t)$ servers to the setup state.
\item Reactive with extra capacity. This algorithm is similar to Reactive except that we introduce a virtual traffic flow of $p$ jobs per slot. So during each time slot $t$, the algorithm maintains $k_{react}(t)=\left\lceil(\overline{\lambda}(t)+p)/2\right\rceil$ servers on. 
\end{itemize}
Fig. \ref{fig:Stupendous7}-\ref{fig:Stupendous9} plots the average power consumption, queue length and the number of active servers, respectively. It can be seen that all algorithms perform pretty well during first half of the trace. For the second half of the trace, the traffic load is increasing. The Always-on algorithm with mean capacity does not adapt to the traffic so the queue length blows up quickly.
Because of the long setup time, the number of active servers in the Reactive algorithm fails to catch up with the increasing traffic so the queue length also blows up. Our proposed algorithm minimizes the power consumption while stabilizing the queues, thereby outperforming both the Always-on and the Reactive algorithm. Note that the Reactive with extra 200 job capacity is able to achieve a similar delay performance as our proposed algorithm, but with significant extra power consumption.

\begin{figure}[htbp]
   \centering
   \includegraphics[height=2.3in]{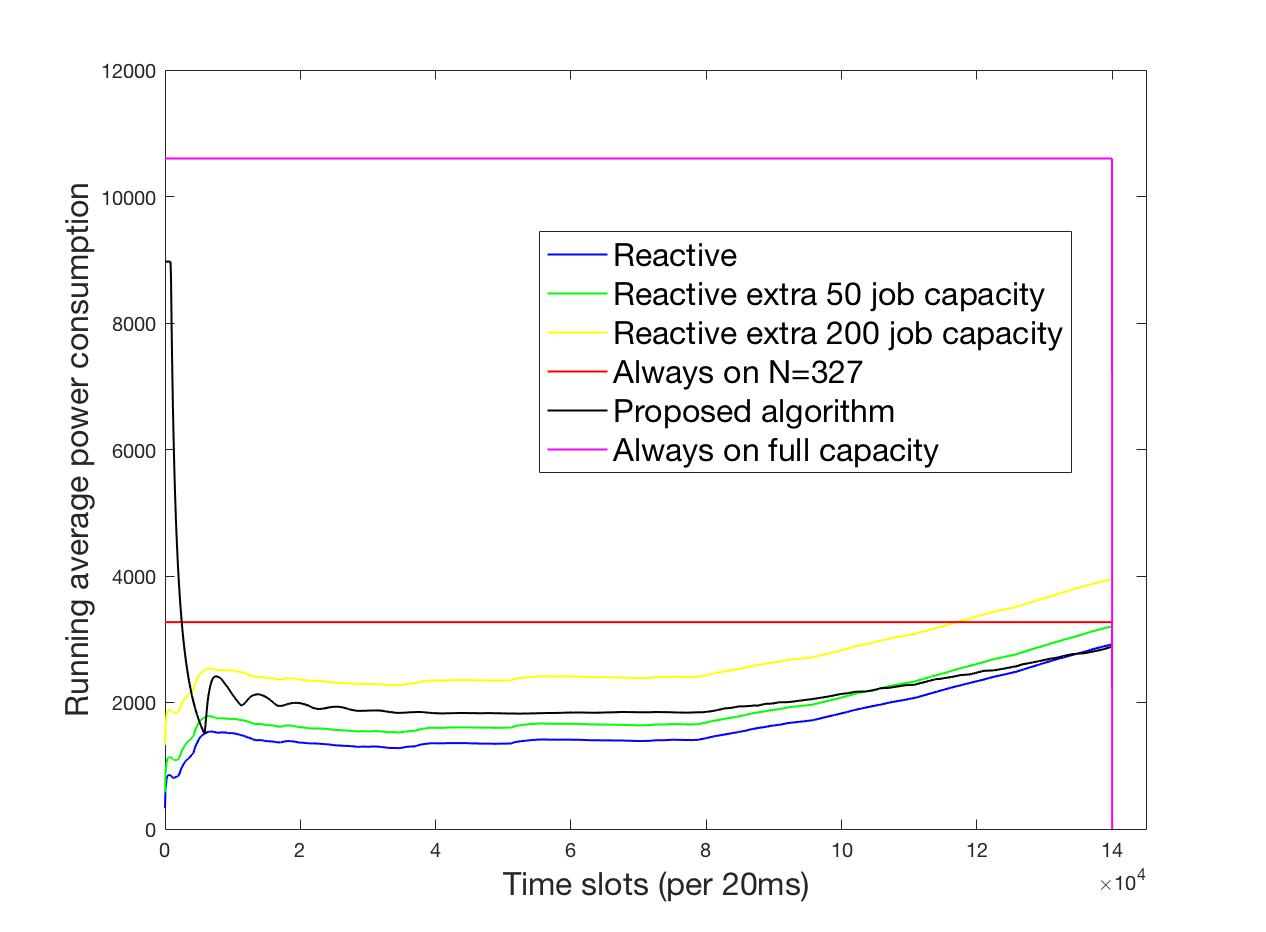} 
   \caption{Running average power consumption from slot 1 to the current slot for different algorithms.}
   \label{fig:Stupendous7}
\end{figure}

\begin{figure}[htbp]
   \centering
   \includegraphics[height=2in]{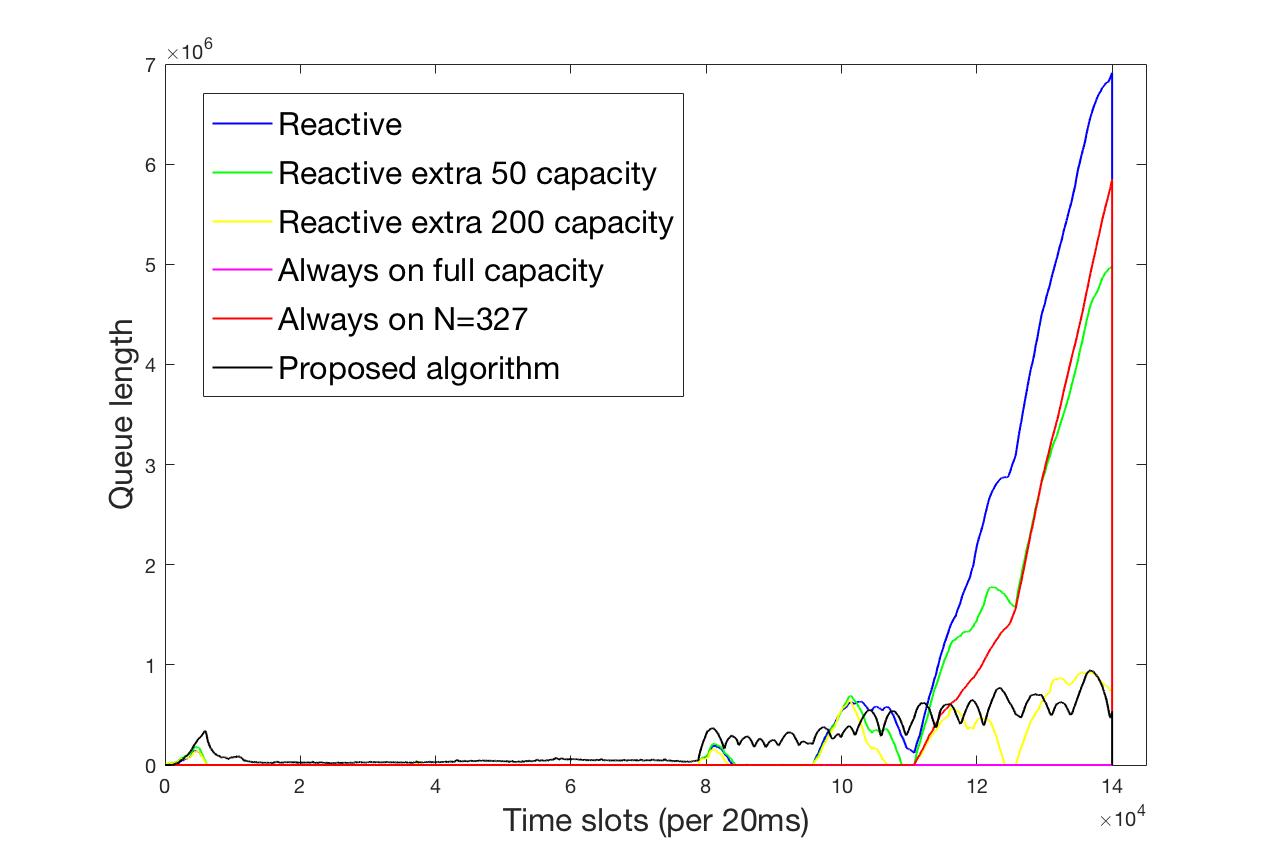} 
   \caption{Instantaneous queue length for different algorithms.}
   \label{fig:Stupendous8}
\end{figure}

\begin{figure}[htbp]
   \centering
   \includegraphics[height=2in]{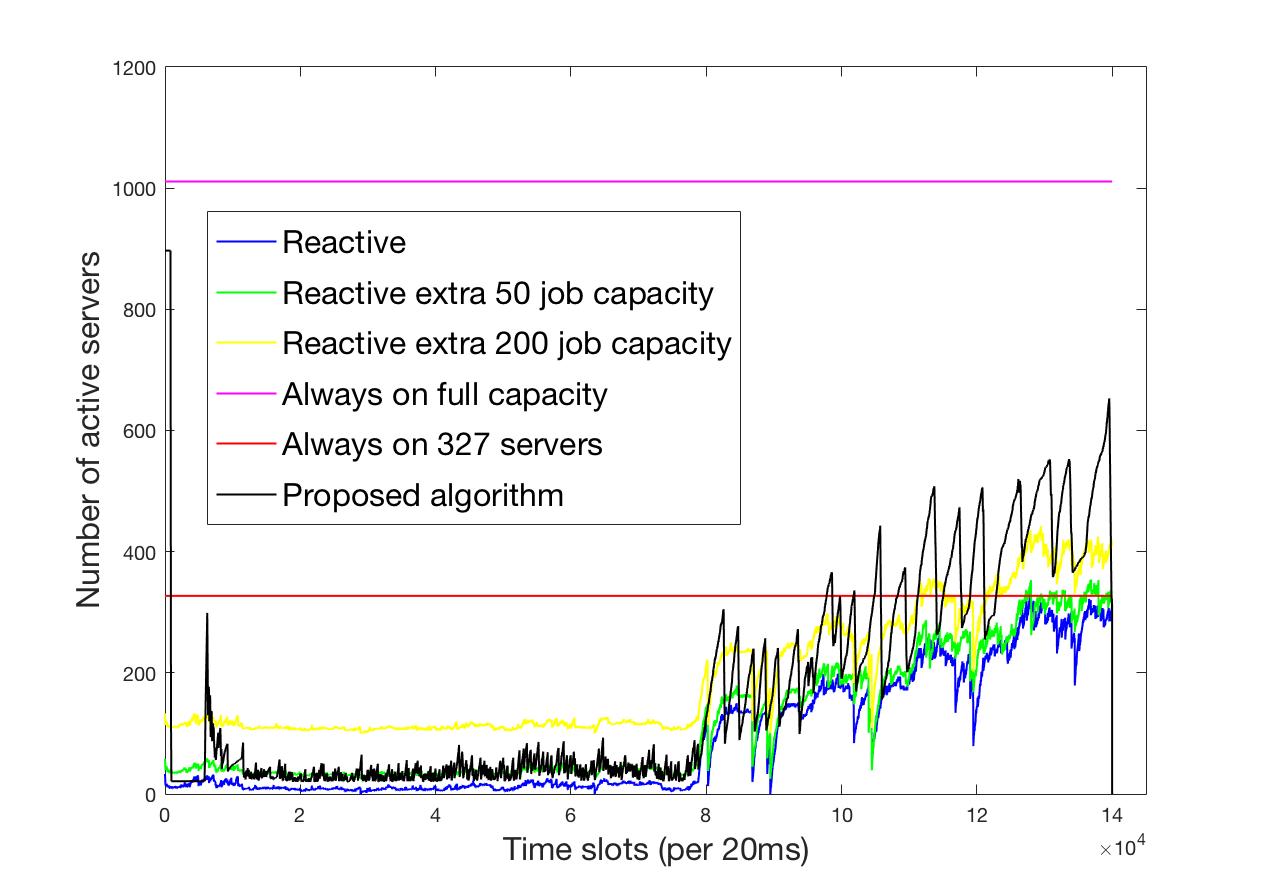} 
   \caption{Number of active servers over time.}
   \label{fig:Stupendous9}
\end{figure}

Finally, we evaluate the influence of different sleep modes on the performance. We keep all the setups the same as before and consider the sleep modes with sleep power consumption  equal to 2 W and 4 W per slot, respectively. Since the Always-on and the Reactive algorithm do not look at the sleep power consumption, their decisions remain the same as before, thus, we superpose the queue length of our proposed algorithm onto the previous Fig. \ref{fig:Stupendous8} and get the queue length comparison in Fig. \ref{fig:queue-length-S}. We see from the plot that increasing the power consumption during the sleep mode only slightly increases the queue length of our proposed algorithm. Fig. \ref{fig:sleep-mode} plots the running average power consumption under different sleep modes. Despite spending more power on the sleep mode, the proposed algorithm can still save considerable amount of power compared to other algorithms while keeping the request queue stable. This shows that our algorithm is empirically robust to the change of sleep mode.
\begin{figure}[htbp]
   \centering
   \includegraphics[height=2in]{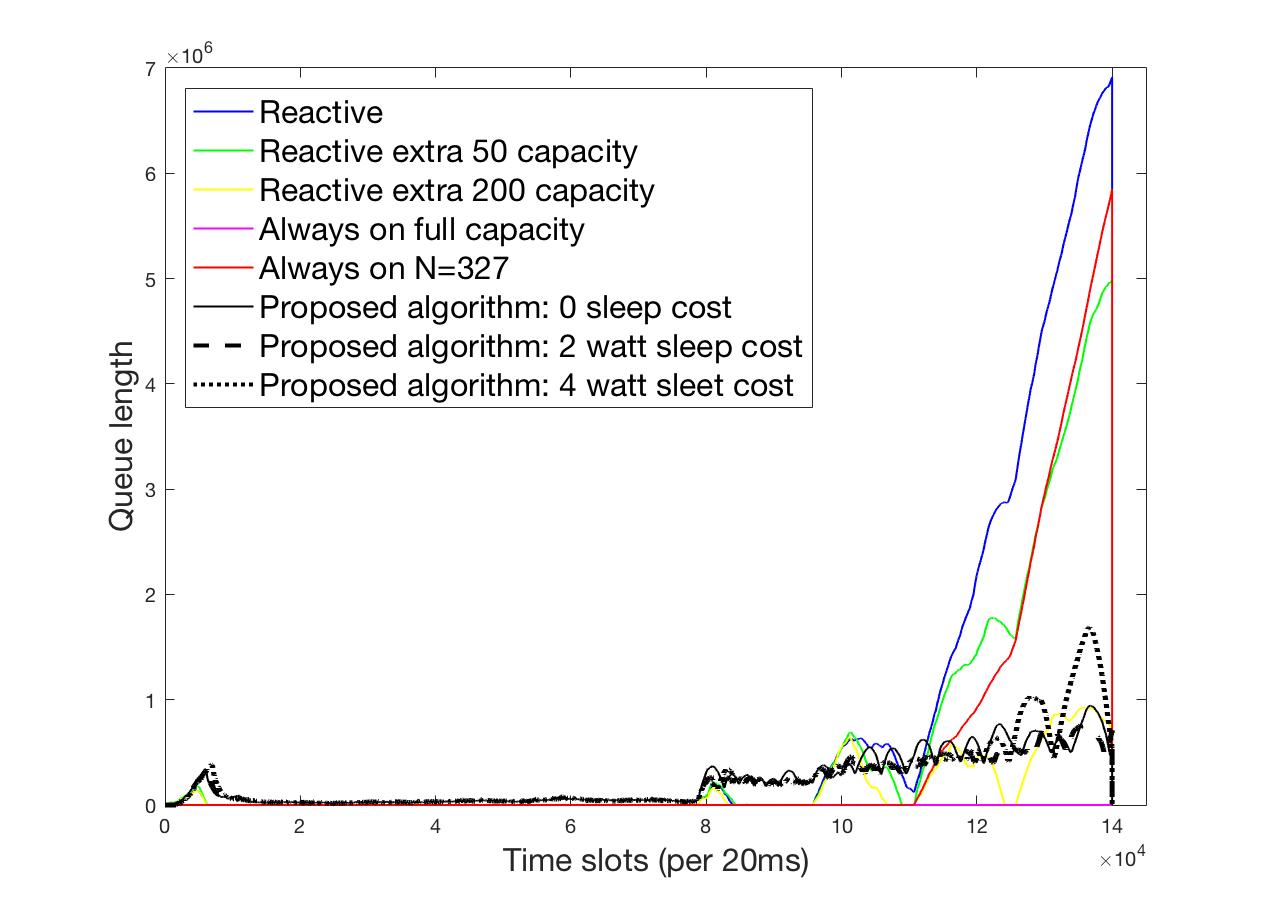} 
   \caption{Instantaneous queue length for different algorithms.}
   \label{fig:queue-length-S}
\end{figure}

\setcounter{figure}{15}
\begin{figure*}[ht]
\centering
 \begin{minipage}{5.6cm}
   \includegraphics[height=4cm] {chapter3/average_cost}
 \end{minipage}
 \begin{minipage}{5.6cm}
   \includegraphics[height=4cm] {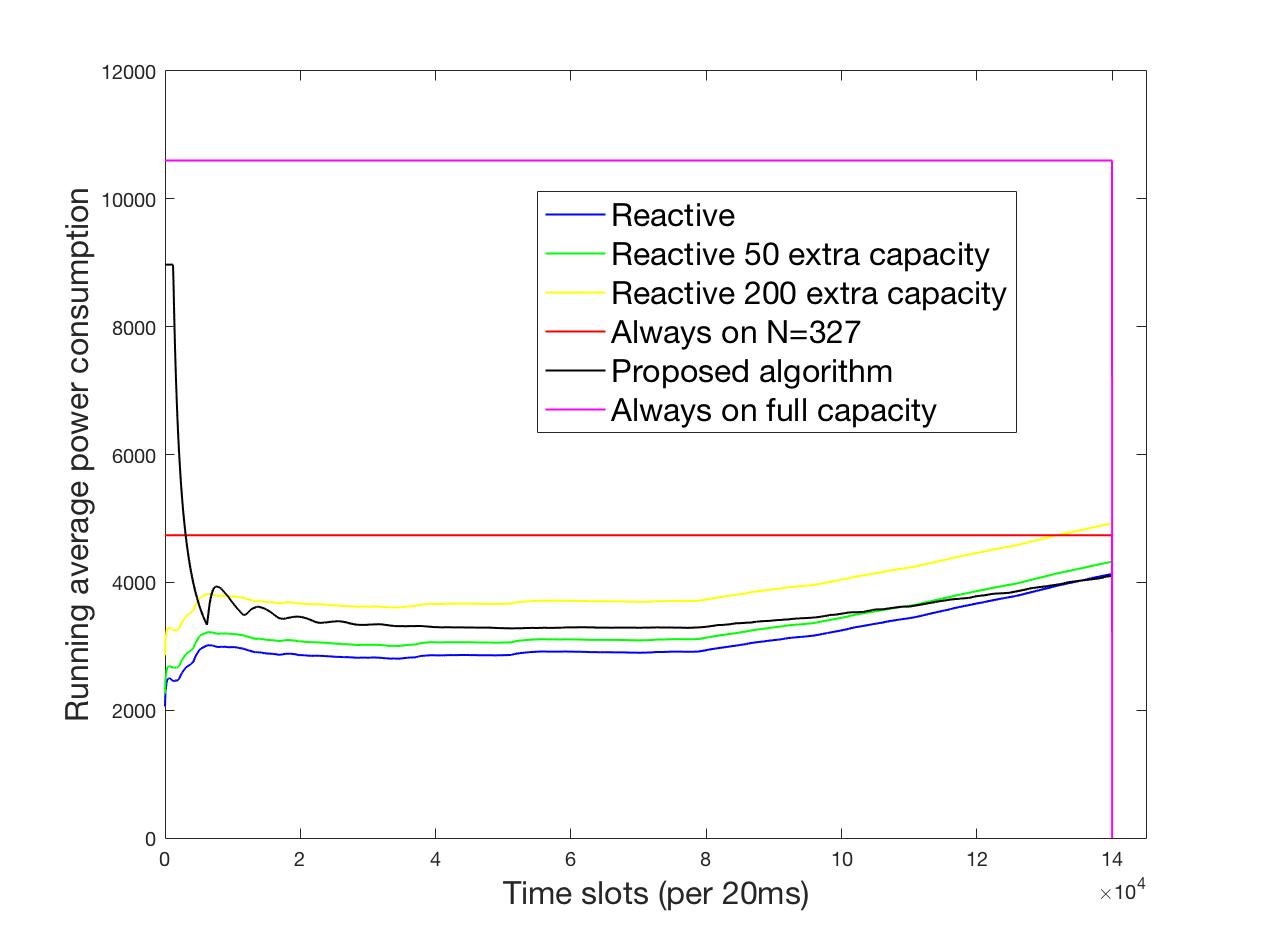}
 \end{minipage}
 \begin{minipage}{5.6cm}
   \includegraphics[height=4cm] {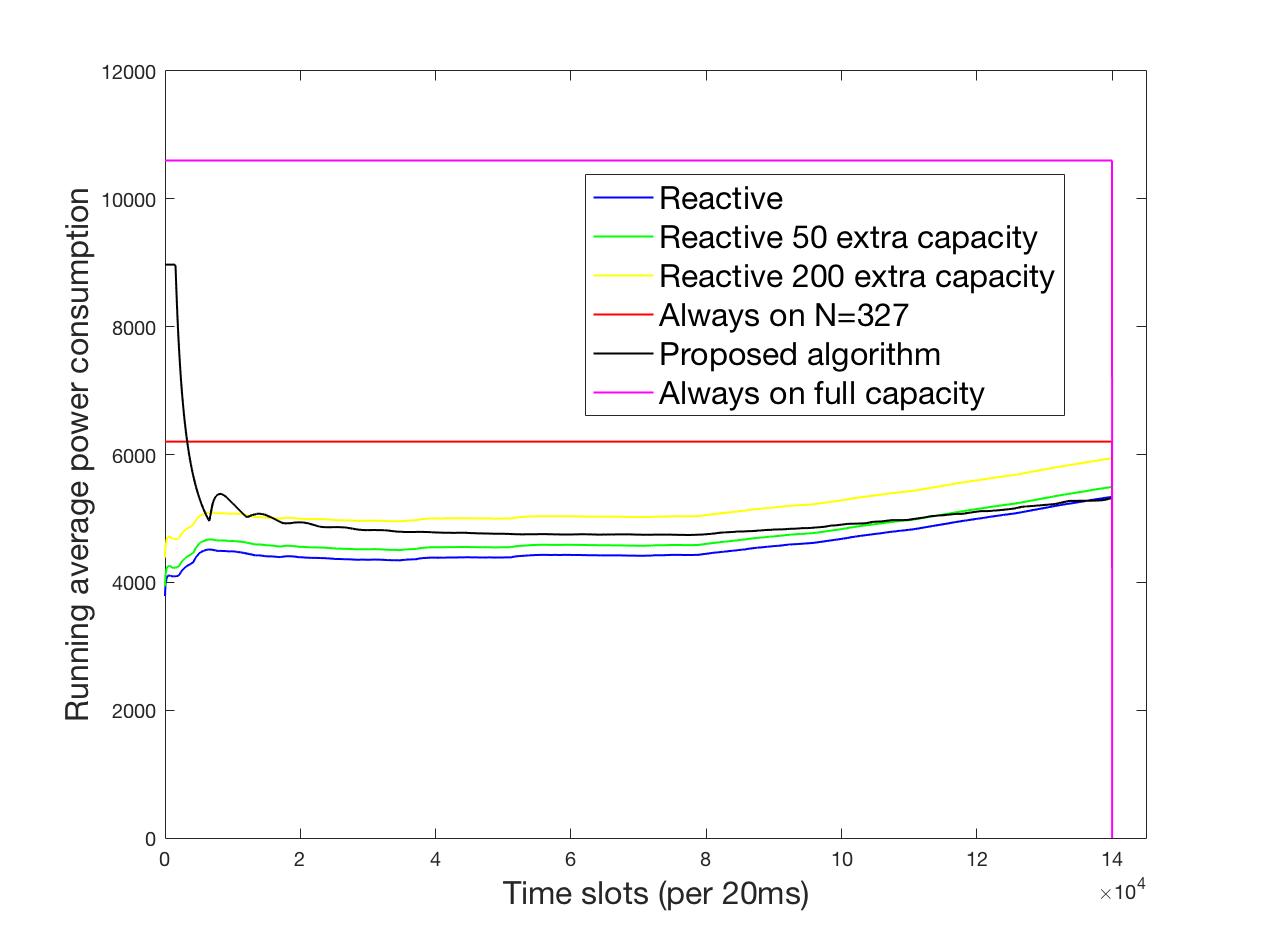}
 \end{minipage}
\caption{Running average power consumption for 0 W sleep cost(left), 2 W sleep cost(middle) and 4 W sleep cost(right)}\label{fig:sleep-mode}
\end{figure*}

\section{Additional lemmas and proofs}
\subsection*{Appendix A--- Proof of Lemma \ref{compute_idle}}
We have \eqref{appendix_A_interim}, as shown at the bottom of this page, holds,
\begin{figure*}
\normalsize
\begin{align}
&D_n[f]=\frac{V\hat{W}_n(\alpha_n[f])m_{\alpha_n[f]}+Ve_n-Q_n(t_f^{n})\mu_n
        +\expect{\frac{B_0}{2}(I_n[f]+\tau_n[f]+1)^2+V\hat{g}(\alpha_n[f])I_n[f]\left|~Q_n(t_f^{n})\right.}}
        {\expect{I_n[f]+\tau_n[f]+1~\left|~Q_n(t_f^{n})\right.}}-\frac{B_0}{2}\nonumber\\
&=\frac{V\hat{W}_n(\alpha_n[f])m_{\alpha_n[f]}+Ve_n-Q_n(t_f^{n})\mu_n+\expect{\frac{B_0}{2}(I_n[f]+m_{\alpha_n}+1)^2
  +\frac{B_0}{2}\sigma_{\alpha_n[f]}^2+V\hat{g}(\alpha_n[f])I_n[f]~\left|~Q_n(t_f^{n})\right.}}\nonumber\\
  &{\expect{I_n[f]+m_{\alpha_n}+1~\left|~Q_n(t_f^{n})\right.}}
  -\frac{B_0}{2}\label{appendix_A_interim}
\end{align}
\end{figure*}
where the first equality follows from the definition $T_n[f]=I_n[f]+\tau_n[f]+1$ and the second equality follows from iterated expectations conditioning on $I_n[f]$ and $\alpha_n[f]$.
For simplicity of notations, let
\begin{align*}
F(\alpha_n[f],I_n[f]) =& V\hat{W}_n(\alpha_n[f])m_{\alpha_n[f]}+Ve_n-Q_n(t_f^{n})\mu_n\\
        &+\frac{B_0}{2}(I_n[f]+m_{\alpha_n[f]}+1)^2+V\hat{g}(\alpha_n[f])I_n[f]\\
        &+\frac{B_0}{2}\sigma_{\alpha_n[f]}^2\\
G(\alpha_n[f],I_n[f]) =& I_n[f]+m_{\alpha_n[f]}+1,
\end{align*}
then
\[D_n[f]=\frac{\expect{F(\alpha_n[f],I_n[f])~|~Q_n(t_f^{n})}}{\expect{G(\alpha_n[f],I_n[f])~|~Q_n(t_f^{n})}}-\frac{B_0}{2}.\]
Meanwhile, given the queue length $Q_n(t_f^{n})$ at frame $f$, denote the benchmark solution over pure decisions as
\begin{equation}\label{det_solution}
m \triangleq \min_{I_n[f]\in\mathbb{N},~I_n[f]\in[1,I_{\max}],\alpha_n[f]\in\mathcal{L}_n}\frac{F(\alpha_n[f],I_n[f])}{G(\alpha_n[f],I_n[f])}.
\end{equation}
Then, for any randomized decision on $\alpha_n[f]$ and $I_n[f]$, its realization within frame $f$ satisfies the following
\[\frac{F(\alpha_n[f],I_n[f])}{G(\alpha_n[f],I_n[f])}\geq m,\]
which implies
\[F(\alpha_n[f],I_n[f])\geq m G(\alpha_n[f],I_n[f]).\]
Taking conditional expectation from both sides gives
\begin{align*}
&\expect{F(\alpha_n[f],I_n[f])~|~Q_n(t_f^{n})}\\
&\geq m\expect{G(\alpha_n[f],I_n[f])~|~Q_n(t_f^{n})}\\
&\Rightarrow~\frac{\expect{F(\alpha_n[f],I_n[f])~|~Q_n(t_f^{n})}}{\expect{G(\alpha_n[f],I_n[f])~|~Q_n(t_f^{n})}}\geq m.
\end{align*}
Thus, it is enough to consider pure decisions only, which boils down to computing \eqref{det_solution}. This proves the lemma.

\subsection*{Appendix B--- Proof of Lemma \ref{bounded_supMG}}
This section is dedicated to prove that $\expect{X_n[f]^2}$ is bounded. First of all, since the idle option set $\mathcal{L}_n$ is finite, denote
\begin{align*}
W_{\max}=\max_{\alpha_n\in\mathcal{L}_n}W_n(\alpha_n)\\
g_{\max}=\max_{\alpha_n\in\mathcal{L}_n}g_n(\alpha_n)
\end{align*}
It is obvious that $|W_n(t)-\overline{W}_n^*|\leq W_{\max}$, $|g_n(t)-\overline{G}_n^*|\leq g_{\max}$, $|e_nH_n(t)-\overline{E}_n^*|\leq e_n$, and $|\mu_nH_n(t)-\overline{\mu}^*|\leq\mu_n$. Combining with the boundedness of queues in lemma \ref{bounded_delay}, it follows
\begin{align*}
|X_n[f]|\leq&\sum_{t=t_f^{n}}^{t=t_{f+1}^{n}-1}\left(V\left(W_{\max}+e_n+g_{\max}\right)
       +\left(Vc_{\max}\right.\right.\\
       &\left.\left.+R_{\max}\right)\mu_n+\left(t-t_f^{n}\right)B_0+\Psi_n\right)\\
      \leq&\left(V(W_{\max}+e_n+g_{\max})+(Vc_{\max}+R_{\max})\mu_n\right.\\
       &\left.+\Psi_n\right)T_n[f]+\frac{T_n[f](T_n[f]-1)B_0}{2}
\end{align*}
Let $B_1\triangleq V(W_{\max}+e_n+g_{\max})+(Vc_{\max}+R_{\max})\mu_n+\Psi_n+B_0/2$, it follows
\[|X_n[f]|\leq B_1T_n[f]+\frac{B_0}{2}T_n[f]^2.\]
Thus,
\[\expect{X_n[f]^2}\leq B_1^2\expect{T_n[f]^2}+B_1B_0\expect{T_n[f]^3}+\frac{B_0^2}{4}\expect{T_n[f]^4}.\]
Notice that $T_n[f]\leq I_n[f]+\tau_n[f]+1$ by \eqref{frame_length}, where $I_n[f]$ is upper bonded by $I_{\max}$ and $\tau_n[f]$ has first four moments bounded by assumption \ref{bounded_moment_assumption}. Thus, $\expect{X_n[f]^2}$ is bounded by a fixed constant.

\subsection*{Appendix C--- Proof of Lemma \ref{true_time_average}}
\begin{proof}
Let's first abbreviate the notation by defining
\begin{align*}
Y(t)=&V(W_n(t)+e_nH_n(t)+g_n(t))-Q_n(t_f^{n})
(\mu_nH_n(t)-\overline{\mu}_n^*)\\
&+\left(t-t_f^{n}\right)B_0.
\end{align*}
For any $T\in[t_f^{n},~t_F^{(n+1)})$, we can bound the partial sums from above by the following
\[\sum_{t=0}^{T-1}Y(t)\leq\sum_{t=0}^{t_f^{n}-1}Y(t)+B_2T_n[F]+\frac{B_0}{2}T_n[F]^2,\]
where $B_0=\frac{1}{2}(R_{\max}+\mu_{\max})\mu_{\max}$ is defined in \eqref{server_decision}, and $B_2\triangleq VW_n+V\mu_ne_n+(Vc_{\max}+R_{\max})\mu_n+B_0/2$. Thus,
\begin{align*}
\frac{1}{T}\sum_{t=0}^{T-1}Y(t)&\leq\frac{1}{T}\sum_{t=0}^{t_f^{n}-1}Y(t)+\frac{1}{T}\left(B_2T_n[F]+\frac{B_0}{2}T_n[F]^2\right)\\
&\leq\max\{a[F],~b[F]\},
\end{align*}
where
\begin{align*}
a[F]\triangleq&\frac{1}{t_f^{n}}\sum_{t=0}^{t_f^{n}-1}Y(t)+\frac{1}{t_f^{n}}\left(B_2T_n[F]+\frac{B_0}{2}T_n[F]^2\right),\\
b[F]\triangleq&\frac{1}{t_{f+1}^{n}}\sum_{t=0}^{t_f^{n}-1}Y(t)+\frac{1}{t_{f+1}^{n}}\left(B_2T_n[F]+\frac{B_0}{2}T_n[F]^2\right).
\end{align*}
Thus, this implies that
\begin{align*}
\limsup_{T\rightarrow\infty}\frac{1}{T}\sum_{t=0}^{T-1}Y(t)
\leq&\limsup_{F\rightarrow\infty}\max\{a[F],~b[F]\}\\
=&\max\left\{\limsup_{F\rightarrow\infty}a[F],~\limsup_{F\rightarrow\infty}b[F]\right\}.
\end{align*}
We then try to work out an upper bound for $\limsup_{F\rightarrow\infty}a[F]$ and $\limsup_{F\rightarrow\infty}b[F]$ respectively.

\begin{enumerate}
  \item Bound for $\limsup_{F\rightarrow\infty}a[F]$:
\begin{align*}
\limsup_{F\rightarrow\infty}a[F]\leq&\limsup_{F\rightarrow\infty}\frac{1}{t_f^{n}}\sum_{t=0}^{t_f^{n}-1}Y(t)
                                    +\limsup_{F\rightarrow\infty}\frac{1}{t_f^{n}}\left(B_2T_n[F]+\frac{B_0}{2}T_n[F]^2\right)\\
                                \leq& V(\overline{W}_n^*+\overline{E}_n^*+\overline{G}_n^*)+\Psi_n
                                    +\limsup_{F\rightarrow\infty}\frac{1}{t_f^{n}}\left(B_2T_n[F]+\frac{B_0}{2}T_n[F]^2\right).
\end{align*}
where the second inequality follows from corollary \ref{corollary_ratio_time_average}.
It remains to show that
\begin{equation}\label{interim_a[F]}
\limsup_{F\rightarrow\infty}\frac{1}{t_f^{n}}\left(B_2T_n[F]+\frac{B_0}{2}T_n[F]^2\right)\leq0.
\end{equation}
Since $t_f^{n}\geq F$, it is enough to show that
\begin{align}
&\limsup_{F\rightarrow\infty}\frac{T_n[F]}{F}=0,   \label{a_1}\\
&\limsup_{F\rightarrow\infty}\frac{T_n[F]^2}{F}=0. \label{a_2}
\end{align}
We prove \eqref{a_2}, and \eqref{a_1} is similar. Since each $T_n[F]=I_n[F]+\tau_n[F]+1$, where $I_n[F]\leq I_{\max}$ and $\tau_n[F]$ has bounded first four moments, the first four moments of $T_n[F]$ must also be bounded and there exists a constant $C>0$ such that
\[\expect{T_n[F]^4}\leq C.\]
For any $\epsilon>0$, define a sequence of events
\[A_F^\epsilon\triangleq\left\{T_n[F]^2>\epsilon F\right\}.\]
According to Markov inequality,
\[Pr\left[A_F^\epsilon\right]\leq\frac{\expect{T_n[F]^4}}{\epsilon^2F^2}\leq\frac{C}{\epsilon^2F^2}.\]
Thus,
\[\sum_{F=1}^\infty Pr\left[A_F^\epsilon\right]\leq\frac{C}{\epsilon^2}\sum_{F=1}^\infty\frac{1}{F^2}\leq\frac{2C}{\epsilon^2}<\infty. \]
By Borel-Cantelli lemma (lemma 1.6.1 in \cite{Durrett}),
\[Pr\left[A_F^\epsilon~\textrm{occurs infinitely often}\right]=0,\]
which implies
\[Pr\left[\limsup_{F\rightarrow\infty}\frac{T_n[F]^2}{F}>\epsilon\right]=0.\]
Since $\epsilon$ is arbitrary, this implies \eqref{a_2}. Similarly, \eqref{a_1} can be proved. Thus, \eqref{interim_a[F]} holds and
\[\limsup_{F\rightarrow\infty}a[F]\leq V(\overline{W}_n^*+\overline{E}_n^*+\overline{G}_n^*)+\Psi_n.\]

\item Bound for $\limsup_{F\rightarrow\infty}b[F]$:
\begin{align*}
\limsup_{F\rightarrow\infty}b[F]\leq&\limsup_{F\rightarrow\infty}\frac{1}{t_f^{n}}\sum_{t=0}^{t_f^{n}-1}Y(t)\cdot\frac{t_f^{n}}{t_{f+1}^{n}}+\limsup_{F\rightarrow\infty}\frac{1}{t_{f+1}^{n}}\left(B_2T_n[F]+\frac{B_0}{2}T_n[F]^2\right).\\
                                \leq&\limsup_{F\rightarrow\infty}\left(\frac{1}{t_f^{n}}\sum_{t=0}^{t_f^{n}-1}Y(t)\right)\cdot\limsup_{F\rightarrow\infty}\frac{t_f^{n}}{t_{f+1}^{n}}\\
                                \leq&\left(V\left(\overline{W}_n^*+\overline{E}_n^*+\overline{G}_n^*\right)+\Psi_n\right)\cdot\limsup_{F\rightarrow\infty}\frac{t_f^{n}}{t_{f+1}^{n}}\\
                                \leq&V\left(\overline{W}_n^*+\overline{E}_n^*+\overline{G}_n^*\right)+\Psi_n,
\end{align*}
where the second inequality follows from \eqref{interim_a[F]}, the third inequality follows from corollary \ref{corollary_ratio_time_average} and the last inequality follows from the fact that $V\left(\overline{W}_n^*+\overline{E}_n^*+\overline{G}_n^*\right)+\Psi_n>0$.
\end{enumerate}
Above all, we proved the lemma.
\end{proof}

\subsection*{Appendix D--- Proof of Theorem \ref{theorem_near_optimal_perform}}
\begin{proof}
Define the drift-plus-penalty(DPP) expression $P(t)$ as follows
\begin{align*}
P(t)=&V\left(c(t)r(t)+\sum_{n=1}^N\left(W_n(t)+e_nH_n(t)+g_n(t)\right)\right)\\
&+\frac{1}{2}\sum_{n=1}^N\left(Q_n(t+1)^2-Q_n(t)^2\right).
\end{align*}
By simple algebra using the queue updating rule \eqref{queue_update}, we can work out the upper bound for $P(t)$ as follows,
\begin{align*}
P(t)\leq& \frac{1}{2}\sum_{n=1}^N(R_n(t)+\mu_n)^2+V\left(c(t)r(t)+\sum_{n=1}^N\left(W_n(t)+e_nH_n(t)+g_n(t)\right)\right)\\
        &  +\sum_{n=1}^NQ_n(t)(R_n(t)-\mu_nH_n(t))\\
    \leq& B_3+V\left(c(t)r(t)+\sum_{n=1}^N\left(W_n(t)+e_nH_n(t)+g_n(t)\right)\right) +\sum_{n=1}^NQ_n(t)(R_n(t)-\mu_nH_n(t))\\
    \leq& B_3+Vc(t)r(t)+\sum_{n=1}^NQ_n(t)\left(R_n(t)-\overline{R}^*_n\right)+V\sum_{n=1}^N\left(W_n(t)+e_nH_n(t)+g_n(t)\right)\\
          &+\sum_{n=1}^NQ_n(t)\left(\overline{\mu}^*_n-\mu_nH_n(t)\right)
\end{align*}
where $B_3=\frac{1}{2}\sum_{n=1}^N(R_{\max}+\mu_n)^2$, the last inequality follows from adding $\sum_{n=1}^NQ_n(t)\overline{\mu}^*_n$ and subtracting $\sum_{n=1}^NQ_n(t)\overline{R}^*$ with the fact that the best randomized stationary algorithm should also satisfy the constraint \eqref{obj_4}, i.e. $\overline{\mu}^*\geq\overline{R}_n^*$.

Now we take the partial average of $P(t)$ from 0 to $T-1$ and take $\limsup_{T\rightarrow\infty}$,
\begin{align}
\limsup_{T\rightarrow\infty}\frac{1}{T}\sum_{n=1}^{T-1}P(t)
\leq& B_3 +
\limsup_{T\rightarrow\infty}\frac{1}{T}\sum_{t=0}^{T-1}\left(Vc(t)r(t)+\sum_{n=1}^NQ_n(t)\left(R_n(t)-\overline{R}^*_n\right)\right)\nonumber\\
&+\sum_{n=1}^N\limsup_{T\rightarrow\infty}\frac{1}{T}\sum_{t=0}^{T-1}\left(V\left(W_n(t)+e_nH_n(t)+g_n(t)\right)\right.\nonumber\\
&\left.+Q_n(t)\left(\overline{\mu}^*_n-\mu_nH_n(t)\right)\right).
\label{DPP_bound}
\end{align}
According to \eqref{prob_1_front_end},
\begin{equation}\label{DPP_sub_bound_1}
\limsup_{T\rightarrow\infty}\frac{1}{T}\sum_{t=0}^{T-1}\left(Vc(t)r(t)+\sum_{n=1}^NQ_n(t)\left(R_n(t)-\overline{R}^*_n\right)\right)
\leq V\overline{C}^*.
\end{equation}
On the other hand,
\begin{align}
&\limsup_{T\rightarrow\infty}\frac{1}{T}\sum_{t=0}^{T-1}\left(V\left(W_n(t)+e_nH_n(t)+g_n(t)\right)+Q_n(t)\left(\overline{\mu}^*_n-\mu_nH_n(t)\right)\right)\nonumber\\
\leq&\limsup_{T\rightarrow\infty}\frac{1}{T}\sum_{t=0}^{T-1}\left(V\left(W_n(t)+e_nH_n(t)+g_n(t)\right)+Q_n(t_f^{n})\left(\overline{\mu}^*_n-\mu_nH_n(t)\right)
       +(t-t_f^{n})B_0\right)\nonumber\\
\leq&V\left(\overline{W}_n^*+\overline{E}_n^*+\overline{G}_n^*\right)+\Psi_n,\label{DPP_sub_bound_2}
\end{align}
where $B_0=\frac{1}{2}(R_{\max}+\mu_{\max})\mu_{\max}$ as defined below \eqref{server_decision}, the first inequality follows from the fact that for any $t\in\left(t_f^{n},~t_{f+1}^{n}\right)$,
\begin{align*}
&Q_n(t)\left(\overline{\mu}^*_n-\mu_nH_n(t)\right)\nonumber\\
\leq& Q_n(t_f^{n})\left(\overline{\mu}^*_n-\mu_nH_n(t)\right)
      +(Q_n(t)-Q_n(t_f^{n}))\left(\overline{\mu}^*_n-\mu_nH_n(t)\right)\\
\leq& Q_n(t_f^{n})\left(\overline{\mu}^*_n-\mu_nH_n(t)\right)
      +\sum_{t=t_f^{n}}^{t_{f+1}^{n}-1}(R_n(t)-\mu_nH_n(t))\left(\overline{\mu}^*_n-\mu_nH_n(t)\right)\\
\leq& Q_n(t_f^{n})\left(\overline{\mu}^*_n-\mu_nH_n(t)\right)+(t-t_f^{n})B_0,
\end{align*}
and the second inequality follows from lemma \ref{true_time_average}. Substitute \eqref{DPP_sub_bound_1} and \eqref{DPP_sub_bound_2} into \eqref{DPP_bound} gives
\begin{align}\label{penultimate_step}
\limsup_{T\rightarrow\infty}\frac{1}{T}\sum_{t=0}^{T-1}P(t)\leq& V\left(\overline{C}^*+\sum_{n=1}^N\left(\overline{W}_n^*+\overline{E}_n^*+\overline{G}_n^*\right)\right)+B_3+\sum_{n=1}^N\Psi_n.
\end{align}
Finally, notice that by telescoping sums,
\begin{align*}
&\limsup_{T\rightarrow\infty}\frac{1}{T}\sum_{t=0}^{T-1}P(t)\\
=&\limsup_{T\rightarrow\infty}\left(\frac{V}{T}\sum_{t=0}^{T-1}(c(t)r(t)
+\sum_{n=1}^N\left(W_n(t)+e_nH_n(t)+g_n(t)\right))+\frac{1}{2}\sum_{n=1}^NQ_n(T)^2\right)\\
\geq&V\cdot\limsup_{T\rightarrow\infty}\frac{1}{T}\sum_{t=0}^{T-1}(c(t)r(t)
+\sum_{n=1}^N\left(W_n(t)+e_nH_n(t)+g_n(t)\right))
\end{align*}
Substitute above inequality into \eqref{penultimate_step} and divide $V$ from both sides give the desired result.
\end{proof}


\chapter{Power Aware Wireless File Downloading and Restless Bandit via Renewal Optimization}
In this chapter, we look at another application of the renewal optimization, namely, the wireless file downloading. We start with a simple single-user file downloading problem and show that this problem can be characterized by a 2 state Markov decision process (MDP) with constraints, for which the drift-plus-penalty (DPP) ratio algorithm (Algorithm \ref{dpp-ratio-algorithm}) applies.
We then consider a more realistic multi-user file downloading and show that this problem is a constrained version of the well-known restless bandit problem, for which we develop a \textit{DPP ratio indexing} heuristic based on the coupled renewal optimization.

\section{System model and problem formulation}

Consider a wireless access point, such as a base station or  femto node, that delivers files to
$N$ different wireless users.  The system operates in slotted time with time slots $t \in \{0, 1, 2, \ldots\}$.
Each user can download at most one file at a time.   File sizes are random and complete delivery of a file requires a random number of time slots. A new file request is made by each user at a random
time after it finishes its previous download.
Let $F_n(t) \in \{0,1\}$ represent the binary \emph{file state process} for user $n \in \{1, \ldots, N\}$.  The state $F_n(t)=1$ means that user $n$ is currently active downloading a file, while the state $F_n(t)=0$ means that user $n$ is currently idle.

Idle times are assumed to be independent and geometrically distributed with parameter $\lambda_n$ for each user $n$, so that the average idle time is $1/\lambda_n$.  Active times depend on the random file size and the transmission decisions that are made.  Every slot $t$, the access point observes which users are active and
decides to serve a subset of at most $M$ users, where $M$ is the maximum number of simultaneous transmissions allowed in the system ($M < N$ is assumed throughout).
The goal is to maximize a weighted sum of throughput subject to a total average power constraint.

The file state processes $F_n(t)$ are coupled controlled Markov chains that form a total state
$(F_1(t), \ldots, F_N(t))$ that can be viewed as a \emph{restless multi-armed
bandit system}.
Such problems are complex due to the inherent curse of dimensionality.

We first compute an online optimal algorithm for 1-user systems, i.e., the case $N=1$.  This simple case avoids the curse of dimensionality and provides valuable intuition.  The optimal policy here is computed via the drift-plus-penalty (DPP) ratio algorithm.
 The resulting algorithm makes a greedy transmission decision that affects success probability and power usage. Next, the algorithm is extended as a low complexity online heuristic for the $N$-user problem, which we call the ``DPP ratio indexing''.   The heuristic has the following desirable properties:
\begin{itemize}
\item  Implementation of the $N$-user heuristic is as simple as comparing indices for $N$ different 1-user problems.

\item The $N$-user heuristic is analytically shown to meet the desired average power constraint.

\item  The $N$-user heuristic is shown in simulation to perform well over a wide range of parameters.  Specifically, it is very close to optimal for example cases where an offline optimal can be computed.

\item The $N$-user heuristic is shown to be optimal in a special case with no power constraint and with certain additional assumptions.  The optimality proof uses a theory of \emph{stochastic coupling} for queueing systems \cite{tassiulas1993dynamic}.
\end{itemize}

Prior work on wireless optimization uses Lyapunov functions to maximize throughput in
cases where the users are assumed to have an infinite amount of data to
send \cite{neely2008fairness,eryilmaz2007fair,georgiadis2006resource,stolyar2005maximizing,tassiulas1993dynamic}, or
when data arrives according to a fixed rate process that does not depend on delays in the
network  (which necessitates \emph{dropping} data if the arrival rate vector is outside of the capacity
region, e.g. \cite{neely2008fairness}).  These models do not consider the interplay between arrivals at the transport layer and file delivery at the network layer.    For example, a web user in a coffee shop may want to evaluate the file she downloaded before initiating another download.  The current work captures this interplay through the binary file state processes $F_n(t)$. This
creates a complex problem of
coupled Markov chains.    This problem is fundamental to
file downloading systems.
The modeling and analysis of these systems is a  significant contribution of the current thesis.

To understand this issue, suppose the data arrival rate is fixed and does not adapt to the service received
over the network.   If this arrival rate exceeds network capacity by a factor of two, then at least half of all data must
be dropped.  This can result in an unusable data stream, possibly one that
contains every odd-numbered packet.  A more practical model assumes that full files must be downloaded and that new downloads are only initiated when previous ones are completed.
A general model in this direction would
allow each user to download up to $K$ files simultaneously.
This thesis considers the case $K=1$, so that each user is either actively downloading a file, or is idle.\footnote{One way to allow a user $n$ to download up to $K$ files simultaneously is as follows:  Define $K$   \emph{virtual users}  with separate binary file state processes.  The transition probability from idle to active in each of these virtual users is $\lambda_n/K$. The conditional rate of total new arrivals for user $n$ (given that $m$ files are currently in progress) is then $\lambda_n(1-m/K)$ for $m \in \{0, 1, \ldots, M\}$.}
 The resulting
system for $N$ users has a nontrivial Markov structure with $2^N$ states.

Since the current problem includes both time-average constraints (on average power expenditure) and instantaneous constraints which restrict the number of users that can be served on one slot, it is more complicated than the weakly coupled systems discussed in previous chapters. More specifically, The latter service restriction is similar to a traditional restless multi-armed bandit (RMAB) system \cite{whittle1988restless}.

 RMAB problem considers a population of $N$ parallel MDPs that continue evolving whether in operation or not (although in different rules). The goal is to choose the MDPs in operation during each time slot so as to maximize the expected reward subject to a constraint on the number of MDPs in operation. The problem is in general complex (see P-SPACE hardness results in \cite{papadimitriou1999complexity}).  A standard low-complexity heuristic for such problems is the \emph{Whittle's index} technique \cite{whittle1988restless}.
 However, the Whittle's index framework applies only when there are two options on each state (active and passive). Further, it does not consider the 
additional time average cost constraints.  The \emph{DPP ratio indexing} algorithm developed in the current work can
be viewed as an alternative indexing scheme that can always be implemented and that incorporates
additional time average constraints.
It is likely that the techniques of the current work can be extended to other constrained
RMAB problems.
Prior work in \cite{tassiulas1993dynamic} develops a Lyapunov drift method for queue stability, and work in
 \cite{Neely2010} develops a drift-plus-penalty (DPP) ratio method for optimization over renewal systems. The current work is the first to use these techniques as a low complexity heuristic for multidimensional Markov problems.

Work in \cite{tassiulas1993dynamic} uses the theory of stochastic coupling to show that a \emph{longest connected queue} algorithm is delay optimal in a multi-dimensional queueing system with \emph{special symmetric assumptions}.   The problem in \cite{tassiulas1993dynamic} is different from that of the current work.  However,
a similar coupling approach is used below to show that, for a special case with no power constraint, the DPP ratio indexing
algorithm  is throughput optimal in certain \emph{asymmetric} cases.
 As a consequence, the proof shows the policy is also optimal for a different setting with $M$ servers, $N$ single-buffer queues, and arbitrary packet arrival rates $(\lambda_1, \ldots, \lambda_N)$.

\section{Single user scenario}

Consider a file downloading system that consists of only one user that repeatedly downloads files.
Let $F(t) \in \{0,1\}$ be the file state process of the user.  State ``1'' means there is a file in the system that has not completed its download,
and  ``0'' means no file is waiting.
The length of each file is independent and is either exponentially distributed or geometrically distributed (described in more detail below).  Let $\overline{B}$ denote the expected file size in bits. Time is slotted. At each slot in which there is an active file for downloading, the user makes a service decision that affects both the downloading success probability and the power expenditure. After a file is downloaded, the system goes idle (state $0$) and remains in the idle state for a random amount of time that is independent and geometrically distributed with parameter $\lambda>0$.

A transmission decision is made on each slot $t$ in which $F(t)=1$.  The decision affects the number of bits that are sent, the probability these bits are successfully received, and the power usage.
Let $\alpha(t)$ denote the decision variable at slot $t$ and let $\mathcal{A}$ represent an abstract action set.  The set $\mathcal{A}$ can represent a collection of modulation and coding
options for each transmission. Assume also that $\mathcal{A}$ contains an idle action denoted as ``0.'' The
decision  $\alpha(t)$ determines the following two values:
\begin{itemize}
  \item The probability of successfully downloading a file $\phi(\alpha(t))$, where $\phi(\cdot)\in[0,1]$ with $\phi(0)=0$.
  \item The power expenditure $p(\alpha(t))$, where $p(\cdot)$ is a nonnegative function with $p(0)=0$.
\end{itemize}
The user chooses $\alpha(t) = 0$ whenever $F(t)=0$.
The user chooses $\alpha(t) \in \mathcal{A}$ for each slot $t$ in which $F(t) = 1$, with the goal of maximizing throughput subject to a time average power constraint. The example where the decision set $\mathcal A$ is finite can be found in the simulation experiment section. 
Here is a simple example where the decision can be continuous:
\begin{example}
Let $\mathcal{A}$ be the set of all possible power allocation options, i.e. $\mathcal A := [p_{\min}, p_{\max}]\cup \{0\}$ where $p_{\min},p_{\max}>0$ are constants. Then, $\alpha(t)\in[p_{\min}, p_{\max}]\cup\{0\}$, $p(\alpha(t)) = \alpha(t)$ and the success probability of downloading a file can be 
$\phi(\alpha(t)) = 1-\exp(-\alpha(t))$. 
\end{example}

The problem can be described by a two state Markov decision process with binary state $F(t)$. Given $F(t)=1$, a file is currently in
the system. This file will finish its download at the end of the slot with probability $\phi(\alpha(t))$. Hence, the transition probabilities out of state $1$ are:
\begin{eqnarray}
Pr[F(t+1) = 0 | F(t)=1] &=& \phi(\alpha(t))  \label{eq:trans1} \\
Pr[F(t+1) = 1 | F(t) = 1] &=& 1-\phi(\alpha(t)) \label{eq:trans2}
\end{eqnarray}
Given $F(t)=0$, the system is idle and will transition to the active state in the next slot with probability $\lambda$,
so that:
\begin{eqnarray}
Pr[F(t+1) = 1 | F(t) = 0] &=& \lambda  \label{eq:trans3} \\
Pr[F(t+1) =0 | F(t) = 0] &=& 1- \lambda \label{eq:trans4}
\end{eqnarray}

Define the throughput, measured by bits per slot, as:
\begin{equation}
    \liminf_{T\rightarrow\infty}\frac{1}{T}\sum_{t=0}^{T-1}\overline{B}\phi(\alpha(t)) \nonumber
\end{equation}
The file downloading problem reduces to the following:
\begin{align}
    \mbox{Maximize:} &~\liminf_{T\rightarrow\infty}\frac{1}{T}\sum_{t=0}^{T-1}\overline{B}\phi(\alpha(t)) \label{eq:prob-1}\\
    \mbox{Subject to:} &~\limsup_{T\rightarrow\infty}\frac{1}{T}\sum_{t=0}^{T-1}p(\alpha(t))\leq \beta \label{eq:prob-2}\\
    & \alpha(t) \in \mathcal{A} \mbox{ $\forall t \in \{0, 1,2, \ldots\}$ such that $F(t) =1$} \label{eq:prob-3} \\
    & \mbox{Transition probabilities satisfy \eqref{eq:trans1}-\eqref{eq:trans4}} \label{eq:prob-4}
\end{align}
where $\beta$ is a positive constant that determines the desired average power constraint. 

\subsection{The memoryless file size assumption}

The above model assumes that file completion success on slot $t$ depends only on the
transmission decision $\alpha(t)$, independent of history.   This implicitly assumes that file
length distributions have a \emph{memoryless property} where the residual file length is independent of the amount already delivered. Further, it is assumed that if the controller selects a transmission rate  that is larger than the residual bits in the file, the remaining portion of the transmission is padded with \emph{fill bits}.  This ensures error events provide no information about the residual file length beyond the already known 0/1 binary file state.  Of course, error probability might be improved by removing padded bits. However,  this affects only the last transmission of a file and has negligible impact when expected file size is large in comparison to the amount that can be transmitted in one slot.  Note that padding
is not needed in the special case when all transmissions send one fixed length packet.

The memoryless property holds when each file $i$ has independent length $B_i$ that is  \emph{exponentially distributed} with mean length $\overline{B}$ bits, so that:
\[ Pr[B_i > x] = e^{-x/\overline{B}} \mbox{ for $x >0$} \]
For example, suppose the \emph{transmission rate} $r(t)$ (in units of bits/slot) and the \emph{transmission success probability} $q(t)$ are given by general functions of $\alpha(t)$:
\begin{eqnarray*}
 r(t) &=& \hat{r}(\alpha(t)) \\
 q(t) &=& \hat{q}(\alpha(t))
 \end{eqnarray*}
 Then the file completion probability $\phi(\alpha(t))$ is the probability that the \emph{residual} amount of bits in the file is less than or equal to $r(t)$, \emph{and} that the transmission of these residual bits is a success.  By the memoryless property of the exponential distribution, the residual file length is distributed the same as the original
 file length.  Thus:
 \begin{eqnarray}
\phi(\alpha(t))
  &=& \hat{q}(\alpha(t))Pr[B_i \leq \hat{r}(\alpha(t))] \nonumber \\
  &=& \hat{q}(\alpha(t)) \int_{0}^{\hat{r}(\alpha(t))} \frac{1}{\overline{B}} e^{-x/\overline{B}} dx \label{eq:approx1}
  \end{eqnarray}

Alternatively, history independence holds when each file $i$ consists of a random number $Z_i$ of fixed length packets, where $Z_i$ is geometrically distributed with mean $\overline{Z} = 1/\mu$.   Assume each transmission
sends exactly one packet, but different power levels affect the transmission success probability $q(t) = \hat{q}(\alpha(t))$.   Then:
\begin{equation} \label{eq:approx2}
\phi(\alpha(t)) = \mu\hat{q}(\alpha(t))
\end{equation}

The memoryless file length assumption allows the file state  to be modeled by a simple
binary-valued process $F(t) \in \{0, 1\}$. However, actual file sizes may not have an exponential or geometric distribution.
 One way to treat general distributions is to \emph{approximate} the file sizes as being memoryless by
using a $\phi(\alpha(t))$ function defined by either \eqref{eq:approx1} or \eqref{eq:approx2}, formed by matching the average file size $\overline{B}$ or average number of packets $\overline{Z}$. The \emph{decisions} $\alpha(t)$ are made according to the algorithm below, but the actual event outcomes that arise from these decisions are not memoryless.   A simulation comparison of this approximation is provided in Section \ref{section:sims}, where it is shown to be remarkably accurate (see Fig. \ref{fig:Stupendous6}).

 The algorithm in this section optimizes over the class of all algorithms that do not use residual file length information.
This maintains low complexity by ensuring a user has a binary-valued
Markov state $F(t) \in \{0,1\}$. While a system controller might know the
residual file length, incorporating this knowledge creates a Markov decision problem with an infinite number of states (one for each possible value of residual length) which significantly complicates the scenario.

\subsection{DPP ratio optimization}

This subsection develops an online algorithm for problem \eqref{eq:prob-1}-\eqref{eq:prob-4}. This algorithm follows from Algorithm \ref{dpp-ratio-algorithm} in Chapter 1 with some customizations towards this application.
First, notice that file state ``$1$'' is recurrent under any decisions for $\alpha(t)$. Denote $t_k$ as the $k$-th time when the system returns to state ``1.''
Define the renewal frame as the time period between $t_k$ and $t_{k+1}$.  Define the \emph{frame size}:
\[ T[k] = t_{k+1} - t_k \]
Notice that $T[k]=1$ for any frame $k$ in which the file does not complete its download.  If the file is completed on frame $k$, then $T[k] = 1 + G_k$, where $G_k$ is a geometric random variable with mean
$\expect{G_k} = 1/\lambda$.   Each frame $k$ involves only a single decision $\alpha(t_k)$ that is made at the beginning of the frame.  Thus, the total power used over the duration of frame $k$ is:
\begin{equation}\label{extra1}
\sum_{t=t_{k}}^{t_{k+1}-1}p(\alpha(t)) = p(\alpha(t_k))
\end{equation}
We treat the time average constraint in \eqref{eq:prob-2} using a virtual queue $Q[k]$ that is updated every frame $k$ by:
\begin{equation}
    Q[k+1]=\max\left\{Q[k]
     + p(\alpha(t_k)) - \beta T[k],~0\right\} \label{eq:q-update}
\end{equation}
with initial condition $Q[0]=0$.
The algorithm is then parameterized by a constant $V\geq0$ which affects a performance tradeoff.  At the beginning of the $k$-th renewal frame, the user observes virtual queue $Q[k]$ and chooses $\alpha(t_k)$ to maximize the following drift-plus-penalty (DPP) ratio:
\begin{equation}\label{e4}
    \max_{\alpha(t_k)\in\mathcal{A}} ~~\frac{V\overline{B}\phi(\alpha(t_k))- Q[k]p(\alpha(t_k))}
        {\mathbb{E}[T[k]|\alpha(t_k)]}
\end{equation}
The numerator of the above ratio adds a ``queue drift term''  $-Q[k]p(\alpha(t_k))$ to the
``current reward term'' $V\overline{B}\phi(\alpha(t_k))$. The intuition
is that  it is desirable to have a large value of current reward, but it is also desirable  to have a large drift (since this tends to decrease queue size).   Creating a weighted sum of these two terms and dividing by the expected frame size
gives a simple index.  The next subsections show that, for the context of the current work, this index leads to an algorithm that pushes throughput arbitrarily close to optimal (depending on the chosen $V$ parameter) with a
strong sample path guarantee on average power expenditure.

The denominator in \eqref{e4} can  easily be computed via the transition model \eqref{eq:trans1}-\eqref{eq:trans4}:
\begin{eqnarray}\label{eq:frame-size}
\mathbb{E}[T[k]|\alpha(t_k)] = 1-\phi(\alpha(t_k)) + \phi(\alpha(t_k))\cdot\left(1+\frac{1}{\lambda}\right) =  1+\frac{\phi(\alpha(t_k))}{\lambda} 
\end{eqnarray}
Thus, \eqref{e4} is equivalent to
\begin{equation}\label{e5}
    \max_{\alpha(t_k)\in\mathcal{A}} ~~\frac{V\overline{B}\phi(\alpha(t_k))- Q[k]p(\alpha(t_k))}
        {1+\phi(\alpha(t_k))/\lambda}
\end{equation}
This gives the following Algorithm \ref{alg:file} for the single-user case:
\begin{algorithm}
\begin{Alg}\label{alg:file}
\begin{itemize}
  \item At each time $t_{k}$, the user observes virtual queue $Q[k]$ and chooses $\alpha(t_k)$ as the solution to (\ref{e5}) (where ties are broken arbitrarily).
  \item The value $Q[k+1]$ is computed according to \eqref{eq:q-update} at the end of the $k$-th frame.
\end{itemize}
\end{Alg}
\end{algorithm}
The expected performance analysis of this algorithm follows from that of Section \ref{sec:simple-analysis} and we omit the details for brevity. In the following, we give a stronger probability 1 performance analysis taking into account the special property of the algorithm in this customized setting.

\subsection{Average power constraints via queue bounds}
In this section, we show that the proposed algorithm makes the virtual queue deterministically bounded.

\begin{lemma} \label{lem:1}
If there is a constant $C\geq 0$ such that $Q[k] \leq C$ for all $k \in \{0, 1, 2, \ldots\}$, then:
\[ \limsup_{T\rightarrow\infty}\frac{1}{T}\sum_{t=0}^{T-1}p(\alpha(t))\leq \beta \]
\end{lemma}
\begin{proof}
From \eqref{eq:q-update}, we know that for each frame $k$:
\begin{equation}
    Q[k+1] \geq Q[k]
     +p(\alpha(t_{k}))-T[k]\beta \nonumber
\end{equation}
Rearranging terms and using $T[k] = t_{k+1} -t_k$ gives:
\[ p(\alpha(t_k)) \leq (t_{k+1} - t_k)\beta + Q[k+1]-Q[k] \]
Fix $K>0$. Summing over $k\in\{0,1,\cdots,K-1\}$  gives:
\begin{eqnarray*}
    \sum_{k=0}^{K-1}p(\alpha(t_{k})) &\leq& (t_{K}-t_0) \beta + Q[K] - Q[0] \\
    &\leq&  t_K\beta + C
\end{eqnarray*}
The sum power over the first $K$ frames is the same as the sum up to time $t_{K}-1$, and so:
\[ \sum_{t=0}^{t_K-1} p(\alpha(t)) \leq t_K \beta + C \]
Dividing by $t_K$ gives:
\[ \frac{1}{t_K}\sum_{t=0}^{t_K-1} p(\alpha(t)) \leq \beta + C/t_K. \]
Taking $K\rightarrow\infty$, then,
\begin{equation} \label{eq:sub-here}
\limsup_{K\rightarrow\infty}\frac{1}{t_K}\sum_{t=0}^{t_K-1} p(\alpha(t)) \leq \beta
\end{equation}
Now for each positive integer $T$, let $K(T)$ be the integer such that $t_{K(T)} \leq T < t_{K(T)+1}$.
Since power
is only used at the first slot of a  frame, one has:
\[ \frac{1}{T}\sum_{t=0}^{T-1} p(\alpha(t))\leq \frac{1}{t_{K(T)}}\sum_{t=0}^{t_{K(T)}-1}p(\alpha(t)) \]
Taking a $\limsup$ as $T\rightarrow\infty$ and using \eqref{eq:sub-here} yields the result.
\end{proof}

In order to show that the queue process under our proposed algorithm is deterministically bounded, we need the following assumption:
\begin{assumption}\label{as:power-assumption}
The following quantities are finite and strictly positive:
\begin{eqnarray*}
p^{min} &=& \min_{\alpha\in\mathcal{A}\setminus\{0\}}p(\alpha)\\
p^{max}&=&\max_{\alpha\in\mathcal{A}\setminus\{0\}}p(\alpha).
\end{eqnarray*}
\end{assumption}

\begin{lemma} \label{lem:2}
Suppose Assumption \ref{as:power-assumption} holds. If $Q[0]=0$, then under our algorithm we have for all $k>0$:
\[ Q[k]\leq \max\left\{\frac{V\overline{B}}{p^{min}}+p^{max}-\beta,0\right\}\]
\end{lemma}

\begin{proof}
First, consider the case when $p^{max} \leq \beta$. From \eqref{eq:q-update} and the fact that $T[k] \geq 1$ for all $k$,
 it is clear  the queue can never increase, and so $Q[k] \leq Q[0]=0$ for all $k>0$.

Next, consider the case when $p^{max} > \beta$. We prove the assertion by induction on $k$.
The result trivially holds for $k=0$.  Suppose it holds at $k=l$ for $l>0$, so that:
\[ Q[l]\leq \frac{V\overline{B}}{p^{min}}+p^{max}-\beta \]
We are going to prove that the same holds for $k=l+1$.  There are two cases:
\begin{enumerate}
  \item $Q[l]\leq\frac{V\overline{B}}{p^{min}}$. In this case we have by \eqref{eq:q-update}:
  \begin{eqnarray*}
   Q[l+1] &\leq& Q[l] + p^{max} - \beta \\
   &\leq& \frac{V\overline{B}}{p^{min}} + p^{max} - \beta
   \end{eqnarray*}

    \item $\frac{V\overline{B}}{p^{min}}<Q[l]\leq\frac{V\overline{B}}{p^{min}}+p^{max}-\beta$.     In this case,
    we use proof by contradiction.
    If $p(\alpha(t_l))=0$ then the queue cannot increase, so:
    \[ Q[l+1] \leq Q[l] \leq   \frac{V\overline{B}}{p^{min}} + p^{max} - \beta  \]
    On the other hand, if $p(\alpha(t_l))>0$ then $p(\alpha(t_l)) \geq p^{min}$ and so the numerator in
    \eqref{e5} satisfies:
    \begin{eqnarray*}
    V\overline{B}\phi(\alpha(t_l)) - Q[l]p(\alpha(t_l))
    &\leq& V\overline{B} - Q[l]p^{min} < 0
    \end{eqnarray*}
    and so the maximizing ratio in \eqref{e5} is negative.  However, the maximizing ratio in \eqref{e5} \emph{cannot} be negative
    because the alternative  choice $\alpha(t_l)=0$ increases the ratio to 0.  This contradiction implies that
    we cannot have $p(\alpha(t_l))>0$.
\end{enumerate}
\end{proof}

The above is a \emph{sample path result} that only assumes parameters satisfy $\lambda >0$, $\overline{B}>0$, and $0 \leq \phi(\cdot) \leq 1$.  Thus, the algorithm meets the average power constraint even if it uses incorrect values for these parameters. The next subsection provides a throughput optimality result when these parameters match the true system values.

\subsection{Optimality over randomized algorithms}

Consider the following class of \emph{i.i.d. randomized algorithms}:  Let $\theta(\alpha)$ be non-negative numbers defined for each $\alpha \in \mathcal{A}$, and suppose they satisfy $\sum_{\alpha \in \mathcal{A}} \theta(\alpha) = 1$.  Let $\alpha^*(t)$ represent a policy that, every slot $t$ for which $F(t)=1$, chooses $\alpha^*(t) \in \mathcal{A}$ by independently selecting strategy $\alpha$ with probability $\theta(\alpha)$.    Then $(p(\alpha^*(t_k)), \phi(\alpha^*(t_k)))$ are independent and identically distributed (i.i.d.) over frames $k$. Under this algorithm, it follows by the law of large numbers that the throughput and power expenditure satisfy (with probability 1):
\begin{eqnarray*}
\lim_{t\rightarrow\infty} \frac{1}{T}\sum_{t=0}^{T-1} \overline{B}\phi(\alpha^*(t)) &=& \frac{\overline{B}\expect{\phi(\alpha^*(t_k))}}{1 + \expect{\phi(\alpha^*(t_k))}/\lambda} \\
\lim_{t\rightarrow\infty} \frac{1}{T}\sum_{t=0}^{T-1} p(\alpha^*(t)) &=& \frac{\expect{p(\alpha^*(t_k))}}{1 + \expect{\phi(\alpha^*(t_k))}/\lambda}
\end{eqnarray*}
It can be shown that optimality of problem \eqref{eq:prob-1}-\eqref{eq:prob-4} can be achieved over this class. Thus, there exists an i.i.d. randomized algorithm $\alpha^*(t)$ that satisfies:
\begin{eqnarray}
\frac{\overline{B}\expect{\phi(\alpha^*(t_k))}}{1 + \expect{\phi(\alpha^*(t_k))}/\lambda} &=& \mu^* \label{eq:iid1} \\
 \frac{\expect{p(\alpha^*(t_k))}}{1 + \expect{\phi(\alpha^*(t_k))}/\lambda} &\leq& \beta \label{eq:iid2}
\end{eqnarray}
where $\mu^*$ is the optimal throughput for the problem \eqref{eq:prob-1}-\eqref{eq:prob-4}.

\subsection{Key feature of the drift-plus-penalty ratio}

Define $\mathcal{F}(t_k)$ as the \emph{system history} up to frame $k$, which includes which includes the actions taken $\alpha(t_0),\cdots,\alpha(t_{k-1}),$ frame lengths $T[0],\cdots,T[k-1]$, the busy period in each frame, the idle period in each frame, and the queue value $Q[k]$ (since
this is determined by the random events before frame $k$).
Consider the algorithm
that, on frame $k$, observes $Q[k]$ and
chooses $\alpha(t_k)$ according to \eqref{e5}.  The following key feature of this algorithm can be
shown (see \cite{Neely2010} for related results):
\begin{align*}
\frac{\expect{-V\overline{B}\phi(\alpha(t_k)) + Q[k]p(\alpha(t_k))|\mathcal{F}(t_k)}}{\expect{1 + \phi(\alpha(t_k))/\lambda|\mathcal{F}(t_k)}}
\leq
\frac{\expect{-V\overline{B}\phi(\alpha^*(t_k)) + Q[k]p(\alpha^*(t_k))|\mathcal{F}(t_k)}}{\expect{1 + \phi(\alpha^*(t_k))/\lambda|\mathcal{F}(t_k)}}
\end{align*}
where $\alpha^*(t_k)$ is any (possibly randomized) alternative decision that is based only on $\mathcal{F}(t_k)$.
 This is an intuitive property:  By design, the algorithm in \eqref{e5} observes $\mathcal{F}(t_k)$ and then chooses a particular action $\alpha(t_k)$ to minimize the ratio over all deterministic actions.  Thus, as can be shown, it also minimizes the ratio over all potentially randomized actions.
Using the (randomized) i.i.d. decision $\alpha^*(t_k)$ from \eqref{eq:iid1}-\eqref{eq:iid2} in the above and noting that this alternative decision
is independent of $\mathcal{F}(t_k)$ gives:
\begin{align}
&\frac{\expect{-V\overline{B}\phi(\alpha(t_k)) + Q[k]p(\alpha(t_k))|\mathcal{F}(t_k)}}{\expect{1 + \phi(\alpha(t_k))/\lambda|\mathcal{F}(t_k)}}
\leq -V\mu^* + Q[k]\beta \label{eq:key-feature}
\end{align}

\subsection{Performance theorem}

\begin{theorem}
Suppose Assumption \ref{as:power-assumption} holds.
The proposed algorithm achieves the constraint $\limsup_{T\rightarrow\infty}\frac{1}{T}\sum_{t=0}^{T-1}p(\alpha(t))\leq \beta$ and yields throughput satisfying (with probability 1):
\begin{equation}\label{e6}
    \liminf_{T\rightarrow\infty}\frac{1}{T}\sum_{t=0}^{T-1}\overline{B}\phi(\alpha(t))\geq \mu^{*}-\frac{C_0}{V}
\end{equation}
where $C_0$ is a constant.\footnote{The constant $C_0$ is independent of $V$ and is given in the proof.}
\end{theorem}
\begin{proof}
First, for any fixed $V$, Lemma \ref{lem:2} implies that the queue is deterministically bounded.  Thus, according to Lemma \ref{lem:1}, the proposed algorithm achieves the constraint 
$$\limsup_{T\rightarrow\infty}\frac{1}{T}\sum_{t=0}^{T-1}p(\alpha(t))\leq \beta.$$ 
The rest is devoted to proving the throughput guarantee \eqref{e6}.

Define:
\begin{equation}
    L(Q[k])=\frac{1}{2}Q[k]^2. \nonumber
\end{equation}
We call this a \emph{Lyapunov function}. Define a frame-based Lyapunov Drift as:
\[ \Delta[k] = L(Q[k+1]) - L(Q[k]) \]
According to \eqref{eq:q-update} we get
\begin{equation}
    Q[k+1]^2\leq \left(Q[k]+p(\alpha(t_{k}))-T[k]\beta\right)^2. \nonumber
\end{equation}
Thus:
\begin{eqnarray*}
\Delta[k]  \leq \frac{(p(\alpha(t_k)) - T[k]\beta)^2}{2} + Q[k](p(\alpha(t_k)) - T[k]\beta)
\end{eqnarray*}
Taking a conditional expectation of the above given $\mathcal{F}(t_k)$ and recalling that $\mathcal{F}(t_k)$ includes the information
$Q[k]$ gives:
\begin{equation} \label{eq:drift}
 \expect{\Delta[k] |\mathcal{F}(t_k)} \leq C_0 + Q[k]\expect{p(\alpha(t_k)) - \beta T[k]| \mathcal{F}(t_k)}
 \end{equation}
where $C_0$ is a constant that satisfies the following for all possible histories $\mathcal{F}(t_k)$:
\[ \expect{\left.\frac{(p(\alpha(t_k)) - T[k]\beta)^2}{2}  \right| \mathcal{F}(t_k)} \leq C_0 \]
Such a constant $C_0$ exists because the power $p(\alpha(t_k))$ is deterministically bounded due to Assumption \ref{as:power-assumption}, and the frame sizes $T[k]$ are bounded in second moment regardless of history according to \eqref{eq:frame-size}.

Adding the ``penalty'' $-\expect{V\overline{B}\phi(\alpha(t_k))|\mathcal{F}(t_k)}$ to both sides of \eqref{eq:drift} gives:
\begin{align*}
&\expect{\Delta[k] - V\overline{B}\phi(\alpha(t_k))|\mathcal{F}(t_k)}  \\
&\leq C_0 +  \expect{-V\overline{B}\phi(\alpha(t_k)) + Q[k](p(\alpha(t_k)) - \beta T[k])| \mathcal{F}(t_k)} \\
&= C_0 - Q[k]\beta\expect{T[k]|\mathcal{F}(t_k)}   +  \frac{\expect{T[k]|\mathcal{F}(t_k)}\expect{-V\overline{B}\phi(\alpha(t_k)) + Q[k]p(\alpha(t_k))|\mathcal{F}(t_k)}}{\expect{T[k]|\mathcal{F}(t_k)}}
\end{align*}
Expanding $T[k]$ in the denominator of the last term  gives:
\begin{align*}
&\expect{\Delta[k] - V\overline{B}\phi(\alpha(t_k))|\mathcal{F}(t_k)}  \nonumber\\
&\leq C_0 - Q[k]\beta\expect{T[k]|\mathcal{F}(t_k)}  + \expect{T[k]|\mathcal{F}(t_k)} 
\frac{\expect{-V\overline{B}\phi(\alpha(t_k)) + Q[k]p(\alpha(t_k))|\mathcal{F}(t_k)}}{\expect{1 + \phi(\alpha(t_k))/\lambda|\mathcal{F}(t_k)}}
\end{align*}
Substituting \eqref{eq:key-feature} into the above expression gives:
\begin{align}
&\expect{\Delta[k] - V\overline{B}\phi(\alpha(t_k))|\mathcal{F}(t_k)} \nonumber  \\
&\leq  C_0 - Q[k]\beta\expect{T[k]|\mathcal{F}(t_k)}+ \expect{T[k]|\mathcal{F}(t_k)} (-V\mu^* + \beta Q[k])\nonumber \\
&= C_0 -V\mu^* \expect{T[k] | \mathcal{F}(t_k)}  \label{eq:dude}
\end{align}
Rearranging gives:
\begin{equation} \label{eq:dpp}
\expect{\Delta[k] + V(\mu^*T[k] - \overline{B}\phi(\alpha(t_k)))|\mathcal{F}(t_k)} \leq C_0 
\end{equation}
This implies that $\Delta[k] + V(\mu^*T[k] - \overline{B}\phi(\alpha(t_k))) - C_0$ is a supermartingale difference sequence.
Furthermore, we already know the queue $Q[k]$ is deterministically bounded, it follows that:
\[ \sum_{k=1}^{\infty} \frac{\expect{\Delta[k]^2}}{k^2} < \infty \]
This, together with \eqref{eq:dpp}, implies by Lemma \ref{prob-1-converge} that (with probability 1):
\[  \limsup_{K\rightarrow\infty} \frac{1}{K}\sum_{k=0}^{K-1} \left[\mu^* T[k] - \overline{B}\phi(\alpha(t_k))\right] \leq \frac{C_0}{V} \]
Thus, for any $\epsilon>0$ one has for all sufficiently large $K$:
\[ \frac{1}{K}\sum_{k=0}^{K-1}[\mu^* T[k] - \overline{B} \phi(\alpha(t_k))] \leq \frac{C_0}{V} + \epsilon \]
Rearranging implies that for all sufficiently large $K$:
\begin{eqnarray*}
\frac{\sum_{k=0}^{K-1} \overline{B}\phi(\alpha(t_k))}{\sum_{k=0}^{K-1} T[k]} &\geq& \mu^*  - \frac{(C_0/V + \epsilon)}{\frac{1}{K}\sum_{k=0}^{K-1}T[k]}\\
&\geq& \mu^* - (C_0/V + \epsilon)
\end{eqnarray*}
where the final inequality holds because $T[k] \geq 1$ for all $k$. Thus:
\[ \liminf_{K\rightarrow\infty} \frac{\sum_{k=0}^{K-1} \overline{B}\phi(\alpha(t_k))}{\sum_{k=0}^{K-1} T[k]} \geq \mu^* - (C_0/V + \epsilon) \]
The above holds for all $\epsilon>0$.  Taking a limit as $\epsilon\rightarrow 0$ implies:
\[ \liminf_{K\rightarrow\infty} \frac{\sum_{k=0}^{K-1} \overline{B}\phi(\alpha(t_k))}{\sum_{k=0}^{K-1} T[k]} \geq \mu^* - C_0/V. \]
Notice that $\phi(\alpha(t))$ only changes at the boundary of each frame and remains 0 within the frame. Thus, we can replace the sum over frames $k$ by a sum over slots $t$. The desired result follows.
\end{proof}

The theorem shows that throughput can be pushed within $O(1/V)$ of the optimal value $\mu^*$, where $V$ can be chosen as large as desired to ensure throughput is arbitrarily close to optimal.  The tradeoff is a queue bound that grows linearly with $V$ according to Lemma \ref{lem:2}, which affects the convergence time required for the constraints to be close to the desired time averages (as described in the proof of Lemma \ref{lem:1}).

\section{Multi-user file downloading}

\begin{figure}[htbp]
   \centering
   \includegraphics[height=2.5in]{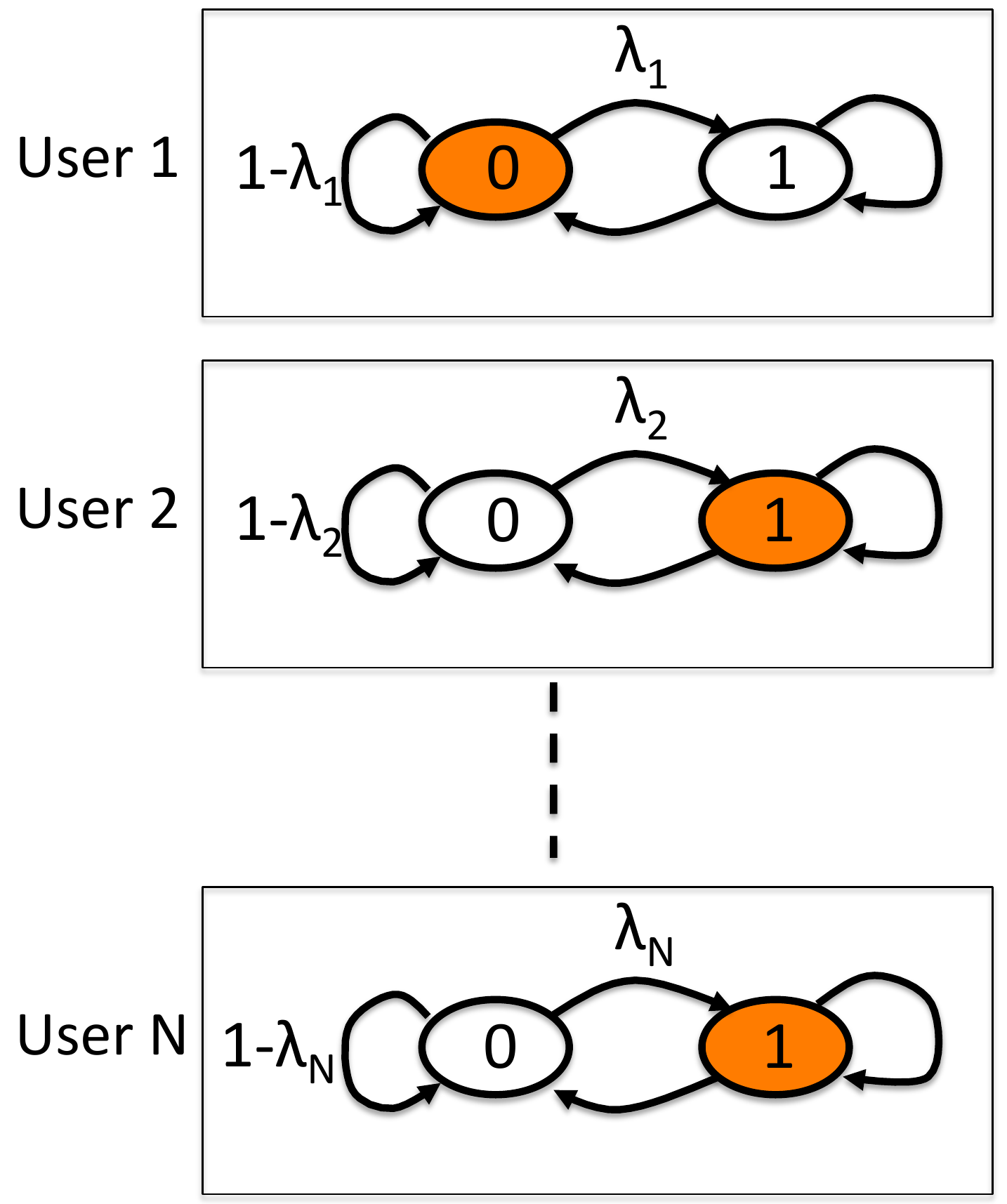} 
   \caption{A system with $N$ users.  The shaded node for each user $n$ indicates the current file state $F_n(t)$  of that user.  There are $2^N$ different state vectors.}
   \label{fig:multi-user-picture}
\end{figure}

This section considers a multi-user file downloading system that
consists of $N$ single-user subsystems.  Each subsystem is similar to the single-user system described in the previous section.
Specifically, for  the $n$-th user (where $n \in \{1, \ldots, N\}$):
\begin{itemize}
\item The file state process is $F_n(t) \in \{0,1\}$.
\item The transmission decision is $\alpha_n(t) \in \mathcal{A}_n$, where $\mathcal{A}_n$ is an abstract set of transmission options for user $n$.
\item The power expenditure on slot $t$ is $p_n(\alpha_n(t))$.
\item The success probability on a slot $t$ for which $F_n(t)=1$ is $\phi_n(\alpha_n(t))$, where $\phi_n(\cdot)$ is the function that describes file completion probability for user $n$.
\item The idle period parameter is $\lambda_n>0$.
\item The average file size is $\overline{B}_n$ bits.
\end{itemize}
Assume that the random variables associated with different subsystems are mutually independent.  The resulting Markov decision problem has $2^N$ states, as shown in Fig. \ref{fig:multi-user-picture}. The transition probabilities for each active user depends on which users are selected for transmission and on the corresponding transmission modes.
 This is a \emph{restless bandit system} because there can also be transitions for non-selected users (specifically, it is possible to transition from inactive to active).

To control the downloading process, there is a
central server with only $M$ threads ($M<N$), meaning that \emph{at most $M$ jobs can be processed simultaneously}. So at each time slot, the server has to make decisions selecting at most $M$ out of $N$ users to transmit a portion of their files. These decisions are further restricted by a global time average power constraint. The goal is to maximize the aggregate throughput, which is defined as
\begin{equation}
    \liminf_{T\rightarrow\infty}\frac{1}{T}\sum_{t=0}^{T-1}\sum_{n=1}^Nc_{n}\overline{B}_n\phi(\alpha_{n}(t)) \nonumber
\end{equation}
where $c_1, c_2, \ldots, c_N$ are a collection of positive weights that can be used to prioritize users.
Thus, this multi-user file downloading problem reduces  to the following:
\begin{align}
   \mbox{Max:} &~\liminf_{T\rightarrow\infty}\frac{1}{T}\sum_{t=0}^{T-1}\sum_{n=1}^N
    c_{n}\overline{B}_n\phi_n(\alpha_{n}(t)) \label{eq:multi-1} \\
    \mbox{S.t.:}&~\limsup_{T\rightarrow\infty}\frac{1}{T}\sum_{t=0}^{T-1}\sum_{n=1}^Np_n(\alpha_{n}(t))\leq \beta \label{eq:multi-2}\\
    &~\sum_{n=1}^NI(\alpha_{n}(t))\leq M~~\forall t\in\{0,1,2,\cdots\} \label{eq:mutli-3} \\
    &~Pr[F_n(t+1)=1~|~F_n(t)=0]=\lambda_n\label{eq:multi-4}\\
    &~Pr[F_n(t+1)=0~|~F_n(t)=1]=\phi_n(\alpha_n(t))\label{eq:multi-5}
\end{align}
where the constraints \eqref{eq:multi-4}-\eqref{eq:multi-5} hold for all $n \in \{1, \ldots, N\}$ and $t \in \{0, 1, 2,\ldots\}$, and
where $I(\cdot)$ is the indicator function defined as:
\begin{equation}
    I(x)=\left\{
           \begin{array}{ll}
             0, & \hbox{if $x=0$;} \\
             1, & \hbox{otherwise.}
           \end{array}
         \right.\nonumber
\end{equation}

\subsection{DPP ratio indexing algorithm} \label{multi:indexing}
This section develops our indexing algorithm for the multi-user case using the single-user case as a stepping stone. The major difficulty is  the instantaneous constraint $\sum_{n=1}^NI(\alpha_{n}(t))\leq M$.  Temporarily neglecting this constraint,
we use Lyapunov optimization to deal with the time average power constraint first.

We introduce a virtual queue $Q(t)$, which is again 0 at $t=0$. Instead of updating it on a frame basis, the server updates this queue every slot as follows:
\begin{equation}\label{m2}
    Q(t+1)=\max\left\{Q(t)+\sum_{n=1}^Np_n(\alpha_n(t))-\beta,0\right\}.
\end{equation}
Define $\mathcal{N}(t)$ as the set of users beginning their renewal frames at time $t$, so that $F_n(t)=1$ for all such users.
 In general, $\mathcal{N}(t)$ is a subset of $\mathcal{N}=\{1,2,\cdots,N\}$. Define $|\mathcal{N}(t)|$ as the number of users
 in the set $\mathcal{N}(t)$.

At each time slot $t$, the server observes the queue state $Q(t)$ and chooses $(\alpha_1(t), \ldots, \alpha_N(t))$
in a manner similar to the single-user case. Specifically, for each user $n \in \mathcal{N}(t)$ define:
\begin{equation}
    g_n(\alpha_n(t))\triangleq\frac{Vc_n\overline{B}_n\phi_n(\alpha_{n}(t))- Q(t)p_n(\alpha_{n}(t))}
        {1+\phi_n(\alpha_n(t))/\lambda_n} \label{eq:multi-alg1}
\end{equation}
This
is similar to the expression \eqref{e5} used in the single-user optimization.   Call $g_n(\alpha_n(t))$ a \emph{reward}.
Now define an index for each subsystem $n$ by:
\begin{equation}\label{m5}
    \gamma_n(t)\triangleq\max_{\alpha_n(t)\in\mathcal{A}_n}g_n(\alpha_n(t))
\end{equation}
which is the maximum possible reward one can get from the $n$-th subsystem at time slot $t$. Thus, it is natural to define the following myopic algorithm:  Find the (at most) $M$ subsystems in $\mathcal{N}(t)$ with the greatest rewards, and serve these with their corresponding
optimal $\alpha_n(t)$ options in $\mathcal{A}_n$ that maximize $g_n(\alpha_n(t))$.
\begin{algorithm}
\begin{Alg}\label{alg:ratio-index}~
\begin{itemize}
  \item At each time slot $t$, the server observes virtual queue state $Q(t)$ and computes the indices using (\ref{m5}) for all $n\in\mathcal{N}(t)$.
  \item Activate the $\min[M, |\mathcal{N}(t)|]$ subsystems with greatest indices, using their corresponding actions $\alpha_n(t) \in \mathcal{A}_n$ that
  maximize $g_n(\alpha_n(t))$.
  \item Update $Q(t)$ according to \eqref{m2}  at the end of each slot $t$.
\end{itemize}
\end{Alg}
\end{algorithm}

\subsection{Theoretical performance analysis}

In this subsection, we show that the above algorithm always satisfies the desired time average power constraint. We adopt the following assumption:
\begin{assumption}\label{as:constants}
The following quantities are finite and strictly positive.
\begin{eqnarray*}
p^{min}_n &=& \min_{\alpha_n\in\mathcal{A}_n\setminus\{0\}}p_n(\alpha_n) \\
p^{min} &=& \min_np^{min}_n \\
p^{max}_n &=& \max_{\alpha_n\in\mathcal{A}_n}p_n(\alpha_n) \\
c^{max} &=& \max_{n}c_n \\
\overline{B}^{max} &=& \max_n \overline{B}_n
\end{eqnarray*}
\end{assumption}

\begin{lemma} \label{lem:3}
Suppose Assumption \ref{as:constants} holds. Then, the
queue $\{Q(t)\}_{t=0}^{\infty}$  is deterministically bounded under Algorithm \ref{alg:ratio-index}. Specifically, we have for all $t \in \{0, 1, 2, \ldots\}$:
\[ Q(t)\leq \max\left\{\frac{Vc^{max}\overline{B}^{max}}{p^{min}}+\sum_{n=1}^Np^{max}_n-\beta,0\right\}\]
\end{lemma}
\begin{proof}
First, consider the case when $\sum_{n=1}^Np^{max}_n \leq \beta$. Since $Q(0)=0$, it is clear from the updating
rule \eqref{m2} that $Q(t)$ will remain 0 for all $t$.

Next, consider the case when $\sum_{n=1}^Np^{max}_n > \beta$. We prove the assertion by induction on $t$. The result
trivially holds for $t=0$. Suppose at $t=t'$, we have:
\[ Q(t')\leq\frac{Vc^{max}\overline{B}^{max}}{p^{min}}+\sum_{n=1}^Np^{max}_n-\beta \]
We are going to prove that the same statement holds for $t=t'+1$. We further divide it into two cases:
\begin{enumerate}
  \item $Q(t')\leq\frac{Vc^{max}\overline{B}^{max}}{p^{min}}$. In this case, since the queue increases by at most
  $\sum_{n=1}^Np^{max}_n - \beta$ on one slot, we have:
  \[ Q(t'+1) \leq \frac{Vc^{max}\overline{B}^{max}}{p^{min}} + \sum_{n=1}^Np^{max}_n - \beta\]

  \item $\frac{Vc^{max}\overline{B}^{max}}{p^{min}}<Q(t')\leq\frac{Vc^{max}\overline{B}^{max}}{p^{min}}+\sum_{n=1}^Np^{max}_n-\beta$. In this case, since $\phi_n(\alpha_n(t'))\leq 1$, there is no possibility that $Vc_n\overline{B}_n\phi_n(\alpha_n(t'))\geq Q(t')p_n(\alpha_n(t'))$ unless $\alpha_n(t')=0$. Thus, the DPP ratio indexing algorithm of minimizing
  \eqref{eq:multi-alg1} chooses  $\alpha_n(t')=0$ for all $n$. Thus, all indices are 0. This implies that $Q(t'+1)$ cannot increase, and we get $Q(t'+1)\leq\frac{Vc^{max}\overline{B}^{max}}{p^{min}}+\sum_{n=1}^Np^{max}_n-\beta$.
\end{enumerate}
\end{proof}

\begin{theorem}
The proposed DPP ratio indexing algorithm achieves the constraint:
\[ \limsup_{T\rightarrow\infty}\frac{1}{T}\sum_{t=0}^{T-1}\sum_{n=1}^Np_n(\alpha_{n}(t))\leq \beta \]
\end{theorem}
\begin{proof}
First of all, similar to Lemma \ref{lem:1}, one can show that if $Q(t)\leq C$ for some constant $C>0$ and any $t\in\{0,1,2,\cdots\}$, then, $\limsup_{T\rightarrow\infty}\frac{1}{T}\sum_{t=0}^{T-1}\sum_{n=1}^Np_n(\alpha_{n}(t))\leq \beta$. Using Lemma \ref{lem:3} we finish the proof.
\end{proof}


\section{Multi-user optimality in a special case}\label{section:coupling}
 In general, it is very difficult to prove optimality of the above multi-user algorithm.
There are mainly two reasons.
The first reason is that multiple users
might renew themselves asynchronously, making it
difficult to define a ``renewal frame'' for the whole system.
Thus, the proof technique in Theorem 1 is infeasible.
The second reason is that, even without the time average constraint, the problem
degenerates into a standard restless bandit problem where the optimality of
 indexing is not guaranteed.

This section considers
a special case of the multi-user file downloading problem where
the DPP ratio indexing algorithm is provably optimal.
The special case has no time average power constraint.  Further,
for each user $n \in \{1, \ldots, N\}$:
\begin{itemize}
  \item Each file consists of a random number of fixed length packets with mean $\overline{B}_n=1/\mu_n$.
  \item The decision set $\mathcal{A}_n= \{0,1\}$, where 0 stands for ``idle'' and 1 stands for ``download.'' If $\alpha_n(t)=1$, then user $n$ successfully downloads a single packet.
  \item $\phi_n(\alpha_n(t)) = \mu_n\alpha_n(t)$.
  \item Idle time is geometrically distributed with mean $1/\lambda_n$.
  \item The special case \emph{$\mu_n = 1-\lambda_n$ is assumed}.
\end{itemize}
 The assumption that the file length and idle time parameters $\mu_n$ and $\lambda_n$
satisfy $\mu_n = 1-\lambda_n$ is restrictive. However, there exists certain queueing system which admits \emph{exactly the same markov dynamics} as the system considered here when the assumption holds
(described in Section \ref{subsection:single_buffer} below). More importantly, it allows us to implement the stochastic coupling idea to prove the optimality.

The goal is to maximize the sum throughput (in units of packets/slot), which is defined as:
\begin{equation}\label{thput}
    \liminf_{T\rightarrow\infty}\frac{1}{T}\sum_{t=0}^{T-1}\sum_{n=1}^N \overline{B}_n\phi(\alpha_{n}(t)).
\end{equation}
In this special case, the multi-user file downloading problem reduces to the following:
\begin{align}
   \mbox{Max:} &~\liminf_{T\rightarrow\infty}\frac{1}{T}\sum_{t=0}^{T-1}\sum_{n=1}^N\alpha_{n}(t) \label{sp:multi-1} \\
    \mbox{S.t.:}&~\sum_{n=1}^N\alpha_{n}(t)\leq M~~\forall t\in\{0,1,2,\cdots\} \label{sp:mutli-2} \\
        &~\alpha_n(t) \in \{0, F_n(t)\} \label{sp:multi-2b}\\
    &~Pr[F_n(t+1)=1~|~F_n(t)=0]=\lambda_n\label{sp:multi-3}\\
    &~Pr[F_n(t+1)=0~|~F_n(t)=1]=\alpha_{n}(t)(1-\lambda_n)\label{sp:multi-4}
\end{align}
where the equality \eqref{sp:multi-4} uses the fact that $\mu_n = 1-\lambda_n$.
A picture that illustrates the Markov structure
of constraints \eqref{sp:multi-2b}-\eqref{sp:multi-4} is given in Fig. \ref{fig:WTF1}

\subsection{A system with $N$ single-buffer queues}\label{subsection:single_buffer}

The above model, with the assumption $\mu_n = 1-\lambda_n$, is structurally equivalent to the following:
Consider a system of $N$ single-buffer queues, $M$ servers, and independent Bernoulli packet arrivals with rates $\lambda_n$
to each queue $n \in \{1, \ldots, N\}$.  This considers \emph{packet arrivals} rather than \emph{file arrivals}, so there are no file length variables and no parameters $\mu_n$ in this interpretation.  Let $\mathbf{A}(t) = (A_1(t), \ldots, A_N(t))$ be the binary-valued vector of packet arrivals on slot $t$, assumed to be i.i.d. over slots and independent in each coordinate.
 Assume all packets have the same size and each queue has a single buffer that can store just one packet.
 Let $F_n(t)$ be 1 if queue $n$ has a packet at the beginning of slot $t$, and $0$ else.  Each server can transmit at most 1 packet per slot.  Let $\alpha_n(t)$ be 1 if queue $n$ is served on slot $t$, and $0$ else.
 An arrival $A_n(t)$ occurs at the end of slot $t$ and is accepted only if queue $n$ is empty at the end of the slot (such as when it was served on that slot).  Packets that are not accepted are dropped.
 The Markov dynamics are described by the same figure as before, namely, Fig. \ref{fig:WTF1}.  Further,
the problem of maximizing throughput is given by the \emph{same} equations \eqref{sp:multi-1}-\eqref{sp:multi-4}. Thus, although the variables of the two problems have different interpretations, the problems are structurally equivalent.
For simplicity of exposition, the remainder of this section uses
this single-buffer queue interpretation.

\begin{figure}
   \centering
   \includegraphics[height=2.5in]{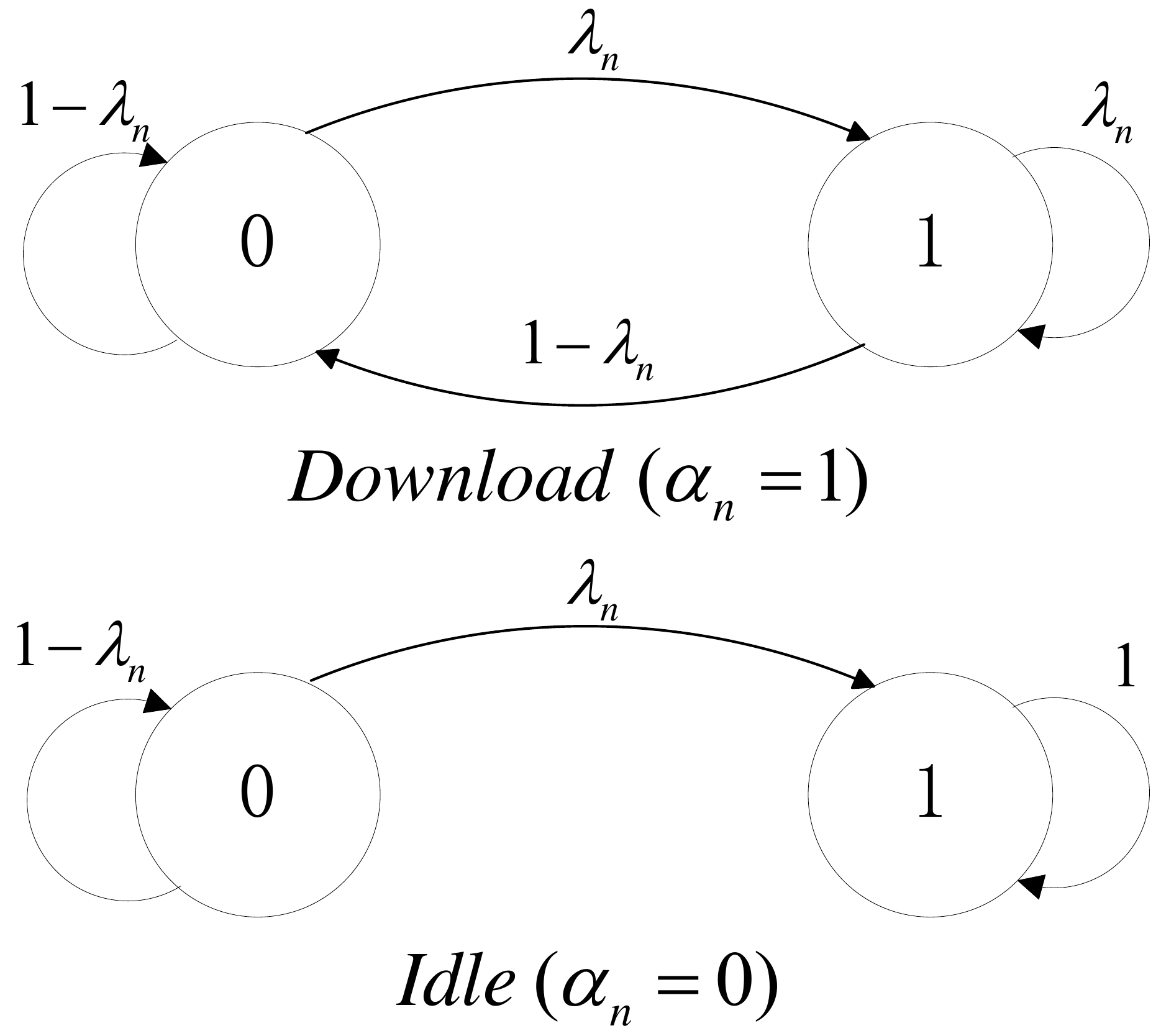} 
   \caption{Markovian dynamics of the $n$-th system.}
   \label{fig:WTF1}
\end{figure}

\subsection{Optimality of the indexing algorithm}
Since there is no power constraint, for any $V>0$ the
DPP ratio indexing policy \eqref{m5} in Section \ref{multi:indexing} reduces
to the following (using $c_n=1$, $Q(t)\equiv0$):   If there are fewer than $M$ non-empty queues, serve all of them. Else, serve the $M$ non-empty queues with the largest values of $\gamma_n$, where:
\begin{equation}
    \gamma_n=\frac{1}{1+(1-\lambda_n)/\lambda_n}=\lambda_n. \nonumber
\end{equation}
Thus, the DPP ratio indexing algorithm in this context reduces to serving the (at most $M$) non-empty queues with the  largest $\lambda_n$ values each time slot.  For the remainder of this section, this is called the Max-$\lambda$ policy.
The following theorem shows that Max-$\lambda$ is optimal in this context.

\begin{theorem} \label{thm:max-lambda}
The Max-$\lambda$ policy is optimal for the problem  \eqref{sp:multi-1}-\eqref{sp:multi-4}.  In particular, under the single-buffer queue interpretation, it maximizes throughput over all policies that transmit on each slot $t$
without knowledge of the arrival vector
$\mathbf{A}(t)$.
\end{theorem}

For the $N$ single-buffer queue interpretation,
the total throughput is equal to the raw arrival rate $\sum_{i=1}^N \lambda_i$ minus the packet drop rate.
Intuitively, the reason Max-$\lambda$ is optimal is that it chooses to leave packets in the queues that are least likely to induce packet drops. An example comparison of the throughput gap between Max-$\lambda$ and Min-$\lambda$ policies is given in Section \ref{sec:additional-proof}.

The proof of Theorem \ref{thm:max-lambda} is divided into two parts. The first part uses stochastic coupling techniques to prove that Max-$\lambda$ dominates all alternative \emph{work-conserving} policies.  A policy is \emph{work-conserving} if it does not allow any server to be idle when it could be used to serve a non-empty queue.
 The second part of the proof
 shows that throughput cannot be increased by considering non-work-conserving policies.

\subsection{Preliminaries on stochastic coupling}\label{prelim}

Consider two discrete time processes $\mathcal{X}\triangleq\{X(t)\}_{t=0}^{\infty}$ and $\mathcal{Y}\triangleq\{Y(t)\}_{t=0}^{\infty}$.  The notation $\mathcal{X} =_{st} \mathcal{Y}$ means that $\mathcal{X}$ and $\mathcal{Y}$ are \emph{stochastically equivalent}, in that they are described by the same probability law.  Formally, this means that
their joint distributions are the same, so for all $t \in \{0, 1, 2, \ldots\}$ and all $(z_0, \ldots, z_t) \in \mathcal{R}^{t+1}$:
\begin{eqnarray*}
&&Pr[X(0)\leq z_0, \ldots, X(t)\leq z_t] \\
&&= Pr[Y(0)\leq z_0, \ldots, Y(t) \leq z_t]
\end{eqnarray*}
The notation $\mathcal{X}\leq_{st}\mathcal{Y}$ means that $\mathcal{X}$ is \emph{stochastically less than or equal to} $\mathcal{Y}$, as defined by the following theorem.
\begin{theorem}  \label{sto_theorem}
(\cite{tassiulas1993dynamic}) The following three statements are equivalent:
\begin{enumerate}
  \item $\mathcal{X}\leq_{st}\mathcal{Y}$.
  \item $Pr[g(X(0), X(1), \cdots, X(t))>z]\leq Pr[g(Y(0),$ $Y(1), \cdots, Y(t))>z]$ for all $t\in \mathbb{Z}^{+}$, all $z$, and for all functions $g:\mathcal{R}^{n}\rightarrow\mathcal{R}$ that are measurable and nondecreasing in all coordinates.
  \item There exist two stochastic processes $\mathcal{X}'$ and $\mathcal{Y}'$ on a common  probability space that satisfy $\mathcal{X}=_{st}\mathcal{X}'$, $\mathcal{Y}=_{st}\mathcal{Y}'$, and $X'(t)\leq Y'(t)$ for every $t\in \mathbb{Z}^{+}$.
\end{enumerate}
\end{theorem}

The following additional notation is used in the proof of Theorem \ref{thm:max-lambda}.

\begin{itemize}
  \item Arrival vector $\{\mathbf{A}(t)\}_{t=0}^{\infty}$, where $\mathbf{A}(t)\triangleq[A_1(t)$ $A_2(t)~\cdots~A_N(t)]$. Each $A_n(t)$ is an independent binary random variable that takes $1$ w.p. $\lambda_n$ and $0$ w.p. $1-\lambda_n$.
  \item Buffer state vector $\{\mathbf{F}(t)\}_{t=0}^{\infty}$, where $\mathbf{F}(t)\triangleq[F_1(t)$ $F_2(t)~\cdots~F_N(t)]$. So $F_n(t)=1$ if queue $n$ has a packet at the beginning of slot $t$, and $F_n(t)=0$ else.
  \item Total packet process $\mathcal{U}\triangleq\{U(t)\}_{t=0}^{\infty}$, where $U(t)\triangleq\sum_{n=1}^NF_n(t)$  represents the total number of packets in the system on slot $t$.  Since each queue can hold at most one packet, we have $0 \leq U(t) \leq N$ for all slots $t$.
\end{itemize}

\subsection{Stochastic ordering of buffer state process}

The next lemma is the key to proving Theorem \ref{thm:max-lambda}.  The lemma considers the multi-queue system with a fixed but arbitrary initial buffer state $\mathbf{F}(0)$. The arrival process $\mathbf{A}(t)$ is as defined above.
Let $\mathcal{U}^{\mbox{\tiny Max-$\lambda$}}$ be the total packet process under the Max-$\lambda$ policy.
Let $\mathcal{U}^{\pi}$ be the corresponding process starting from the same initial state $\mathbf{F}(0)$ and having the same arrivals $\mathbf{A}(t)$, but with an arbitrary work-conserving policy $\pi$.

\begin{lemma} \label{sto_file_stat}
The total packet processes $\mathcal{U}^\pi$ and $\mathcal{U}^{\mbox{\tiny Max-$\lambda$}}$ satisfy:
\begin{equation}\label{st_order}
   \mathcal{U}^{\pi}\leq_{st}\mathcal{U}^{\mbox{\tiny Max-$\lambda$}}
\end{equation}
\end{lemma}
\begin{proof}
Without loss of generality, assume the queues are sorted so that $\lambda_n\leq\lambda_{n+1},~n=1,2,\cdots,N-1$. Define $\{\mathbf{F}^\pi(t)\}_{t=0}^{\infty}$ as the buffer state vector under policy $\pi$.  Define
$\{\mathbf{F}^{\mbox{\tiny Max-$\lambda$}}(t)\}_{t=0}^{\infty}$
as the corresponding buffer states under the Max-$\lambda$ policy.
By assumption the initial states satisfy $\mathbf{F}^\pi(0)=\mathbf{F}^{\mbox{\tiny Max-$\lambda$}}(0)$.  Next,  we construct a \emph{third process} $\mathcal{U}^\lambda$
with a \emph{modified} arrival vector process $\{\mathbf{A}^\lambda(t)\}_{t=0}^{\infty}$ and a corresponding
buffer state vector $\{\mathbf{F}^\lambda(t)\}_{t=0}^{\infty}$ (with the same initial state
$\mathbf{F}^\lambda(0)=\mathbf{F}^\pi(0)$), which satisfies:
\begin{enumerate}
  \item $\mathcal{U}^\lambda$ is also generated from the Max-$\lambda$ policy.
  \item $\mathcal{U}^\lambda =_{st} \mathcal{U}^{\mbox{\tiny Max-$\lambda$}}$. Since the total packet process is completely determined by the initial state, the scheduling policy, and the arrival process, it is enough to construct $\{\mathbf{A}^\lambda(t)\}_{t=0}^{\infty}$ so that it is of the same probability law as $\{\mathbf{A}(t)\}_{t=0}^{\infty}$.
  \item $U^\pi(t)\leq U^\lambda(t)~\forall t\geq0$.
\end{enumerate}

Since the arrival process $\mathbf{A}(t)$ is i.i.d. over slots, in order to guarantee 2) and 3), it is sufficient to construct $\mathbf{A}^\lambda(t)$ coupled with $\mathbf{A}(t)$ for each $t$ so that the following two properties hold for all $t\geq0$:
\begin{itemize}
  \item  The random variables $\mathbf{A}(t)$ and $\mathbf{A}^\lambda(t)$ have the same probability law. Specifically, both  produce arrivals according to Bernoulli processes that are independent over queues and over time, with $Pr[A_n(t)=1]=Pr[A_n^\lambda(t)=1] = \lambda_n$ for all $n \in \{1, \ldots, N\}$.
  \item For all $j\in\{1,2,\cdots,N\}$,
  \begin{equation}\label{partial_order}
            \sum_{n=1}^jF^\pi_{n}(t)\leq\sum_{n=1}^jF^\lambda_{n}(t),
  \end{equation}
\end{itemize}

The construction is based on an induction.

At $t=0$ we have $\mathbf{F}^\pi(0)=\mathbf{F}^\lambda(0)$. Thus, \eqref{partial_order} naturally holds for $t=0$.
Now fix $\tau \geq 0$ and assume \eqref{partial_order} holds for all slots up to time $t=\tau$.  If $\tau\geq 1$, further assume the arrivals $\{\mathbf{A}^\lambda(t)\}_{t=0}^{\tau-1}$ have been constructed to have the same probability law as $\{\mathbf{A}(t)\}_{\tau=0}^{\tau-1}$. Since arrivals on slot $\tau$ occur at the \emph{end} of slot $\tau$, the arrivals $\mathbf{A}^\lambda(\tau)$ must be constructed.
We are going to show  there exists an $\mathbf{A}^\lambda(\tau)$ that is \emph{coupled}  with $\mathbf{A}(\tau)$ so that it has the same probability law and it also ensures \eqref{partial_order} holds for $t=\tau+1$.

Since arrivals occur after the transmitting action, we divide the analysis into two parts. First, we analyze the temporary buffer states after the transmitting action \emph{but before arrivals occur}. Then, we define arrivals $\mathbf{A}^\lambda(\tau)$ at the end of slot $\tau$ to achieve the desired coupling.

Define $\tilde{\mathbf{F}}^\pi(\tau)$ and $\tilde{\mathbf{F}}^{\lambda}(\tau)$ as the \emph{temporary buffer states} right after the transmitting action at slot $\tau$ but before arrivals occur under policy $\pi$ and policy Max-$\lambda$, respectively.  Thus, for each queue $n \in \{1, \ldots, N\}$:
\begin{eqnarray}
\tilde{F}_n^\pi(\tau) &=& F_n^\pi(\tau) - \alpha_n^\pi(\tau) \label{eq:f-pi-dude} \\
\tilde{F}_n^\lambda(\tau) &=& F_n^\lambda(\tau) - \alpha_n^\lambda(\tau) \label{eq:f-lambda-dude}
\end{eqnarray}
where $\alpha_n^\pi(\tau)$ and $\alpha_n^\lambda(\tau)$ are the slot $\tau$ decisions under policy
$\pi$ and Max-$\lambda$, respectively. Since
 \eqref{partial_order} holds for $j=N$ on slot $\tau$, the total number of packets at the start of slot $\tau$ under policy $\pi$ is less than or equal to that of using Max-$\lambda$.   Since both policies $\pi$ and Max-$\lambda$ are work-conserving,
 it is impossible for policy $\pi$ to transmit more packets than Max-$\lambda$ during slot $\tau$.
 This implies:
  \begin{equation}\label{FpileqFlambda}
\sum_{n=1}^N\tilde{F}_n^\pi(\tau)\leq\sum_{n=1}^N\tilde{F}^\lambda_{n}(\tau).
\end{equation}
 Indeed, if $\pi$ transmits the same number of packets as Max-$\lambda$ on slot $\tau$, then \eqref{FpileqFlambda} clearly holds.  On the other hand, if $\pi$ transmits \emph{fewer} packets than Max-$\lambda$, it must transmit fewer than $M$ packets (since $M$ is the number of servers).  In this case, the work-conserving nature of $\pi$ implies that all non-empty queues were served, so that $\tilde{F}_n^\pi(\tau)=0$ for all $n$ and \eqref{FpileqFlambda} again holds.
We now claim the following holds:
\begin{lemma}  \label{lem:new}
\begin{equation}\label{temp_order}
\sum_{n=1}^j\tilde{F}_n^\pi(\tau)\leq\sum_{n=1}^j\tilde{F}^\lambda_{n}(\tau)~~\forall j\in\{1,2,\cdots,N\}.
\end{equation}
\end{lemma}
\begin{proof}
See Section \ref{sec:additional-proof}.
\end{proof}

Now
let $j^\pi(l)$ and $j^\lambda(l)$ be the subscript of $l$-th \emph{empty} temporary buffer (with order starting from the first queue) corresponding to $\tilde{\mathbf{F}}^\pi(\tau)$ and $\tilde{\mathbf{F}}^{\lambda}(\tau)$, respectively.
It  follows from \eqref{temp_order} that the $\pi$ system on slot $\tau$ has at  least as many empty temporary buffer states as the Max-$\lambda$ policy, and:
\begin{equation}\label{empty_order}
    j^\pi(l)\leq j^\lambda(l)~~\forall l\in\{1,2,\cdots,K(\tau)\}
\end{equation}
where $K(\tau)\leq N$ is the the number of empty temporary buffer states under Max-$\lambda$ at time slot $\tau$.
Since $\lambda_i\leq\lambda_j$ if and only if $i\leq j$, \eqref{empty_order} further implies that
\begin{equation}\label{lambda_order}
    \lambda_{j^\pi(l)}\leq\lambda_{j^\lambda(l)}~~\forall l\in\{1,2,\cdots,K(\tau)\}.
\end{equation}

Now construct the arrival vector $\mathbf{A}^\lambda(\tau)$ for the system with the Max-$\lambda$ policy  in the following way:
\begin{align}
    A_{j^\pi(l)}(\tau)=1&\Rightarrow A^\lambda_{j^\lambda(l)}(\tau)=1 ~~~w.p.~1\label{coupling1}\\
    A_{j^\pi(l)}(\tau)=0&\Rightarrow\left\{
                                  \begin{array}{ll}
                                    A^\lambda_{j^\lambda(l)}(\tau)=0 , & \hbox{w.p.~$\frac{1-\lambda_{j^\lambda(l)}}{1-\lambda_{j^\pi(l)}}$;} \\
                                    A^\lambda_{j^\lambda(l)}(\tau)=1,  & \hbox{w.p.~$\frac{\lambda_{j^\lambda(l)}-\lambda_{j^\pi(l)}}{1-\lambda_{j^\pi(l)}}$.}
                                  \end{array}
                                \right.\label{coupling2}
\end{align}
Notice that \eqref{coupling2} uses valid probability distributions because of \eqref{lambda_order}.
This establishes the slot $\tau$ arrivals for the Max-$\lambda$ policy for all of its $K(\tau)$ queues with empty temporary buffer states.  The slot $\tau$ arrivals for its queues with non-empty temporary buffers
will be dropped and hence do not affect the queue states on slot $\tau+1$.  Thus, we define arrivals $A_j^\lambda(\tau)$ to be independent of all other quantities and to be Bernoulli with $Pr[A_j^\lambda(\tau)=1]=\lambda_j$ for all $j$ in the set:
\[ j\in\{1,2,\cdots,N\}\setminus\{j^\lambda(1),\cdots,j^\lambda(K(\tau))\} \]
Now we verify that $\mathbf{A}(\tau)$ and $\mathbf{A}^\lambda(\tau)$ have the same probability law.
First condition on knowledge of $K(\tau)$ and the particular $j^\pi(l)$ and $j^\lambda(l)$ values for $l \in \{1, \ldots, K(\tau)\}$.
All queues $j$ with non-empty temporary buffer states on slot $\tau$ under Max-$\lambda$ were defined to have arrivals $A_{j}^\lambda(\tau)$ as independent Bernoulli variables with $Pr[A_j^\lambda(\tau)=1]=\lambda_j$.
It remains to  verify those queues within $\{j^\lambda(1),\cdots,j^\lambda(K(\tau))\}$.
According to \eqref{coupling2}, for any queue $j^\lambda(l)$ in set  $\{j^\lambda(1),\cdots,j^\lambda(K(\tau))\}$, it follows
\begin{eqnarray*}
Pr\left[A^\lambda_{j^\lambda(l)}(\tau)=0\right]&=&(1-\lambda_{j^\pi(l)})\frac{1-\lambda_{j^\lambda(l)}}{1-\lambda_{j^\pi(l)}}\nonumber\\
&=&1-\lambda_{j^\lambda(l)}
\end{eqnarray*}
and so $Pr[A_j^\lambda(\tau)=1]=\lambda_j$ for all $j \in \{j^\lambda(l)\}_{l=1}^{K(\tau)}$.
Further, mutual independence of $\{A_{j^\pi(l)}(\tau)\}_{l=1}^{K(\tau)}$ implies mutual independence of
$\{A_{j^\lambda(l)}(\tau)\}_{l=1}^{K(\tau)}$.   Finally, these quantities are conditionally
independent of events before slot $\tau$, given knowledge of $K(\tau)$ and the particular $j^\pi(l)$ and $j^\lambda(l)$ values for $l \in \{1, \ldots, K(\tau)\}$.
Thus, conditioned on this knowledge,
 $\mathbf{A}(\tau)$ and $\mathbf{A}^\lambda(\tau)$ have the same probability law.  This holds for all possible values of the conditional knowledge $K(\tau)$  and $j^\pi(l)$ and $j^\lambda(l)$.  It follows that $\mathbf{A}(\tau)$ and $\mathbf{A}^\lambda(\tau)$ have the same (unconditioned) probability law.

Finally, we show that the coupling relations \eqref{coupling1} and \eqref{coupling2} produce such $\mathbf{F}^\lambda(\tau+1)$ satisfying
\begin{equation}\label{conclusion}
    \sum_{n=1}^jF^\pi_{n}(\tau+1)\leq\sum_{n=1}^jF^\lambda_{n}(\tau+1),~\forall~j\in\{1,2,\cdots,N\}.
\end{equation}
According to \eqref{coupling1} and \eqref{coupling2},
\[A_{j^\pi(l)}(\tau)\leq A^\lambda_{j^\lambda(l)}(\tau),~~\forall l\in\{1,\cdots,K(\tau)\},\]
thus,
\begin{equation}\label{arrival_ineq}
\sum_{i=1}^{l}A_{j^\pi(i)}(\tau)\leq \sum_{i=1}^{l}A^\lambda_{j^\lambda(i)}(\tau),~~\forall l\in\{1,\cdots,K(\tau)\}.
\end{equation}
Pick any $j\in\{1,2,\cdots,N\}$. Let $l^\pi$ be the number of empty temporary buffers within the first $j$ queues
under policy $\pi$, i.e.
\[l^\pi=\max_{j^\pi(l)\leq j}l\]
Similarly define:
\[l^\lambda=\max_{j^\lambda(l)\leq j}l.\]
Then, it follows:
\begin{eqnarray}
\sum_{n=1}^jF^\pi_{n}(\tau+1)&=&\sum_{n=1}^j\tilde{F}^\pi_{n}(\tau)+\sum_{i=1}^{l^\pi}A_{j^\pi(i)}(\tau) \label{F(t+1)}\\
\sum_{n=1}^jF^\lambda_{n}(\tau+1)&=&\sum_{n=1}^j\tilde{F}^\lambda_{n}(\tau)+\sum_{i=1}^{l^\lambda}A^\lambda_{j^\lambda(i)}(\tau)\label{barF(t+1)}
\end{eqnarray}
We know that $l^\pi \geq l^\lambda$.  So there are two cases:
\begin{itemize}
  \item If $l^\pi=l^\lambda$, then from \eqref{F(t+1)}:
  \begin{eqnarray*}
  \sum_{n=1}^jF_n^\pi(\tau+1) &=& \sum_{n=1}^j\tilde{F}_n^\pi(\tau) + \sum_{i=1}^{l^{\lambda}}A_{j^\pi(i)}(\tau)\\
  &\leq& \sum_{n=1}^j\tilde{F}_n^\lambda(\tau) + \sum_{i=1}^{l^{\lambda}}A_{j^{\lambda}(i)}(\tau) \\
  &=& \sum_{n=1}^j F_n^\lambda(\tau+1)
  \end{eqnarray*}
where the inequality follows from \eqref{temp_order} and from \eqref{arrival_ineq} with
$l=l^{\lambda}$.
Thus, \eqref{conclusion} holds.
  \item If $l^\pi>l^\lambda$, then from \eqref{F(t+1)}:
\begin{align}
\sum_{n=1}^jF^\pi_{n}(\tau+1)=&\sum_{n=1}^j\tilde{F}^\pi_{n}(\tau)+\sum_{i=1}^{l^\lambda}A_{j^\pi(i)}(\tau)\nonumber\\
&+\sum_{i=l^\lambda+1}^{l^\pi}A_{j^\pi(i)}(\tau)\nonumber\\
\leq&\sum_{n=1}^j\tilde{F}^\lambda_{n}(\tau)+\sum_{i=1}^{l^\lambda}A_{j^\pi(i)}(\tau)\nonumber\\
\leq&\sum_{n=1}^j\tilde{F}^\lambda_{n}(\tau)+\sum_{i=1}^{l^\lambda}A^\lambda_{j^\lambda(i)}(\tau) \nonumber\\
=&\sum_{n=1}^jF^\lambda_{n}(\tau+1).\nonumber
\end{align}
where the first inequality follows from the fact that
\begin{eqnarray*}
\sum_{i=l^\lambda+1}^{l^\pi}A_{j^\pi(i)}(\tau)&\leq& l^\pi-l^\lambda\\
&=&(j-l^\lambda)-(j-l^\pi)  \nonumber\\
&=&\sum_{n=1}^j\tilde{F}^\lambda_{n}(\tau)-\sum_{n=1}^j\tilde{F}^\pi_{n}(\tau),\nonumber
\end{eqnarray*}
and the second inequality follows from \eqref{arrival_ineq}.
\end{itemize}
Thus, \eqref{partial_order} holds for $t=\tau+1$ and the induction step is done.
\end{proof}

\begin{corollary} \label{corollary:1}
The Max-$\lambda$ policy maximizes throughput within the class of work-conserving policies.
\end{corollary}
\begin{proof}
Let $S^\pi(t)$ be the number of packets transmitted under any work-conserving policy $\pi$ on slot $t$,
and let $S^{\mbox{\tiny Max-$\lambda$}}(t)$ be the corresponding process under policy Max-$\lambda$.  Lemma
\ref{sto_file_stat} implies $\mathcal{U}^\pi(t) \leq_{st} \mathcal{U}^{\mbox{\tiny Max-$\lambda$}}$. Then:
\begin{eqnarray*}
 \expect{S^\pi(t)}&=&\expect{\min[U^\pi(t), M]}   \\
  &\leq&\expect{\min[U^{\mbox{\tiny Max-$\lambda$}}(t),M]} \\
  &=& \expect{S^{\mbox{\tiny Max-$\lambda$}}(t)}
\end{eqnarray*}
where the inequality follows from Theorem \ref{sto_theorem}, with the understanding that $g(U(0),\ldots, U(t))\triangleq\min[U(t),M]$ is a function that is nondecreasing in all coordinates.
\end{proof}

\subsection{Extending to non-work-conserving policies}

Corollary \ref{corollary:1} establishes optimality of Max-$\lambda$ over the class of all work-conserving policies.
To complete the proof of Theorem \ref{thm:max-lambda}, it remains to show that throughput cannot be increased by allowing for non-work-conserving policies.  It suffices to show that for any non-work-conserving policy, there exists a work-conserving policy that gets the same or better throughput.  The proof is straightforward and we give only a
proof sketch for brevity.
Consider any non-work-conserving policy $\pi$, and let $F_n^\pi(t)$ be its buffer state process on slot $t$ for each queue $n$.  For the same initial buffer state and arrival process, define the work-conserving policy $\pi'$ as follows:  Every slot $t$, policy $\pi'$ initially allocates the $M$ servers to exactly the same queues as policy $\pi$.  However, if some of these queues are empty under policy $\pi'$, it reallocates those servers to any non-empty queues that are not yet allocated servers (in keeping with the work-conserving property).  Let $F_n^{\pi'}(t)$ be the buffer state process for queue $n$ under policy $\pi'$.
 It is not difficult to show that $F_n^\pi(t) \geq F_n^{\pi'}(t)$ for all queues $n$ and all slots $t$.
 Therefore, on every slot $t$, the amount of \emph{blocked arrivals} under policy $\pi$ is always greater than or equal to that under policy $\pi'$.  This implies the throughput under policy $\pi$ is less than or equal to that of policy $\pi'$.

\section{Simulation experiments} \label{section:sims}

In this section, we demonstrate near optimality of the multi-user DPP ratio indexing algorithm by extensive simulations. In the first part, we simulate the case in which the file length distribution is geometric, and show that the suboptimality gap is extremely small. In the second part, we test the robustness of our algorithm for more general scenarios in which the file length distribution is not geometric.
For simplicity, it is assumed throughout that all transmissions send a fixed sized packet, all files are an integer number of these packets, and that decisions $\alpha_n(t) \in \mathcal{A}_n$ affect the success probability of the transmission as well as the power expenditure.

\subsection{DPP ratio indexing with geometric file length} \label{subsection:geometric-sim}
In the first simulation we use $N=8$, $M=4$ with action set $\mathcal{A}_n=\{0,1\}~\forall n$; The settings are generated randomly and specified in Table I, and the constraint $\beta=5$.

\begin{table}
\begin{center}
\caption{Problem parameters}
\begin{tabular}{|l|l|l|l|l|l|l|}
  \hline
   User & $\lambda_n$ & $\mu_n$ & $\phi_n(1)$ & $c_n$ & $p_n(1)$  \\ \hline
  1 & 0.0028 & 0.5380 & 0.4842 & 4.7527 & 3.9504  \\
  2 & 0.4176 & 0.5453 & 0.4908 & 2.0681 & 3.7391  \\
  3 & 0.0888 & 0.5044 & 0.4540 & 2.8656 & 3.5753  \\
  4 & 0.3181 & 0.6103 & 0.5493 & 2.4605 & 2.1828  \\
  5 & 0.4151 & 0.9839 & 0.8855 & 4.5554 & 3.1982  \\
  6 & 0.2546 & 0.5975 & 0.5377 & 3.9647 & 3.5290  \\
  7 & 0.1705 & 0.5517 & 0.4966 & 1.5159 & 2.5226  \\
  8 & 0.2109 & 0.7597 & 0.6837 & 3.6364 & 2.5376  \\
  \hline
\end{tabular}
\end{center}
\end{table}

The algorithm is run for 1 million slots in each trial and each point is the average of 100 trials. We compare the performance of our algorithm with the optimal randomized policy. The optimal policy is computed by constructing composite states (i.e. if there are three users where user 1 is at state 0, user 2 is at state 1 and user 3 is at state 1, we view 011 as a composite state), and then reformulating this MDP into a linear program (see \cite{fox1966markov}) with $\mathbf{5985}$ variables and $\mathbf{258}$ constraints.

In Fig. \ref{fig:Stupendous2}, we show that as our tradeoff parameter $V$ gets larger, the objective value approaches the optimal value and achieves a near optimal performance. Fig. \ref{fig:Stupendous3} and Fig. \ref{fig:Stupendous4} show that $V$ also affects the virtual queue size and the constraint gap. As $V$ gets larger, the average virtual queue size becomes larger and the gap becomes smaller. We also plot the upper bound of queue size we derived from Lemma \ref{lem:3} in Fig. \ref{fig:Stupendous3}, demonstrating that the queue is bounded. In order to show that $V$ is indeed a trade-off parameter affecting the convergence time, we plotted Fig. \ref{fig:Stupendous5}. It can be seen from the figure that as $V$ gets larger, the number of time slots needed for the running average to roughly converge to the optimal power expenditure becomes larger.

\begin{figure}[htbp]
   \centering
   \includegraphics[height=2.5in]{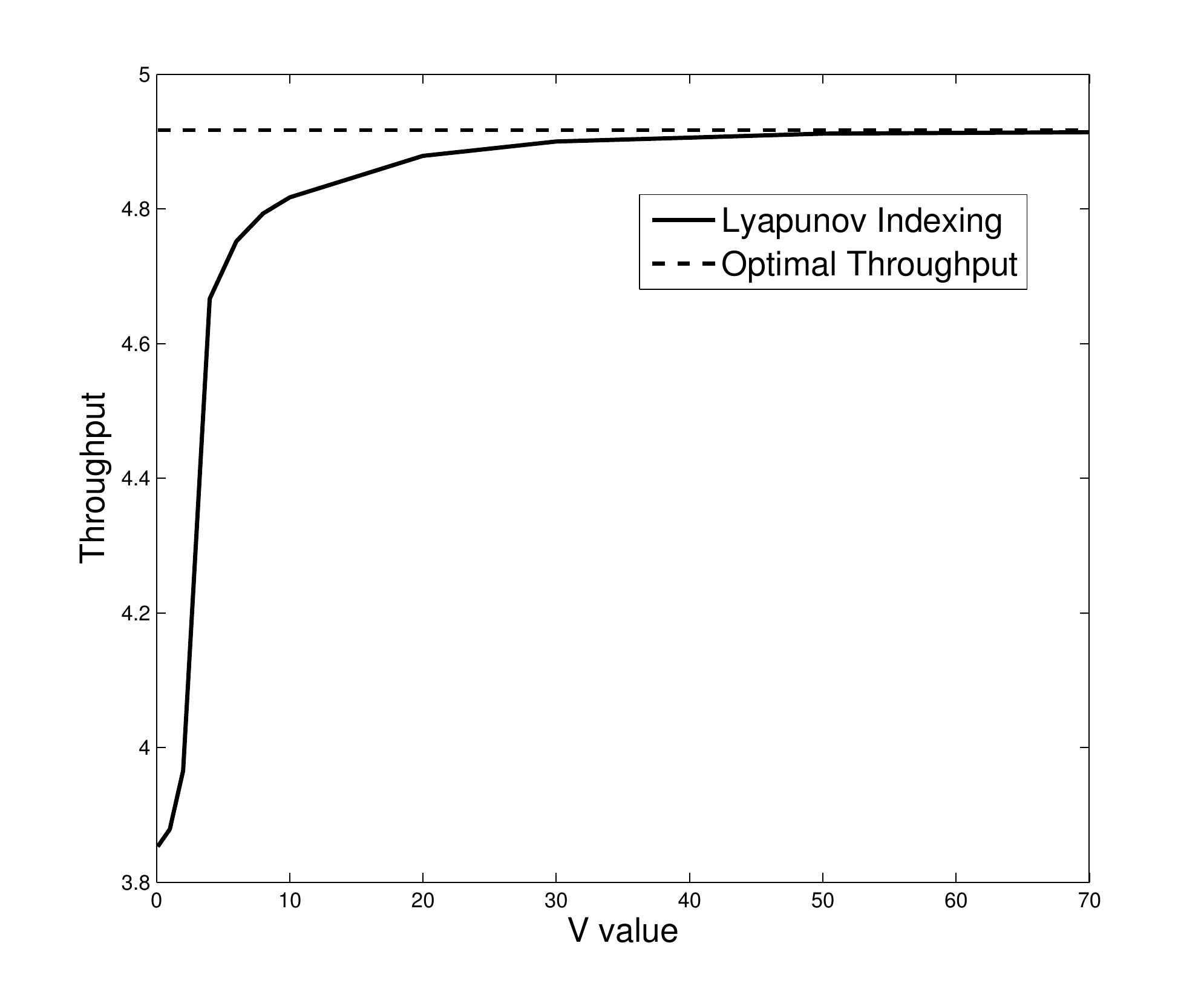} 
   \caption{Throughput versus tradeoff parameter V}
   \label{fig:Stupendous2}
\end{figure}

\begin{figure}[htbp]
   \centering
   \includegraphics[height=2.5in]{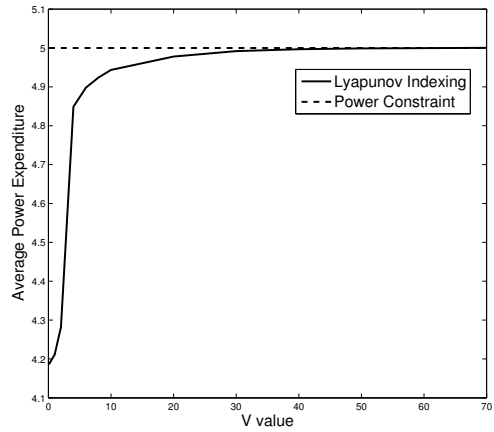} 
   \caption{The time average power consumption versus tradeoff parameter $V$.}
   \label{fig:Stupendous3}
\end{figure}

\begin{figure}[htbp]
   \centering
   \includegraphics[height=2.5in]{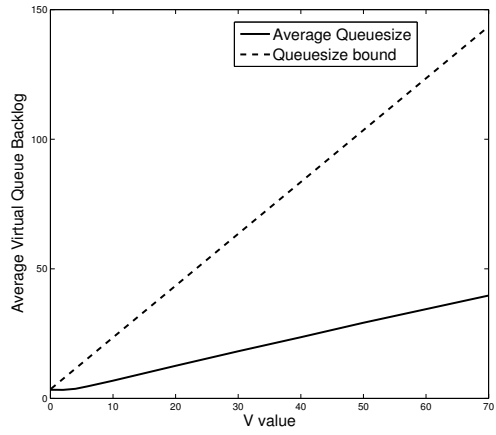} 
   \caption{Average virtual queue backlog versus tradeoff parameter $V$.}
   \label{fig:Stupendous4}
\end{figure}

\begin{figure}[htbp]
   \centering
   \includegraphics[height=2.5in]{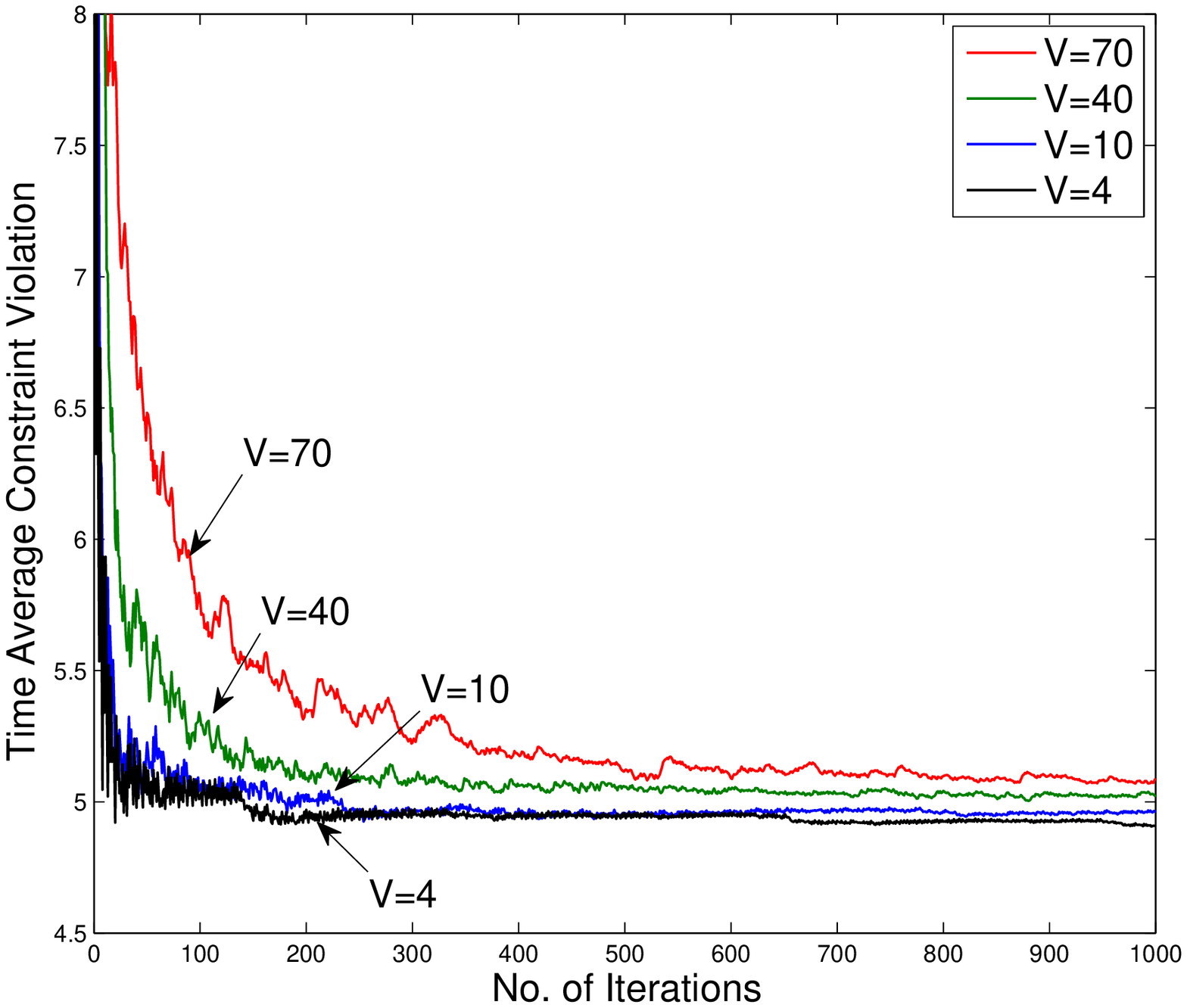} 
   \caption{Running average power consumption versus tradeoff parameter $V$.}
   \label{fig:Stupendous5}
\end{figure}

In the second simulation, we explore the parameter space and demonstrate that in general the suboptimality gap of our algorithm is negligible. First, we define the relative error as the following:
\begin{equation}\label{e20}
    \textrm{relative error}=\frac{|OBJ-OPT|}{OPT}
\end{equation}
where $OBJ$ is the objective value after running 1 million slots of our algorithm and $OPT$ is the optimal value. We first explore the system parameters by letting $\lambda_{n}$'s and $\mu_n$'s take random numbers within 0 and 1, letting $c_n$ take random number within 1 and 5, choosing $V=70$ and fixing the remaining
parameters the same as the last experiment. We conduct 1000 Monte-Carlo experiments and calculate the average relative error, which is \textbf{0.00083}.

Next, we explore the control parameters by letting the $p_{n}(1)$ take random number within 2 and 4, and letting $\phi_n(1)/\mu_n$ values random numbers between 0 and 1, choosing $V=70$ and fixing the remaining parameters the same as the first simulation. The relative error is \textbf{0.00057}. Both experiments show that the suboptimality gap is extremely small.

\subsection{DPP ratio indexing with non-memoryless file lengths}
In this part, we test the sensitivity of the algorithm to different file length distributions. In particular, the uniform distribution and the Poisson distribution are implemented respectively, while our algorithm still treats them as a geometric distribution with same mean. We then compare their throughputs with the geometric case.

We use $N=9$, $M=4$ with action set $\mathcal{A}_n=\{0,1\}~\forall n$. The settings are specified in Table II with constraint $\beta=5$. Notice that for geometric and uniform distribution, the file lengths are taken to be integer values. The algorithm is run for 1 million slots in each trial and each point is the average of 100 trials.
\begin{table}
\begin{center}
\caption{Problem parameters under geometric, uniform and poisson distribution}
\begin{tabular}{|l|l|l|l|l|l|l|l|}
  \hline
  User & $\mu_n$ & Unif. & Poiss. & $\lambda_n$ & $\phi_n(1)$ & $c_n$ & $p_n(1)$ \\
   &~ & interval & mean &~&~&~& \\ \hline
   1 & 1/3 & [1,5] & 3 & 0.4955 & 0.1832 & 4.3261 & 2.8763 \\
   2 & 1/2 & [1,3] & 2 & 0.1181 & 0.4187 & 1.6827 & 2.0549 \\
  3 & 1/2 & [1,3] & 2 & 0.1298 & 0.4491 & 1.9483 & 2.1469 \\
  4 & 1/7 & [1,13] & 7 & 0.4660 & 0.0984 & 2.7495 & 3.4472 \\
  5  & 1/4 & [1,7] & 4 & 0.1661 & 0.1742 & 1.5535 & 3.2801 \\
  6 & 1/3 & [1,5] & 3 & 0.2124 & 0.3101 & 4.3151 & 3.5648 \\
  7 & 1/2 & [1,3] & 2 & 0.5295 & 0.4980 & 3.6701 & 2.4680 \\
  8 & 1/5 & [1,9] & 5 & 0.2228 & 0.1971 & 4.0185 & 2.2984 \\
  9 & 1/4 & [1,7] & 4 & 0.0332 & 0.1986 & 3.0411 & 2.5747 \\
  \hline
\end{tabular}
\end{center}
\end{table}

While the decisions are made using these values, the affect of these decisions incorporates the actual (non-memoryless) file sizes.
Fig. \ref{fig:Stupendous6}  shows the throughput-versus-$V$ relation for the two non-memoryless cases and the memoryless case with matched means. The performance of all three is similar.  This illustrates that
the indexing algorithm is robust under different file length distributions.

\begin{figure}[htbp]
   \centering
   \includegraphics[height=2.5in]{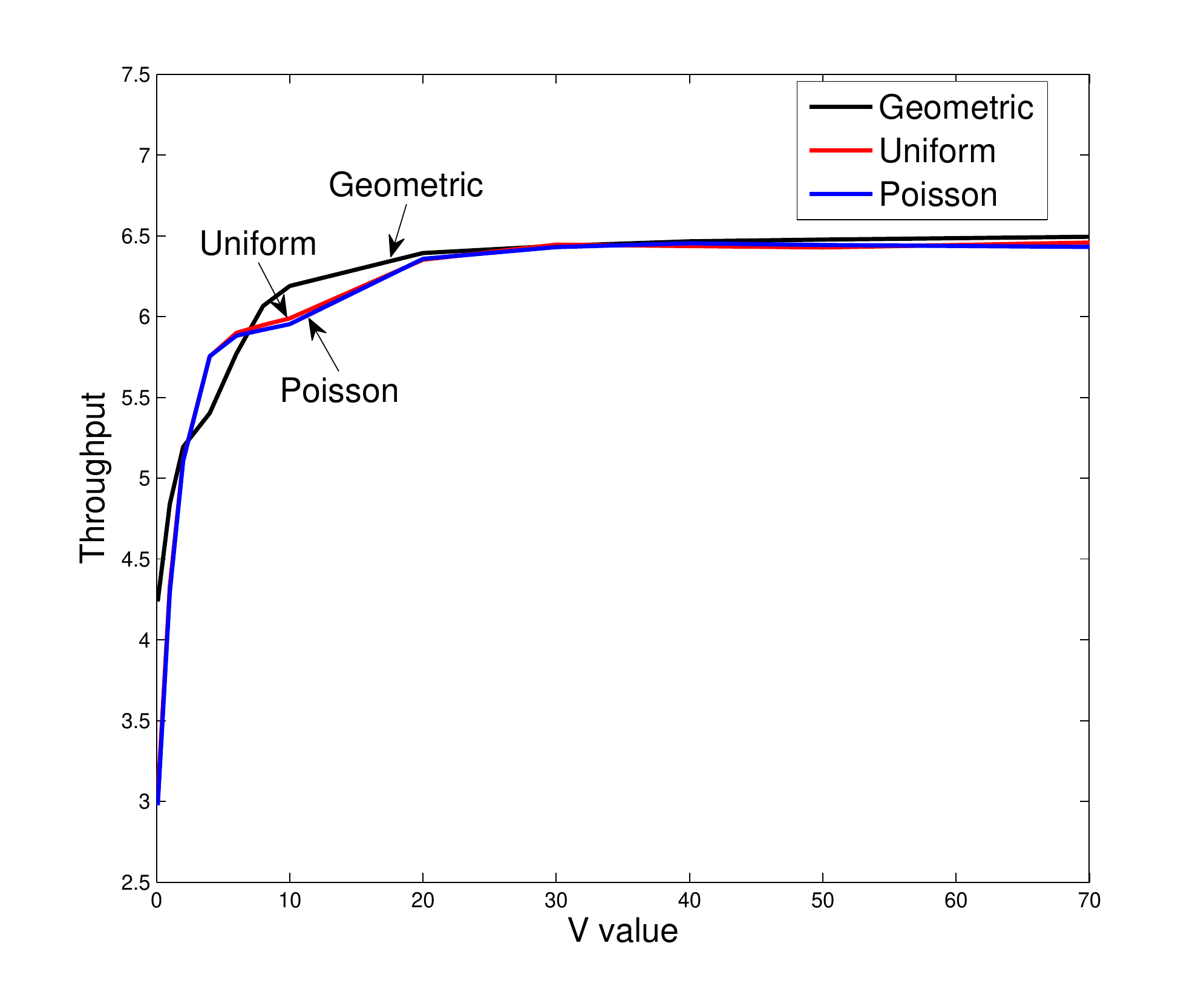} 
   \caption{Throughput versus tradeoff parameter $V$ under different file length distributions.}
   \label{fig:Stupendous6}
\end{figure}

\section{Additional lemmas and proofs}\label{sec:additional-proof}
\subsection{Comparison of Max-$\lambda$ and Min-$\lambda$}

This section shows that different work conserving policies can give different throughput for
the $N$ single-buffer queue problem of Section \ref{subsection:single_buffer}.
Suppose we have two single-buffer queues and one server.  Let $\lambda_1, \lambda_2$ be the arrival rates of the i.i.d. Bernoulli arrival processes for queues 1 and 2. Assume $\lambda_1 \neq \lambda_2$.
There are 4 system states: $(0,0),~(0,1),~(1,0),~(1,1)$, where state $(i,j)$ means queue 1 has $i$ packets and queue 2 has $j$ packets.  Consider the (work conserving) policy of giving queue 1 strict priority over queue 2.  This is equivalent to the Max-$\lambda$ policy when $\lambda_1>\lambda_2$, and is equivalent to the Min-$\lambda$ policy when $\lambda_1 < \lambda_2$.  Let $\theta(\lambda_1, \lambda_2)$ be the steady state throughput.
Then:
\[ \theta(\lambda_1, \lambda_2) = p_{1,0} + p_{0,1} + p_{1,1}  \]
 where $p_{i,j}$ is the steady state probability of the resulting discrete time Markov chain.  One can solve the global balance equations to show that $\theta(1/2, 1/4) > \theta(1/4, 1/2)$, so that the Max-$\lambda$ policy has a higher throughput than the Min-$\lambda$ policy. In particular, it can be shown that:
 \begin{itemize}
 \item Max-$\lambda$ throughput: $\theta(1/2,1/4) = 0.7$
 \item Min-$\lambda$ throughput: $\theta(1/4, 1/2) \approx 0.6786$
 \end{itemize}


\subsection{Proof of Lemma \ref{lem:new}}
This section proves that:
\begin{equation}\label{temp_order-appendix}
\sum_{n=1}^j\tilde{F}_n^\pi(\tau)\leq\sum_{n=1}^j\tilde{F}^\lambda_{n}(\tau)~~\forall j\in\{1,2,\cdots,N\}.
\end{equation}

The case $j=N$ is already established from \eqref{FpileqFlambda}.
Fix $j \in \{1, 2, \ldots, N-1\}$.  Since $\pi$ cannot transmit more packets than Max-$\lambda$ during slot $\tau$,
inequality \eqref{temp_order-appendix} is proved by considering two cases:
\begin{enumerate}
  \item Policy $\pi$ transmits less packets than policy Max-$\lambda$. Then $\pi$ transmits less than $M$ packets during slot $\tau$.  The work-conserving nature of $\pi$  implies all non-empty queues were served, so $\tilde{F}^\pi_n(\tau)=0$ for all $n$ and \eqref{temp_order-appendix} holds.
  \item Policy $\pi$ transmits the same number of packets as policy Max-$\lambda$. In this case, consider the temporary buffer states of the last $N-j$ queues under policy Max-$\lambda$. If  $\sum_{n=j+1}^N\tilde{F}^\lambda_{n}(\tau)=0$, then clearly the following holds
      \begin{equation}\label{reverse_order}
      \sum_{n=j+1}^{N}\tilde{F}^\pi_{n}(\tau)\geq\sum_{n=j+1}^N\tilde{F}^\lambda_{n}(\tau).
      \end{equation}
      Subtracting \eqref{reverse_order} from \eqref{FpileqFlambda} immediately gives \eqref{temp_order-appendix}.
      If  $\sum_{n=j+1}^N\tilde{F}^\lambda_{n}(\tau)>0$, then all $M$ servers of the Max-$\lambda$ system were devoted to serving the largest $\lambda_n$ queues.  So only packets in the last $N-j$ queues could be transmitted by Max-$\lambda$
      during the slot $\tau$.  In particular, $\alpha_n^\lambda(\tau)=0$ for all $n \in \{1, \ldots, j\}$, and so (by \eqref{eq:f-lambda-dude}):
      \begin{equation} \label{eq:last-dude}
      \sum_{n=1}^j\tilde{F}^\lambda_{n}(\tau)=\sum_{n=1}^jF^\lambda_{n}(\tau)
      \end{equation}
      Thus:
      \begin{align}
      \sum_{n=1}^j\tilde{F}^\pi_{n}(\tau)&\leq\sum_{n=1}^jF_n^\pi(\tau) \label{eq:thus-dude1} \\
      &\leq\sum_{n=1}^jF^\lambda_{n}(\tau)\label{eq:thus-dude2} \\
      &=\sum_{n=1}^{j}\tilde{F}^\lambda_{n}(\tau), \label{eq:thus-dude3}
      \end{align}
      where \eqref{eq:thus-dude1} holds by \eqref{eq:f-pi-dude}, \eqref{eq:thus-dude2} holds because
      \eqref{partial_order} is true on slot $t=\tau$, and the last equality holds by \eqref{eq:last-dude}. This proves \eqref{temp_order-appendix}.
\end{enumerate}


\chapter{Opportunistic Scheduling over Renewal Systems}
This chapter considers an opportunistic scheduling problem over a single renewal system. Different from previous chapters, we consider teh scenario where at the beginning of each renewal frame, the  controller observes a random event  and then chooses an action in response to the event, which affects the duration of the frame, the amount of resources used, and a penalty metric. The goal is to make frame-wise decisions so as to minimize the time average penalty subject to time average resource constraints. This problem has applications to task processing and communication in data networks, as well as to certain classes of Markov decision problems. 
We formulate the problem as a dynamic fractional program and propose an adaptive algorithm which uses an empirical accumulation as a feedback parameter. A key feature of the proposed algorithm is that it does not require knowledge of the random event statistics and potentially allows (uncountably) infinite event sets. We prove the algorithm satisfies all desired constraints and achieves $O(\epsilon)$ near optimality with probability 1.

\section{Introduction}
Consider a system that operates over the timeline of real numbers $t \geq 0$.  The timeline is divided into back-to-back periods called \emph{renewal frames} and the start of each frame is called a \emph{renewal} (see Fig. \ref{fig:renewal}).   The system state is refreshed at each renewal.    At the start of each renewal frame $n \in \{0, 1, 2, \dots\}$ the controller observes a random event $\omega[n]\in\Omega$ and then takes an action $\alpha[n]$ from an action set $\mathcal{A}$ in response to $\omega[n]$.   The pair $(\omega[n], \alpha[n])$ affects: (i)  the duration of that renewal frame; (ii)  a vector of resource expenditures for that frame;  (iii) a penalty incurred on that frame.   The goal is to choose actions over time to minimize time average penalty subject to time average constraints on the resources without knowing any statistic of $\omega[n]$. We call such a problem \textit{opportunistic scheduling over renewal systems}. 
\begin{figure}[htbp]
   \centering
   \includegraphics[height=1in]{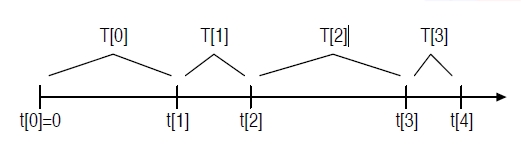} 
   \caption{An illustration of a sequence of renewal frames.}
   \label{fig:renewal}
\end{figure}

\subsection{Example applications} 

This problem has applications to task processing in computer networks, 
and certain generalizations of Markov decision problems.  

\begin{itemize} 
\item Task processing networks:  Consider a device that processes tasks back-to-back.  Each renewal period corresponds to the time required to complete a single task.  The random event $\omega[n]$ observed corresponds to a vector of task parameters, including the type, size, and resource requirements for that particular task.  The action
set $\mathcal{A}$ consists of different processing mode options, and the specific action  $\alpha[n]$ determines the processing time, energy expenditure, and task quality.  In this case, task quality can be defined as a negative penalty, and the goal is to maximize time average quality subject to power constraints and task completion rate constraints. A specific example of this sort is the following file downloading problem: Consider a wireless device that repeatedly downloads files.  The device has two states: \emph{active} (wants to download a file) and \emph{idle} (does not want to download a file). Renewals occur at the start of each new active state.
Here, $\omega[n]$ denotes the observed wireless channel state, which affects the success probability of downloading a file (and thereby affects the transition probability from active to idle).  This example is discussed further in the simulation section (Section \ref{simulation}). 

\item Hierarchical Markov decision problems:  Consider a slotted \emph{two-timescale Markov decision processes (MDP)}
over an infinite horizon and with constraints on average cost per slot.  
An MDP is run on the lower level, with a special state that is recurrent under any sequence of actions. The
renewals are defined as revisitation times to that state.  On a higher level, 
a random event $\omega$ is observed upon each revisitation to the renewal state on the lower level. Then, a decision is made on the higher level in response to $\omega$, which in turn affects the transition probability and penalty/cost received per slot on the lower level until the next renewal.
Such a problem is a generalization of classical MDP problem (e.g. \cite{Ro02}, \cite{Be01}) and has been considered previously in \cite{wernz2013multi}, \cite{chang2003multitime} with discrete finite state and full information on both levels. A heuristic method is also proposed in  \cite{wernz2013multi} when some of the information is unknown. The algorithm of the current chapter does not require 
knowledge of the statistics of $\omega$ and allows the event set $\Omega$ to be potentially (uncountably) infinite.
\end{itemize} 

\subsection{Previous approaches on renewal systems} 
Most works on optimization over renewal systems consider the simpler scenario of knowing the probability distribution of $\omega[n]$. In such a case, one can show via the renewal-reward theory that the problem can be solved (offline) by finding the solution to a linear fractional program. This idea has been applied to solve MDPs in the seminal work \cite{Fo66}. 
Methods for solving linear fractional programs can also be found, for example, in \cite{Sc83, BV04}. 
However, the practical limitations of such an offline algorithm are twofold: First, if the event set 
$\Omega$ is large, then, there are too many probabilities $Pr(\omega[n] = \omega),~\omega\in\Omega$ to estimate and the corresponding offline optimization problem may be difficult to solve even if all probabilities are estimated accurately. Second, generic offline 
optimization solvers may not take advantage of the special renewal 
structure of the system.  One notable example is the treatment of 
power and delay minimization for a multi-class M/G/1 queue in \cite{Yao02, LN14}, 
where the renewal structure allows a well known $c$-$\mu$ rule 
for delay minimization to be extended to treat both power and 
delay constraints.


The work in \cite{Neely2010,Ne09} presents a new \emph{drift-plus-penalty (DPP) ratio} algorithm solving renewal optimizations knowing the distribution of $\omega[n]$. The algorithm treats the constraints via \emph{virtual queues} so that one only requires to minimize an unconstrained ratio during every renewal frame. The algorithm provably meets all constraints and achieves asymptotic
near-optimality. The works  \cite{wang2015dynamic, urgaonkar2015dynamic} show that the edge cloud server migration problem can be formulated as a specific renewal optimization. Using a variant of the DPP ratio algorithm, they show that solving a simple stochastic shortest path problem during every renewal frame gives near-optimal performance. The work \cite{wei2018asynchronous} solves a more general asynchronous optimization over parallel renewal systems, though the knowledge of the random event statistics is still required. It is worth noting that the work \cite{Ne09} also proposes a heuristic algorithm when the distribution of 
$\omega[n]$ is not known. That algorithm is partially analyzed:  It is shown that if a certain process converges, then the algorithm converges to a near-optimal point.  However, whether or not such a process converges is unknown.


\subsection{Other related works}
The renewal optimization problem considered in this chapter is a generalization of stochastic optimization
over fixed time slots. Such problems are categorized based on whether or not the random event is observed
before the decision is made.  Cases where the random event is observed before taking actions are often referred to as \textit{opportunistic scheduling problems}. Over the past decades, many algorithms have been proposed including max-weight (\cite{tassiulas1990stability, tassiulas1993dynamic}), Lyapunov optimization (\cite{eryilmaz2006joint, eryilmaz2007fair, Neely2010, georgiadis2006resource}), fluid model methods (\cite{stolyar2005maximizing, eryilmaz2007fair}), and dual subgradient methods (\cite{lin2004joint, ribeiro2010ergodic}) are often used.  

Cases where the random events are not observed are referred to as  \textit{online learning problems}. Various algorithms are developed for unconstrained learning including the weighted majority algorithm (\cite{littlestone1994weighted}), multiplicative weighting algorithm (\cite{freund1999adaptive}), following the perturbed leader (\cite{hutter2005adaptive}) and online gradient descent (\cite{zinkevich2003online, hazan2014beyond}). The resource constrained learning problem is studied in \cite{mahdavi2012trading} and \cite{wu2015algorithms}.
Online learning with an underlying MDP structure is also treated using modified multiplicative weighting (\cite{even2005experts}) and improved following the perturbed leader (\cite{yu2009markov}).



\subsection{Our contributions} 
In this work, we focus on opportunistic scheduling over renewal systems and propose a new algorithm that runs online (i.e. takes actions in response to each observed 
$\omega[n]$). Unlike prior works, the proposed algorithm requires neither the statistics of $\omega[n]$ nor explicit estimation of them, and is fully analyzed with convergence properties that 
hold with probability 1. From a technical perspective, we prove near-optimality of the algorithm by showing  
asymptotic stability of a customized process, relying on a novel construction of exponential supermartingales which could be of independent interest.
We complement our theoretical results with
simulation experiments on a time varying constrained MDP.

\section{Problem Formulation and Preliminaries}\label{formulation}
Consider a system where the time line is divided into back-to-back time periods called frames. At the beginning of frame $n$ ($n\in\{0,1,2,\cdots\}$), a controller observes the realization of a random variable $\omega[n]$, which is an i.i.d. copy of a random variable taking values in a compact set $\Omega\in\mathbb{R}^q$ with distribution function unknown to the controller.
Then, after observing the random event $\omega[n]$, the controller chooses an action vector $\alpha[n]\in\mathcal{A}$. Then, the tuple $(\omega[n],~\alpha[n])$ induces the following random variables:
\begin{itemize}
\item The penalty received during frame $n$: $y[n]$.
\item The length of frame $n$: $T[n]$.
\item A vector of resource consumptions during frame $n$:
$\mathbf{z}[n]=[z_1[n],~z_2[n],~\cdots,~z_L[n]]$.
\end{itemize}
We assume that \textit{given $\alpha[n]=\alpha$ and $\omega[n]=\omega$ at frame $n$, $(y[n],T[n],\mathbf{z}[n])$ is a random vector independent of the outcomes of previous frames}, with \emph{known} expectations. We then denote these conditional expectations as 
\begin{align*}
\hat{y}(\omega,\alpha)=&\expect{y[n]~|~\omega,\alpha},\\
\hat{T}(\omega,\alpha)=&\expect{T[n]~|~\omega,\alpha},\\
\hat{\mathbf{z}}(\omega,\alpha)=&\expect{\hat{\mathbf{z}}[n]~|~\omega,\alpha},
\end{align*}
which are all deterministic functions of $\omega$ and $\alpha$. This notation is useful when we want to highlight the action $\alpha$ we choose. 
The analysis assumes a single action in response to the observed $\omega[n]$ at each frame. Nevertheless, an ergodic MDP can fit into this model by defining the action as a selection of a policy to implement over that frame so that the corresponding $\hat{y}(\omega,\alpha)$, $\hat{T}(\omega,\alpha)$ and $\hat{\mathbf{z}}(\omega,\alpha)$ are expectations over the frame under the chosen policy.

Let
\begin{align*}
\overline{y}[N]&=\frac1N\sum_{n=0}^{N-1}y[n],\\
\overline{T}[N]&=\frac1N\sum_{n=0}^{N-1}T[n],\\
\overline{z}_l[N]&=\frac1N\sum_{n=0}^{N-1}z_l[n]~~~l\in\{1,2,\cdots,L\}.
\end{align*}
The goal is to minimize the time average penalty subject to $L$ constraints on resource consumptions. Specifically, we aim to solve the following fractional programming problem:
\begin{align}
\min~~&\limsup_{N\rightarrow\infty}\frac{\overline{y}[N]}{\overline{T}[N]}\label{prob-1}\\
\textrm{s.t.}~~& \limsup_{N\rightarrow\infty}\frac{\overline{z}_l[N]}{\overline{T}[N]}\leq c_l,~~\forall l\in\{1,2,\cdots,L\},\\
&\alpha[n]\in\mathcal{A},~\forall n\in\{0,1,2,\cdots\} \label{prob-3},
\end{align}
where $c_l,~l\in\{1,2,\cdots,L\}$ are nonnegative constants, and both the minimum and constraint are taken in an almost sure sense. Finally, we use $\theta^*$ to denote the minimum that can be achieved by solving above optimization problem. For simplicity of notation, let
\begin{equation}\label{def-K}
K[n]=\sqrt{\sum_{l=1}^L(z_l[n]-c_lT[n])^2}.
\end{equation}

\subsection{Assumptions}
Our main result requires the following assumptions, their importance will become clear as we proceed. We begin with the following boundedness assumption:
\begin{assumption}[Exponential type]\label{bounded-assumption}
Given $\omega[n]=\omega\in\Omega$ and $\alpha[n]=\alpha\in\mathcal{A}$ for a fixed $n$, it holds that 
$T[n]\geq1$ with probability 1 and $y[n],~K[n],~T[n]$ are of exponential type, i.e. there exists a constant $\eta>0$ s.t.
\begin{align*}
&\expect{\left.\exp\left(\eta \big|y[n]\big|\right)~\right|\omega,\alpha}\leq B+1,\\
&\expect{\left.\exp\left(\eta \big|K[n]\big|\right)~\right|\omega,\alpha}\leq B+1,\\
&\expect{\left.\exp\left(\eta \big|T[n]\big|\right)~\right|\omega,\alpha}\leq B+1,
\end{align*}
where $B$ is a positive constant.
\end{assumption}

The following proposition is a simple consequence of the above assumption:
\begin{prop}\label{prop-1}
~~Suppose Assumption \ref{bounded-assumption} holds.
Let $X[n]$ be any of the three random variables $y[n]$, $K[n]$ and $T[n]$ for a fixed $n$. Then, given 
any $\omega[n]=\omega\in\Omega$ and $\alpha[n]=\alpha\in\mathcal{A}$,
\begin{align*}
\expect{\left.\big|X[n]\big|~\right|\omega,\alpha}\leq B/\eta,~~
\expect{\left.X[n]^2~\right|\omega,\alpha}\leq 2B/\eta^2.
\end{align*}
\end{prop}
The proof follows from the inequality: 
$$
B+1\geq\expect{\left.e^{\eta \big|X[n]\big|}~\right|\omega,\alpha}
\geq 1+\eta\cdot \expect{\left.\big|X[n]\big|~\right|\omega,\alpha}
+ \frac{\eta^2}{2}\cdot\expect{\left.X[n]^2~\right|\omega,\alpha}.
$$ 

\begin{assumption}\label{optimal-assumption}
~~There exists a positive constant $\theta_{\max}$ large enough so that the optimal objective of $\eqref{prob-1}-\eqref{prob-3}$, denoted as $\theta^*$, falls into $[0,\theta_{\max})$ with probability 1.
\end{assumption}

\begin{remark}
~~If $\theta^*<0$, then, we shall find a constant $c$ large enough so that $\theta^*+c\geq0$. Then, define a new penalty $y'[n]=y[n]+cT[n]$. It is easy to see that minimizing $\limsup_{N\rightarrow\infty}\overline{y}[N]/\overline{T}[N]$ is equivalent to minimizing $\limsup_{N\rightarrow\infty}\overline{y'}[N]/\overline{T}[N]$ and the optimal objective of the new problem is $\theta^*+c$, which is nonnegative. 
\end{remark}

\begin{assumption}\label{assumption-for-algorithm}
~~Let $\left(\hat y(\omega,\alpha),~\hat T(\omega,\alpha),~\hat{\mathbf{z}}(\omega,\alpha)\right)$ be the performance vector under a certain $(\omega,\alpha)$ pair. Then, for any fixed $\omega\in\Omega$, the set of achievable performance vectors over all $\alpha\in\mathcal{A}$ is compact.
\end{assumption}

In order to state the next assumption, we need the notion of  \textit{randomized stationary policy}. We start with the definition:
\begin{definition}[Randomized stationary policy]\label{RSP}
A randomized stationary policy is an algorithm that at the beginning of each frame $n$, after observing the random event $\omega[n]$, the controller chooses $\alpha^*[n]$ with a conditional probability that is the same for all $n$. 
\end{definition}

\begin{assumption}[Bounded achievable region]\label{compact}
Let
\[(\overline{y},~\overline{T},~\overline{\mathbf{z}})\triangleq\expect{(\hat y(\omega[0],\alpha^*[0]),~\hat T(\omega[0],\alpha^*[0]),~\hat{\mathbf{z}}(\omega[0],\alpha^*[0]))}\]
be the one-shot average of one randomized stationary policy. Let $\mathcal{R}\subseteq\mathbb{R}^{L+2}$ be the set of all achievable one-shot averages $(\overline{y},~\overline{T},~\overline{\mathbf{z}})$. Then, $\mathcal{R}$ is bounded.
\end{assumption}

\begin{assumption}[$\xi$-slackness]\label{slack}
 There exists a randomized stationary policy $\alpha^{(\xi)}[n]$ such that the following holds,
\[\frac{\expect{\hat{z}_l\left(\omega[n],\alpha^{(\xi)}[n]\right)}}
{\expect{\hat{T}(\omega[n],\alpha^{(\xi)}[n])}}=c_l-\xi,~~\forall l\in\{1,2,\cdots,L\},\]
where $\xi>0$ is a constant.
\end{assumption}

\begin{remark}[Measurability issue]
~~We implicitly assume the policies for choosing $\alpha$ in reaction to $\omega$ result in a measurable $\alpha$, so that $T[n]$, $y[n]$, $\mathbf{z}[n]$ are valid random variables and the expectations in Assumption \ref{compact} and \ref{slack} are well defined. This assumption is mild. For example, when the sets $\Omega$ and $\mathcal{A}$ are finite, it holds for any randomized stationary policy. More generally, if $\Omega$ and $\mathcal{A}$ are measurable subsets of some separable metric spaces, this holds whenever the conditional probability in Definition \ref{RSP} is ``regular'' (see \cite{Durrett} for discussions on regular conditional probability), and $T[n]$, $y[n]$, $\mathbf{z}[n]$ are continuous functions on $\Omega\times\mathcal{A}$. 
\end{remark}

\section{An Online Algorithm}\label{online-section}
We define a vector of virtual queues $\mathbf{Q}[n]=[Q_1[n]~Q_2[n]~\cdots~Q_L[n]]$ which are 0 at $n=0$ and updated as follows:
\begin{equation}\label{queue-update}
Q_l[n+1]=\max\{Q_l[n]+z_l[n]-c_lT[n],0\}.
\end{equation}
The intuition behind this virtual queue idea is that if the algorithm can stabilize $Q_l[n]$, then the ``arrival rate'' $\overline{z}_l[N]/\overline{T}[N]$ is below ``service rate'' $c_l$ and the constraint is satisfied. The proposed algorithm then proceeds as in Algorithm \ref{online-algorithm} via two fixed parameters $V>0$, $\delta>0$, and an additional process $\theta[n]$ that is initialized to be $\theta[0]=0$.
For any real number $x$, the notation $[x]_0^{\theta_{\max}}$ stands for ceil and floor function:
\[[x]_0^{\theta_{\max}}=\left\{
                     \begin{array}{ll}
                       \theta_{\max}, & \hbox{if $x\in(\theta_{\max},+\infty)$;} \\
                       x, & \hbox{if $x\in[0,\theta_{\max}]$;} \\
                       0, & \hbox{if $x\in(-\infty, 0)$.}
                     \end{array}
                   \right.
\]
Note that we can rewrite \eqref{DPP} as the following deterministic form:
$$
 V\left(\hat{y}(\omega[n],\alpha[n])-\theta[n]\hat{T}(\omega[n],\alpha[n])\right)
 +\sum_{l=1}^LQ_l[n]\left(\hat{z}_l(\omega[n],\alpha[n])-c_l\hat{T}(\omega[n],\alpha[n])\right),
$$
Thus, Algorithm \ref{online-algorithm} proceeds by observing $\omega[n]$ on each frame $n$ and then choosing $\alpha[n]$ in $\mathcal{A}$ to minimize the above deterministic function.
We can now see that we only use knowledge of current realization $\omega[n]$, not statistics of $\omega[n]$. Also, the compactness assumption (Assumption \ref{assumption-for-algorithm}) guarantees that the minimum of \eqref{DPP} is always achievable.

\begin{algorithm}
\begin{itemize}
  \item At the beginning of each frame $n$, the controller observes $Q_l[n]$, $\theta[n]$, $\omega[n]$ and chooses action $\alpha[n]\in\mathcal{A}$ to minimize the following function:
\begin{equation}\label{DPP}
  \expect{\left. V(y[n]-\theta[n]T[n])+\sum_{l=1}^LQ_l[n](z_l[n]-c_lT[n])\right|Q_l[n],\theta[n],\omega[n]}.
\end{equation}
  \item Update $\theta[n]$:
  \[\theta[n+1]=\left[\frac{1}{(n+1)^{\delta}}\sum_{i=0}^{n}\left(y[i]-\theta[i]T[i]
  +\frac{1}{V}\sum_{l=1}^LQ_l[i](z_l[i]-c_lT[i])\right)\right]_{0}^{\theta_{\max}}.\]
  \item Update virtual queues $Q_l[n]$:
  \[Q_l[n+1]=\max\{Q_l[n]+z_l[n]-c_lT[n],0\},~l=1,2,\cdots,L.\]
\end{itemize}
\caption{Online renewal optimization:}
\label{online-algorithm}
\end{algorithm}

\section{Feasibility Analysis}
In this section, we prove that the proposed algorithm gives a sequence of actions $\{\alpha[n]\}_{n=0}^{\infty}$ which satisfies all desired constraints with probability 1. Specifically, we show that all virtual queues are stable with probability 1, in which we leverage an important lemma from \cite{hajek1982hitting} to obtain a exponential bound for the norm of $\mathbf{Q}[n]$.

\subsection{The drift-plus-penalty bound}
The start of our proof uses the drift-plus-penalty methodology. For a general introduction on this topic, see \cite{neely2012stability} for more details. 
We define the 2-norm function of the virtual queue vector as:
\[\|\mathbf{Q}[n]\|^2=\sum_{l=1}^LQ_l[n]^2.\]
Define the \textit{Lyapunov drift} $\Delta(\mathbf{Q}[n])$ as
\[\Delta(\mathbf{Q}[n])=\frac12\left(\|\mathbf{Q}[n+1]\|^2-\|\mathbf{Q}[n]\|^2\right).\]
Next, define the penalty function at frame $n$ as $V(y[n]-\theta[n]T[n])$, where $V>0$ is a fixed trade-off parameter. Then, the drift-plus-penalty methodology suggests that we can stabilize the virtual queues by choosing an action $\alpha[n]\in\mathcal{A}$ to greedily minimize the following drift-plus-penalty expression, with the observed $\mathbf{Q}[n]$, $\omega[n]$ and $\theta[n]$:
\[\expect{\left.V(y[n]-\theta[n]T[n])+\Delta(\mathbf{Q}[n])\right|Q_l[n],\theta[n],\omega[n]}.\]
The penalty term $V(y[n]-\theta[n]T[n])$ uses the $\theta[n]$ variable, which 
depends on events from all previous frames. This penalty does not
fit the rubric of \cite{neely2012stability} and convergence of the algorithm does
not follow from prior work. A significant thrust of the current chapter is convergence
analysis under such a penalty function.

In order to obtain an upper bound on $\Delta(\mathbf{Q}[n])$, we square both sides of \eqref{queue-update} and use the fact that $\max\{x,0\}^2\leq x^2$,
\begin{align}\label{dpp-relation}
Q_l[n+1]^2\leq Q_l[n]^2+(z_l[n]-c_lT[n])^2+2Q_l[n](z_l[n]-c_lT[n]).
\end{align}
Summing the above over all $l \in \{1, \ldots, L\}$ and dividing by $2$ gives
$$ \Delta(\mathbf{Q}[n]) \leq \frac{1}{2}\sum_{l=1}^L (z_l[n]-c_lT[n])^2 + \sum_{l=1}^LQ_l[n](z_l[n]-c_lT[n])$$
Adding $V(y[n]-\theta[n]T[n])$ to both sides and taking conditional expectations gives
\begin{align}
&\expect{\left.V(y[n]-\theta[n]T[n])+\Delta(\mathbf{Q}[n])\right|Q_l[n],\theta[n],\omega[n]}\nonumber\\
\leq& \expect{\left.V(y[n]-\theta[n]T[n])+\sum_{l=1}^LQ_l[n](z_l[n]-c_lT[n])\right|Q_l[n],\theta[n],\omega[n]}
+\frac12\sum_{l=1}^L\expect{(z_l[n]-c_lT[n])^2}\nonumber\\
\leq& \expect{\left.V(y[n]-\theta[n]T[n])+\sum_{l=1}^LQ_l[n](z_l[n]-c_lT[n])\right|Q_l[n],\theta[n],\omega[n]}
+\frac{B^2}{\eta^2}.
\label{dpp-upperbound}
\end{align}
where the last inequality follows from Proposition \ref{prop-1}.
Thus, as we have already seen in Algorithm \ref{online-algorithm},
the proposed algorithm observes the vector $\mathbf{Q}[n]$, the random event $\omega[n]$ and $\theta[n]$ at frame $n$, and minimizes the right hand side of \eqref{dpp-upperbound}. 

\subsection{Bounds on the virtual queue process and feasibility}
In this section, we show how the bound \eqref{dpp-upperbound} leads to the feasibility of the proposed algorithm.
Define $\mathcal{H}_n$ as the system history information up until frame $n$. Formally, $\{\mathcal{H}_n\}_{n=0}^{\infty}$ is a filtration where each $\mathcal{H}_n$ is the $\sigma$-algebra generated by all the random variables before frame $n$. Notice that since $\mathbf{Q}[n]$ and $\theta[n]$ depend only on the events before frame $n$, $\mathcal{H}_n$ contains both $\mathbf{Q}[n]$ and $\theta[n]$.
 The following important lemma gives a stability criterion for any given real random process with certain negative drift property:

\begin{lemma}[Theorem 2.3 of \cite{hajek1982hitting}]\label{master-bound}
Let $R[n]$ be a real random process over $n\in \{0,1,2,\cdots\}$ satisfying the following two conditions for a fixed $r>0$:
\begin{enumerate}
\item For any $n$, $\expect{\left.e^{r(R[n+1]-R[n])}\right| \mathcal{H}_n}\leq \Gamma$, for some $\Gamma>0$.
\item Given $R[n]\geq\sigma$, $\expect{\left.e^{r(R[n+1]-R[n])}\right| \mathcal{H}_n}\leq \rho$, with some $\rho\in(0,1)$.
\end{enumerate}
Suppose further that $R[0]\in\mathbb{R}$ is given and finite, then, at every $n\in\{0,1,2,\cdots\}$, the following bound holds:
\[\expect{e^{rR[n]}}\leq \rho^ne^{rR[0]}+\frac{1-\rho^n}{1-\rho}\Gamma e^{r\sigma}.\]
\end{lemma}

Thus, in order to show the stability of the virtual queue process, it is enough to test the above two conditions with
$R[n]=\|\mathbf{Q}[n]\|$. The following lemma shows that $\|\mathbf{Q}[n]\|$ satisfies these two conditions:
\begin{lemma}[Drift condition]\label{geometric-bound}
Let $R[n]=\|\mathbf{Q}[n]\|$, then, it satisfies the two conditions in Lemma \ref{master-bound} with the following constants:
\begin{align*}
\Gamma&=B,\\
r&=\min\left\{\eta,\frac{\xi\eta^2}{4B}\right\},\\
\sigma&=C_0V,\\
\rho&=1-\frac{r\xi}{2}+\frac{2B}{\eta^2}r^2<1.
\end{align*}
where
$C_0=\frac{2B^2}{V\xi\eta^2}+\frac{2(\theta_{\max}+1)B}{\xi\eta}-\frac{\xi}{4V}$.
\end{lemma}
%
%

The central idea of the proof is to plug the $\xi$-slackness policy specified in Assumption \ref{slack} into the right hand side of \eqref{dpp-upperbound}. A similar idea has been presented in the
Lemma 6 of \cite{wei2015probabilistic} under the bounded increment of the virtual queue process. Here, we generalize the idea to the case where the increment of the virtual queues contains exponential type random variables $z_l[n]$ and $T[n]$. Note that
the boundedness of $\theta[n]$ is crucial for the argument to hold, which justifies the truncation of pseudo average in the algorithm. Lemma \ref{master-bound} is proved in the Appendix \ref{sec:proof}.

Combining the above two lemmas, we immediately have the following corollary:
\begin{corollary}[Exponential decay]\label{exponential-queue-bound}
Given $\mathbf{Q}[0]=0$, the following holds for any $n\in\{0,1,2,\cdots\}$ under the proposed algorithm,\\
\begin{equation}\label{eq:exp-queue-bound}
\expect{e^{r\|\mathbf{Q}[n]\|}}\leq D,
\end{equation}
where
\[D=1+\frac{B}{1-\rho} e^{rC_0V},\]
and $r,~\rho,~C_0$ are as defined in Lemma \ref{geometric-bound}. Furthermore, we have
$\expect{\|Q[n]\|}\leq \frac{1}{r}\log(1+ \frac{B}{1-\rho} e^{rC_0V})$, i.e. the queue size is $\mathcal{O}(V)$.
\end{corollary}
The bound on $\expect{\|Q[n]\|}$ follows readily from \eqref{eq:exp-queue-bound} via Jensen's inequality.
With Corollary \ref{exponential-queue-bound} in hand, we can prove the following theorem:
\begin{theorem}[Feasibility]\label{feasibility}
All constraints in \eqref{prob-1}-\eqref{prob-3} are satisfied under the proposed algorithm with probability 1.
\end{theorem}
\begin{proof}[Proof of Theorem \ref{feasibility}]
By queue updating rule \eqref{queue-update}, for any $n$ and any $l\in\{1,2,\cdots,L\}$, one has
  \[Q_l[n+1]\geq Q_l[n]+z_l[n]-c_lT[n].\]
Fix $N$ as a positive integer. Then, summing over all $n\in\{0,1,2,\cdots,N-1\}$,
  \[Q_l[N]\geq Q_l[0]+\sum_{n=0}^{N-1}(z_l[n]-c_lT[n]).\]
  Since $Q_l[0]=0,~\forall l$ and $T[n]\geq1,~\forall n$,
  \begin{equation}\label{inter-constraint-violation}
  \frac{\sum_{n=0}^{N-1}z_l[n]}{\sum_{n=0}^{N-1}T[n]}-c_l\leq\frac{Q_l[N]}{\sum_{n=0}^{N-1}T[n]}\leq\frac{Q_l[N]}{N}.
  \end{equation}
Define the event
\[A_{N}^{(\varepsilon)}=\{Q_l[N]>\varepsilon N\}.\]
By the Markov inequality and Corollary \ref{exponential-queue-bound}, for any $\varepsilon>0$, we have
\begin{align*}
Pr(Q_l[N]>\varepsilon N)
\leq&Pr\left(r\|\mathbf Q[N]\|> r\varepsilon N\right)\\
=&Pr\left(e^{r\|\mathbf Q[N]\|}> e^{r\varepsilon N}\right)\\
\leq&\frac{\expect{e^{r\|\mathbf Q[N]\|}}}{e^{r\varepsilon N}}\leq De^{-r\varepsilon N},
\end{align*}
where $r$ is defined in Corollary \ref{exponential-queue-bound}.
Thus, we have
\begin{align*}
\sum_{N=0}^{\infty}Pr(Q_l[N]>\varepsilon N)\leq D\sum_{N=0}^{\infty}e^{-r\varepsilon N}<+\infty.
\end{align*}
Thus, by the Borel-Cantelli lemma \cite{Durrett},
\[Pr\left(A_{N}^{(\varepsilon)}~\textrm{occurs infinitely often}\right)=0.\]
Since $\varepsilon>0$ is arbitrary, letting $\varepsilon\rightarrow0$ gives
\[Pr\left(\lim_{N\rightarrow\infty}\frac{Q_l[N]}{N}=0\right)=1.\]
Finally, taking the $\limsup_{N\rightarrow\infty}$ from both sides of \eqref{inter-constraint-violation} and substituting in the  above equation gives the claim.
\end{proof}

\section{Optimality Analysis}
In this section, we show that the proposed algorithm achieves time average penalty within $\mathcal{O}(1/V)$ of the optimal objective $\theta^*$. Since the algorithm meets all the constraints, it follows,
\[\limsup_{n\rightarrow\infty}\frac{\sum_{i=0}^{n-1}y[i]}{\sum_{i=0}^{n-1}T[i]}\geq\theta^*,~~w.p.1.\]
Thus, it is enough to prove the following theorem:
\begin{theorem}[Near optimality]\label{theorem_average_converge}
For any $\delta\in(1/3,1)$ and $V\geq1$, the objective value produced by the proposed algorithm is near optimal with
\[\limsup_{n\rightarrow\infty}\frac{\sum_{i=0}^{n-1}y[i]}{\sum_{i=0}^{n-1}T[i]}\leq\theta^*+\frac{B^2}{\eta^2V},~w.p.1,\]
i.e. the algorithm achieves $\mathcal{O}(1/V)$ near optimality.
\end{theorem}
\begin{remark}
Combining Theorem \ref{theorem_average_converge} with Corollary \ref{exponential-queue-bound}, we see that the tuning parameter $V$ plays a trade-off between the sub-optimality and the virtual queue bound (i.e. the constraint violation). In particular, our result recovers the classical 
$[\mathcal{O}(1/V),~\mathcal{O}(V)]$ trade-off in the work of opportunistic scheduling \cite{Neely2010}. 
\end{remark}

In order to prove Theorem \ref{theorem_average_converge}, we introduce the following notation:
\begin{align*}
&\textrm{original pseudo average}:~~\hat{\theta}[n]\triangleq\frac{1}{(n+1)^\delta}\sum_{i=0}^{n}\left(y[i]-\theta[i]T[i]+\frac{1}{V}\sum_{l=1}^LQ_l[i](z_l[i]-c_lT[i])\right),\\
&\textrm{tamed pseudo average}:~~\theta[n]\triangleq\left[\frac{1}{(n+1)^\delta}\sum_{i=0}^{n}\left(y[i]-\theta[i]T[i]+\frac{1}{V}\sum_{l=1}^LQ_l[i](z_l[i]-c_lT[i])\right)\right]_0^{\theta_{\max}}.
\end{align*}

\subsection{Relation between $\hat\theta[n]$ and $\theta[n]$}
We start with a preliminary lemma illustrating that the original pseudo average $\hat\theta[n]$ behaves almost the same as the tamed pseudo average $\theta[n]$. Note that $\theta[n]$ can be written as:
$$\theta[n] = [ \hat{\theta}[n]]_0^{\theta_{max}}. $$
 
\begin{lemma}[Equivalence relation]\label{properties}
For any $x\in(0,\theta_{\max})$,
\begin{enumerate}
\item $\theta[n]\geq x$ if and only if $\hat{\theta}[n]\geq x$.
\item $\theta[n]\leq x$ if and only if $\hat{\theta}[n]\leq x$.
\item $\limsup_{n\rightarrow\infty}\theta[n]\leq x$ if and only if $\limsup_{n\rightarrow\infty}\hat\theta[n]\leq x$.
\item $\limsup_{n\rightarrow\infty}\theta[n]\geq x$ if and only if $\limsup_{n\rightarrow\infty}\hat\theta[n]\geq x$.
\end{enumerate}
\end{lemma}

This lemma is intuitive and the proof is shown in the Appendix \ref{sec:proof}. We will prove results on $\hat{\theta}[n]$ which extend naturally to $\theta[n]$ via Lemma \ref{properties}.

The key idea of proving Theorem \ref{theorem_average_converge} is to bound the original pseudo average process $\hat{\theta}[n]$ asymptotically from above by $\theta^*$, which is Theorem \ref{theorem-asymptotic-upperbound} below. We then prove Theorem \ref{theorem-asymptotic-upperbound} through the following three steps:
\begin{itemize}
\item We construct a truncated version of $\hat\theta[n]$, namely $\tilde\theta[n]$, which has the same limit as $\hat{\theta}[n]$ (Lemma \ref{truncation-lemma} below), so that it is enough to show $\tilde\theta[n]\leq \theta^*$ asymptotically.
\item  For the process $\tilde\theta[n]$, we bound the moments of the hitting time, namely, the time interval between two consecutive visits to the region $\{\tilde\theta[n]\leq\theta^*\}$, by constructing a dominating exponential supermartingale and bounding its size.
 (Lemma \ref{exp-supMG} and \ref{time-moment-bound} below).

\item
We show that $\tilde\theta[n]>\theta^*$ only finitely often asymptotically (with probability 1) using the bounded moments of the hitting time.  
\end{itemize}

\subsection{Towards near optimality (I): Truncation}
The following lemma states that  the optimality of \eqref{prob-1}-\eqref{prob-3} is achievable within the closure of the set of all one-shot averages specified in Assumption \ref{compact}:
\begin{lemma}[Stationary optimality]  \label{optimal-stationary-lemma}
Let $\theta^*$ be the optimal objective of \eqref{prob-1}-\eqref{prob-3}. Then, there exists a tuple $(y^*,~T^*,~\mathbf{z}^*)\in\overline{\mathcal{R}}$, the closure of $\mathcal{R}$, such that the following hold:
\begin{align}
&y^*/T^*=\theta^*   \label{iid1}\\
&z_l^*/T^*\leq c_l ,~\forall l\in\{1,2,\cdots,L\}, \label{iid2}
\end{align}
i.e. the optimality is achievable within $\overline{\mathcal{R}}$.
\end{lemma}
The proof of this lemma is similar to the proof of Theorem 4.5 as well as Lemma 7.1 of \cite{Neely2010}. We omit the details for brevity.

We start the truncation by picking up an $\varepsilon_0>0$ small enough so that $\theta^*+\varepsilon_0/V<\theta_{\max}$.
We aim to show $\limsup_{n\rightarrow\infty}\theta[n]\leq\theta^*+\varepsilon_0/V$. By Lemma \ref{properties},
it is enough to show $\limsup_{n\rightarrow\infty}\hat{\theta}[n]\leq\theta^*+\varepsilon_0/V$. The following lemma tells us it is enough to prove it on a further term-wise truncated version of $\hat{\theta}[n]$.

\begin{lemma}[Truncation lemma]\label{truncation-lemma}
Consider the following alternative pseudo average $\{\tilde{\theta}[n]\}_{n=0}^{\infty}$ obtained by truncating each 
summand such that $\tilde{\theta}[0]=0$ and
\begin{align*}
\tilde{\theta}[n+1]=\frac{1}{(n+1)^{\delta}}\sum_{i=0}^n\left[\left(y[i]-\theta[i]T[i]+\frac1V\sum_{l=1}^LQ_l[i](z_l[i]-c_lT[i])\right)
\wedge \left(\left(\frac{2}{\eta}+\frac{4\sqrt{L}}{\eta rV}\right)\log^2(i+1)\right)\right],
\end{align*}
where $a\wedge b\triangleq\min\{a,b\}$, $\eta$ is defined in Assumption \ref{bounded-assumption}
and $r$ is defined in Lemma \ref{geometric-bound}. Then, we have
\[\limsup_{n\rightarrow\infty}\hat{\theta}[n]=\limsup_{n\rightarrow\infty}\tilde{\theta}[n].\]
\end{lemma}
\begin{proof}[Proof of Lemma \ref{truncation-lemma}]
Consider any frame $i \in \{0, 1, 2, \ldots\}$ such that there is a discrepancy between the summand of $\hat{\theta}[n]$ and $\tilde{\theta}[n]$, i.e.
\begin{align}\label{discrepancy-condition}
y[i]-\theta[i]T[i]+\frac1V\sum_{l=1}^LQ_l[i](z_l[i]-c_lT[i])>\left(\frac{2}{\eta}+\frac{4\sqrt{L}}{\eta rV}\right)\log^2(i+1),
\end{align}
By the Cauchy-Schwartz inequality, this implies
\begin{align*}
y[i]-\theta[i]T[i]+\frac1V\sqrt{\sum_{l=1}^LQ_l[i]^2}\sqrt{\sum_{l=1}^L(z_l[i]-c_lT[i])^2}>\left(\frac{2}{\eta}+\frac{4\sqrt{L}}{\eta rV}\right)\log^2(i+1).
\end{align*}
Thus, at least one of the following three events happened:
\begin{enumerate}
\item $A_i\triangleq \left\{y[i]-\theta[i]T[i]>\frac{2}{\eta}\log^2(i+1)\right\}$.
\item $B_i\triangleq\left\{\sqrt{\sum_{l=1}^LQ_l[i]^2}>\frac{2\sqrt{L}}{r}\log(i+1)\right\}$.
\item $E_i\triangleq\left\{K[i]>\frac{2}{\eta}\log(i+1)\right\}$.
\end{enumerate}
where $K[i]$ is defined in \eqref{def-K}. Indeed, the occurence of one of the three events is necessary for \eqref{discrepancy-condition} to happen.
We then argue that these three events jointly occur only finitely many times. Thus, as $n\rightarrow\infty$, the discrepancies are negligible.

Assume the event $A_i$ occurs, then, since $y[i]-\theta[i]T[i]\leq y[i]$, it follows
$y[i]>\frac{2}{\eta}\log^2(i+1)$.
Then, we have
\begin{align*}
Pr(A_i)\leq&Pr\left(y[i]>\frac{2}{\eta}\log^2(i+1)\right)\\
=&Pr\left(e^{\eta y[i]}>e^{2\log^2(i+1)}\right)\\
\leq&\frac{\expect{e^{\eta y[i]}}}{(i+1)^{2\log(i+1)}}\leq\frac{B}{(i+1)^{2\log(i+1)}},
\end{align*}
where the second to last inequality follows from the Markov inequality  
and the last inequality follows from Assumption \ref{bounded-assumption}.

Assume the event $B_i$ occurs, then, we have
\begin{align*}
\|\mathbf{Q}[i]\|=\sqrt{\sum_{l=1}^LQ_l[i]^2}>\frac{2\sqrt{L}}{r}\log(i+1)\geq\frac{2}{r}\log(i+1).
\end{align*}
Thus, 
\begin{align*}
Pr(B_i)\leq& Pr\left(\|\mathbf{Q}[i]\|>\frac{2}{r}\log(i+1)\right)\\
=&Pr\left(e^{r\|\mathbf{Q}[i]\|}>e^{2\log (i+1)}\right)\\
\leq&\frac{\expect{e^{r\|\mathbf{Q}[i]\|}}}{(i+1)^2}\leq\frac{D}{(i+1)^2},
\end{align*}
where the second to last inequality follows from the 
Markov inequality and the last inequality follows from Corollary \ref{exponential-queue-bound}.

Assume the event $E_i$ occurs.
Again, by Assumption \ref{bounded-assumption} and the Markov inequality,
\begin{align*}
Pr(E_i)=&Pr\left(K[i]>\frac{2}{\eta}\log(i+1)\right)\\
=&Pr\left(e^{\eta K[i]}>e^{2\log(i+1)}\right)\\
\leq&\frac{\expect{e^{\eta K[i]}}}{(i+1)^2}\leq\frac{B}{(i+1)^2},
\end{align*}
where the last inequality follows from Assumption \ref{bounded-assumption} again.
Now, by a union bound,
$$Pr(A_i\cup B_i\cup E_i)\leq Pr(A_i)+Pr(B_i)+Pr(E_i)\leq\frac{B}{(i+1)^{2\log(i+1)}}+\frac{B+D}{(i+1)^2},$$
and thus,
$$\sum_{i=0}^{\infty}Pr(A_i\cup B_i\cup E_i)\leq\sum_{i=0}^{\infty}
\left(\frac{B}{(i+1)^{2\log(i+1)}}+\frac{B+D}{(i+1)^2}\right)<\infty$$
By the Borel-Cantelli lemma, we have the joint event $A_i\cup B_i\cup E_i$ occurs only finitely many times with probability 1, and our proof is finished.
\end{proof}

Lemma \ref{truncation-lemma} is crucial for the rest of the proof. Specifically, it creates an alternative sequence $\tilde{\theta}[n]$ which has the following two properties:
\begin{enumerate}
\item We know exactly what the upper bound of each of the summands is, whereas in $\hat{\theta}[n]$, there is no exact bound for the summand due to $Q_l[i]$ and other exponential type random variables.
\item For any $n\in\mathbb{N}$, we have $\tilde{\theta}[n]\leq\hat{\theta}[n]$. Thus, if $\tilde{\theta}[n]\geq\theta^*+\varepsilon_0/V$ for some $n$, then, $\hat{\theta}[n]\geq\theta^*+\varepsilon_0/V$.
\end{enumerate}

\subsection{Towards near optimality (II): Exponential supermartingale}
The following preliminary lemma demonstrates a negative drift property for each of the summands in $\tilde{\theta}[n]$.
\begin{lemma}[Key feature inequality]\label{key-feature}
For any $\varepsilon_0>0$, if $\theta[i]\geq\theta^*+\varepsilon_0/V$, then, we have
\begin{align*}
\expect{\left.\left(y[i]-\theta[i]T[i]+\frac1V\sum_{l=1}^LQ_l[i](z_l[i]-c_lT[i])\right)
\wedge\left(\left(\frac{2}{\eta}+\frac{4\sqrt{L}}{\eta rV}\right)\log^2(i+1)\right)\right|\mathcal{H}_i}\leq-\varepsilon_0/V,
\end{align*}
\end{lemma}
\begin{proof}[Proof of Lemma \ref{key-feature}]
Since the proposed algorithm minimizes \eqref{DPP} over all possible decisions in $\mathcal{A}$, it must achieve value less than or equal to that of any randomized stationary algorithm $\alpha^*[i]$. This in turn implies,
\begin{align*}
&\expect{\left.\left(y[i]-\theta[i]T[i]+\frac1V\sum_{l=1}^LQ_l[i](z_l[i]-c_lT[i])\right)\right|\mathcal{H}_i,\omega[i]}\\
\leq&\expect{\left.\left(\hat{y}(\omega[i],\alpha^*[i])-\theta[i]\hat T(\omega[i],\alpha^*[i])
+\frac1V\sum_{l=1}^LQ_l[i](\hat z_l(\omega[i],\alpha^*[i])-c_l\hat T(\omega[i],\alpha^*[i]))\right)\right|\mathcal{H}_i,\omega[i]}.
\end{align*}
Taking expectation from both sides with respect to $\omega[i]$ and using the fact that randomized stationary algorithms are i.i.d. over frames and independent of $\mathcal{H}_i$, we have
\begin{align*}
&\expect{\left.\left(y[i]-\theta[i]T[i]+\frac1V\sum_{l=1}^LQ_l[i](z_l[i]-c_lT[i])\right)\right|\mathcal{H}_i,\omega[i]}
\leq\overline{y}-\theta[i]\overline T
+\frac1V\sum_{l=1}^LQ_l[i](\overline z_l-c_l\overline T),
\end{align*}
for any $(\overline{y},\overline{T},\overline{\mathbf{z}})\in\mathcal{R}$. Since $(y^*,T^*,\mathbf{z}^*)$  specified in Lemma \ref{optimal-stationary-lemma} is in the closure of $\mathcal{R}$, we can replace $(\overline{y},\overline{T},\overline{\mathbf{z}})$ by the tuple $(y^*,T^*,\mathbf{z}^*)$ and the inequality still holds. This gives
\begin{align*}
&\expect{\left.\left(y[i]-\theta[i]T[i]+\frac1V\sum_{l=1}^LQ_l[i](z_l[i]-c_lT[i])\right)\right|\mathcal{H}_i,\omega[i]}\\
\leq& y^*-\theta[i]T^*+\frac1V\sum_{l=1}^LQ_l[i](z_l^*-c_l T^*),\\
=& T^*\left(y^*/T^*-\theta[i]+\frac1V\sum_{l=1}^LQ_l[i](z_l^*/T^*-c_l)\right)\\
\leq&T^*(\theta^*-\theta[i])\leq-\varepsilon_0/V,
\end{align*}
where the second to last inequality follows from \eqref{iid1} and \eqref{iid2}, and the last inequality follows from $\theta[i]\geq\theta^*+\varepsilon_0/V$ and $T[i]\geq1$.
Finally, since $a\wedge b\leq a$ for any real numbers $a,b$, it follows,
\begin{align*}
&\expect{\left.\left(y[i]-\theta[i]T[i]+\frac1V\sum_{l=1}^LQ_l[i](z_l[i]-c_lT[i])\right)
\wedge\left(\left(\frac{2}{\eta}+\frac{4\sqrt{L}}{\eta rV}\right)\log^2(i+1)\right)\right|\mathcal{H}_i}\\
\leq&
\expect{\left.\left(y[i]-\theta[i]T[i]+\frac1V\sum_{l=1}^LQ_l[i](z_l[i]-c_lT[i])\right)\right|\mathcal{H}_i}
\leq-\varepsilon_0/V,
\end{align*}
and the claim follows.
\end{proof}

Define $n_k$ as the frame where $\tilde{\theta}[n]$ visits the set $(-\infty,~\theta^*+\varepsilon_0/V)$ for the $k$-th time with the following conventions: 1. If $\tilde{\theta}[n]\in(-\infty,~\theta^*+\varepsilon_0/V)$ and $\tilde{\theta}[n+1]\in(-\infty,~\theta^*+\varepsilon_0/V)$, then we count them as two times. 2. When $k=1$, $n_1$ is equal to 0. Define the \textit{hitting time} $S_{n_k}$ as
\[S_{n_k}=n_{k+1}-n_k.\]
The goal is to obtain a moment bound on this quantity when $\tilde\theta[n_k+1]\geq\theta^*+\varepsilon_0/V$ (otherwise, this quantity is 1). In order to do so, we introduce a new process as follows. For any $n_k$, define 
\begin{align}\label{Fn-construction}
F[n]\triangleq\sum_{i=n_k}^{n-1}\left(y[i]-\theta[i]T[i]+\frac1V\sum_{l=1}^LQ_l[i](z_l[i]-c_lT[i])\right)
\wedge\left(\left(\frac{2}{\eta}+\frac{4\sqrt{L}}{\eta rV}\right)\log^2(i+1)\right),~\forall n>n_k,
\end{align}

The following lemma shows that indeed this $F[n]$ is closely related to $\tilde\theta[n]$. It plays an important role in proving Lemma \ref{time-moment-bound}:
\begin{lemma}\label{comparison-lemma}
~~For any $n>n_k$, if $\tilde{\theta}[n]\geq\theta^*+\varepsilon_0/V$, then, $F[n]\geq0$.
\end{lemma}
\begin{proof}[Proof of Lemma \ref{comparison-lemma}]
Suppose $\tilde{\theta}[n]\geq\theta^*+\varepsilon_0/V$, then, the following holds
\[\theta^*+\varepsilon_0/V\leq\tilde{\theta}[n]=\frac{n_k^\delta}{n^\delta}\tilde\theta[n_k]+\frac{1}{n^\delta}F[n].\]
Thus,
$$F[n]\geq n^\delta(\theta^*+\varepsilon_0/V)-n_k^\delta\tilde\theta[n_k].$$
Since at the frame $n_k$, $\tilde\theta[n_k]<\theta^*+\varepsilon_0/V$, it follows,
\[F[n]\geq \left(n^\delta-n_k^\delta\right)(\theta^*+\varepsilon_0/V).\]
Since $\theta^*+\varepsilon_0/V\geq0$, it follows $F[n]\geq0$ and the claim follows.
\end{proof}
%
%
 Recall our goal is to bound the hitting time $S_{n_k}$ of the process $\tilde\theta[n]$ when $\{\tilde\theta[n_k+1]\geq\theta^*+\varepsilon_0/V\}$, with a strictly negative drift property as Lemma \ref{key-feature}. A classical approach analyzing the hitting time of a stochastic process came from Wald's construction of martingale for sequential analysis (see, for example,  \cite{wald1944cumulative} for details). Later,
  \cite{hajek1982hitting} extended this idea to analyze the stability of a queueing system with drift condition by a supermartingale construnction. Here, we take one step further by considering the following supermartingale construction based on $F[n]$:
\begin{lemma}[Exponential Supermartingale]\label{exp-supMG}
Fix $\varepsilon_0>0$ and $V\geq\max\left\{\frac{\varepsilon_0\eta}{4\log^22}-\frac{2\sqrt{L}}{r},~1\right\}$ such that $\theta^*+\varepsilon_0/V<\theta_{\max}$.
Define a new random process $G[n]$ starting from $n_k+1$ with
\[G[n]\triangleq\frac{\exp\left(\lambda_nF[n\wedge(n_k+S_{n_k})]\right)}{\prod_{i=n_k+1}^{n\wedge (n_k+S_{n_k})}\rho_i}\mathbf{1}_{\{\tilde\theta[n_k+1]\geq\theta^*+\varepsilon_0/V\}},\]
where for any set $A$, $\mathbf{1}_A$ is the indicator function which takes value 1 if $A$ is true and 0 otherwise.
For any $n\geq n_k+1$, $\lambda_n$ and $\rho_n$ are defined as follows:
\begin{align*}
\lambda_n=&\frac{\varepsilon_0}{2Ve\left(\frac{2}{\eta}+\frac{4\sqrt{L}}{\eta rV}\right)^2\log^4(n+1)},\\
\rho_n=&1-\frac{\varepsilon_0^2}{4V^2e\left(\frac{2}{\eta}+\frac{4\sqrt{L}}{\eta rV}\right)^2\log^4(n+1)}.
\end{align*}
Then, the process $G[n]$ is measurable with respect to $\mathcal{H}_n$, $\forall n\geq n_k+1$, and furthermore, it is a supermartingale with respect to the filtration $\{\mathcal{H}_n\}_{n\geq n_k+1}$.
\end{lemma}
The proof of Lemma \ref{exp-supMG} is shown in Appendix \ref{sec:proof}.
\begin{remark}
~~If the increments $F[n+1]-F[n]$ were to be bounded, then, we could adopt the similar construction as that of \cite{hajek1982hitting}. However, in our scenario $F[n+1]-F[n]$ is of the order $\log^2(n+1)$, which is increasing and unbounded. Thus, we need decreasing exponents $\lambda_n$ and increasing weights $\rho_n$ to account for that. Furthermore, the indicator function indicates that we are only interested in the scenario $\{\tilde\theta[n_k+1]\geq\theta^*+\varepsilon_0/V\}$.
\end{remark}

The following lemma uses the previous result to bound the conditional fourth moment of the hitting time $S_{n_k}$. 
\begin{lemma}\label{time-moment-bound}
~~Given
$V\geq\max\left\{\frac{\varepsilon_0\eta}{4\log^22}-\frac{2\sqrt{L}}{r},~1\right\}$ as in Lemma \ref{exp-supMG}, for any $\beta\in(0,1/5)$ and any $\varepsilon_0>0$ such that $\theta^*+\varepsilon_0/V<
\theta_{\max}$,
there exists a positive constant $C_{\beta,V,\varepsilon_0}\simeq\mathcal{O}\left(V^{10}\beta^{-20}\varepsilon_0^{-10}\right)$, such that
\[\expect{S_{n_k}^4|\mathcal{H}_{n_k}}\leq C_{\beta,V,\varepsilon_0}(n_k+2)^{4\beta},~~\forall k\geq1.\]
\end{lemma}

\begin{proof}[Proof of Lemma \ref{time-moment-bound}]
First of all, from Lemma \ref{exp-supMG} gives that $G[n]$ is a supermartingale starting from $n_k+1$, thus, we have the following chains of inequalities for any $n\geq n_k+1$:
\begin{align*}
G[n_k+1]=&\expect{G[n_k+1]~|~\mathcal{H}_{n_k+1}}\\
\geq&\expect{G[n]~|~\mathcal{H}_{n_k+1}}\\
=&\expect{\left.\frac{e^{\lambda_n F[n\wedge(n_k+S_{n_k})]}}
{\prod_{i=n_k+1}^{n}\rho_i}\mathbf{1}_{\{\tilde\theta[n_k+1]\geq\theta^*+\varepsilon_0/V\}}~\right|~\mathcal{H}_{n_k+1}}\\
\geq&\expect{\left.\frac{e^{\lambda_n F[n\wedge(n_k+S_{n_k})]}}
{\prod_{i=n_k+1}^{n}\rho_i}\mathbf{1}_{\{S_{n_k}\geq n-n_k+1\}}\mathbf{1}_{\{\tilde\theta[n_k+1]\geq\theta^*+\varepsilon_0/V\}}~\right|~\mathcal{H}_{n_k+1}}\\
\geq&\frac{1}
{\prod_{i=n_k+1}^{n}\rho_i}Pr\left[\left.S_{n_k}\geq n-n_k+1,~\tilde\theta[n_k+1]\geq\theta^*+\varepsilon_0/V~\right|~\mathcal{H}_{n_k+1}\right],
\end{align*}
where the first inequality uses the supermartingale property and the
 last inequality uses Lemma \ref{comparison-lemma} that on the set $\{S_{n_k}\geq n-n_k+1\}$, $n\wedge(n_k+S_{n_k})=n$ and $F[n]\geq0$. By definition of $G[n_k+1]$,
\begin{align*}
G[n_k+1]=\frac{e^{\lambda_{n_k+1} F[n_k+1]}}
{\rho_{n_k+1}}
\leq\frac{e^{\lambda_{n_k+1}\left(\frac{2}{\eta}+\frac{4\sqrt{L}}{\eta rV}\right)\log^2(n_k+2)}}{\rho_{n_k+1}}
\leq\frac43e,
\end{align*}
where the first inequality follows from the definition of $F[n]$, and the second inequality follows from the assumption that $V\geq\frac{\varepsilon_0\eta}{4\log^22}-\frac{2\sqrt{L}}{r}$, thus,
$\lambda_{n_k+1}\leq\frac{1}{\left(\frac{2}{\eta}+\frac{4\sqrt{L}}{\eta rV}\right)\log^2(n_k+2)}$ and $\rho_{n_k+1}\geq1-\frac{\log^22}{2e}>\frac34$. Thus,
it follows,
\[Pr\left[\left.S_{n_k}\geq n-n_k+1,~\tilde\theta[n_k+1]\geq\theta^*+\varepsilon_0/V~\right|~\mathcal{H}_{n_k+1}\right]\leq \left(\prod_{i=n_k+1}^{n}\rho_i\right)\cdot \frac{4}{3}e.\]
Now, we bound the fourth moment of hitting time:
\begin{align*}
&\expect{\left. S_{n_k}^4~\right|~\mathcal{H}_{n_k+1}}\\
=&\sum_{m=1}^{\infty}m^4Pr\left[\left.S_{n_k}= m~\right|~\mathcal{H}_{n_k+1}\right]\\
\leq&\sum_{m=1}^{\infty}\left((m+1)^4-m^4\right)Pr\left[\left.S_{n_k}\geq m+1,~\tilde\theta[n_k+1]\geq\theta^*+\varepsilon_0/V~\right|~\mathcal{H}_{n_k+1}\right]+1\\
\leq&4\sum_{m=1}^{\infty}(m+1)^3Pr\left[\left.S_{n_k}\geq m+1,~\tilde\theta[n_k+1]\geq\theta^*+\varepsilon_0/V~\right|~\mathcal{H}_{n_k+1}\right]+1\\
\leq&1+\frac{16}{3}e\sum_{m=1}^\infty(m+1)^3\prod_{i=n_k+1}^{n_k+m}\rho_i.
\end{align*}
Thus, it remains to
show there exists a constant $C$ on the order $\mathcal{O}\left(V^{10}\beta^{-20}\varepsilon_0^{-10}\right)$ such that
\[\sum_{m=1}^\infty(m+1)^3\prod_{i=n_k+1}^{n_k+m}\rho_i\leq C(n_k+2)^{4\beta},\]
which is given is Appendix \ref{computation}. This implies there exists a $C_{\beta,V,\varepsilon_0}$ so that
\[\expect{\left. S_{n_k}^4\right|\mathcal{H}_{n_k+1}}\leq C_{\beta, V,\varepsilon_0}(n_k+2)^{4\beta}.\]
Thus,
\begin{align*}
\expect{\left. S_{n_k}^4\right|\mathcal{H}_{n_k}}&=\expect{\expect{\left. S_{n_k}^4\right|\mathcal{H}_{n_k+1}}|\mathcal{H}_{n_k}}
\leq\expect{C_{\beta, V,\varepsilon_0}(n_k+2)^{4\beta}|\mathcal{H}_{n_k}}
=C_{\beta, V,\varepsilon_0}(n_k+2)^{4\beta},
\end{align*}
where the last equality follows from the fact that $n_k\in\mathcal{H}_{n_k}$.
This finishes the proof. 
\end{proof}

\subsection{An asymptotic upper bound on $\theta[n]$}\label{sec:finish-proof}
So far, we have proved that if we pick any $\varepsilon_0>0$ such that $\theta^*+\varepsilon_0/V<\theta_{\max}$, then,  the inter-visiting time has bounded conditional fourth moment. We aim to show that $\limsup_{n\rightarrow\infty}\hat\theta[n]\leq\theta^*$ with probability 1. By Lemma \ref{truncation-lemma}, it is enough to show $\limsup_{n\rightarrow\infty}\tilde\theta[n]\leq\theta^*$. To do so, we need the following Second Borel-Cantelli lemma:
\begin{lemma}[Theorem 5.3.2. of \cite{Durrett}]\label{second-borel}
Let $\mathcal{F}_k,~k\geq1$ be a filtration with $\mathcal{F}_1=\{\emptyset,\Omega\}$, and $A_k,~k\geq1$ be a sequence of events with $A_k\in\mathcal{F}_{k+1}$, then
\[\{A_k~\textrm{occurs infinitely often}\}=\left\{\sum_{k=1}^{\infty}Pr(A_k|\mathcal{F}_{k})=\infty\right\}\]
\end{lemma}

\begin{theorem}[Asymptotic upper bound]\label{theorem-asymptotic-upperbound}
For any 
$\delta\in(1/3,1)$ and $V\geq1$, the following hold,
\[\limsup_{n\rightarrow\infty}\hat\theta[n]\leq\theta^*,~~w.p.1,\]
and
\[\limsup_{n\rightarrow\infty}\theta[n]\leq\theta^*,~~w.p.1.\]
\end{theorem}
\begin{proof}[Proof of Theorem \ref{theorem-asymptotic-upperbound}]
First of all, since the inter-hitting time $S_{n_k}$ has finite fourth moment, each inter-hitting time is finite with probability 1, and thus the process $\{\tilde{\theta}[n]\}_{n=0}^{\infty}$ will visit $(-\infty,\theta^*+\varepsilon_0/V)$ infinitely many times with probability 1.
Then, we pick any $\epsilon>0$ and define the following sequence of events:
\begin{equation}\label{def-A}
A_k\triangleq\left\{\frac{S_{n_k}}{n_k^{1/3}}>\epsilon\right\},~k=1,2,\cdots.
\end{equation}
For any fixed $k$, by Conditional Markov inequality, the following holds with probability 1:
\begin{align*}
Pr(A_k|\mathcal{H}_{n_k})=&Pr\left(\left.S_{n_k}^4>\epsilon^4 n_k^{4/3}\right|\mathcal{H}_{n_k}\right)\\
\leq&\frac{\expect{S_{n_k}^4|\mathcal{H}_{n_k}}}{\epsilon^4 n_k^{4/3}}\\
\leq&\frac{C_{\beta,V,\varepsilon_0}(n_k+2)^{4\beta}}{\epsilon^4 n_k^{4/3}}\\
\leq&\frac{C_{\beta,V,\varepsilon_0}}{\epsilon^4}n_k^{-4/3+4\beta}+\frac{C_{\beta,V,\varepsilon_0}2^{4\beta}}{\epsilon^4n_k^{4/3}}\\
\leq&\frac{C_{\beta,V,\varepsilon_0}}{\epsilon^4}k^{-4/3+4\beta}+\frac{C_{\beta,V,\varepsilon_0}2^{4\beta}}{\epsilon^4}k^{-4/3},
\end{align*}
where the second inequality follows from Lemma \ref{time-moment-bound} with $\beta\in(0,1/5)$, the third inequality follows from the fact that $(a+b)^x\leq a^x+b^x,~\forall a,b\geq0$ and $x\in(0,1)$. The last inequality follows from the fact that the inter-hitting time takes at least one frame and thus $n_k\geq k$.

Choose $\mathcal{F}_k=\mathcal{H}_{n_k}$ and $A_k$ as is defined in \eqref{def-A}. Then, for any $\beta\in(0,1/12)$, we have with probability 1,
\begin{align*}
\sum_{k=1}^{\infty}Pr(A_k|\mathcal{H}_{n_k})\leq
\sum_{k=1}^{\infty}\left(\frac{C_{\beta,V,\varepsilon_0}}{\epsilon^4}k^{-4/3+4\beta}+\frac{C_{\beta,V,\varepsilon_0}2^{4\beta}}{\epsilon^4}k^{-4/3}
\right)<\infty.
\end{align*}
Now by Lemma \ref{second-borel},
\[Pr\left(A_k~\textrm{occurs infinitely often}\right)=0.\]
Since the process $\{\tilde{\theta}[n]\}_{n=0}^{\infty}$ visits $(-\infty,\theta^*+\varepsilon_0/V)$ infinitely many times with probability 1, 
\[\limsup_{n\rightarrow\infty}\frac{S_{n_k}}{n_k^{1/3}}
=\limsup_{k\rightarrow\infty}\frac{S_{n_k}}{n_k^{1/3}}\leq\epsilon,~w.p.1,\]
Since $\epsilon>0$ is arbitrary, let $\epsilon\rightarrow0$ gives
\begin{equation}\label{bound-on-returning}
\lim_{n\rightarrow\infty}\frac{S_{n_k}}{n_k^{1/3}}=0,~w.p.1.
\end{equation}
Finally, we show how this convergence result leads to the bound of $\tilde{\theta}[n]$. According to the updating rule of $\tilde{\theta}[n]$, for any frame $n$ such that $n_k<n\leq n_{k+1}$,
\begin{align*}
\tilde{\theta}[n]=&(\frac{n_k}{n})^{\delta}\tilde{\theta}[n_k]
               +\frac{1}{n^{\delta}}\sum_{i=n_k}^{n-1}\left(y[i]-\theta[i]T[i]+\frac{1}{V}Q[i](z[i]-cT[i])
      \right)\wedge\left(\left(\frac{2}{\eta}+\frac{4\sqrt{L}}{\eta rV}\right)\log^2(i+1)\right)\\
\leq&(\frac{n_k}{n})^{\delta}\left(\theta^*+\frac{\varepsilon_0}{V}\right)
      +\frac{1}{n^{\delta}}\sum_{i=n_k}^{n-1}\left(\left(\frac{2}{\eta}+\frac{4\sqrt{L}}{\eta rV}\right)\log^2(i+1)\right)\\
\leq&(\frac{n_k}{n})^{\delta}\left(\theta^*+\frac{\varepsilon_0}{V}\right)
+\frac{1}{n^{\delta}}S_{n_k}\left(\frac{2}{\eta}+\frac{4\sqrt{L}}{\eta rV}\right)\log^2n,
\end{align*}
where the first inequality follows from the fact that $\tilde\theta[n_k]<\theta^*+\varepsilon_0/V$.
Now, we take the $\limsup_{n\rightarrow\infty}$ from both sides and analyze each single term on the right hand side:
\begin{align*}
&1\geq\limsup_{n\rightarrow\infty}(\frac{n_k}{n})^{\delta}
 \geq\limsup_{k\rightarrow\infty}(\frac{n_k}{n_k+S_{n_k}})^{\delta}
 =\limsup_{k\rightarrow\infty}(\frac{1}{1+\frac{S_{n_k}}{n_k}})^{\delta}=1,~w.p.1,\\
&\limsup_{n\rightarrow\infty}\frac{S_{n_k}}{n^{\delta}}\left(\frac{2}{\eta}+\frac{4\sqrt{L}}{\eta rV}\right)\log^2n
\leq\limsup_{n\rightarrow\infty}\frac{S_{n_k}}{n_k^{1/3}}\cdot
\limsup_{n\rightarrow\infty}\frac{\left(\frac{2}{\eta}+\frac{4\sqrt{L}}{\eta rV}\right)\log^2n}{n^{\delta-1/3}}=0,~w.p.1,
\end{align*}
where we apply the convergence result \eqref{bound-on-returning} in the second line. Thus,
\[\limsup_{n\rightarrow\infty}\tilde{\theta}[n]\leq\theta^*+\frac{\varepsilon_0}{V},~w.p.1.\]
By Lemma \ref{truncation-lemma} we have $\limsup_{n\rightarrow\infty}\hat{\theta}[n]\leq\theta^*+\varepsilon_0/V$. Finally, by Lemma \ref{properties},  and the fact that $\theta^*+\varepsilon_0/V\in(0,\theta_{\max})$, we have  $\limsup_{n\rightarrow\infty}\theta[n]\leq\theta^*+\varepsilon_0/V$. Since this holds for any $\varepsilon_0>0$ small enough, let $\varepsilon_0\rightarrow0$ finishes the proof.
\end{proof}

\subsection{Finishing the proof of near optimality}
With the help of previous analysis on $\theta[n]$, we are ready to prove our main theorem, with the following lemma on strong law of large numbers for martingale difference sequences:
\begin{lemma}[Corollary 4.2 of \cite{neely2012stability}]\label{SLLN}
Let $\{\mathcal{F}_i\}_{i=0}^{\infty}$ be a filtration and let $\{X(i)\}_{i=0}^{\infty}$ be a real-valued random process such that $X(i)\in\mathcal{F}_{i+1},~\forall i$. Suppose there is a finite constant $C$ such that $\expect{X(i)|\mathcal{F}_i}\leq C,~\forall i$, and 
\[\sum_{i=1}^{\infty}\frac{\expect{X(i)^2}}{i^2}<\infty.\]
Then,
\[\limsup_{n\rightarrow\infty}\frac{1}{n}\sum_{i=0}^{n-1}X(i)\leq C,~~w.p.1.\]
\end{lemma}

\begin{proof}[Proof of Theorem \ref{theorem_average_converge}]
Recall for any $n$, the empirical accumulation without ceil and floor function is
\[\hat{\theta}[n]=\frac{1}{n^{\delta}}\sum_{i=0}^{n-1}\left(y[i]-\theta[i]T[i]+\frac{1}{V}\sum_{l=1}^LQ_l[i](z_l[i]-c_lT[i])\right).\]
Dividing both sides by $\sum_{i=0}^{n-1}T[i]/n^{\delta}$ yields
\begin{align*}
\frac{\hat{\theta}[n]}{\frac{1}{n^{\delta}}\sum_{i=0}^{n-1}T[i]}
=&\frac{\sum_{i=0}^{n-1}\left(y[i]-\theta[i]T[i]+\frac{1}{V}\sum_{l=1}^LQ_l[i](z_l[i]-c_lT[i])\right)}{\sum_{i=0}^{n-1}T[i]}\\
=&\frac{\sum_{i=0}^{n-1}\left(y[i]+\frac{1}{V}\sum_{l=1}^LQ_l[i](z_l[i]-c_lT[i])\right)}{\sum_{i=0}^{n-1}T[i]}
    -\frac{\sum_{i=0}^{n-1}\theta[i]T[i]}{\sum_{i=0}^{n-1}T[i]}.
\end{align*}
Moving the last term to the left hand side and taking the $\limsup_{n\rightarrow\infty}$ from both sides gives
\begin{align*}
\limsup_{n\rightarrow\infty}\left(\frac{\hat{\theta}[n]}{\frac{1}{n^{\delta}}\sum_{i=0}^{n-1}T[i]}
+\frac{\sum_{i=0}^{n-1}\theta[i]T[i]}{\sum_{i=0}^{n-1}T[i]}\right)
\geq&\limsup_{n\rightarrow\infty}\frac{\sum_{i=0}^{n-1}y[i]}{\sum_{i=0}^{n-1}T[i]}
    +\frac{\sum_{i=0}^{n-1}\frac{1}{V}\sum_{l=1}^LQ_l[i](z_l[i]-c_lT[i])}{\sum_{i=0}^{n-1}T[i]}\\
\geq&\limsup_{n\rightarrow\infty}\frac{\sum_{i=0}^{n-1}y[i]}{\sum_{i=0}^{n-1}T[i]}
+\frac{1}{2}\frac{\|\mathbf{Q}[n]\|^2-\sum_{i=0}^{n-1}\sum_{l=1}^L(z_l[i]-c_lT[i])^2}{V\sum_{i=0}^{n-1}T[i]}\\
\geq&\limsup_{n\rightarrow\infty}\frac{\sum_{i=0}^{n-1}y[i]}{\sum_{i=0}^{n-1}T[i]}
-\frac{1}{2V}\limsup_{n\rightarrow\infty}\frac{1}{n}\sum_{i=0}^{n-1}K[i]^2,
\end{align*}
where the second inequality follows from inequality \eqref{dpp-relation} and telescoping sums, and the last inequality follows from $T[n]\geq1$, $\|\mathbf{Q}[n]\|^2\geq0$ and $K[i]=\sqrt{\sum_{l=1}^L(z_l[i]-c_lT[i])^2}$. Now we use Lemma \ref{SLLN} with $X(i)=K[i]^2$ to bound the second term. Since $K[i]$ is of exponential type by Assumption \ref{bounded-assumption}, we know that $\expect{K[i]^2|\mathcal{H}_n}\leq 2B^2/\eta^2$. Furthermore, $\expect{K[i]^4}\leq24B^4/\eta^4$. Thus,
\[\sum_{i=1}^{\infty}\frac{\expect{K[i]^4}}{i^2}<\infty.\]
Thus, all assumptions in Lemma \ref{SLLN} are satisfied and we conclude that
\[\limsup_{n\rightarrow\infty}\frac{1}{n}\sum_{i=0}^{n-1}K[i]^2\leq \frac{2B^2}{\eta^2},~w.p.1.\]
This implies,
\[\limsup_{n\rightarrow\infty}\left(\frac{\hat{\theta}[n]}{\frac{1}{n^{\delta}}\sum_{i=0}^{n-1}T[i]}
+\frac{\sum_{i=0}^{n-1}\theta[i]T[i]}{\sum_{i=0}^{n-1}T[i]}\right)
\geq\limsup_{n\rightarrow\infty}\frac{\sum_{i=0}^{n-1}y[i]}{\sum_{i=0}^{n-1}T[i]}-\frac{B^2}{\eta^2V}.\]
By Theorem \ref{theorem-asymptotic-upperbound}, $\hat{\theta}[n]$ is asymptotically upper bounded. Since $\delta<1$ and $T[n]\geq1$, it follows $\frac{1}{n^{\delta}}\sum_{i=0}^{n-1}T[i]=\mathcal{O}(n^{1-\delta})$, which goes to infinity as $n\rightarrow\infty$. Thus,
\[\limsup_{n\rightarrow\infty}\frac{\hat{\theta}[n]}{\frac{1}{n^{\delta}}\sum_{i=0}^{n-1}T[i]}\leq0,\]
and thus,
\[\limsup_{n\rightarrow\infty}\frac{\sum_{i=0}^{n-1}\theta[i]T[i]}{\sum_{i=0}^{n-1}T[i]}
\geq\limsup_{n\rightarrow\infty}\frac{\sum_{i=0}^{n-1}y[i]}{\sum_{i=0}^{n-1}T[i]}
-\frac{B^2}{\eta^2V}.\]
By Theorem \ref{theorem-asymptotic-upperbound} again, $\theta[n]$ is asymptotically upper bounded by $\theta^*$. Based on this result, it is easy to show the following
\[\limsup_{n\rightarrow\infty}\frac{\sum_{i=0}^{n-1}\theta[i]T[i]}{\sum_{i=0}^{n-1}T[i]}\leq\theta^*.\]
Thus, we finally get
\[\limsup_{n\rightarrow\infty}\frac{\sum_{i=0}^{n-1}y[i]}{\sum_{i=0}^{n-1}T[i]}\leq\theta^*+\frac{B^2}{\eta^2V},\]
finishing the proof.
\end{proof}

\section{Simulation experiments}\label{simulation}
In this section, we demonstrate the performance of our proposed algorithm through an application scenario on single user file downloading. We show that this problem can be formulated as a two state constrained online MDP and solved using our proposed algorithm.

Consider a slotted time system where $t\in\{0,1,2,\cdots\}$, and one user is repeatedly downloading files.
We use $F(t)\in\{0,1\}$ to denote the system file state at time slot $t$.  State ``1'' indicates there is an active file in the system for downloading and state ``0'' means there is no file and the system is idle.
Suppose the user can only download 1 file at each time, and the user cannot observe the file length. Each file contains an integer number of packets which is independent and geometrically distributed with expected length equal to 1.

During each time slot where there is an active file for downloading (i.e. $F(t)=1$), the user first observes the channel state
$\omega(t)$, which is the i.i.d. random variable taking values in $\Omega=\{0.2, 0.5, 0.8\}$ with equal probabilities, and delay penalty $s(t)$,  which is also an i.i.d. random variable taking values in $\{1,3,5\}$ with equal probability. Then, the user
makes a service action $\alpha(t)\in\mathcal{A}=\{0, 0.3, 0.6, 0.9\}$. The pair $(\omega(t),\alpha(t))$ affects the following quantities:
\begin{itemize}
\item The success probability of downloading a file at time $t$: $\phi(\alpha(t),\omega(t))\triangleq\alpha(t)\cdot\omega(t)$.
\item The resource consumption $p(\alpha(t))$ at time $t$. We assume $p(0)=0$, $p(0.3)=1$, $p(0.6)=2$ and $p(0.9)=4$.
\end{itemize}
After a file is downloaded, the system goes idle (i.e. $F(t)=0$) and stays there for a random amount of time that is independent and geometrically
distributed with mean equal to 2. The goal is to minimize the time average delay penalty subject to a resource constraint that the time average resource consumption cannot exceed 1.

In \cite{wei2015power}, a similar optimization problem is considered but without random events $\omega(t)$ and $s(t)$, which can be formulated as a two state constrained MDP. Here, using the same logic, we can formulate our optimization problem as a two state constrained online MDP.
Given $F(t)=1$, the file will finish its download at the end of this time slot with probability $\phi(\alpha(t),\omega(t))$. Thus, the transition probabilities out of state 1 are:
\begin{align*}
&Pr[F(t+1)=0|F(t)=1]=\phi(\alpha(t),\omega(t))\\
&Pr[F(t+1)=1|F(t)=1]=1-\phi(\alpha(t),\omega(t)),
\end{align*}
On the other hand, given $F(t)=0$, the system is idle and will transition to the active state in the next slot with probability $\lambda$:
\begin{align*}
&Pr[F(t+1)=1|F(t)=0]=\lambda\\
&Pr[F(t+1)=0|F(t)=0]=1-\lambda,
\end{align*}

Now, we characterize this online MDP through renewal frames and show that it can be solved using the proposed algorithm in Section \ref{formulation}. First, notice that the state ``1'' is recurrent under any action $\alpha(t)$. We denote $t_n$ as the $n$-th time slot when the system returns to state ``1''. Define the renewal frame as the time period between $t_n$ and $t_{n+1}$ with frame size
\[T[n]=t_{n+1}-t_n.\]
Furthermore, since the system does not have any control options in state ``0'', the controller makes exactly one decision during each frame and this decision is made at the beginning of each frame. Thus, we can write out the optimization problem as follows:
\begin{align*}
\min~~&\limsup_{N\rightarrow\infty}\frac{\sum_{n=0}^{N-1}\alpha(t_n)s(t_n)}{\sum_{n=0}^{N-1}T[n]}\\
s.t.~~&\limsup_{N\rightarrow\infty}\frac{\sum_{n=0}^{N-1}p(\alpha(t_n))}{\sum_{n=0}^{N-1}T[n]}\leq1,
~\alpha(t_n)\in\mathcal{A}.
\end{align*}
Subsequently, in order to apply our algorithm, we can define the virtual queue $Q[n]$ as $Q[0]=0$ with updating rule
\[Q[n+1]=\max\{Q[n]+p(\alpha(t_n))-T[n],0\}.\]
Notice that for any particular action $\alpha(t_n)\in\mathcal{A}$ and random event $\omega(t_n)\in\Omega$, we can always compute $\expect{T[n]}$ as
\begin{align*}
\expect{T[n]}&=1-\phi(\alpha(t_n),\omega(t_n))+\phi(\alpha(t_n),\omega(t_n))\left(1+\frac{1}{\lambda}\right)\\
&=1+2\alpha(t_n)\omega(t_n),
\end{align*}
where the second equality follows by substituting $\lambda=0.5$ and $\phi(\alpha(t_n),\omega(t_n))=\alpha(t_n)\omega(t_n)$. Thus, for each $\alpha(t_n)\in\mathcal{A}$, the expression \eqref{DPP} can be computed.

In each of the simulations, each data point is the time average of 2 million slots. We compare the performance of the proposed algorithm with the optimal randomized policy. The optimal policy is computed by formulating the MDP into a linear program with the knowledge of the distribution on $\omega(t)$ and $s(t)$. See \cite{Fo66} for details of this linear program formulation.

In Fig. \ref{fig:Stupendous1}, we plot the performance of our algorithm verses $V$ parameter for different $\delta$ value.
We see from the plots that as $V$ gets larger, the time averages approaches the optimal value and achieves a near optimal performance for $\delta$ roughly between $0.4$ and $1$. A more obvious relation between performance and $\delta$ value is shown in Fig. \ref{fig:Stupendous2}, where we fix $V=300$ and plot the performance of the algorithm verses $\delta$ value. It is clear from the plots that the algorithm fails whenever $\delta$ is too small ($\delta< 0.3$) or too big ($\delta>1$). This meets the statement of Theorem \ref{theorem_average_converge} that the algorithm works for $\delta\in(1/3,1)$.

\begin{figure}[htbp]
   \centering
   \includegraphics[height=3.5in]{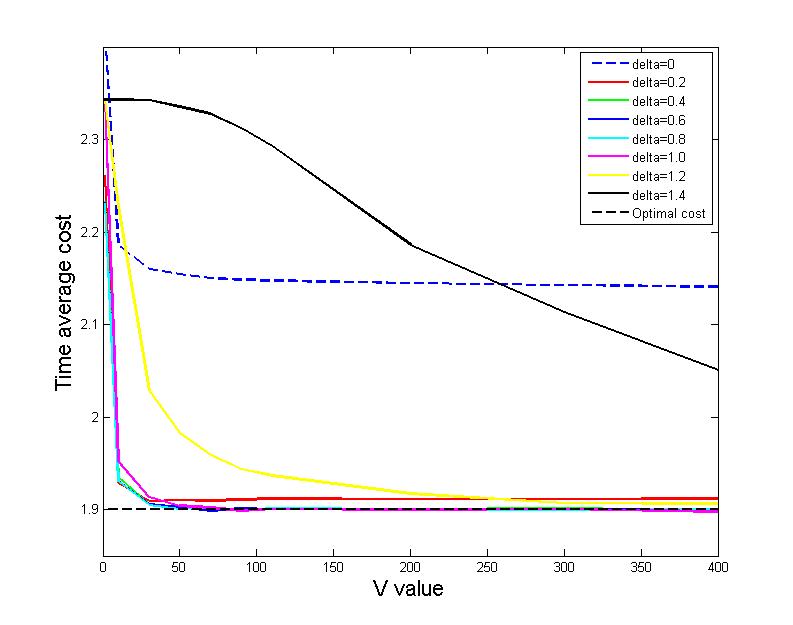} 
   \caption{Time average penalty versus tradeoff parameter V}
   \label{fig:Stupendous1}
\end{figure}

\begin{figure}[htbp]
   \centering
   \includegraphics[height=3.5in]{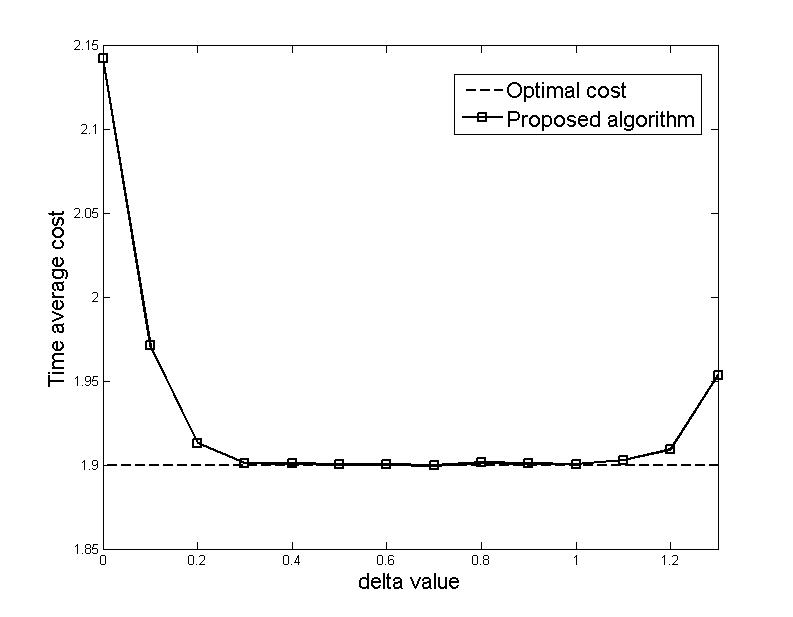} 
   \caption{Time average penalty versus $\delta$ parameter with fixed $V=300$.}
   \label{fig:Stupendous2}
\end{figure}

In Fig. \ref{fig:Stupendous3}, we plot the time average resource consumption verses $V$ value. We see from the plots that the algorithm is always feasible for different $V$'s and $\delta$'s, which meets the statement of Theorem \ref{feasibility}. Also, as $V$ gets larger, the constraint gap tends to be smaller.
In Fig. \ref{fig:Stupendous4}, we plot the average virtual queue size verses $V$ value. It shows that the average queue size gets larger as $V$ get larger. To see the implications,
recall from the proof of Theorem \ref{feasibility}, the inequality \eqref{inter-constraint-violation} implies that the virtual queue size $Q_l[N]$ affects the rate that the algorithm converges down to the feasible region. Thus, if the average virtual queue size is large, then, it takes longer for the algorithm to converge. This demonstrates that $V$ is indeed a trade-off parameter which trades the sub-optimality gap for the convergence rate.

\begin{figure}[htbp]
   \centering
   \includegraphics[height=3.5in]{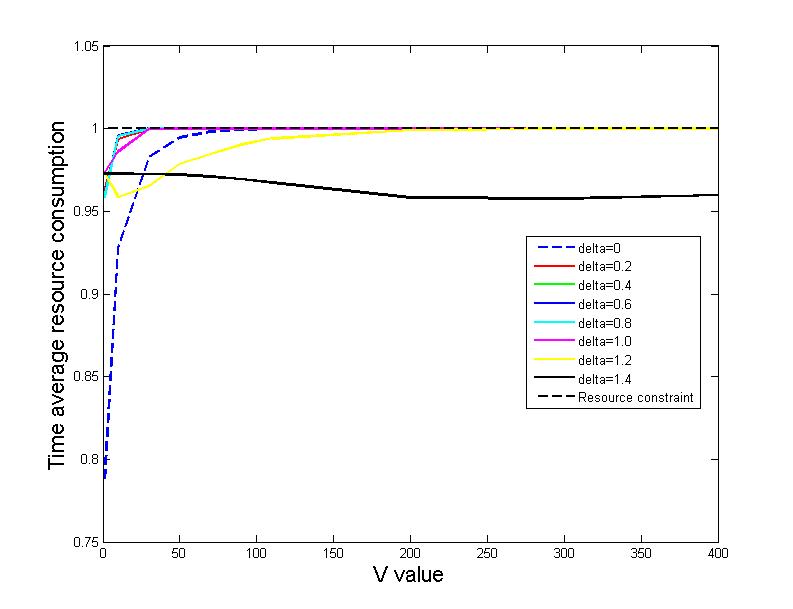} 
   \caption{Time average resource consumption versus tradeoff parameter $V$.}
   \label{fig:Stupendous3}
\end{figure}

\begin{figure}[htbp]
   \centering
   \includegraphics[height=3.5in]{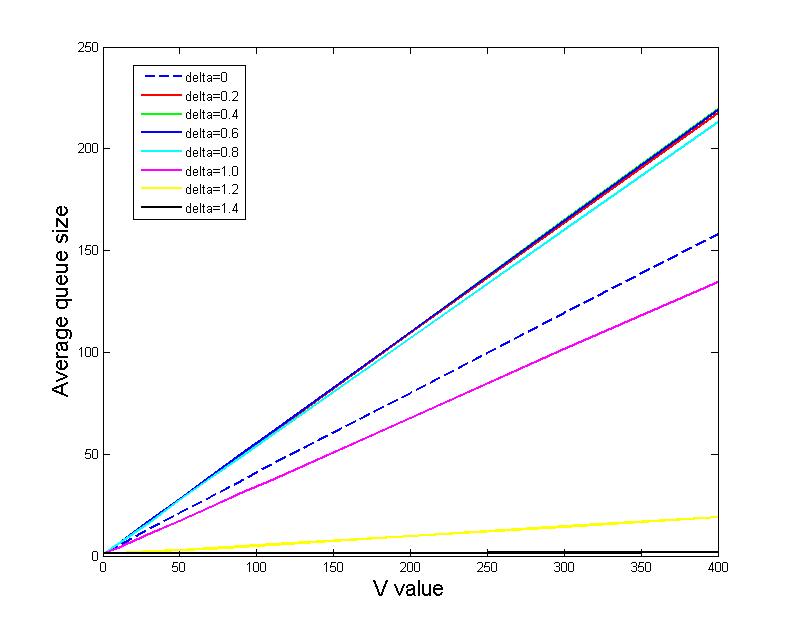} 
   \caption{Time average virtual queue size versus tradeoff parameter $V$.}
   \label{fig:Stupendous4}
\end{figure}

\section{Additional proofs}\label{sec:proof}

\begin{proof}[Proof of Lemma \ref{geometric-bound}]
We begin by bounding the difference $\left|\|\mathbf{Q}[n+1]\|-\|\mathbf{Q}[n]\|\right|$ for any $n$:
\begin{align*}
\big|\|\mathbf{Q}[n+1]\|-\|\mathbf{Q}[n]\|\big|
\leq&\|\mathbf{Q}[n+1]- \mathbf{Q}[n]\|\\
=&\sqrt{\sum_{l=1}^L\big( \max\{Q_l[n] + z_l[n]-c_lT[n],~0\} - Q_l[n]\big)^2}\\
\leq&\sqrt{\sum_{l=1}^L(z_l[n]-c_lT[n])^2}=K[n],
\end{align*}
where the first inequality follows from triangle inequality and
the last inequality follows from the fact that for any $a,b\in\mathbb{R}$, $|\max\{a+b,0\}-a|\leq |b|$. 
Thus, it follows,
\begin{align*}
\left|\expect{\left.\|\mathbf{Q}[n+1]\|-\|\mathbf{Q}[n]\|\right|\mathcal{H}_n}\right|
\leq \expect{\left.K[n]\right|\mathcal{H}_n}\leq \frac{B}{\eta},
\end{align*}
which follows from Proposition \ref{prop-1}. Also, we have
\begin{align*}
\expect{\left.e^{r(\|\mathbf{Q}[n+1]\|-\|\mathbf{Q}[n]\|)}\right|\mathcal{H}_n}
\leq&\expect{\left.\exp\left(rK[n]\right)\right|\mathcal{H}_n}\\
\leq&\expect{\left.\exp\left(\eta K[n]\right)\right|\mathcal{H}_n}
\leq B\triangleq \Gamma
\end{align*}
where the second to last inequality follows by substituting the definition $r=\min\left\{\eta,\frac{\xi\eta^2}{4B}\right\}\leq\eta$ and the last inequality follows from Assumption \ref{bounded-assumption}.

Next, suppose $\|\mathbf{Q}[n]\|> \sigma\triangleq C_0V$. Then, since the proposed algorithm minimizes the term on the right hand side of \eqref{dpp-upperbound} over all possible decisions at frame $n$, it must achieve smaller value on that term compared to that of $\xi$-slackness policy $\alpha^{(\xi)}[n]$ specified in Assumption \ref{slack}. Formally, this is
\begin{align*}
&\expect{\left.\sum_{l=1}^LQ_l[n](z_l[n]-c_lT[n])+V(y[n]-\theta[n]T[n])~\right|~\mathcal{H}_n,\omega[n]}\\
\leq&\expect{\left.\sum_{l=1}^LQ_l[n](z_l^{(\xi)}[n]-c_lT^{(\xi)}[n])+V(y^{(\xi)}[n]-\theta[n]T^{(\xi)}[n])~\right|~\mathcal{H}_n,\omega[n]}.
\end{align*}
where we used the fact that $\theta[n]$ and $\mathbf{Q}[n]$ are in $\mathcal{H}_n$.
Substitute this bound into the right hand side of \eqref{dpp-upperbound}  and take expectation from both sides regarding $\omega[n]$ gives
\begin{align*}
&\expect{\Delta[n]+V(y[n]-\theta[n]T[n])~|~\mathcal{H}_n}\\
\leq&\expect{\left.\sum_{l=1}^LQ_l[n](z_l^{(\xi)}[n]-c_lT^{(\xi)}[n])+V(y^{(\xi)}[n]-\theta[n]T^{(\xi)}[n])~\right|~\mathcal{H}_n}
+B^2/\eta^2.
\end{align*}
Since $\Delta[n]=\frac12(\|\mathbf{Q}[n+1]\|^2-\|\mathbf{Q}[n]\|^2)$,
This implies
\begin{align*}
&\expect{\|\mathbf{Q}[n+1]\|^2-\|\mathbf{Q}[n]\|^2~|~\mathcal{H}_n}\\
\leq&2B^2/\eta^2+2\expect{\left.\sum_{l=1}^LQ_l[n](z_l^{(\xi)}[n]-c_lT^{(\xi)}[n])+V(y^{(\xi)}[n]-\theta[n]T^{(\xi)}[n])
-V(y[n]-\theta[n]T[n])
\right|\mathcal{H}_n}\\
\leq&2B^2/\eta^2+2\sum_{l=1}^LQ_l[n]\expect{\left.z_l^{(\xi)}[n]-c_lT^{(\xi)}[n]
\right|\mathcal{H}_n}+2V\frac{B+\theta_{\max}B}{\eta}\\
\leq&2B^2/\eta^2+2V\frac{B+\theta_{\max}B}{\eta}-2\xi\sum_{l=1}^LQ_l[n]\\
\leq&2B^2/\eta^2+2V\frac{B+\theta_{\max}B}{\eta}-2\xi\|\mathbf{Q}[n]\|,
\end{align*}
where the second inequality follows from applying Proposition \ref{prop-1} to bound $\expect{T[n]|\mathcal{H}_n}$ as well as the fact that $0<\theta[n]< \theta_{\max}$, and the third inequality
follows from the $\xi$-slackness property as well as the assumption that $z_l^{(\xi)}[n]$ is i.i.d. over slots and hence independent of $Q_l[n]$. This further implies
\begin{align*}
&\expect{\|\mathbf{Q}[n+1]\|^2~|~\mathcal{H}_n}\\
\leq&\|\mathbf{Q}[n]\|^2-2\xi\|\mathbf{Q}[n]\|+2B^2/\eta^2+2V\frac{B+\theta_{\max}B}{\eta}\\
=&\|\mathbf{Q}[n]\|^2-2\xi\|\mathbf{Q}[n]\|+2B^2/\eta^2+2V\frac{B+\theta_{\max}B}{\eta}-\frac{\xi^2}{4}+\frac{\xi^2}{4}\\
=&\|\mathbf{Q}[n]\|^2-2\xi\|\mathbf{Q}[n]\|+\frac{2B^2/\eta^2+2V\frac{B+\theta_{\max}B}{\eta}-\frac{\xi^2}{4}}{\xi}\cdot\xi+\frac{\xi^2}{4}\\
=&\|\mathbf{Q}[n]\|^2-2\xi\|\mathbf{Q}[n]\|+C_0V\cdot\xi+\frac{\xi^2}{4}\\
\leq&\|\mathbf{Q}[n]\|^2-\xi\|\mathbf{Q}[n]\|+\frac{\xi^2}{4}=\left(\|\mathbf{Q}[n]\|-\frac\xi2\right)^2,
\end{align*}
where we use the fact that $C_0=\frac{2B^2}{V\xi\eta^2}+\frac{2}{\xi}\frac{B+\theta_{\max}B}{\eta}-\frac{\xi}{4V}$ and also the assumption that $\|\mathbf{Q}[n]\|\geq C_0V$.
Now take the square root from both sides gives
\[\sqrt{\expect{\|\mathbf{Q}[n+1]\|^2~|~\mathcal{H}_n}}\leq\|\mathbf{Q}[n]\|-\frac\xi2.\]
By concavity of $\sqrt{x}$ function, we have $\expect{\left.\|\mathbf{Q}[n+1]\|~\right|~\mathcal{H}_n}\leq\sqrt{\expect{\|\mathbf{Q}[n+1]\|^2~|~\mathcal{H}_n}}$, thus,
\begin{equation}\label{pre-conclusion}
\expect{\left.\|\mathbf{Q}[n+1]\|~\right|~\mathcal{H}_n}\leq\|\mathbf{Q}[n]\|-\frac\xi2.
\end{equation}
Finally, we claim that this gives that under the condition $\|\mathbf{Q}[n]\|> \sigma\triangleq C_0V$,
\begin{equation}\label{conclusion}
\expect{\left.e^{r(\|\mathbf{Q}[n+1]\|-\|\mathbf{Q}[n]\|)}\right|\mathcal{H}_n}\leq\rho\triangleq
1-\frac{r\xi}{2}+\frac{2B}{\eta^2}r^2<1.
\end{equation}
To see this, we expand $\expect{\left.e^{r(\|\mathbf{Q}[n+1]\|-\|\mathbf{Q}[n]\|)}\right|\mathcal{H}_n}$ using Taylor series as follows:
\begin{align*}
&\expect{\left.e^{r(\|\mathbf{Q}[n+1]\|-\|\mathbf{Q}[n]\|)}\right|\mathcal{H}_n}\\
=&1+r\expect{\left.\|\mathbf{Q}[n+1]\|-\|\mathbf{Q}[n]\|\right|\mathcal{H}_n}
+r^2\sum_{k=2}^{\infty}\frac{r^{k-2}\expect{\left.(\|\mathbf{Q}[n+1]\|-\|\mathbf{Q}[n]\|)^k\right|\mathcal{H}_n}}{k!}\\
\leq&1- \frac{r\xi}{2}
+r^2\sum_{k=2}^{\infty}\frac{r^{k-2}\expect{\left.(\|\mathbf{Q}[n+1]\|-\|\mathbf{Q}[n]\|)^k\right|\mathcal{H}_n}}{k!}\\
\leq&1- \frac{r\xi}{2}
+r^2\sum_{k=2}^{\infty}\frac{\eta^{k-2}\expect{\left.(\|\mathbf{Q}[n+1]\|-\|\mathbf{Q}[n]\|)^k\right|\mathcal{H}_n}}{k!}\\
=&1- \frac{r\xi}{2}+r^2\frac{\left(\expect{\left.e^{\eta(\|\mathbf{Q}[n+1]\|-\|\mathbf{Q}[n]\|)}\right|\mathcal{H}_n}
-\eta\expect{\left.\|\mathbf{Q}[n+1]\|-\|\mathbf{Q}[n]\|\right|\mathcal{H}_n}-1\right)}{\eta^2}\\
\leq&1- \frac{r\xi}{2}+\frac{B+\eta\cdot\frac{B}{\eta}}{\eta^2}r^2\\
\leq&1-\frac{r\xi}{2}+\frac{2B}{\eta^2}r^2=\rho,
\end{align*}
where the first inequality follows from \eqref{pre-conclusion}, the second inequality follows from $r\leq\eta$, and the second
to last inequality follows from Proposition \ref{prop-1}.

Finally, notice that the above quadratic function on $r$ attains the minimum at the point
$r=\frac{\xi\eta^2}{4B}$ with value $1-\frac{\xi^2\eta^2}{8B}<1$, and this function is strictly decreasing when
$$r\in\left(0, \frac{\xi\eta^2}{4B}\right).$$
Thus, our choice of
$$r=\min\left\{\eta,\frac{\xi\eta^2}{4B}\right\}
\leq\frac{\xi\eta^2}{4B}$$
ensures that $\rho$ is strictly less than 1 and the proof is finished.
\end{proof}

\begin{proof}[Proof of Lemma \ref{properties}]
If $\theta[n]=y$ for some $y\in[0,\theta_{\max}]$, then, $\hat\theta[n]$ falls into one of the following three cases:
\begin{itemize}
\item $\hat\theta[n]=y$.
\item $y=\theta_{\max}$ and $\hat\theta[n]>\theta_{\max}$.
\item $y=0$ and $\hat\theta[n]<0$.
\end{itemize}
Then, we prove the above four properties based on these three cases.

1) If $\theta[n]=y\geq x$ for some $y$, then, the first two cases immediately imply $\hat\theta[n]\geq x$. If $y=0$, then, we have $x\leq0$, which violates the assumption that $x\in(0,\theta_{\max})$. Thus, the third case is ruled out. On the other hand, if $\hat\theta[n]\geq x$, then, obviously, $\theta[n]\geq x$.

2) If $\theta[n]=y\leq x$ for some $y$, then the last two cases immediately imply $\hat\theta[n]\leq x$. If $y=\theta_{\max}$, then, we have $x\geq y_{\max}$, which violates the assumption that $x\in(0,\theta_{\max})$. Thus, the first case is ruled out. On the other hand, if $\hat\theta[n]\leq x$, then, obviously, $\theta[n]\leq x$.

3) If $\limsup_{n\rightarrow\infty}\theta[n]\leq x$, then, for any $\epsilon>0$ such that $x+\epsilon<y_{\max}$, there exists an $N$ large enough so that $\theta[n]\leq x+\epsilon,~\forall n\geq N$. Then, by property 2), $\hat\theta[n]\leq x+\epsilon,~\forall n\geq N$, which implies $\limsup_{n\rightarrow\infty}\hat\theta[n]\leq x+\epsilon$. Let $\epsilon\rightarrow0$ gives $\limsup_{n\rightarrow\infty}\hat\theta[n]\leq x$.  One the other hand, if $\limsup_{n\rightarrow\infty}\hat\theta[n]\leq x$, then, obviously, $\limsup_{n\rightarrow\infty}\theta[n]\leq x$.

4) If $\liminf_{n\rightarrow\infty}\theta[n]\geq x$, then, for any $\epsilon>0$ such that $x-\epsilon>0$ there exists an $N$ large enough so that $\theta[n]\geq x-\epsilon,~\forall n\geq N$. Then, by property 1), $\hat\theta[n]\geq x-\epsilon,~\forall n\geq N$, which implies $\limsup_{n\rightarrow\infty}\hat\theta[n]\leq x-\epsilon$. Let $\epsilon\rightarrow0$ gives $\limsup_{n\rightarrow\infty}\hat\theta[n]\geq x$. One the other hand, if $\limsup_{n\rightarrow\infty}\hat\theta[n]\leq x$, then, obviously, $\limsup_{n\rightarrow\infty}\theta[n]\leq x$.
\end{proof}

\begin{proof}[Proof of Lemma \ref{exp-supMG}]
The proof is divided into two parts. The first part contains some technical preliminaries showing $G[n]$ is measurable respect to $\mathcal{H}_n,~\forall n\geq n_k+1$, and the second part contains computations to prove the supermartingale claim.
\begin{itemize}
\item \textit{Technical preliminaries:}
First of all, for any fixed $k$, since $n_k$ is a random variable on the integers, we need to justify that $\{\mathcal{H}_n\}_{n\geq n_k+1}$ is indeed a filtration.  
First, it is obvious that $n_k$ a valid stopping time, i.e. 
$$\{n_k\leq t\}\in\mathcal{H}_t,~\forall t\in\mathbb{N}.$$
Then, any $n=n_k+s$ with some constant $s\in\mathbb{N}^+$ is also a valid stopping time because
$$\{n\leq t\}=\{n_k\leq t-s\}\in\mathcal{H}_{(t-s)\vee0}\subseteq\mathcal{H}_t,~\forall t\in\mathbb{N},$$
where $a\vee b\triangleq\max\{a,b\}$. Thus, by definition of stopping time $\sigma$-algebra from \cite{Durrett}, we know that for any $n\geq n_k+1$, $\mathcal{H}_n$ can be written as the collection of all sets $A$ that have $A\cap\{n\leq t\}\in\mathcal{H}_t,~\forall t\in\mathbb{N}$\footnote{An intuitive interpretation is that when $n\leq t$, the set $A$ is contained in the information known until $t$.}. Now, pick $1\leq s_1\leq s_2$ as constants, and if a set $A\in\mathcal{H}_{n_k+s_1}$, then, 
$$A\cap\{n_k+s_2\leq t\}=A\cap\{n_k+s_1\leq t-(s_2-s_1)\}\in\mathcal{H}_{(t-(s_2-s_1))\vee 0}\subseteq\mathcal{H}_t.$$
Thus, $\mathcal{H}_{n_k+s_1}\subseteq\mathcal{H}_{n_k+s_2}$ and $\{\mathcal{H}_n\}_{n\geq n_k+1}$ is indeed a filtration.

Since $\tilde\theta[n_k+1]$ is determined by the realization up to frame $n_k$, it follows, for any $t\in\mathbb{N}^+$,
$$\{\tilde\theta[n_k+1]\geq\theta^*+\varepsilon_0/V\}\cap\{n_k+1\leq t\}
=\cup_{s=1}^t\{\tilde\theta[s]\geq\theta^*+\varepsilon_0/V\}\in\mathcal{H}_t,$$
which implies that $\{\tilde\theta[n_k+1]\geq\theta^*+\varepsilon_0/V\}\in\mathcal{H}_{n_k+1}$. Since 
$\{\mathcal{H}_n\}_{n\geq n_k+1}$ is a filtration, it follows $\{\tilde\theta[n_k+1]\geq\theta^*+\varepsilon_0/V\}\in\mathcal{H}_n$ for any $n\geq n_k+1$. By the same methodology, we can show that
$\{\tilde\theta[n]<\theta^*+\varepsilon_0/V\}\in\mathcal{H}_n,~\forall n\geq n_k+1$, which in turn implies, $\{S_{n_k}+n_k\leq n\}\in\mathcal{H}_n$ and $\{S_{n_k}\geq n-n_k+1\}\in\mathcal{H}_n$. 
Overall, the function $G[n]$ is measurable respect to $\mathcal{H}_n,~\forall n\geq n_k+1$.

\item \textit{Proof of supermartingale claim:} It is obvious that $|G[n]|<\infty$, thus, in order to prove $G[n]$ is a supermartingale, it is enough to show that
\begin{equation}\label{sup_MG_condition}
\expect{\left.G[n+1]-G[n]\right|\mathcal{H}_n}\leq0,~\forall n\geq n_k+1.
\end{equation}
First, on the set $\{S_{n_k}\leq n-n_k\}$, we have
\[\expect{\left.(G[n+1]-G[n])\mathbf{1}_{\{S_{n_k}+n_k\leq n\}}\right|\mathcal{H}_n}=\expect{\left.(G[n]-G[n])\mathbf{1}_{\{S_{n_k}+n_k\leq n\}}\right|\mathcal{H}_n}=0.\]
It is then sufficient to show the inequality \eqref{sup_MG_condition} holds on the set $\{S_{n_k}\geq n-n_k+1\}$. Since
\begin{align*}
&\expect{G[n+1]\mathbf{1}_{\{S_{n_k}\geq n-n_k+1\}}|\mathcal{H}_n}\\
=&\expect{\left.\frac{e^{\lambda_{n+1} F[(n+1)\wedge(n_k+S_{n_k})]}}{\prod_{i=n_k+1}^{(n+1)\wedge (n_k+S_{n_k})}\rho_i}~\right|~\mathcal{H}_{n}}\mathbf{1}_{\{\tilde\theta[n_k+1]\geq\theta^*+\varepsilon_0/V\}}\mathbf{1}_{\{S_{n_k}\geq n-n_k+1\}}\\
=&\expect{\left.\frac{e^{\lambda_{n+1} F[n+1]}}{\prod_{i=n_k+1}^{n+1}\rho_i}~\right|~\mathcal{H}_{n}}
 \mathbf{1}_{\{\tilde\theta[n_k+1]\geq\theta^*+\varepsilon_0/V\}}\mathbf{1}_{\{S_{n_k}\geq n-n_k+1\}}\\
=&\frac{e^{\lambda_{n+1} F[n]}}{\prod_{i=n_k+1}^{n}\rho_i}
  \expect{\left.\frac{e^{\lambda_{n+1} (F[n+1]-F[n])}}{\rho_{n+1}}~\right|~\mathcal{H}_{n}}
  \mathbf{1}_{\{\tilde\theta[n_k+1]\geq\theta^*+\varepsilon_0/V\}}\mathbf{1}_{\{S_{n_k}\geq n-n_k+1\}}\\
\leq&\frac{e^{\lambda_{n} F[n]}}{\prod_{i=n_k+1}^{n}\rho_i}
  \expect{\left.\frac{e^{\lambda_{n+1} (F[n+1]-F[n])}}{\rho_{n+1}}~\right|~\mathcal{H}_{n}}
  \mathbf{1}_{\{\tilde\theta[n_k+1]\geq\theta^*+\varepsilon_0/V\}}\mathbf{1}_{\{S_{n_k}\geq n-n_k+1\}}\\
=&G[n]\expect{\left.\frac{e^{\lambda_{n+1} (F[n+1]-F[n])}}{\rho_{n+1}}~\right|~\mathcal{H}_{n}}
\mathbf{1}_{\{\tilde\theta[n_k+1]\geq\theta^*+\varepsilon_0/V\}}\mathbf{1}_{\{S_{n_k}\geq n-n_k+1\}},
\end{align*}
where $\mathbf{1}_{\{\tilde\theta[n_k+1]\geq\theta^*+\varepsilon_0/V\}}$ and $\mathbf{1}_{\{S_{n_k}\geq n-n_k+1\}}$ can be moved out of the expectation because $\{\tilde\theta[n_k+1]\geq\theta^*+\varepsilon_0/V\}\in\mathcal{H}_{n}$ and
$\{S_{n_k}\geq n-n_k+1\}\in\mathcal{H}_{n}$,
and the only inequality follows from the following argument: On the set $\{S_{n_k}\geq n-n_k+1\}$, $\{\tilde\theta[n]\geq\theta^*+\varepsilon_0/V\}$, thus, by Lemma \ref{comparison-lemma}, $F[n]\geq0$ and
using the fact $\lambda_n>\lambda_{n+1}$, we have $\lambda_{n+1} F[n]\leq\lambda_{n} F[n]$.
Thus, it is sufficient to show that on the set $\{S_{n_k}\geq n-n_k+1\}\cap\{\tilde\theta[n_k+1]\geq\theta^*+\varepsilon_0/V\}$, we have
\[\expect{\left.\frac{e^{\lambda_{n+1} (F[n+1]-F[n])}}{\rho_{n+1}}~\right|~\mathcal{H}_{n}}\leq1.\]
By Taylor expansion, we have
\begin{align*}
&\expect{\left.e^{\lambda_{n+1} (F[n+1]-F[n])}~\right|~\mathcal{H}_{n}}\\
=&1+\lambda_{n+1}\expect{F[n+1]-F[n]~|~\mathcal{H}_{n}}+\sum_{k=2}^{\infty}\frac{\lambda_{n+1}^k}{k!}\expect{(F[n+1]-F[n])^k~|~\mathcal{H}_{n}}\\
=&1+\lambda_{n+1}\expect{F[n+1]-F[n]~|~\mathcal{H}_{n}}
+\lambda_{n+1}^2\sum_{k=2}^{\infty}\frac{\lambda_{n+1}^{k-2}}{k!}\expect{(F[n+1]-F[n])^k~|~\mathcal{H}_{n}}\\
\leq&1-\frac{\lambda_{n+1}\varepsilon_0}{V}
+\lambda_{n+1}^2\sum_{k=2}^{\infty}\frac{\lambda_{n+1}^{k-2}}{k!}\expect{(F[n+1]-F[n])^k~|~\mathcal{H}_{n}},
\end{align*}
where the last inequality comes from the following argument: On the set $\{S_{n_k}\geq n-n_k+1\}$, $\tilde{\theta}[n_k+1]\geq\theta^*+\varepsilon_0/V$, thus,
by the definition of $\tilde\theta[n]$, we have
$\hat{\theta}[n]\geq\tilde\theta[n]\geq\theta^*+\varepsilon_0/V$, and Lemma \ref{properties} gives
$\theta[n]\geq\theta^*+\varepsilon_0/V$, then, by Lemma \ref{key-feature}, we have
\[\expect{F[n+1]-F[n]~|~\mathcal{H}_{n}}\leq-\frac{\varepsilon_0}{V}.\]

Now, by the assumption that $V\geq\frac{\varepsilon_0\eta}{4\log^22}-\frac{2\sqrt{L}}{r}$, we have
$\lambda_{n+1}\leq\frac{1}{\left(\frac{2}{\eta}+\frac{4\sqrt{L}}{\eta rV}\right)\log^2(n+1)}$, which follows from simple algebraic manipulations. Using the fact that $|F[n+1]-F[n]|\leq \left(\frac{2}{\eta}+\frac{4\sqrt{L}}{\eta rV}\right)\log^2(n+1)$, we have
\begin{align*}
&\expect{\left.e^{\lambda_{n+1} (F[n+1]-F[n])}~\right|~\mathcal{H}_{n}}\\
\leq&1-\frac{\lambda_{n+1}\epsilon_0}{V}+\lambda_{n+1}^2\sum_{k=2}^{\infty}
\frac{\left(\frac{1}{\left(\frac{2}{\eta}+\frac{4\sqrt{L}}{\eta rV}\right)\log^2(n+1)}\right)^{k-2}}{k!}
\expect{\left.\left(\left(\frac{2}{\eta}+\frac{4\sqrt{L}}{\eta rV}\right)\log^2(n+1)\right)^k~\right|~\mathcal{H}_{n}}\\
=&1-\frac{\lambda_{n+1}\epsilon_0}{V}+\lambda_{n+1}^2\sum_{k=2}^{\infty}\frac{1}{k!}
\left(\left(\frac{2}{\eta}+\frac{4\sqrt{L}}{\eta rV}\right)\log^2(n+1)\right)^2\\
\leq&1-\frac{\lambda_{n+1}\epsilon_0}{V}+\lambda_{n+1}^2e\left(\frac{2}{\eta}+\frac{4\sqrt{L}}{\eta rV}\right)^2\log^4(n+1)=\rho_{n+1},
\end{align*}
where the final inequality follows by completing the third term back to Taylor series which is equal to $e$.
Overall, the inequality \eqref{sup_MG_condition} holds and $G[n]$ is a supermartingale.
\end{itemize}
\end{proof}

\section{Computation of Asymptotics}\label{computation}
In this appendix, we show that there exists a constant $C$ such that
\[\sum_{m=1}^\infty(m+1)^3\prod_{i=n_k+1}^{n_k+m}\rho_i\leq C(n_k+2)^{4\beta}.\]
We first bound $\rho_i$. Let $C_1=\frac{96V^2e\left(\frac{2}{\eta}+\frac{4\sqrt{L}}{\eta rV}\right)^2}{\varepsilon_0^2\beta^4}$, then,
\begin{align*}
\rho_i = &1-\frac{\varepsilon_0^2}{4V^2e\left(\frac{2}{\eta}+\frac{4\sqrt{L}}{\eta rV}\right)^2\log^4(i+1)} \\
&= 1 - \frac{1}{C_1\frac{\beta^4}{24}\log^4(i+1)}\\
&< 1 - \frac{1}{C_1(i + 1)^{\beta}},
\end{align*}
where we used the fact that $\frac{\beta^4}{24}\log^4(i+1)<(i+1)^\beta,~\forall \beta>0, i\geq0$. Next, to bound $\prod_{i=n_k+1}^{n_k+m}\rho_i$, we take the logarithm:
\begin{align*}
\log\left(\prod_{i=n_k+1}^{n_k+m}\rho_i\right)=&\sum_{i=n_k+1}^{n_k+m}\log\rho_i\\
=&\sum_{i=n_k+1}^{n_k+m}\log\left(1 - \frac{1}{C_1(i + 1)^{\beta}}\right)\\
\leq&-\sum_{i=n_k+1}^{n_k+m}\frac{1}{C_1(i + 1)^{\beta}}\\
\leq&-\frac{1}{C_1}\int_{n_k+2}^{n_k+m+1}\frac{1}{x^{\beta}}dx.
\end{align*}
where the first inequality follows from the first order Taylor expansion. Since $\beta<1$, we compute the integral, which gives
\[-\frac{1}{C_1}\int_{n_k+2}^{n_k+m+1}\frac{1}{x^{\beta}}dx
=-\frac{1}{C_1(1-2\beta)}\left((n_k+m+1)^{1-\beta}-(n_k+2)^{1-\beta}\right).\]
Thus,
\begin{align*}
&\sum_{m=1}^\infty(m+1)^3\prod_{i=n_k+1}^{n_k+m}\rho_i\\
\leq&\sum_{m=1}^\infty(m+1)^3e^{-\frac{1}{C_1(1-\beta)}\left((n_k+m+1)^{1-\beta}-(n_k+2)^{1-\beta}\right)}\\
\leq&\int_0^{\infty}(x+2)^3e^{-\frac{1}{C_1(1-\beta)}\left((x+n_k+2)^{1-\beta}-(n_k+2)^{1-\beta}\right)}dx
+(3C_1(1-\beta))^4,
\end{align*}
where the last inequality follows from the fact that the integrand is monotonically decreasing when $x>3C_1(1-\beta)$, thus, the integral dominates the sum on the tail $x>3C_1(1-\beta)$. For the part where $x\leq 3C_1(1-\beta)$, the maximum of the integrand is bounded by $(3C_1(1-\beta))^3$. Thus, the total difference of such approximation is bounded by $(3C_1(1-\beta))^4$.
Then, we try to estimate the integral. Notice that
\[\frac{d}{dx}e^{-\frac{1}{C_1(1-\beta)}(x+n_k+2)^{1-\beta}}
=-\frac{1}{C_1}e^{-\frac{1}{C_1(1-\beta)}(x+n_k+2)^{1-\beta}}(x+n_k+2)^{-\beta},\]
we do integration-by-parts, which gives
\begin{align*}
&\int_0^{\infty}(x+2)^3e^{-\frac{1}{C_1(1-\beta)}\left((x+n_k+2)^{1-\beta}-(n_k+2)^{1-\beta}\right)}dx\\
=&\int_0^{\infty}(x+2)^3(x+n_k+2)^{\beta}(x+n_k+2)^{-\beta}e^{-\frac{1}{C_1(1-\beta)}(x+n_k+2)^{1-\beta}}dx
\cdot e^{\frac{1}{C_1(1-\beta)}(n_k+2)^{1-\beta}}\\
=&8C_1(n_k+2)^{\beta}+\int_{0}^{\infty}C_1\left(3(x+2)^2(x+n_k+2)^{\beta}+\beta(x+2)^3(x+n_k+2)^{\beta-1}\right)
e^{-\frac{1}{C_1(1-\beta)}\left((x+n_k+2)^{1-\beta}-(n_k+2)^{1-\beta}\right)}dx.
\end{align*}
Since $5\beta\leq1$ and $n_k\geq1$, we have $x+n_k+2\geq x+2$, which implies $(x+2)^3(x+n_k+2)^{\beta-1}\leq(x+2)^2(x+n_k+2)^{\beta}$, thus,
\begin{align*}
&\int_0^{\infty}(x+2)^3e^{-\frac{1}{C_1(1-\beta)}\left((x+n_k+2)^{1-\beta}-(n_k+2)^{1-\beta}\right)}dx\\
\leq&8C_1(n_k+2)^{\beta}+\int_{0}^{\infty}4C_1(x+2)^2(x+n_k+2)^{\beta}e^{-\frac{1}{C_1(1-\beta)}\left((x+n_k+2)^{1-\beta}-(n_k+2)^{1-\beta}\right)}dx.
\end{align*}
Repeat above procedure 3 more times, we have
\begin{align*}
&\int_0^{\infty}(x+2)^3e^{-\frac{1}{C_1(1-\beta)}\left((x+n_k+2)^{1-\beta}-(n_k+2)^{1-\beta}\right)}dx\\
\leq&8C_1(n_k+2)^{\beta}+16C_1^2(n_k+2)^{2\beta}+24C_1^3(n_k+2)^{3\beta}+24C_1^4(n_k+2)^{4\beta}\\
&+\int_0^{\infty}24C_1^4(x+n_k+2)^{4\beta-1}e^{-\frac{1}{C_1(1-\beta)}\left((x+n_k+2)^{1-\beta}-(n_k+2)^{1-\beta}\right)}dx\\
\leq&8C_1(n_k+2)^{\beta}+16C_1^2(n_k+2)^{2\beta}+24C_1^3(n_k+2)^{3\beta}+24C_1^4(n_k+2)^{4\beta}
+24C_1^5\leq C(n_k+2)^{4\beta},
\end{align*}
for some $C$ on the order of $C_1^5$ (which is $\mathcal{O}\left(V^{10}\beta^{-20}\varepsilon_0^{-10}\right)$),
where the second to last inequality follows from $4\beta-1\leq-\beta$ and thus, we replace $(x+n_k+2)^{4\beta-1}$ with $(x+n_k+2)^{-\beta}$ and do a direct integration. Overall, we proved the claim.


\chapter{Online Learning in Weakly Coupled Markov Decision Processes}
In this chapter, we consider online learning over weakly coupled Markov decision processes.
We develop a new distributed online algorithm where each
MDP makes its own decision each slot after observing a multiplier computed from past
information. While the scenario is significantly more challenging than the classical online
learning context, the algorithm is shown to have a tight $\mathcal{O}(\sqrt{T})$ regret and constraint
violations simultaneously over a time horizon $T$.

\section{Problem formulation and related works}

This chapter considers online constrained Markov decision processes (OCMDP) where both the objective and constraint functions can vary each time slot after the decision is made.  We assume a slotted time scenario with time slots $t \in \{0, 1, 2, \ldots\}$.  The OCMDP consists of $K$ parallel Markov decision processes with indices $k \in \{1, 2, \ldots, K\}$.  The $k$-th MDP has state space $\mathcal{S}^{(k)}$, action space $\mathcal{A}^{(k)}$, and transition probability matrix $P_a^{(k)}$ which depends on the chosen action $a \in \mathcal{A}^{(k)}$.  Specifically, $P_a^{(k)} = (P_a^{(k)}(s,s'))$ where 
\[
P_a^{(k)}(s,s')  = Pr\l(s_{t+1}^{(k)}=s'~\l|~s_t^{(k)} = s,~a_t^{(k)}=a\r.\r),
\]
where $s_t^{(k)}$ and $a_t^{(k)}$ are the state and action for system $k$ on slot $t$. 
We assume that both the state space and the action space are finite for all $k\in\{1,2,\cdots,K\}$. 
 
 After each MDP $k \in \{1, \ldots, K\}$ makes the decision at time $t$ (and assuming the current state is $s_t^{(k)} = s$ and the action is $a_t^{(k)}=a)$, the following information is revealed: 
 \begin{enumerate} 
 \item The next state $s_{t+1}^{(k)}$. 
 
 \item A penalty function $f_t^{(k)}(s,a)$ that depends on the current state $s$ and the current action $a$. 
 
 \item A collection of $m$ constraint functions $g_{1,t}^{(k)}(s,a), \ldots, g_{m,t}^{(k)}(s,a)$ that depend on $s$ and $a$. 
 \end{enumerate} 
 The functions $f_t^{(k)}$ and $g_{i,t}^{(k)}$ are all bounded mappings from $\mathcal{S}^{(k)} \times \mathcal{A}^{(k)}$ to $\mathbb{R}$ and represent different types of costs incurred by system $k$ on slot $t$ (depending on the current state and action).  
For example, in a multi-server data center, the different systems $k \in \{1, \ldots, K\}$ can represent different servers, the cost function for a particular server $k$ might represent energy or monetary expenditure for that server, and the constraint costs for server $k$ 
can represent negative rewards such as service rates or qualities.  Coupling between the server systems comes from using all of them to collectively support a common stream of arriving jobs. 

 A key aspect of this general problem is that the functions $f_t^{(k)}$ and $g_{i,t}^{(k)}$ are unknown until after the slot $t$ decision is made. Thus, the precise costs incurred by each system are only known at the end of the slot.  
 For a fixed time horizon of $T$ slots, the overall penalty and constraint accumulation resulting from a policy $\mathscr{P}$ is: 
\begin{equation}\label{main-regret}
F_T(d_0,\mathscr{P}) := \expect{\left.\sum_{t=1}^T\sum_{k=1}^Kf_t^{(k)}\l(a_t^{(k)},s_t^{(k)}\r)\right|~d_0,\mathscr{P}},
\end{equation}
and
\begin{equation*}
G_{i,T}(d_0,\mathscr{P}) := \expect{\left.\sum_{t=1}^T\sum_{k=1}^Kg_{i,t}^{(k)}\l(a_t^{(k)},s_t^{(k)}\r)\right|~d_0,\mathscr{P}}, 
\end{equation*}
where $d_0$ represents a given distribution on the initial joint state vector $(s_0^{(1)}, \cdots, s_0^{(K)})$. Note that $(a_t^{(k)}, s_t^{(k)})$ denotes the state-action pair of the $k$th MDP, which is a pair of random variables determined by $d_0$ and $\mathscr{P}$. Define a constraint set 
\begin{equation}\label{main-constraint}
\mathcal{G}:= \{(\mathscr{P},d_0):~G_{i,T}(d_0,\mathscr{P})\leq 0,~i=1,2,\cdots,m\}.
\end{equation}
Define the regret of a policy $\mathscr{P}$ with respect to a particular joint randomized stationary policy $\Pi$ along with an arbitrary starting state distribution $d_0$ as: 
\[
F_T(d_0,\mathscr{P}) - F_T(d_0,\Pi),
\]
The goal of OCMDP is to choose a policy $\mathscr{P}$  so that both the regret and constraint violations grow sublinearly with respect to $T$, where regret is measured against all feasible joint randomized stationary policies $\Pi$.

Here we give a brief review of the works related to online optimization and online MDPs.
\begin{itemize}
\item \textbf{Online convex optimization (OCO)}: This concerns multi-round cost minimization with arbitrarily-varying convex loss functions.  Specifically, on each slot $t$ the decision maker chooses decisions $x(t)$ within a convex set $\mathcal{X}$ (before observing the loss function $f^t(x)$) in order to minimize the total \emph{regret} compared to the best fixed decision in hindsight, expressed as: 
\begin{align*}
\text{regret}(T) = \sum_{t=1}^{T} f^t(\mathbf{x}(t))  - \min_{\mathbf{x}\in \mathcal{X}} \sum_{t=1}^T f^t(\mathbf{x}).
\end{align*}
See \cite{hazan2016introduction} for an introduction to OCO.  Zinkevich introduced OCO in \cite{zinkevich2003online}
and shows that an online projection gradient descent (OGD) algorithm achieves $O(\sqrt{T})$ regret. This $O(\sqrt{T})$ regret is proven to be the best in \cite{hazan07ML}, although improved performance is possible if all convex loss functions are \emph{strongly} convex.  The OGD decision requires to compute a projection of a vector onto a set $\mathcal{X}$.  For complicated sets $\mathcal{X}$ with functional equality constraints, e.g., $\mathcal{X} = \{x\in \mathcal{X}_0: g_k (\mathbf{x})\leq 0, k\in\{1,2,\ldots,m\}\}$, the projection can have high complexity. To circumvent the projection, work in \cite{mahdavi2012trading,jenatton2016adaptive,yu2016low,chen2017online}
proposes alternative algorithms with simpler per-slot complexity and that satisfy the inequality constraints in the long term (rather than on every slot).  Recently, new primal-dual type algorithms with low complexity are proposed in \cite{neely2017online,hao2017onlinestochastic} to solve more challenging OCO with time-varying functional inequality constraints.

 \item \textbf{Online Markov decision processes}: 
 This extends OCO to allow systems with a more complex Markov structure. This is similar to the setup of the current paper of minimizing the expression  \eqref{main-regret}, but does not have the constraint set \eqref{main-constraint}.  Unlike traditional OCO, the current penalty depends not only on the current action and the current (unknown) penalty function, but on the current system state (which depends on the history of previous actions).  Further, the number of policies can grow exponentially with the sizes of the state and action spaces, so that solutions can be computationally intensive. The work \cite{even2009online} develops an algorithm in this context with 
  $\mathcal{O}(\sqrt{T})$ regret. Extended algorithms and regularization methods are developed 
  in \cite{yu2009markov}\cite{guan2014online}\cite{dick2014online} to reduce complexity and improve dependencies on the number of states and actions.  Online MDP under bandit feedback (where the decision maker can only observe the penalty corresponding to the chosen action) is considered in \cite{yu2009markov}\cite{neu2010online}.

\item \textbf{Constrained MDPs}: 
This aims to solve classical MDP problems with \emph{known} cost functions but subject to additional constraints on the budget or resources. Linear programming methods for MDPs are found, for example, in  \cite{altman1999constrained}, and algorithms beyond LP are found in \cite{neely2011online} \cite{caramanis2014efficient}.  Formulations closest to our setup appear in recent work on weakly coupled MDPs in \cite{boutilier2016budget}\cite{wei2016theory} that have known cost and resource functions.

\item \textbf{Reinforcement Learning (RL)}: This concerns MDPs with some unknown parameters (such as unknown functions and transition probabilities).  Typically, RL makes stronger assumptions than the online setting, such as an environment that is unknown but fixed, whereas the unknown environment in the online context can change over time. Methods for RL are developed in \cite{bertsekas1995dynamic}\cite{sutton1998reinforcement}\cite{lattimore2013sample}\cite{chen2016stochastic}.
\end{itemize}

\section{Preliminaries}\label{sec:assumption}
\subsection{Basic Definitions}
Throughout this paper, given an MDP with state space $\mathcal{S}$ and action space $\mathcal{A}$,
 a \textit{policy} $\mathscr{P}$ defines a (possibly probabilistic) 
 method of choosing actions $a\in\mathcal{A}$ at state $s\in\mathcal{S}$ based on the past information. 
We start with some basic definitions of important classes of policies: 
\begin{definition}
For an MDP, a \textbf{randomized stationary policy} $\pi$ defines an algorithm which, whenever the system is in state $s \in \mathcal{S}$, chooses an action $a \in \mathcal{A}$ according to a fixed conditional probability function $\pi(a|s)$, defined for all $a\in\mathcal{A}$ and $s\in\mathcal{S}$. 
\end{definition}

\begin{definition}\label{def:pp}
For an MDP, a \textbf{pure policy} $\pi$ is a randomized stationary policy with all probabilities equal to either 0 or 1.  That is, 
a pure policy is defined by a  deterministic mapping between states $s \in \mathcal{S}$ and actions $a \in \mathcal{A}$. 
Whenever
the system is in a state $s \in \mathcal{S}$, it always chooses a particular action $a_s \in\mathcal{A}$ (with probability 1).
\end{definition}

Note that if an MDP has a finite state and action space, the set of all pure policies is also finite. 
Consider the MDP associated with a particular system $k \in \{1, \ldots, K\}$. For any randomized stationary policy $\pi$, it holds that $\sum_{a\in\mathcal{A}^{(k)}}\pi(a|s) = 1$ for all  $s\in\mathcal{S}^{(k)}$. Define the transition probability matrix $\mathbf{P}_{\pi}^{(k)}$ under policy $\pi$ to have components as follows:
\begin{equation}\label{transition-matrix}
P_{\pi}^{(k)}(s,s') = \sum_{a\in\mathcal{A}^{(k)}}\pi(a|s)P_{a}^{(k)}(s,s'),~~s,s'\in\mathcal{S}^{(k)}.
\end{equation}
It is easy to verify that $\mathbf{P}_{\pi}^{(k)}$ is indeed a \emph{stochastic matrix}, that is, it 
has rows with nonnegative components that 
sum to 1. Let  $d_0^{(k)}\in[0,1]^{|\mathcal{S}^{(k)}|}$ be an (arbitrary) initial distribution for the $k$-th MDP.
Define the state distribution 
at time $t$ under $\pi$ as $d_{\pi,t}^{(k)}$. By the Markov property of the system, 
we have $d_{\pi,t}^{(k)} = d_{0}^{(k)}\l(\mathbf{P}_{\pi}^{(k)}\r)^t$.    A transition probability matrix $\mathbf{P}_{\pi}^{(k)}$ is 
\emph{ergodic} if it gives rise to a Markov chain that is irreducible and aperiodic.  Since the state space is finite, an ergodic 
matrix $\mathbf{P}_{\pi}^{(k)}$ has a unique stationary distribution denoted $d_{\pi}^{(k)}$, so that  
$d_{\pi}^{(k)}$ is the unique probability vector solving $d=d \mathbf{P}_{\pi}^{(k)}$.

\begin{assumption}[Unichain model]\label{assumption-1}
There exists a universal integer $\widehat{r}\geq1$ such that for any integer $r\geq \widehat{r}$ and every $k \in \{1, \ldots, K\}$, 
 we have the product $\mathbf{P}_{\pi_1}^{(k)}\mathbf{P}_{\pi_2}^{(k)}\cdots \mathbf{P}_{\pi_r}^{(k)}$ is \xcolor{a transition matrix with strictly positive entries} for any sequence of pure policies 
$\pi_1,\pi_2,\cdots,\pi_r$ associated with the $k$th MDP. 
\end{assumption}

\begin{remark}
Assumption \ref{assumption-1}  implies that each MDP $k\in\{1, \ldots, K\}$  is ergodic under any pure policy.  This follows
by taking $\pi_1,\pi_2,\cdots,\pi_r$ all the same in Assumption \ref{assumption-1}. \xcolor{Since the transition matrix of any randomized stationary policy can be formed as a convex combination of those of pure policies, any randomized stationary policy results in an ergodic MDP  for which there is a unique stationary distribution.} Assumption \ref{assumption-1} is easy to check via the following simple sufficient condition.
\end{remark}

\begin{proposition}
Assumption \ref{assumption-1} holds if, for every $k \in \{1, \ldots, K\}$,  
there is a fixed ergodic matrix $\mathbf{P}^{(k)}$ (i.e., a transition probability matrix that defines an irreducible and aperiodic
Markov chain) such that for any pure policy $\pi$ on MDP $k$ we have the decomposition 
$$\mathbf{P}_{\pi}^{(k)} = \delta_{\pi}\mathbf{P}^{(k)} + (1-\delta_{\pi})\mathbf{Q}_{\pi}^{(k)},$$
where $\delta_{\pi}\in(0,1]$ depends on the pure policy $\pi$ and $\mathbf{Q}_{\pi}^{(k)}$ is a stochastic matrix depending on $\pi$.
\end{proposition}
\begin{proof}
Fix $k \in \{1, \ldots, K\}$ and assume every pure policy on MDP $k$ has the above decomposition. 
Since there are only finitely many pure policies, there exists a lower bound $\delta_{\min}>0$ such that $\delta_{\pi}\geq\delta_{\min}$ for every pure policy $\pi$. 
Since $\mathbf{P}^{(k)}$ is an ergodic matrix, there exists an integer $r^{(k)}>0$ large enough such that $(\mathbf{P}^{(k)})^r$ has strictly positive components for all $r \geq r^{(k)}$.  Fix $r \geq r^{(k)}$ and 
let $\pi_1, \ldots, \pi_r$ be any sequence of $r$ pure policies on MDP $k$. Then
\[\mathbf{P}_{\pi_1}^{(k)}\cdots\mathbf{P}_{\pi_r}^{(k)}\geq \delta_{\min}\l(\mathbf{P}^{(k)}\r)^r > 0,\]
where inequality is treated entrywise.
The universal integer $r$ can be taken as the maximum integer $r^{(k)}$ over all $k \in \{1, \ldots, K\}$. 
\end{proof}

\begin{definition}\label{def:jrsp}
A \textbf{joint randomized stationary policy} $\Pi$ on $K$ parallel MDPs defines an algorithm which chooses a joint action $\mathbf{a}:=\l(a^{(1)},~a^{(2)},~\cdots,~a^{(K)}\r)
\in\mathcal{A}^{(1)}\times\mathcal{A}^{(2)}\cdots\times\mathcal{A}^{(K)}$ given the joint state $\mathbf{s}:=\l(s^{(1)},~s^{(2)},~,\cdots,s^{(K)}\r)
\in\mathcal{S}^{(1)}\times\mathcal{S}^{(2)}\cdots\times\mathcal{S}^{(K)}$ according to a fixed conditional probability 
$\Pi\l(\mathbf{a} \l| \mathbf{s} \r.\r)$.
\end{definition}

The following special class of \emph{separable}  policies can be implemented separately over each of the $K$ MDPs and plays a role in both algorithm design and performance analysis.

\begin{definition}\label{def:rsp}
A joint randomized stationary policy $\pi$ is \textbf{separable} if the conditional probabilities 
 $\pi:=\l(\pi^{(1)},~\pi^{(2)},~\cdots,~\pi^{(K)}\r)$ decompose as a product 
 $$\pi\l(\mathbf{a} \l| \mathbf{s} \r.\r) = \prod_{k=1}^K \pi^{(k)}\l(a^{(k)}|s^{(k)}\r)$$
 for all $\mathbf{a} \in \mathcal{A}^{(1)}\times \cdots\times\mathcal{A}^{(K)}$, $\mathbf{s} \in \mathcal{S}^{(1)}\cdots\times\mathcal{S}^{(K)}$.
 \end{definition}

\subsection{Technical assumptions}

The functions $f_{t}^{(k)}$ and $g_{i,t}^{(k)}$ are determined by random processes defined over $t= 0,1,2,\cdots$. Specifically, 
let $\Omega$ be a finite dimensional vector space. Let $\{\omega_t\}_{t=0}^{\infty}$ and $\{\mu_t\}_{t=0}^{\infty}$ be two sequences of random vectors in $\Omega$. Then for all $a\in\mathcal{A}^{(k)}$, $s\in\mathcal{S}^{(k)}$, $i\in\{1,2,\cdots,m\}$ we have
 \begin{align*}
 &g_{i,t}^{(k)}(a,s) = \hat{g}_i^{(k)}\l(a,s,\omega_t\r),\\
 &f_{t}^{(k)}(a,s) =  \hat{f}^{(k)}\l(a,s,\mu_t\r)
 \end{align*}
 where $\hat{g}_i^{(k)}$ and $\hat{f}^{(k)}$ formally define the time-varying functions in terms of the random processes $\omega_t$ and $\mu_t$. 
 It is assumed that the processes $\{\omega_t\}_{t=0}^{\infty}$ and $\{\mu_t\}_{t=0}^{\infty}$ are generated at the start of slot $0$ (before any control actions are taken),
 and 
revealed gradually over time, 
 so that functions $g_{i,t}^{(k)}$ and $f_t^{(k)}$ are only revealed at the end of slot $t$.
 \begin{remark}
 The functions generated at time 0 in this way are also called \textit{oblivious functions} because they are not influenced by control actions. Such an assumption is commonly adopted in previous unconstrained online MDP works (e.g. \cite{even2009online}, \cite{yu2009markov} and \cite{dick2014online}). Further, it is also shown in \cite{yu2009markov} that without this assumption, one can choose a sequence of objective functions against the decision maker in a specifically designed MDP scenario so that one never achieves the sublinear regret. 
 \end{remark}

 The functions are also assumed to be bounded by a universal constant $\Psi$, so that: 
 \begin{multline}\label{function-bounds}
  |\hat{g}_i^{(k)}(a,s,\omega)| \leq \Psi, |\hat{f}^{(k)}(a,s,\mu)| \leq \Psi \quad, \forall k \in \{1, \ldots, K\},    
  \forall a \in \mathcal{A}^{(k)},~s \in \mathcal{S}^{(k)}, ~\forall \omega, \mu \in \Omega.  
 \end{multline}
 It is assumed that $\{\omega_t\}_{t=0}^{\infty}$ is independent, identically distributed (i.i.d.) and independent of $\{\mu_t\}_{t=0}^{\infty}$. Hence, 
 the constraint functions can be arbitrarily correlated on the same slot, but appear i.i.d. over different slots. 
On the other hand, no specific model is imposed on $\{\mu_t\}_{t=0}^{\infty}$. Thus,  
the functions $f_{t}^{(k)}$ can be arbitrarily time varying. 
Let 
$\mathcal{H}_t$ be the system information up to time $t$, then, for any $t\in\{0,1,2,\cdots\}$, $\mathcal{H}_t$ contains
state and action information up to time $t$, i.e.
$\mathbf{s}_0,\cdots,\mathbf{s}_t$, $\mathbf{a}_0,\cdots,\mathbf{a}_t$, and $\{\omega_t\}_{t=0}^{\infty}$ and $\{\mu_t\}_{t=0}^{\infty}$.
Throughout this paper, we make the following assumptions.

\begin{assumption}[Independent transition]\label{assumption:indep-trans}
For each MDP, given the state $s^{(k)}_t\in\mathcal{S}^{(k)}$ and action $a^{(k)}_t\in\mathcal{A}^{(k)}$, the next state $s^{(k)}_{t+1}$ is independent of all other past information up to time $t$ as well as the state transition $s_{t+1}^{(j)},~\forall j\neq k$, i.e., for all
$s \in \mathcal{S}^{(k)}$ it holds that 
$$Pr\l( s_{t+1}^{(k)}=s | \mathcal{H}_t, s_{t+1}^{(j)},~\forall j\neq k \r)=Pr\l( s_{t+1}^{(k)}=s | s_t^{(k)},a_t^{(k)}\r)$$
where 
$\mathcal{H}_t$ contains all past information up to time $t$.
\end{assumption}

Intuitively, this assumption means that all MDPs are running independently in the joint probability space and thus the only coupling among them comes from the constraints, which reflects the notion of \textit{weakly coupled MDPs} in our title. Furthermore, by definition of $\mathcal{H}_t$, given $s_t^{(k)},a_t^{(k)}$, the next transition $s_{t+1}^{(k)}$ is also independent of function paths $\{\omega_t\}_{t=0}^{\infty}$ and $\{\mu_t\}_{t=0}^{\infty}$. 

The following assumption states the constraint set is strictly feasible.

\begin{assumption}[Slater's condition]\label{assumption:slater}
There exists a real value $\eta>0$ and a fixed separable randomized stationary policy $\widetilde{\pi}$ such that
\[
\mathbb{E}\left[\sum_{k=1}^Kg_{i,t}^{(k)}\l(a^{(k)}_t,s^{(k)}_t\r)\Big|~ d_{\widetilde{\pi}},\widetilde{\pi}\right]\leq-\eta,~\forall i\in\{1,2,\cdots,m\},
\]
\xcolor{
where the initial state is $d_{\widetilde{\pi}}$ and is the unique stationary distribution of policy $\widetilde{\pi}$, and the expectation is taken with respect to the random initial state and the stochastic function $g_{i,t}^{(k)}(a,s)$ (i.e., $\omega_t$).  }
\end{assumption}

Slater's condition is a common assumption in convergence time analysis of constrained convex optimization (e.g. \cite{nedic2009approximate}, \cite{bertsekas2009convex}). 
Note that this assumption readily implies the constraint set $\mathcal{G}$ can be achieved by the above randomized stationary policy. Specifically, take $d_0^{(k)}=d_{\widetilde\pi^{(k)}}$ and $\mathscr{P}=\widetilde{\pi}$, then, we have
\[G_{i,T}(d_0,{\tilde\pi}) = \sum_{t=0}^{T-1}\mathbb{E}\left[\sum_{k=1}^Kg_{i,t}^{(k)}\l(a^{(k)}_t,s^{(k)}_t\r)\Big|~ d_{\widetilde{\pi}},\widetilde{\pi}\right]\leq -\eta T<0.
\]

\subsection{The state-action polyhedron}\label{sec:sapoly}
In this section, we recall the well-known linear program formulation of an MDP (see, for example, \cite{altman1999constrained} and \cite{fox1966markov}).
Consider an MDP with a state space $\mathcal{S}$ and an action space $\mathcal{A}$.
Let $\Delta\subseteq \mathbb{R}^{|\mathcal S| |\mathcal{A}|}$ be a probability simplex, i.e.
\[\Delta=\l\{\theta\in\mathbb{R}^{|\mathcal S| |\mathcal{A}|}:~\sum_{(s,a)\in\mathcal{S}\times\mathcal{A}}\theta(s,a)=1,~\theta(s,a)\geq0\r\}.\]
Given a randomized stationary policy $\pi$ with stationary state distribution $d_\pi$, the MDP is a Markov chain with transition matrix $\mathbf{P}_\pi$ given by \eqref{transition-matrix}. Thus, it must satisfy the following balance equation:
\[
\sum_{s\in\mathcal{S}}d_\pi(s)P_\pi(s,s') =  d_\pi(s'),~\forall s'\in\mathcal{S}.
\]
Defining $\theta(a,s) = \pi(a|s)d_\pi(s)$ and substituting the definition of transition probability \eqref{transition-matrix} into the above equation gives
\[
\sum_{s\in\mathcal{S}}\sum_{a\in\mathcal{A}}\theta(s,a)P_a(s,s') = \sum_{a\in\mathcal{A}}\theta(s',a),
~~\forall s'\in\mathcal{S}.
\]
The variable $\theta(a,s)$ is often interpreted as a stationary probability of being at state $s\in\mathcal{S}$ and taking action $a\in\mathcal{A}$ under some randomized stationary policy. 
The state action polyhedron $\Theta$ is then defined as 
\[
\Theta := \l\{\theta\in\Delta:~\sum_{s\in\mathcal{S}}\sum_{a\in\mathcal{A}}\theta(s,a)P_a(s,s') = \sum_{a\in\mathcal{A}}\theta(s',a),
~~\forall s'\in\mathcal{S}  \r\}.
\]
Given any $\theta\in\Theta$, one can recover a randomized stationary policy $\pi$ at any state $s\in\mathcal{S}$ as 
\begin{equation}\label{solution-to-LP}
\pi(a|s) =
\begin{cases}
 \frac{\theta(a,s)}{\sum_{a\in\mathcal{A}}\theta(a,s)},~~&\textrm{if}~\sum_{a\in\mathcal{A}}\theta(a,s)\neq0,\\
 0,~~&\textrm{otherwise}.
\end{cases}
\end{equation}

Given any fixed penalty function $f(a,s)$, the best policy minimizing the penalty (without constraint) is a randomized stationary policy given by the solution to the following linear program (LP):
\begin{align}\label{LP1}
\min~~\langle\mathbf{f},\theta\rangle,~~s.t.~~\theta\in\Theta.
\end{align}
where $\mathbf{f}:=[f(a,s)]_{a\in\mathcal{A},~s\in\mathcal{S}}$. Note that for any policy $\pi$ given by the state-action pair $\theta$ according to \eqref{solution-to-LP},
\begin{align*}
\l\langle \mathbf{f},\theta \r\rangle = \mathbb{E}_{s\sim d_\pi, a\sim \pi(\cdot | s)}\l[f(a,s)\r],
\end{align*}
Thus, $\l\langle \mathbf{f},\theta \r\rangle$ is often referred to as the stationary state penalty of policy $\pi$.

It can also be shown that any state-action pair in the set $\Theta$ can be achieved by a convex combination of state-action vectors of pure policies, and thus all corner points of the polyhedron $\Theta$ are from pure policies. As a consequence, the best randomized stationary policy solving \eqref{LP1} is always a pure policy.

\subsection{Preliminary results on MDPs}\label{sec:prelim-thm}
In this section, we give preliminary results regarding the properties of our weakly coupled MDPs under randomized stationary policies. The proofs can be found in Appendix \ref{proof:section2}. We start with a lemma on the uniform mixing of MDPs.

\begin{lemma}\label{lemma:mixing1}
\xcolor{Suppose Assumption \ref{assumption-1} and \ref{assumption:indep-trans} hold. }
There exists a positive integer $r$ and a constant $\tau\geq1$ such that for any two state distributions $d_1$ and $d_2$,
\begin{multline*}
\sup_{\pi_1^{(k)},\cdots,\pi_r^{(k)}}\l\|  \l(d_1^{(k)}-d_2^{(k)}\r)\mathbf{P}_{\pi_1^{(k)}}^{(k)}\mathbf{P}_{\pi_2^{(k)}}^{(k)}\cdots \mathbf{P}_{\pi_r^{(k)}}^{(k)}  \r\|_1
\leq 
e^{-1/\tau}\l\| d_1^{(k)}-d_2^{(k)}\r\|_1,~\forall k\in\{1,2,\cdots,K\}
\end{multline*}
where the supremum is taken with respect to \textbf{any} sequence of $r$ randomized stationary policies $\l\{\pi_1^{(k)},\cdots,\pi_r^{(k)}\r\}$.
\end{lemma}


For the $k$-th MDP, let $\Theta^{(k)}$ be its state-action polyhedron according to the definition in Section \ref{sec:sapoly}. For any joint randomized stationary policy, let  $\theta^{(k)}$ be the marginal state-action probability vector on the $k$-th MDP, i.e. for any joint state-action distribution $\Phi(\mathbf{a},\mathbf{s})$ where $\mathbf{a}\in \mathcal{A}^{(1)}\times\cdots\times\mathcal{A}^{(K)}$ and $\mathbf{s}\in\mathcal{S}^{(1)}\times\cdots\times\mathcal{S}^{(K)}$, we have
 $\theta^{(k)}(a^{(k)}, s^{(k)}) = \sum_{a^{(j)}, s^{(j)},~j\neq k}\Phi(\mathbf{a},\mathbf{s})$.

We have the following lemma:

\begin{lemma}\label{lemma:prod-chain}
\xcolor{Suppose Assumption \ref{assumption-1} and \ref{assumption:indep-trans} hold.}
Consider the product MDP with product state space $\mathcal{S}^{(1)}\times\cdots\times \mathcal{S}^{(K)}$ and action space $\mathcal{A}^{(1)}\times\cdots\times \mathcal{A}^{(K)}$. Then, for any joint randomized stationary policy, the following hold:
\begin{enumerate}
\item  The product MDP is irreducible and aperiodic.
\item  The marginal stationary state-action probability vector $\theta^{(k)}\in\Theta^{(k)},~\forall k\in\{1,2,\cdots,K\}$.
\end{enumerate}
\end{lemma}

An immediate conclusion we can draw from this lemma is that given any penalty and constraint functions $\mathbf{f}^{(k)}$ and $\mathbf{g}_{i}^{(k)}$, $k=1,2,\cdots,K$, the stationary penalty and constraint value of any joint randomized stationary policy can be expressed as 
$$
\sum_{k=1}^K\l\langle\mathbf{f}^{(k)},\theta^{(k)} \r\rangle,~\sum_{k=1}^K\l\langle\mathbf{g}_i^{(k)},\theta^{(k)} \r\rangle,~~i=1,2,\cdots,m,
$$
with $\theta^{(k)}\in\Theta^{(k)}$. 
This in turn implies such stationary state-action probabilities $\{\theta^{(k)}\}_{k=1}^{K}$ can also be realized via a separable randomized stationary policy
$\pi$ with 
\begin{equation}\label{eq:relation}
\pi^{(k)}(a|s) = \frac{\theta^{(k)}(a,s)}{\sum_{a\in\mathcal{A}^{(k)}}\theta^{(k)}(a,s)},~a\in\mathcal{A}^{(k)},~s\in\mathcal{S}^{(k)},
\end{equation}
and the corresponding stationary penalty and constraint value can also be achieved via this policy. This fact implies that when considering the stationary state performance only, the class of separable randomized stationary policies is large enough to cover all possible stationary penalty and constraint values.

\xcolor{
In particular, 
let $\tilde\pi=\l( \tilde\pi^{(1)},\cdots,\tilde\pi^{(K)}\r)$ be the separable randomized stationary policy associated with the Slater condition (Assumption \ref{assumption:slater}).  Using the fact that the constraint functions $\mathbf{g}_{i,t}^{(k)},k=1,2,\cdots,K$ (i.e. $w_t$) are i.i.d.and Assumption \ref{assumption:indep-trans} on independence of probability transitions, we have the constraint functions $g_{i,t}^{(k)}$ and the state-action pairs at any time $t$ are mutuallly independent.  Thus, }
\[
\mathbb{E}\left[\sum_{k=1}^Kg_{i,t}^{(k)}\l(a^{(k)}_t,s^{(k)}_t\r)\Big|~ d_{\widetilde{\pi}},\widetilde{\pi}\right]
= \sum_{k=1}^K\l\langle\expect{\mathbf{g}_{i,t}^{(k)}}, \tilde\theta^{(k)} \r\rangle,
\]
where $\tilde\theta^{(k)}$ corresponds to $\tilde\pi$ according to \eqref{eq:relation}.

Then, Slater's condition
can be translated to the following: There exists a sequence of state-action probabilities $\{\tilde\theta^{(k)}\}_{k=1}^{K}$ from a separable randomized stationary policy such that $\tilde\theta^{(k)}\in\Theta^{(k)},~\forall k$, and
\begin{equation}\label{slater-2}
\sum_{k=1}^K\l\langle\expect{\mathbf{g}_{i,t}^{(k)}}, \tilde\theta^{(k)} \r\rangle\leq-\eta,~~i=1,2,\cdots,m,
\end{equation}
The assumption on separability does not lose generality in the sense that if there is no separable randomized stationary policy that satisfies \eqref{slater-2}, then, there is no \textit{joint} randomized stationary policy that satisfies \eqref{slater-2} either.

\subsection{The blessing of slow-update property in online MDPs}\label{sec:dis-slow-change}
\xcolor{
 The current state of an MDP depends on previous states and actions. 
As a consequence, the slot $t$ penalty not only depends on the current penalty function and current action,  but also on the system history.  This complication does not arise in classical online convex optimization (\cite{hazan2016introduction},\cite{zinkevich2003online}) as there is no notion of ``state'' and the slot $t$ penalty depends only on the slot $t$ penalty function and action. 
}

\xcolor{
Now imagine a virtual system where, on each slot $t$, a policy $\pi_t$ is chosen (rather than an action).  Further imagine the MDP immediately reaching its corresponding stationary distribution $d_{\pi_t}$. Then the states and actions on previous slots do not matter and the slot $t$ performance depends only on the chosen policy $\pi_t$ and on the current penalty and constraint functions.  This imaginary system now has a structure similar to classical online convex optimization as in the Zinkevich scenario \cite{zinkevich2003online}. 
}

\xcolor{
A key feature of online convex optimization algorithms as in \cite{zinkevich2003online} is that they update their decision variables slowly.  For a fixed time scale $T$ over which 
$\mathcal{O}(\sqrt{T})$ regret is desired, the decision variables are typically changed no more than a distance $\mathcal{O}(1/\sqrt{T})$ from one slot to the next.  An important insight in prior (unconstrained) MDP works(e.g. \cite{dick2014online}, \cite{even2009online}, and \cite{yu2009markov})
 is that such slow updates also guarantee the ``approximate'' convergence of an MDP to its stationary distribution.  As a consequence, one can design the decision policies under the imaginary assumption that the system instantly reaches its stationary distribution, and later bound the error between the true system and the imaginary system.  If the error is on the same order as the desired $\mathcal{O}(\sqrt{T})$ regret, then this approach works.  This idea serves as a cornerstone of our algorithm design of the next section, which treats the case of multiple weakly coupled systems with both objective functions and constraint functions. 
}

\section{OCMDP algorithm}\label{sec:algorithm}
\xcolor{
Our proposed algorithm is distributed in the sense that each time slot, each MDP solves its own subproblem and the constraint violations are controlled by a simple update of global multipliers called ``virtual queues'' at the end of each slot.}
Let $\Theta^{(1)},~\Theta^{(2)},~\cdots,~\Theta^{(K)}$ be the state-action polyhedra of $K$ MDPs, respectively.
Let $\theta_t^{(k)}\in\Theta^{(k)}$ be a state-action vector at time slot $t$. 
At $t = 0$, each MDP chooses its initial state-action vector $\theta_0^{(k)}$ resulting from any \textit{separable} randomized stationary policy $\pi_0^{(k)}$.
For example, one could choose a uniform policy $\pi^{(k)}(a|s) = 1/\l|\mathcal{A}^{(k)}\r|,~\forall s\in\mathcal{S}^{(k)}$, solve the equation $d_{\pi_0^{(k)}} = d_{\pi_0^{(k)}}\mathbf{P}_{\pi_0^{(k)}}^{(k)}$ to get a probability vector $d_{\pi_0^{(k)}}$, and obtain $\theta_0^{(k)}(a,s) = d_{\pi_0^{(k)}}(s)/\l|\mathcal{A}^{(k)}\r|$. For each constraint $i\in\{1,2,\cdots,m\}$, let $Q_i(t)$ be a \textit{virtual queue} defined over slots $t=0,1,2,\cdots$ with the initial condition $Q_i(0) = Q_i(1) = 0$, and update equation:
\begin{equation}\label{Q-update}
Q_i(t+1) = \max\left\{ Q_i(t) + \sum_{k=1}^K \l\langle\mathbf{g}_{i,t-1}^{(k)},\theta_t\r\rangle,~0\right\},~\forall t\in\{1,2,3,\cdots\}.
\end{equation}
Our algorithm uses two parameters $V>0$ and $\alpha>0$ and makes decisions as follows: At the start of each slot $t\in\{1,2,3,\cdots\}$, 
\begin{itemize}
\item The $k$-th MDP observes $Q_i(t),~i=1,2,\cdots,m$ and chooses $\theta_t^{(k)}$ to solve the following subproblem:
\begin{equation}\label{optimize}
\theta_t^{(k)} = \textrm{argmin}_{\theta\in\Theta^{(k)}}\left\langle 
V\mathbf{f}_{t-1}^{(k)}+\sum_{i=1}^mQ_i(t)\mathbf{g}_{i,t-1}^{(k)},\theta\right\rangle 
+\alpha\l\|\theta-\theta_{t-1}^{(k)}\r\|_2^2.
\end{equation}
\item Construct the randomized stationary policy $\pi_t^{(k)}$ according to \eqref{solution-to-LP} with 
$\theta=\theta_t^{(k)}$, and choose the action $a_t^{(k)}$ at $k$-th MDP according to the conditional distribution $\pi_t^{(k)}\l(\cdot|s^{(k)}_t\r)$. 
\item Update the virtual queue $Q_i(t)$ according to \eqref{Q-update} for all $i=1,2,\cdots,m$.
\end{itemize}

\begin{remark}
Note that for any slot $t\geq1$, this algorithm gives a \textit{separable} randomized stationary policy, so that each MDP chooses its own policy based on its own function 
$\mathbf{f}_{t-1}^{(k)}$, $\mathbf{g}_{i,t-1}^{(k)},i\in\{1,2,\cdots,m\}$, and a common multiplier 
$\mathbf{Q}(t) := \l(Q_1(t),\cdots,Q_m(t)\r)$. \lcolor{Furthermore, note that \eqref{optimize} is a convex quadratic program (QP). Standard theory of QP (e.g. \cite{ye1989extension}) shows that the computation complexity solving \eqref{optimize} is $poly\l(\l|\mathcal{S}^{(k)}\r|\l|\mathcal{A}^{(k)}\r|\r)$ for each $k$. Thus, the total computation complexity over all MDPs during each round is $poly\l(K\l|\mathcal{S}^{(k)}\r|\l|\mathcal{A}^{(k)}\r|\r)$.}
\end{remark}
\lcolor{
\begin{remark}
The quadratic term $\alpha\l\|\theta-\theta_{t-1}^{(k)}\r\|_2^2$ in \eqref{optimize} penalizes the deviation of $\theta$ from the previous decision variable $\theta_{t-1}^{(k)}$. Thus, under proper choice of $\alpha$, the distance between $\theta_{t}^{(k)}$ and $\theta_{t-1}^{(k)}$ would be very small, which is the slow update condition we need according to Section \ref{sec:dis-slow-change}.
\end{remark}
}

The next lemma shows that solving \eqref{optimize} is in fact a projection onto the state-action polyhedron. For any set $\mathcal{X}\in\mathbb{R}^n$ and a vector $\mathbf{y}\in\mathbb{R}^n$, define the projection operator $\mathcal{P}_\mathcal{X}(\mathbf{y})$ as 
\[
\mathcal{P}_\mathcal{X}(\mathbf{y}) = \textrm{arginf}_{\mathbf{x}\in\mathcal{X}}\|\mathbf{x} - \mathbf{y}\|_2.
\]
\begin{lemma}
Fix an $\alpha>0$ and $t\in\{1,2,3,\cdots\}$. The $\theta_t$ that solves \eqref{optimize} is
\[
\theta_t^{(k)} = \mathcal{P}_{\Theta^{(k)}}\left( \theta_{t-1}^{(k)} - \frac{\mathbf{w}_t^{(k)}}{2\alpha}\right),
\]
where 
$
\mathbf{w}_t^{(k)} = V\mathbf{f}_{t-1}^{(k)}+\sum_{i=1}^mQ_i(t)\mathbf{g}_{i,t-1}^{(k)}\in\mathbb{R}^{|\mathcal{A}^{(k)}||\mathcal{S}^{(k)}|}.
$
\end{lemma}
\begin{proof}
By definition, we have
\begin{align*}
\theta_t^{(k)} =& \textrm{argmin}_{\theta\in\Theta^{(k)}} \l\langle \mathbf{w}_t^{(k)} , \theta \r\rangle + \alpha \l\| \theta - \theta^{(k)}_{t-1} \r\|_2^2\\
=& \textrm{argmin}_{\theta\in\Theta^{(k)}}   \l\langle \mathbf{w}_t^{(k)} , \theta - \theta^{(k)}_{t-1} \r\rangle + \alpha \l\| \theta - \theta^{(k)}_{t-1} \r\|_2^2  \\
&+   \l\langle \mathbf{w}_t^{(k)}, \theta^{(k)}_{t-1} \r\rangle\\
=& \textrm{argmin}_{\theta\in\Theta^{(k)}}~  \alpha\cdot \l(\l\langle \l.\mathbf{w}_t^{(k)}\r/\alpha , \theta - \theta^{(k)}_{t-1} \r\rangle + \l\| \theta - \theta^{(k)}_{t-1} \r\|_2^2\r)\\
&+   \l\langle \mathbf{w}_t^{(k)}, \theta^{(k)}_{t-1} \r\rangle\\
=& \textrm{argmin}_{\theta\in\Theta^{(k)}}~ \alpha\cdot\l\| \theta -  \theta^{(k)}_{t-1} + \l.\mathbf{w}_t^{(k)}\r/2\alpha \r\|_2^2  \\
=&\mathcal{P}_{\Theta^{(k)}}\left( \theta_{t-1}^{(k)} - \left.\mathbf{w}_t^{(k)}\right/2\alpha\right),
\end{align*}
finishing the proof.
\end{proof}

\subsection{Intuition of the algorithm and roadmap of analysis}
The intuition of this algorithm follows from the discussion in Section \ref{sec:dis-slow-change}. Instead of the Markovian regret \eqref{main-regret} and constraint set \eqref{main-constraint}, we work on the imaginary system that after the decision maker chooses any joint policy $\Pi_t$ and the penalty/constraint functions are revealed, the $K$ parallel Markov chains reach stationary state distribution right away, with state-action probability vectors $\l\{\theta_t^{(k)}\r\}_{k=1}^K$ for $K$ parallel MDPs. Thus there is no Markov state in such a system anymore and the corresponding stationary
 penalty and constraint function value at time $t$ can be expressed as $\sum_{k=1}^K\l\langle\mathbf{f}_t^{(k)},\theta_t^{(k)}\r\rangle$ and 
$\sum_{k=1}^K\l\langle\mathbf{g}_{i,t}^{(k)},\theta_t^{(k)}\r\rangle,~i=1,2,\cdots,m$, respectively. As a consequence, we are now facing a relatively easier task 
of minimizing the following regret:
\begin{equation}\label{stat-regret}
\sum_{t=0}^{T-1}\sum_{k=1}^K\expect{\l\langle\mathbf{f}_t^{(k)},\theta_t^{(k)}\r\rangle} - \sum_{t=0}^{T-1}\sum_{k=1}^K\expect{\l\langle\mathbf{f}_t^{(k)},\theta^{(k)}_*\r\rangle},
\end{equation}
where $\l\{\theta^{(k)}_*\r\}_{k=1}^K$ are the state-action probabilities corresponding to the best fixed joint randomized stationary policy within the following stationary constraint set 
\begin{equation}\label{stat-constraint}
\overline{\mathcal{G}}:=\left\{ \theta^{(k)}\in\Theta^{(k)},~k\in\{1,2,\cdots,K\}: \right.
\left.\sum_{k=1}^K\l\langle\expect{\mathbf{g}_{i,t}^{(k)}},\theta^{(k)}\r\rangle \leq 0,~i=1,2,\cdots,m\right\},
\end{equation}
with the assumption that Slater's condition \eqref{slater-2} holds.

\xcolor{
To analyze the proposed algorithm, we need to tackle the following two major challenges:
\begin{itemize}
\item Whether or not the policy decision of the proposed algorithm would yield $\mathcal{O}(\sqrt{T})$ regret and constraint violation on the imaginary system that reaches steady state instantaneously on each slot.
\item Whether the error between the imaginary and true systems can be bounded by $\mathcal{O}(\sqrt{T})$. 
\end{itemize}
}

\xcolor{
In the next section, we answer these questions via a multi-stage analysis piecing together the results of MDPs from Section \ref{sec:prelim-thm} with
multiple ingredients from convex analysis and stochastic queue analysis. We first show the 
$\mathcal{O}(\sqrt{T})$ regret and constraint violation in the imaginary online linear program incorporating a new regret analysis procedure with a stochastic drift analysis for queue processes. 
Then, we show if the benchmark randomized stationary algorithm always starts from its stationary state, then,
the discrepancy of regrets between the imaginary and true systems can be controlled via the slow-update property of the proposed algorithm together with the properties of MDPs developed in Section \ref{sec:prelim-thm}. Finally, for the problem with arbitrary non-stationary starting state, we reformulate it as a perturbation on the aforementioned stationary state problem and 
analyze the perturbation via Farkas' Lemma.
}

\section{Convergence time analysis}\label{sec:convergence-analysis}

\subsection{Stationary state performance: An online linear program}\label{sec:stat-analysis}


Let $\mathbf{Q}(t):=[Q_1(t),~Q_2(t),~\cdots,~Q_m(t)]$ be the virtual queue vector and $L(t) = \frac12\|\mathbf{Q}(t)\|_2^2$. Define the drift
$\Delta(t) := L(t+1) - L(t)$.

\subsubsection{Sample-path analysis}
This section develops a couple of bounds given a sequence of penalty functions $f_{0}^{(k)},f_{1}^{(k)},\cdots,f_{T-1}^{(k)}$ and constraint functions $g_{i,0}^{(k)},g_{i,1}^{(k)},\cdots,g_{i,T-1}^{(k)}$.
The following lemma provides bounds for virtual queue processes:
\begin{lemma}\label{lemma:pre-Q-bound}
For any $i\in\{1,2,\cdots,m\}$ at $T\in\{1,2,\cdots\}$, the following holds under the virtual queue update \eqref{Q-update},
\[
\sum_{t=1}^{T}\sum_{k=1}^K\l\langle \mathbf{g}_{i,t-1}^{(k)}, \theta_{t-1}^{(k)} \r\rangle 
\leq Q_i(T+1) - Q_i(1)   
+  \Psi\sum_{t=1}^T\sum_{k=1}^K\sqrt{\l| \mathcal{A}^{(k)} \r|\l| \mathcal{S}^{(k)} \r|}\l\| \theta^{(k)}_t - \theta^{(k)}_{t-1} \r\|_2,
\]
\xcolor{where $\Psi>0$ is the constant defined in \eqref{function-bounds}.}
\end{lemma}
\begin{proof}
By the queue updating rule \eqref{Q-update}, for any $t\in\mathbb{N}$,
\begin{align*}
&Q_i(t+1)  \\
=& \max\l\{ Q_i(t) + \sum_{k=1}^K\l\langle \mathbf{g}_{i,t-1}^{(k)}, \theta_{t}^{(k)} \r\rangle,0 \r\} \\
\geq& Q_i(t) + \sum_{k=1}^K\l\langle \mathbf{g}_{i,t-1}^{(k)}, \theta_{t}^{(k)} \r\rangle\\
=& Q_i(t) + \sum_{k=1}^K\l\langle \mathbf{g}_{i,t-1}^{(k)}, \theta_{t-1}^{(k)} \r\rangle + \sum_{k=1}^K\l\langle \mathbf{g}_{i,t-1}^{(k)}, \theta_t^{(k)}-\theta_{t-1}^{(k)} \r\rangle\\
\geq&  Q_i(t) + \sum_{k=1}^K\l\langle \mathbf{g}_{i,t-1}^{(k)}, \theta_{t-1}^{(k)} \r\rangle - \sum_{k=1}^K\l\|g_{i,t-1}^{(k)}\r\|_2\l\|\theta_t^{(k)}-\theta_{t-1}^{(k)}\r\|_2, 
\end{align*}
Note that the constraint functions are deterministically bounded,
\[
\l\|g_{i,t-1}^{(k)}\r\|_2^2\leq \l| \mathcal{A}^{(k)} \r|\l| \mathcal{S}^{(k)} \r|\Psi^2.
\]
Substituting this bound into the above queue bound and rearranging the terms finish the proof.
\end{proof}

The next lemma provides a bound for the drift $\Delta(t)$.
\begin{lemma}\label{D-bound}
For any slot $t\geq1$, we have 
\[
\Delta(t)\leq \frac12mK^2\Psi^2+\sum_{i=1}^mQ_i(t)\sum_{k=1}^K\l\langle \mathbf{g}_{i,t-1}^{(k)}, \theta_{t}^{(k)} \r\rangle .
\]
\end{lemma}
\begin{proof}
By definition, we have
\begin{align*}
\Delta(t) =& \frac12\|\mathbf{Q}(t+1)\|_2^2 - \frac12\|\mathbf{Q}(t)\|_2^2\\
\leq&\frac12\sum_{i=1}^m\l( \l(Q_i(t) + \sum_{k=1}^K\l\langle \mathbf{g}_{i,t-1}^{(k)}, \theta_{t}^{(k)} \r\rangle \r)^2 - Q_i(t)^2 \r)\\
=& \sum_{i=1}^mQ_i(t)\sum_{k=1}^K\l\langle \mathbf{g}_{i,t-1}^{(k)}, \theta_{t}^{(k)} \r\rangle 
+ \frac12\sum_{i=1}^m\l(\sum_{k=1}^K\l\langle \mathbf{g}_{i,t-1}^{(k)}, \theta_{t}^{(k)} \r\rangle\r)^2.
\end{align*}
Note that by the queue update \eqref{Q-update}, we have
$$
\l| \sum_{k=1}^K \l\langle \mathbf{g}_{i,t-1}^{(k)}, \theta_{t}^{(k)} \r\rangle\r|
\leq K\l\| \mathbf{g}_{i,t-1}^{(k)} \r\|_\infty\l\| \theta_{t}^{(k)} \r\|_1\leq K\Psi.
$$
Substituting this bound into the drift bound finishes the proof.
\end{proof}

Consider a convex set $\mathcal{X}\subseteq\mathbb{R}^n$.
Recall that for a fixed real number $c>0$, a function $h:\mathcal{X}\rightarrow\mathbb{R}$ is said to be \textit{$c$-strongly convex}, if $h(x) - \frac{c}{2}\|x\|_2^2$ is convex over 
$x\in\mathcal{X}$. It is easy to see that if $q:\mathcal{X}\rightarrow\mathbb{R}$ is convex, $c>0$ and $b\in\mathbb{R}^n$, the function $q(x) + \frac{c}{2}\|x-b\|_2^2$ is $c$-strongly convex. Furthermore,  if the function $h$ is $c$-strongly convex that is minimized at a point $x_{\min}\in\mathcal{X}$, then (see, e.g., Corollary 1 in \cite{YuNeely17SIOPT}):
\begin{equation}\label{strongly-convex}
h(x_{\min})\leq h(y) - \frac{c}{2}\|y-x_{\min}\|_2^2, ~~\forall y\in\mathcal{X}.
\end{equation}
The following lemma is a direct consequence of the above strongly convex result. It also demonstrates the key property of our minimization subproblem \eqref{optimize}.

\begin{lemma}\label{strong-convex-bound}
The following bound holds for any $k\in\{1,2,\cdots,K\}$ and any fixed $\theta_*^{(k)}\in\Theta^{(k)}$:
\begin{multline}\label{DPP-bound}
V  \l\langle \mathbf{f}_{t-1}^{(k)}, \theta_t^{(k)} - \theta_{t-1}^{(k)} \r\rangle + \sum_{i=1}^mQ_i(t)
\l\langle\mathbf{g}_{i,t-1}^{(k)},\theta_t^{(k)}\r\rangle + \alpha\|\theta_t^{(k)}-\theta_{t-1}^{(k)}\|_2^2\\
\leq V\l\langle \mathbf{f}_{t-1}^{(k)}, \theta_*^{(k)} - \theta_{t-1}^{(k)} \r\rangle + \sum_{i=1}^mQ_i(t)
\l\langle\mathbf{g}_{i,t-1}^{(k)},\theta_*^{(k)}\r\rangle + \alpha\|\theta_*^{(k)} - \theta_{t-1}^{(k)}\|_2^2
-\alpha\|\theta_*^{(k)}-\theta_t^{(k)}\|_2^2.
\end{multline}
\end{lemma}
This lemma follows easily from the fact that the proposed algorithm \eqref{optimize} gives $\theta_t^{(k)}\in\Theta^{(k)}$ minimizing the left hand side, which is a strongly convex function, and then, applying \eqref{strongly-convex}, with 
\[
h\l(\theta^{(k)}_*\r) = V  \l\langle \mathbf{f}_{t-1}^{(k)}, \theta^{(k)}_* - \theta_{t-1}^{(k)} \r\rangle + \sum_{i=1}^mQ_i(t)
\l\langle\mathbf{g}_{i,t-1}^{(k)},\theta^{(k)}_*\r\rangle  
+ \alpha\l\|\theta^{(k)}_*-\theta_{t-1}^{(k)}\r\|_2^2
\]

Combining the previous two lemmas gives the following ``drift-plus-penalty'' bound.
\begin{lemma}
For any fixed $\{\theta_*^{(k)}\}_{k=1}^K$ such that $\theta_*^{(k)}\in\Theta^{(k)}$ and $t\in\mathbb{N}$, we have the following bound,
\begin{multline}\label{dpp-bound}
\Delta(t) + V\sum_{k=1}^{K}\l\langle \mathbf{f}_{t-1}^{(k)}, \theta_t^{(k)} - \theta_{t-1}^{(k)} \r\rangle + \alpha\sum_{k=1}^K\|\theta^{(k)}_t-\theta^{(k)}_{t-1}\|_2^2\\
\leq \frac32mK^2\Psi^2 + V\sum_{k=1}^K\l\langle \mathbf{f}_{t-1}^{(k)}, \theta_*^{(k)} - \theta_{t-1}^{(k)} \r\rangle + \sum_{i=1}^mQ_i(t-1)\\
\cdot\sum_{k=1}^K
\l\langle\mathbf{g}_{i,t-1}^{(k)},\theta_*^{(k)}\r\rangle 
+ \alpha\sum_{k=1}^K\|\theta_*^{(k)} - \theta_{t-1}^{(k)}\|_2^2
-\alpha\sum_{k=1}^K\|\theta_*^{(k)}-\theta_t^{(k)}\|_2^2
\end{multline}
\end{lemma}
\begin{proof}
Using Lemma \ref{D-bound} and then Lemma \ref{strong-convex-bound}, we obtain
\begin{align}
&\Delta(t) + V\sum_{k=1}^{K}\l\langle \mathbf{f}_{t-1}^{(k)}, \theta_t^{(k)} - \theta_{t-1}^{(k)} \r\rangle + \alpha\sum_{k=1}^K\|\theta^{(k)}_t-\theta^{(k)}_{t-1}\|_2^2\nonumber\\
\leq& \frac12mK^2\Psi^2+\sum_{i=1}^mQ_i(t)\sum_{k=1}^K\l\langle \mathbf{g}_{i,t-1}^{(k)}, \theta_{t}^{(k)} \r\rangle 
+ V\sum_{k=1}^{K}\l\langle \mathbf{f}_{t-1}^{(k)}, \theta_t^{(k)} - \theta_{t-1}^{(k)} \r\rangle 
+ \alpha\sum_{k=1}^K\|\theta^{(k)}_t-\theta^{(k)}_{t-1}\|_2^2\nonumber\\
\leq& \frac12mK^2\Psi^2+ \sum_{k=1}^{K}\l\langle \mathbf{f}_{t-1}^{(k)}, \theta_*^{(k)} - \theta_{t-1}^{(k)} \r\rangle + \sum_{i=1}^mQ_i(t)
\sum_{k=1}^K\l\langle\mathbf{g}_{i,t-1}^{(k)},\theta_*^{(k)}\r\rangle 
+ \alpha\sum_{k=1}^K\|\theta_*^{(k)} - \theta_{t-1}^{(k)}\|_2^2   \nonumber \\
&-\alpha\sum_{k=1}^K\|\theta_*^{(k)}-\theta_t^{(k)}\|_2^2. \label{inter-2}
\end{align}
Note that by the queue updating rule \eqref{Q-update}, we have for any $t\geq2$,
\[
|Q_i(t)-Q_i(t-1)|\leq \l| \sum_{k=1}^K \l\langle \mathbf{g}_{i,t-2}^{(k)}, \theta_{t-1}^{(k)} \r\rangle\r|
\leq K\l\| \mathbf{g}_{i,t-2}^{(k)} \r\|_\infty\l\| \theta_{t-1}^{(k)} \r\|_1 
\leq K\Psi,
\]
and for $t=1$, $Q_i(t)-Q_i(t-1)=0$ by the initial condition of the algorithm. Also, we have for any $\theta_*^{(k)}\in\Theta^{(k)}$,
\[
\l|\sum_{k=1}^K\l\langle\mathbf{g}_{i,t-1}^{(k)},\theta_*^{(k)}\r\rangle   \r|  \leq K\l\| \mathbf{g}_{i,t-2}^{(k)} \r\|_\infty\l\| \theta_*^{(k)} \r\|_1\leq K\Psi.
\]
Thus, we have 
\[
\sum_{i=1}^mQ_i(t)
\sum_{k=1}^K\l\langle\mathbf{g}_{i,t-1}^{(k)},\theta_*^{(k)}\r\rangle
\leq \sum_{i=1}^mQ_i(t-1)
\sum_{k=1}^K\l\langle\mathbf{g}_{i,t-1}^{(k)},\theta_*^{(k)}\r\rangle  
+ mK^2\Psi^2.
\]
Substituting this bound into \eqref{inter-2} finishes the proof.
\end{proof}

\subsubsection{Objective bound}

\begin{theorem}\label{thm:stationary-regret}
For any $\{\theta_*^{(k)}\}_{k=1}^K$ in the constraint set \eqref{stat-constraint} and any $T\in\{1,2,3,\cdots\}$,
the proposed algorithm has the following stationary state performance bound:
\begin{multline*}
\frac1T\sum_{t=0}^{T-1}\expect{\sum_{k=1}^K\l\langle \mathbf{f}_t^{(k)},\theta_t^{(k)} \r\rangle} \leq \frac1T\sum_{t=0}^{T-1}\expect{\sum_{k=1}^K\l\langle \mathbf{f}_t^{(k)},\theta_*^{(k)} \r\rangle}\\
+\frac{2\alpha K}{TV}+\frac{mK^2\Psi^2}{T} + \frac{V\Psi^2}{2\alpha}\sum_{k=1}^K\l|\mathcal{S}^{(k)}\r|\l|\mathcal{A}^{(k)}\r|+ \frac32\frac{mK^2\Psi^2}{V},
\end{multline*}
In particular, choosing $\alpha=T$ and $V=\sqrt{T}$ gives the $\mathcal{O}(\sqrt{T})$ regret
\begin{multline*}
\frac1T\sum_{t=0}^{T-1}\expect{\sum_{k=1}^K\l\langle \mathbf{f}_t^{(k)},\theta_t^{(k)} \r\rangle} \leq \frac1T\sum_{t=0}^{T-1}\expect{\sum_{k=1}^K\l\langle \mathbf{f}_t^{(k)},\theta_*^{(k)} \r\rangle}\\
+\l(2K + \frac{\Psi^2}{2}\sum_{k=1}^K\l|\mathcal{S}^{(k)}\r|\l|\mathcal{A}^{(k)}\r|+ \frac52mK^2\Psi^2\r)\frac{1}{\sqrt{T}}.
\end{multline*}
\end{theorem}
\begin{proof}
\xcolor{First of all, note that $\{\mathbf{g}_{i,t-1}^{(k)}\}_{k=1}^K$ is i.i.d. and independent of all system history up to $t-1$, and thus independent of $Q_i(t-1),~i=1,2,\cdots,m$.} We have
\begin{equation}\label{neg-drift-2}
\expect{Q_i(t-1)\l\langle\mathbf{g}_{i,t-1}^{(k)},\theta^{(k)}_*\r\rangle}  
= \expect{Q_i(t-1)}\expect{\sum_{k=1}^K\l\langle\mathbf{g}_{i,t-1}^{(k)},\theta^{(k)}_*\r\rangle}\leq 0
\end{equation}
where the last inequality follows from the assumption that $\{\theta^{(k)}_*\}_{k=1}^K$ is in the constraint set \eqref{stat-constraint}. Substituting $\theta^{(k)}_*$ into \eqref{dpp-bound}, taking expectation with respect to both sides and using \eqref{neg-drift-2} give
\begin{multline*}
\expect{\Delta(t)} + V\expect{\sum_{k=1}^{K}\l\langle \mathbf{f}_{t-1}^{(k)}, \theta_t^{(k)} - \theta_{t-1}^{(k)} \r\rangle} + \alpha\expect{\sum_{k=1}^K\|\theta^{(k)}_t-\theta^{(k)}_{t-1}\|_2^2}\\
\leq \frac32mK^2\Psi^2 + V\expect{\sum_{k=1}^K\l\langle \mathbf{f}_{t-1}^{(k)}, \theta^{(k)}_* - \theta_{t-1}^{(k)} \r\rangle} 
+ \alpha\expect{\sum_{k=1}^K\|\theta^{(k)}_* - \theta_{t-1}^{(k)}\|_2^2}
-\alpha\expect{\sum_{k=1}^K\|\theta^{(k)}_*-\theta_t^{(k)}\|_2^2},
\end{multline*}
where the second inequality follows from \eqref{neg-drift-2}.
Note that for any $k$, completing the squares gives
\begin{multline*}
V\l\langle \mathbf{f}_{t-1}^{(k)}, \theta_t^{(k)} - \theta_{t-1}^{(k)} \r\rangle+\alpha\|\theta^{(k)}_t-\theta^{(k)}_{t-1}\|_2^2\\
\geq \l\| \sqrt{\frac\alpha2}\l(\theta_t^{(k)}-\theta_{t-1}^{(k)}\r) + \frac{V}{2\sqrt{\alpha/2}}\mathbf{f}_{t-1}^{(k)} \r\|_2^2 - \frac{V^2\Psi^2\l|\mathcal{S}^{(k)}\r|\l|\mathcal{A}^{(k)}\r|}{2\alpha}.
\end{multline*}
Substituting this inequality into the previous bound and rearranging the terms give
\begin{multline*}
V\expect{\sum_{k=1}^K\l\langle \mathbf{f}_{t-1}^{(k)}, \theta_{t-1}^{(k)} \r\rangle}\leq V\expect{\sum_{k=1}^K\l\langle \mathbf{f}_{t-1}^{(k)}, \theta_*^{(k)} \r\rangle}-\expect{\Delta(t)}  
+  \frac{V^2\sum_{k=1}^K\Psi^2\l|\mathcal{S}^{(k)}\r|\l|\mathcal{A}^{(k)}\r|}{2\alpha}+ \frac32mK^2\Psi^2 \\
+\alpha\expect{\sum_{k=1}^K\|\theta^{(k)}_* - \theta_{t-1}^{(k)}\|_2^2}
-\alpha\expect{\sum_{k=1}^K\|\theta^{(k)}_* - \theta_t^{(k)}\|_2^2}.
\end{multline*}
Taking telescoping sums from 1 to $T$ and dividing both sides by $TV$ gives,
\begin{align*}
\frac1T\sum_{t=1}^T\expect{\sum_{k=1}^K\l\langle \mathbf{f}_{t-1}^{(k)}, \theta_{t-1}^{(k)} \r\rangle}
\leq& \expect{\sum_{k=1}^K\l\langle \mathbf{f}_{t-1}^{(k)}, \theta^{(k)}_* \r\rangle} 
+ \frac{L(0)-L(T+1)}{VT} + \frac{V\sum_{k=1}^K\Psi^2\l|\mathcal{S}^{(k)}\r|\l|\mathcal{A}^{(k)}\r|}{2\alpha} \\
&+ \frac32\frac{mK^2\Psi^2}{V}+\frac{\alpha\expect{\sum_{k=1}^K\|\theta^{(k)}_* - \theta_{T-1}^{(k)}\|_2^2}
-\alpha\expect{\sum_{k=1}^K\|\theta^{(k)}_*-\theta_{T}^{(k)}\|_2^2}}{VT}\\
\leq&  \expect{\sum_{k=1}^K\l\langle \mathbf{f}_{t-1}^{(k)}, \theta^{(k)}_* \r\rangle} + \frac{V\sum_{k=1}^K\Psi^2\l|\mathcal{S}^{(k)}\r|\l|\mathcal{A}^{(k)}\r|}{2\alpha}  
+ \frac32\frac{mK^2\Psi^2}{V}
+\frac{2\alpha K}{VT},
\end{align*}
where we use the fact that $L(0)=0$ and $ \|\theta_*^{(k)} - \theta_{T-1}^{(k)}\|_2^2\leq  \|\theta_*^{(k)} - \theta_{T-1}^{(k)}\|_1\leq2$.
\end{proof}

\subsubsection{A drift lemma and its implications}
From Lemma \ref{lemma:pre-Q-bound}, we know that in order to get the constraint violation bound, we need to look at the size of the virtual queue $Q_i(T+1),~i=1,2,\cdots,m$.
The following drift lemma serves as a cornerstone for our goal.
\begin{lemma}[Lemma 5 of \cite{hao2017onlinestochastic}] \label{lemma:drift-bound}
Let $\{\Omega, \mathcal{F}, P\}$ be a probability space.
Let $\{Z(t), t\geq 1\}$ be a discrete time stochastic process adapted to a filtration $\{\mathcal{F}_{t-1}, t\geq 1\}$ with $Z(1) = 0$ and \lcolor{$\mathcal{F}_{0}=\{\emptyset,\Omega\}$}.  Suppose there exist integer $t_{0}>0$, real constants $\lambda\in \mathbb{R}$, $\delta_{\max} > 0$ and $0< \zeta \leq \delta_{\max}$ such that 
\begin{align} 
\vert Z(t+1) - Z(t) \vert \leq& \delta_{\max}, \\
\mathbb{E}[ Z(t+t_{0}) - Z(t) | \mathcal{F}_{t-1}] \leq & \left\{ \begin{array}{cc} t_{0}\delta_{\max}, &\text{if}~ Z(t) < \lambda \\  
-t_{0}\zeta , &\text{if}~ Z(t) \geq \lambda \end{array}\right.. \label{eq:stochastic-process-drift-condition}
\end{align}
hold for all $t\in \{1,2,\ldots\}$. Then, the following holds:
$$\mathbb{E}[Z(t)] \leq  \lambda + t_{0}\delta_{\max}+ t_{0}\frac{4\delta_{\max}^{2}}{\zeta} \log\big[\frac{8\delta_{\max}^{2}}{\zeta^{2}}\big], \forall t\in\{1,2,\ldots\}.$$
\end{lemma}

Note that a special case of above drift lemma for $t_0=1$ dates back to the seminal paper of Hajek (\cite{hajek1982hitting}) bounding the size of a random process with strongly negative drift. Since then, its power has been demonstrated in various scenarios ranging from steady state queue bound (\cite{eryilmaz2012asymptotically}) to feasibility analysis of stochastic optimization (\cite{wei2016online}). The current generalization to a multi-step drift is first considered in \cite{hao2017onlinestochastic}. 

This lemma is useful in the current context due to the following lemma, whose proof can be found in Appendix \ref{proof:section4}.

\begin{lemma}\label{lemma:queue-drift-bound}
Let $\mathcal{F}_t,~t\geq1$ be the system history functions up to time $t$,  including $f^{(k)}_0,\cdots,f^{(k)}_{t-1}$, $g^{(k)}_{0,i},\cdots,g^{(k)}_{t-1,i}$, $i=1,2,\cdots,m,~k=1,2,\cdots,K$, and 
$\mathcal{F}_0$ is a null set. Let $t_0$ be an arbitrary positive integer, then, we have
\begin{align*}
\big \vert \Vert \mathbf{Q}(t+1)\Vert_2 - \Vert \mathbf{Q}(t)\Vert_2 \big\vert \leq& \sqrt{m}K\Psi, 
\end{align*}
\[
\mathbb{E}[\Vert \mathbf{Q}(t+t_0)\Vert_2 - \Vert \mathbf{Q}(t)\Vert_2 \big|  \mathcal{F}(t-1)] 
\leq  \left\{ \begin{array}{cc} t_0 \sqrt{m}K\Psi, &\text{if}~ \Vert \mathbf{Q}(t)\Vert <  \lambda \\  - t_0 \frac{\eta}{2}, &\text{if}~ \Vert \mathbf{Q}(t)\Vert \geq  \lambda
\end{array}\right.
\]
where 
$\lambda =\frac{8VK\Psi  +  3mK^2\Psi^2 + 4K\alpha + t_0(t_0-1)m\Psi 
 + 2mK\Psi\eta t_0 + \eta^2 t_0^2}{\eta t_0}$.
\end{lemma}

Combining the previous two lemmas gives the virtual queue bound as 
\lcolor{
\begin{multline*}
\expect{\|\mathbf{Q}(t)\|_2}  
\leq \frac{8VK\Psi  +  3mK^2\Psi^2 + 4K\alpha + t_0(t_0-1)m\Psi 
 + 2mK\Psi\eta t_0 + \eta^2 t_0^2}{\eta t_0} +  t_0\sqrt{m}K\Psi\\
 +\frac{4t_{0} mK^2\Psi^2}{\eta} \log\big[\frac{8mK^2\Psi^2}{\eta^2}\big].
\end{multline*}
}
We then choose $t_0=\sqrt{T}$, $V=\sqrt{T}$ and $\alpha=T$, which implies that
\begin{equation}\label{expected-Q-bound}
\expect{\|\mathbf{Q}(t)\|_2}\leq C(m,K,\Psi,\eta)\sqrt{T},
\end{equation}
where 
\lcolor{
$C(m,K,\Psi,\eta) = \frac{8K\Psi}{\eta} + \frac{3mK^2\Psi^2}{\eta^2} + \frac{4K+m\Psi}{\eta} + 2mK\Psi + \eta + \sqrt{m}K\Psi+\frac{4mK^2\Psi^2}{\eta} \log\big[\frac{8mK^2\Psi^2}{\eta^2}\big].$}

\subsubsection{The slow-update condition and constraint violation}
In this section, we prove the slow-update property of the proposed algorithm, which not only implies the the $\mathcal{O}(\sqrt{T})$ constraint violation bound, but also plays a key role in Markov analysis.
\begin{lemma}\label{lemma:slow-update}
The sequence of state-action vectors $\theta_t^{(k)},~t\in\{1,2,\cdots,T\}$ satisfies 
\[
\expect{\|\theta_t^{(k)}-\theta_{t-1}^{(k)}\|_2}\leq \frac{\sqrt{m|\mathcal{A}^{(k)}||\mathcal{S}^{(k)}|}\Psi\expect{\|\mathbf{Q}(t)\|_2}}{2\alpha}
 + \frac{\sqrt{|\mathcal{A}^{(k)}||\mathcal{S}^{(k)}|}\Psi V}{2\alpha}.
\]
In particular,choosing $V=\sqrt{T}$ and $\alpha=T$ gives a slow-update condition
\begin{equation}\label{one-step-update}
\expect{\|\theta_t^{(k)}  -  \theta_{t-1}^{(k)}\|_2}\leq\frac{\sqrt{|\mathcal{A}^{(k)}||\mathcal{S}^{(k)}|}\Psi+C\sqrt{m|\mathcal{A}^{(k)}||\mathcal{S}^{(k)}|}\Psi}{2\sqrt{T}},
\end{equation}
where $C= C(m,K,\Psi,\eta)$ is defined in \eqref{expected-Q-bound}.
\end{lemma}

\begin{proof}[Proof of Lemma \ref{lemma:slow-update}]
First, choosing $\theta = \theta_{t-1}$ in \eqref{DPP-bound} gives
\begin{multline*}
V\l\langle \mathbf{f}_{t-1}^{(k)}, \theta_t^{(k)}-\theta_{t-1}^{(k)} \r\rangle + \sum_{i=1}^mQ_i(t)
\l\langle\mathbf{g}_{i,t-1}^{(k)},\theta_t^{(k)}\r\rangle + \alpha\|\theta_t^{(k)}-\theta_{t-1}^{(k)}\|_2^2\\
\leq  \sum_{i=1}^mQ_i(t)
\langle\mathbf{g}_{i,t-1}^{(k)},\theta_{t-1}^{(k)}\rangle 
-\alpha\|\theta_{t-1}^{(k)}-\theta_t^{(k)}\|_2^2.
\end{multline*}
Rearranging the terms gives
\begin{align*}
2\alpha\|\theta_t^{(k)}-\theta_{t-1}^{(k)}\|_2^2  
\leq& - V\langle \mathbf{f}_{t-1}^{(k)}, \theta_t^{(k)}-\theta_{t-1}^{(k)}\rangle
-\sum_{i=1}^mQ_i(t)\langle\mathbf{g}_{i,t-1}^{(k)},\theta_t^{(k)} -\theta_{t-1}^{(k)}\rangle\\
\leq&V\|\mathbf{f}_{t-1}^{(k)}\|_2\cdot\|\theta_t^{(k)}-\theta_{t-1}^{(k)}\|_2 + \sum_{i=1}^mQ_i(t)
\|\mathbf{g}_{i,t-1}^{(k)}\|_2\cdot\|\theta_t^{(k)}-\theta_{t-1}^{(k)}\|_2\\
\leq&V\|\mathbf{f}_{t-1}\|_2\cdot\|\theta_t^{(k)}-\theta_{t-1}^{(k)}\|_2 
+ \|\mathbf{Q}(t)\|_2
\sqrt{\sum_{i=1}^m\|\mathbf{g}_{i,t-1}^{(k)}\|_2^2}\|\theta_t^{(k)}-\theta_{t-1}^{(k)}\|_2,
\end{align*}
where the second and third inequality follow from Cauchy-Schwarz inequality. Thus, it follows
\[
\l\|\theta_t^{(k)}-\theta_{t-1}^{(k)}\r\|_2
\leq\frac{V\|\mathbf{f}_{t-1}^{(k)}\|_2+\|\mathbf{Q}(t)\|_2\cdot
\sqrt{\sum_{i=1}^m\|\mathbf{g}_{i,t-1}^{(k)}\|_2^2}}{2\alpha}.
\]
Applying the fact that $\|\mathbf{f}_{t-1}^{(k)}\|_2\leq 
\sqrt{|\mathcal{A}^{(k)}||\mathcal{S}^{(k)}|}\Psi$, $\|\mathbf{g}_{i,t-1}^{(k)}\|_2\leq \sqrt{|\mathcal{A}^{(k)}||\mathcal{S}^{(k)}|}\Psi$ and taking expectation from both sides give the first bound in the lemma. The second bound follows directly from the first bound by further substituting \eqref{expected-Q-bound}.
\end{proof}

\begin{theorem}\label{thm:constraint-violation}
The proposed algorithm has the following stationary state constraint violation bound:
\[
\frac1T\sum_{t=0}^{T-1}\expect{\sum_{k=1}^K\l\langle \mathbf{g}_{i,t}^{(k)},\theta_t^{(k)} \r\rangle}  
\leq \frac{1}{\sqrt{T}}\l(C+\sum_{k=1}^K\sqrt{m|\mathcal{A}^{(k)}||\mathcal{S}^{(k)}|}\Psi C
+\sum_{k=1}^K|\mathcal{A}^{(k)}||\mathcal{S}^{(k)}|\Psi^2\r),
\]
where $C=C(m,K,\Psi,\eta)$ is defined in \eqref{expected-Q-bound}.
\end{theorem}
\begin{proof}
Taking expectation from both sides of Lemma \ref{lemma:pre-Q-bound} gives
\[
\sum_{t=1}^{T}\expect{\sum_{k=1}^K\l\langle \mathbf{g}_{i,t-1}^{(k)}, \theta_{t-1}^{(k)} \r\rangle } 
\leq \expect{Q_i(T+1)} +  \Psi\sum_{t=1}^T\sum_{k=1}^K\sqrt{\l| \mathcal{A}^{(k)} \r|\l| \mathcal{S}^{(k)} \r|}  \expect{\l\| \theta^{(k)}_t - \theta^{(k)}_{t-1} \r\|_2}.
\]
Substituting the bounds \eqref{expected-Q-bound} and \eqref{one-step-update} in to the above inequality gives the desired result.
\end{proof}

\subsection{Markov analysis}\label{sec:markov}
So far, we have shown that our algorithm achieves an $\mathcal{O}(\sqrt{T})$ regret and constraint violation simultaneously regarding the stationary online linear program \eqref{stat-regret} with constraint set given by \eqref{stat-constraint} in the imaginary system.
In this section, we show how these stationary state results lead to a tight performance bound on the original true online MDP problem \eqref{main-regret} and \eqref{main-constraint} comparing to any joint randomized stationary algorithm starting from its stationary state.

\subsubsection{Approximate mixing of MDPs}
Let $\mathcal{F}_t,~t\geq1$ be the set of system history functions up to time $t$,  including $f^{(k)}_0,\cdots,f^{(k)}_{t-1}$, $g^{(k)}_{0,i},\cdots,g^{(k)}_{t-1,i}$, $i=1,2,\cdots,m,~k=1,2,\cdots,K$, and $\mathcal{F}_0$ is a null set. Let $d_{\pi_t^{(k)}}$ be the stationary state distribution at $k$-th MDP under the randomized stationary policy $\pi_t^{(k)}$ in the proposed algorithm. 
Let $v_t^{(k)}$ be the true state distribution at time slot $t$ under the proposed algorithm given the function path $\mathcal{F}_{T}$ \xcolor{and starting state $d_0^{(k)}$, i.e. for any $s\in\mathcal{S}^{(k)}$, $v_t^{(k)}(s):=Pr\l(s_t^{(k)}=s | \mathcal{F}_{T}\r)$ and $v_0^{(k)}=d_0^{(k)}$.}

The following lemma provides a key estimate on the distance between stationary distribution and true distribution at each time slot $t$. 
It builds upon the slow-update condition (Lemma \ref{lemma:slow-update}) of the proposed algorithm and uniform mixing bound of general MDPs 
(Lemma \ref{lemma:mixing1}).

\begin{lemma}\label{lemma:distance-dist}
\xcolor{
Consider the proposed algorithm with $V=\sqrt{T}$ and $\alpha=T$. 
For any initial state distribution $\{d_0^{(k)}\}_{k=1}^K$ and any $t\in\{0,1,2,\cdots,T-1\}$, we have}

\[
\expect{\l\| d_{\pi_t^{(k)}} - v_t^{(k)} \r\|_1  }
\leq \left.\tau r\l(\l|\mathcal{A}^{(k)}\r|\l|\mathcal{S}^{(k)}\r|\Psi  + C\sqrt{m} \l|\mathcal{A}^{(k)}\r|\l|\mathcal{S}^{(k)}\r|\Psi \r)\right/2\sqrt{T}
+2e^{-\frac{t}{\tau r}+1},
\]

where $\tau$ and $r$ are mixing parameters defined in Lemma \ref{lemma:mixing1} and $C$ is an absolute constant defined in \eqref{expected-Q-bound}.
\end{lemma}

\begin{proof}[Proof of Lemma \ref{lemma:distance-dist}]
By Lemma \ref{lemma:slow-update} we know that for any $t\in\{1,2,\cdots,T\}$,
\[
\expect{\l\| \theta_t^{(k)}-\theta_{t-1}^{(k)} \r\|_2}\leq \frac{\sqrt{\l|\mathcal{A}^{(k)}\r|\l|\mathcal{S}^{(k)}\r|}\Psi  + C\sqrt{m\l|\mathcal{A}^{(k)}\r|\l|\mathcal{S}^{(k)}\r|}\Psi}{2\sqrt{T}},
\]
Thus,
\[
\expect{\l\| \theta_t^{(k)}-\theta_{t-1}^{(k)} \r\|_1}\leq \frac{\l|\mathcal{A}^{(k)}\r|\l|\mathcal{S}^{(k)}\r|\Psi  + C\sqrt{m}\l|\mathcal{A}^{(k)}\r|\l|\mathcal{S}^{(k)}\r|\Psi}{2\sqrt{T}},
\]
Since for any $s\in\mathcal{S}^{(k)}$, 
$\big|d_{\pi_t^{(k)}}(s)-d_{\pi_{t-1}^{(k)}}(s)\big| = \Big| \sum_{a\in\mathcal{A}^{(k)}}\theta_t^{(k)} (a,s)-\theta_{t-1}^{(k)}(a,s) \Big|
\leq  \sum_{a\in\mathcal{A}^{(k)}}\Big| \theta_t^{(k)} (a,s)-\theta_{t-1}^{(k)}(a,s) \Big|, $
it then follows

\begin{multline}\label{slow-update-dist}
\expect{\l\| d_{\pi_t^{(k)}}-d_{\pi_{t-1}^{(k)}} \r\|_1}\leq \expect{\l\| \theta_t^{(k)}-\theta_{t-1}^{(k)} \r\|_1}
\leq
\frac{\l|\mathcal{A}^{(k)}\r|\l|\mathcal{S}^{(k)}\r|\Psi  + C\sqrt{m}\l|\mathcal{A}^{(k)}\r|\l|\mathcal{S}^{(k)}\r|\Psi}{2\sqrt{T}}.
\end{multline}

Now, we use the above relation to bound $\expect{\l\| d_{\pi_t^{(k)}} - v_t^{(k)} \r\|_1  }$ for any $t\geq r$.
\begin{align}
\expect{\l\| d_{\pi_t^{(k)}} - v_t^{(k)} \r\|_1  }  
\leq& \expect{\l\| d_{\pi_t^{(k)}} - d_{\pi_{t-1}^{(k)}} \r\|_1  }  +  \expect{\l\| d_{\pi_{t-1}^{(k)}} - v_t^{(k)} \r\|_1  }  \nonumber\\
\leq&\frac{\l|\mathcal{A}^{(k)}\r|\l|\mathcal{S}^{(k)}\r|\Psi  + C\sqrt{m}\l|\mathcal{A}^{(k)}\r|\l|\mathcal{S}^{(k)}\r|\Psi}{2\sqrt{T}}
+ \expect{\l\| d_{\pi_{t-1}^{(k)}} - v_t^{(k)} \r\|_1  }   \nonumber\\
=& \frac{\l|\mathcal{A}^{(k)}\r|\l|\mathcal{S}^{(k)}\r|\Psi  + C\sqrt{m}\l|\mathcal{A}^{(k)}\r|\l|\mathcal{S}^{(k)}\r|\Psi}{2\sqrt{T}} 
+ \expect{\l\| \l(d_{\pi_{t-1}^{(k)}} - v_{t-1}^{(k)}\r)\mathbf{P}_{\pi_{t-1}^{(k)}}^{(k)} \r\|_1  }, \label{eq:procedure}
\end{align}
where the second inequality follows from the slow-update condition \eqref{slow-update-dist} and the final equality follows from the fact that given the function path $\mathcal{F}_T$, the following holds
\begin{equation}\label{eq:transfer}
d_{\pi_{t-1}^{(k)}} - v_t^{(k)}  = \l(d_{\pi_{t-1}^{(k)}} - v_{t-1}^{(k)}\r)\mathbf{P}_{\pi^{(k)}_{t-1}}^{(k)}.
\end{equation}
To see this, note that from the proposed algorithm, the policy $\pi^{(k)}_t$ is determined by $\mathcal{F}_{T}$. Thus, by definition of stationary distribution, given $\mathcal{F}_{T}$, we know that $d_{\pi_{t-1}^{(k)}}= d_{\pi_{t-1}^{(k)}}\mathbf{P}_{\pi_{t-1}^{(k)}}^{(k)}$, and it is enough to show that given $\mathcal{F}_{T}$,
\begin{equation*}
v_t^{(k)} = v_{t-1}^{(k)}\mathbf{P}_{\pi_{t-1}^{(k)}}^{(k)}.
\end{equation*}
First of all, the state distribution $v_t^{(k)}$ is determined by $v_{t-1}^{(k)}$, $\pi_{t-1}^{(k)}$ and probability transition from $s_{t-1}$ to $s_t$, which are in turn determined by
$\mathcal{F}_{T}$. Thus, given $\mathcal{F}_{T}$,
for any $s\in\mathcal{S}^{(k)}$, 
\begin{align*}
v_t^{(k)}(s) = \sum_{s'\in\mathcal{S}^{(k)}}Pr(s_t = s | s_{t-1} = s', \mathcal{F}_{T}) v_{t-1}^{(k)}(s'),
\end{align*}
and 
\begin{align*}
Pr(s_{t}=s | s_{t-1}=s', \mathcal{F}_{T}) 
=& \sum_{a\in\mathcal{A}^{(k)}}Pr(s_t=s | a_t=a, s_{t-1}=s', \mathcal{F}_{T})Pr(a_t=a| s_{t-1}=s', \mathcal{F}_{T})\\
=& \sum_{a\in\mathcal{A}^{(k)}}P_a(s',s)Pr(a_t= a| s_{t-1}=s', \mathcal{F}_{T})\\
=& \sum_{a\in\mathcal{A}^{(k)}}P_a(s',s)\pi^{(k)}_{t-1}(a|s') = P_{\pi_{t-1}^{(k)}}(s',s),
\end{align*}
where the second inequality follows from the Assumption \ref{assumption:indep-trans}, the third equality follows from the fact that $\pi^{(k)}_{t-1}$ is determined by $\mathcal{F}_{T}$,
thus, for any $t$,
\begin{equation*}
\pi_{t}^{(k)}(a \big| s') = Pr(a_t=a | s_{t-1}=s', \mathcal{F}_{T}),~\forall a\in\mathcal{A}^{(k)},~s'\in\mathcal{S}^{(k)},
\end{equation*}
and the last equality follows from the definition of transition probability \eqref{transition-matrix}. This gives
\[
v_t^{(k)}(s) = \sum_{s'\in\mathcal{S}^{(k)}}P_{\pi_{t-1}^{(k)}}(s',s) v_{t-1}^{(k)}(s'),
\]
and thus \eqref{eq:transfer} holds.

We can iteratively apply the procedure \eqref{eq:procedure} $r$ times as follows
\begin{align*}
&\expect{\l\| d_{\pi_t^{(k)}} - v_t^{(k)} \r\|_1  }\\
\leq& \frac{\l|\mathcal{A}^{(k)}\r|\l|\mathcal{S}^{(k)}\r|\Psi  + C\sqrt{m}\l|\mathcal{A}^{(k)}\r|\l|\mathcal{S}^{(k)}\r|\Psi}{2\sqrt{T}}
 + \expect{\l\| \l(d_{\pi_{t-1}^{(k)}} - d_{\pi_{t-2}^{(k)}}\r)\mathbf{P}_{\pi_{t-1}^{(k)}}^{(k)} \r\|_1  }
 + \expect{\l\| \l(d_{\pi_{t-2}^{(k)}} - v_{t-1}^{(k)}\r)\mathbf{P}_{\pi_{t-1}^{(k)}}^{(k)} \r\|_1  }\\
 \leq& 2\cdot\frac{\l|\mathcal{A}^{(k)}\r|\l|\mathcal{S}^{(k)}\r|\Psi  + C\sqrt{m}\l|\mathcal{A}^{(k)}\r|\l|\mathcal{S}^{(k)}\r|\Psi}{2\sqrt{T}}
 + \expect{\l\| \l(d_{\pi_{t-2}^{(k)}} - v_{t-1}^{(k)}\r)\mathbf{P}_{\pi_{t-1}^{(k)}}^{(k)} \r\|_1  }\\
 =&2\cdot\frac{\l|\mathcal{A}^{(k)}\r|\l|\mathcal{S}^{(k)}\r|\Psi  + C\sqrt{m}\l|\mathcal{A}^{(k)}\r|\l|\mathcal{S}^{(k)}\r|\Psi}{2\sqrt{T}}
 + \expect{\l\| \l(d_{\pi_{t-2}^{(k)}} - v_{t-2}^{(k)}\r)\mathbf{P}_{\pi_{t-2}^{(k)}}^{(k)}\mathbf{P}_{\pi_{t-1}^{(k)}}^{(k)} \r\|_1  }\\
 \leq&\cdots
 \leq r\cdot\frac{\l|\mathcal{A}^{(k)}\r|\l|\mathcal{S}^{(k)}\r|\Psi  + C\sqrt{m}\l|\mathcal{A}^{(k)}\r|\l|\mathcal{S}^{(k)}\r|\Psi}{2\sqrt{T}}
 + \expect{\l\| \l(d_{\pi_{t-r}^{(k)}} - v_{t-r}^{(k)}\r)\mathbf{P}_{\pi_{t-r}^{(k)}}^{(k)}\cdots\mathbf{P}_{\pi_{t-1}^{(k)}}^{(k)} \r\|_1  },
\end{align*} 
where the second inequality follows from the nonexpansive property in $\ell_1$ norm
of the stochastic matrix $\mathbf{P}_{\pi_{t-1}^{(k)}}^{(k)}$ that
\[
\l\| \l(d_{\pi_{t-1}^{(k)}} - d_{\pi_{t-2}^{(k)}}\r)\mathbf{P}_{\pi_{t-1}^{(k)}}^{(k)} \r\|_1\leq \l\| d_{\pi_{t-1}^{(k)}} - d_{\pi_{t-2}^{(k)}}\r\|_1,
\]
and then using the slow-update condition \eqref{slow-update-dist} again.
By Lemma \ref{lemma:mixing1}, we have
\begin{align*}
\expect{\l\| d_{\pi_t^{(k)}} - v_t^{(k)} \r\|_1  }\leq r\cdot\frac{\l|\mathcal{A}^{(k)}\r|\l|\mathcal{S}^{(k)}\r|\Psi  + C\sqrt{m}\l|\mathcal{A}^{(k)}\r|\l|\mathcal{S}^{(k)}\r|\Psi}{2\sqrt{T}}
+e^{-1/\tau}\expect{\l\| d_{\pi_{t-r}^{(k)}} - v_{t-r}^{(k)} \r\|_1}.
\end{align*}
Iterating this inequality down to $t=0$ gives
\begin{align*}
\expect{\l\| d_{\pi_t^{(k)}} - v_t^{(k)} \r\|_1  }
\leq&
\sum_{j=0}^{\lfloor t/\tau\rfloor} e^{-j/\tau} \cdot r\cdot\frac{\l|\mathcal{A}^{(k)}\r|\l|\mathcal{S}^{(k)}\r|\Psi  + C\sqrt{m}\l|\mathcal{A}^{(k)}\r|\l|\mathcal{S}^{(k)}\r|\Psi}{2\sqrt{T}}\\
&+ \expect{\l\|d_{\pi_0^{(k)}} - v_0^{(k)}\r\|_1}e^{-\lfloor t/r\rfloor/\tau}\\
\leq&
\sum_{j=0}^{\lfloor t/\tau\rfloor} e^{-j/\tau} \cdot r\cdot\frac{\l|\mathcal{A}^{(k)}\r|\l|\mathcal{S}^{(k)}\r|\Psi  + C\sqrt{m}\l|\mathcal{A}^{(k)}\r|\l|\mathcal{S}^{(k)}\r|\Psi}{2\sqrt{T}}
+ 2e^{-\lfloor t/r\rfloor/\tau}\\
\leq& \int_{x=0}^\infty e^{-x/\tau}dx\cdot r\cdot\frac{\l|\mathcal{A}^{(k)}\r|\l|\mathcal{S}^{(k)}\r|\Psi  + C\sqrt{m}\l|\mathcal{A}^{(k)}\r|\l|\mathcal{S}^{(k)}\r|\Psi}{2\sqrt{T}}
+ 2e^{- \frac{t}{r\tau}+1}\\
\leq&\tau r\cdot\frac{\l|\mathcal{A}^{(k)}\r|\l|\mathcal{S}^{(k)}\r|\Psi  + C\sqrt{m}\l|\mathcal{A}^{(k)}\r|\l|\mathcal{S}^{(k)}\r|\Psi}{2\sqrt{T}}
+ 2e^{- \frac{t}{r\tau}+1}
\end{align*}
finishing the proof.
\end{proof}

\subsubsection{Benchmarking against policies starting from stationary state}
Combining the results derived so far, we have the following regret bound regarding any randomized stationary policy $\Pi$ starting from its stationary state distribution $d_\Pi$ such that 
$(d_\Pi,\Pi)$ in the constraint set $\mathcal{G}$ defined in \eqref{main-constraint}.

\begin{theorem}\label{thm:final-1}
Let $\mathscr{P}$ be the sequence of randomized stationary policies resulting from the proposed algorithm with $V=\sqrt{T}$ and $\alpha=T$. Let $d_0$ be the starting state of the proposed algorithm.
For any randomized stationary policy $\Pi$ starting from its stationary state distribution $d_\Pi$ such that 
$(d_\Pi,\Pi)\in\mathcal{G}$, we have
\begin{align*}
&F_T(d_0,\mathscr{P}) - F_T(d_\Pi,\Pi)\leq \mathcal{O}\l(m^{3/2}K^2\sum_{k=1}^K\l|\mathcal{A}^{(k)}\r|\l|\mathcal{S}^{(k)}\r|\cdot\sqrt{T}\r),\\
&G_{i,T}(d_0,\mathscr{P})\leq \mathcal{O}\l(m^{3/2}K^2\sum_{k=1}^K\l|\mathcal{A}^{(k)}\r|\l|\mathcal{S}^{(k)}\r|\cdot\sqrt{T}\r),~i=1,2,\cdots,m.
\end{align*}
\end{theorem}

\begin{proof}[Proof of Theorem \ref{thm:final-1}]
First of all, by Lemma \ref{lemma:prod-chain}, for any randomized stationary policy $\Pi$, there exists some stationary state-action probability vectors $\{\theta_*^{(k)}\}_{k=1}^K$ such that $\theta_*^{(k)}\in\Theta^{(k)}$, 
$$F_T(d_\Pi,\Pi) = \sum_{t=0}^{T-1}\sum_{k=1}^K\l\langle\expect{\mathbf{f}_t},\theta^{(k)}_*\r\rangle,$$ 
and $G_{i,T}(d_\Pi,\Pi) = \sum_{t=0}^{T-1}\sum_{k=1}^K\l\langle\expect{\mathbf{g}_{i,t}},\theta^{(k)}_*\r\rangle$. As a consequence, $(d_\Pi,\Pi)\in\mathcal{G}$ implies $G_{i,T}(d_\Pi,\Pi) = \sum_{t=0}^{T-1}\sum_{k=1}^K\l\langle\expect{\mathbf{g}_{i,t}},\theta^{(k)}_*\r\rangle\leq0,~\forall i\in\{1,2,\cdots,m\}$ and it follows $\{\theta_*^{(k)}\}_{k=1}^K$ is in the imaginary constraint set $\overline{\mathcal{G}}$ defined in \eqref{stat-constraint}. Thus, we are in a good shape applying Theorem \ref{thm:stationary-regret} from imaginary systems.

We then split $F_T(d_0,\mathscr{P}) - F_T(d_\Pi,\Pi)$ into two terms:
\begin{align*}
F_T(d_0,\mathscr{P}) - F_T(d_0,\Pi)
\leq&
\underbrace{\left|\expect{\left.\sum_{t=0}^{T-1}\sum_{k=1}^Kf^{(k)}_t(a_t^{(k)},s_t^{(k)})\right|~d_0,\mathscr{P}} 
- \sum_{t=0}^{T-1}\sum_{k=1}^K\expect{\l\langle\mathbf{f}_t^{(k)},\theta_t^{(k)}\r\rangle}\right|}_{\text{(I)}}\\
&+ \underbrace{ \sum_{t=0}^{T-1}\sum_{k=1}^K\l(\expect{\l\langle\mathbf{f}_t^{(k)},\theta_t^{(k)}\r\rangle} - \l\langle\expect{\mathbf{f}_t},\theta^{(k)}_*\r\rangle\r) }_{\text{(II)}}.
\end{align*}
By Theorem \ref{thm:stationary-regret}, we get 
\begin{equation}\label{final-1}
(\text{II})\leq \l(2K + \frac{\Psi^2}{2}\sum_{k=1}^K\l|\mathcal{S}^{(k)}\r|\l|\mathcal{A}^{(k)}\r|+ \frac52mK^2\Psi^2\r)\sqrt{T}.
\end{equation}

We then bound (I). Consider each time slot $t\in\{0,1,\cdots,T-1\}$.  We have
\begin{align*}
\expect{\l\langle\mathbf{f}_t^{(k)},\theta_t^{(k)}\r\rangle }
=\sum_{s\in\mathcal{S}^{(k)}}\sum_{a\in\mathcal{A}^{(k)}}\expect{ d_{\pi^{(k)}_t}(s) \pi^{(k)}_t(a|s)f_t^{(k)}(a,s)}
\end{align*}

\[
\expect{\left.f_t^{(k)}(a_t^{(k)},s_t^{(k)})\right|~d_0,\mathscr{P}}
=\sum_{s\in\mathcal{S}^{(k)}} \sum_{a\in\mathcal{A}^{(k)}} \expect{v^{(k)}_t(s)\pi^{(k)}_t(a|s)f_t^{(k)}(a,s)},
\]
where the first equality follows from the definition of $\theta_t^{(k)}$ and the second equality follows from the following: Given a specific function path $\mathcal{F}_T$, the policy 
$\pi_t^{(k)}$ and the true state distribution $v_t^{(k)}$ are fixed. Thus, we have,
\[
\expect{\left.f_t^{(k)}(a_t^{(k)},s_t^{(k)})\right|~d_0,\mathscr{P},\mathcal{F}_T} 
= \sum_{s\in\mathcal{S}^{(k)}} \sum_{a\in\mathcal{A}^{(k)}}v^{(k)}_t(s)\pi^{(k)}_t(a|s)f_t^{(k)}(a,s).
\]
Taking the full expectation regarding the function path gives the result.
Thus,
\begin{align*}
&\left|\expect{\left.f_t^{(k)}(a_t^{(k)},s_t^{(k)})\right|~d_0,\mathscr{P}}-\expect{\l\langle\mathbf{f}_t^{(k)},\theta_t^{(k)}\r\rangle}\right|\\
\leq&\left| \sum_{s\in\mathcal{S}^{(k)}}\sum_{a\in\mathcal{A}^{(k)}}\expect{ \l(v_t^{(k)}(s) - d_{\pi_t^{(k)}}(s)\r)\pi^{(k)}_t(a|s) } \right| \Psi\\
\leq& \expect{\left\|  v_t^{(k)} - d_{\pi_t^{(k)}} \right\|_1} \Psi\\
\leq& \frac{\tau r\l(1 + C\sqrt{m}\r)\l|\mathcal{A}^{(k)}\r|\l|\mathcal{S}^{(k)}\r|\Psi^2 }{2\sqrt{T}}
+2e^{-\frac{t}{\tau r}+1}\Psi
\end{align*}
where the last inequality follows from Lemma \ref{lemma:distance-dist}.
Thus, it follows,

\begin{align}
\text{(I)}\leq& \sum_{t=0}^{T-1}\sum_{k=1}^K\l(\frac{\tau r\l(1 + C\sqrt{m}\r)\l|\mathcal{A}^{(k)}\r|\l|\mathcal{S}^{(k)}\r|\Psi^2 }{2\sqrt{T}}
+2e^{-\frac{t}{\tau r}+1}\Psi\r)   \nonumber\\
\leq& \sum_{k=1}^K\l(\tau r\l(1 + C\sqrt{m}\r)\l|\mathcal{A}^{(k)}\r|\l|\mathcal{S}^{(k)}\r|\Psi^2 \r) \sqrt{T}
+ 2\Psi K\int_{t=0}^{T-1}e^{-\frac{x}{\tau r}+1}dx    \nonumber\\
\leq& \tau r\Psi^2\l(1 + C\sqrt{m}\r)\sum_{k=1}^K\l|\mathcal{A}^{(k)}\r|\l|\mathcal{S}^{(k)}\r| \cdot \sqrt{T}
+2e\Psi K\tau r. \label{final-2}
\end{align}

Overall, combining \eqref{final-1},\eqref{final-2} and substituting the constant $C = C(m,K,\Psi,\eta)$ defined in \eqref{expected-Q-bound} 
gives the objective regret bound.

For the constraint violation, we have
\[
G_{i,T}(d_0,\mathscr{P}) = \underbrace{\expect{\left.\sum_{t=0}^{T-1}\sum_{k=1}^Kg_{i,t}^{(k)}(a_t,s_t)\right|~d_0,\mathscr{P}} 
- \sum_{t=1}^T\sum_{k=1}^K\l\langle\expect{\mathbf{g}_{i,t}^{(k)}},\theta_t\r\rangle}_{\text{(IV)}} 
+ \underbrace{\sum_{t=1}^T\sum_{k=1}^K\l\langle\expect{\mathbf{g}_{i,t}^{(k)}},\theta_t\r\rangle}_{\text{(V)}}.
\]
The term (V) can be readily bounded using Theorem \ref{thm:constraint-violation} as 
\[
\sum_{t=0}^{T-1}\expect{\sum_{k=1}^K\l\langle \mathbf{g}_{i,t}^{(k)},\theta_t^{(k)} \r\rangle}   
\leq \l(C+\sum_{k=1}^K\sqrt{m|\mathcal{A}^{(k)}||\mathcal{S}^{(k)}|}\Psi C
+\sum_{k=1}^K|\mathcal{A}^{(k)}||\mathcal{S}^{(k)}|\Psi^2\r)\sqrt{T}.
\]
For the term (IV), we have
\begin{align*}
\expect{\l\langle\mathbf{g}_{i,t}^{(k)},\theta_t^{(k)}\r\rangle }
=\sum_{s\in\mathcal{S}^{(k)}}\sum_{a\in\mathcal{A}^{(k)}}\expect{ d_{\pi^{(k)}_t}(s) \pi^{(k)}_t(a|s)g_{i,t}^{(k)}(a,s)}
\end{align*}

\[
\expect{\left.g_{i,t}^{(k)}(a_t^{(k)},s_t^{(k)})\right|~d_0,\mathscr{P}}
=\sum_{s\in\mathcal{S}^{(k)}} \sum_{a\in\mathcal{A}^{(k)}} \expect{v^{(k)}_t(s)\pi^{(k)}_t(a|s)g_{i,t}^{(k)}(a,s)},
\]
where the first equality follows from the definition of $\theta_t^{(k)}$ and the second equality follows from the following:
Given a specific function path $\mathcal{F}_T$, the policy $\pi_t^{(k)}$ and the true state distribution $v_t^{(k)}$ are fixed. Thus, we have,
\[
\expect{\left.g_t^{(k)}(a_t^{(k)},s_t^{(k)})\right|~d_0,\mathscr{P},\mathcal{F}_T} 
= \sum_{s\in\mathcal{S}^{(k)}} \sum_{a\in\mathcal{A}^{(k)}}v^{(k)}_t(s)\pi^{(k)}_t(a|s)g_t^{(k)}(a,s).
\]
Taking the full expectation regarding the function path gives the result.
Then, repeat the same proof as that of \eqref{final-2} gives
\[
\text{(IV)}\leq \tau r\Psi^2\l(1 + C\sqrt{m}\r)\sum_{k=1}^K\l|\mathcal{A}^{(k)}\r|\l|\mathcal{S}^{(k)}\r| \cdot \sqrt{T}
+2e\Psi K\tau r. 
\]
This finishes the proof of constraint violation.
\end{proof}

\section{A more general regret bound against policies with arbitrary starting state}\label{sec:perturb}
Recall that Theorem \ref{thm:final-1} compares the proposed algorithm with any randomized stationary policy $\Pi$ starting from its stationary state distribution $d_\Pi$, so that $(d_\Pi,\Pi)\in\mathcal{G}$. In this section, we generalize Theorem \ref{thm:final-1} and obtain a bound of the regret against all $(d_0,\Pi)\in\mathcal{G}$ where $d_0$ is an arbitrary starting state distribution (not necessarily the stationary state distribution). The main technical difficulty doing such a generalization is as follows: For any randomized stationary policy $\Pi$ such that $(d_0,\Pi)\in\mathcal{G}$, let  
$\{\theta_*^{(k)}\}_{k=1}^K$ be the stationary state-action probabilities such that $\theta_*^{(k)}\in\Theta^{(k)}$ and $G_{i,T}(d_\Pi,\Pi) = \sum_{t=0}^{T-1}\sum_{k=1}^K\l\langle\expect{\mathbf{g}_{i,t}},\theta^{(k)}_*\r\rangle$. For some finite horizon $T$, there might exist some ``low-cost" starting state distribution $d_0$ such that 
$G_{i,T}(d_0,\Pi) < G_{i,T}(d_\Pi,\Pi)$ for some $i\in\{1,2,\cdots,m\}$. As a consequence, one coud have 
\[
G_{i,T}(d_0,\Pi) \leq 0,~\text{and}~\sum_{t=0}^{T-1}\sum_{k=1}^K\l\langle\expect{\mathbf{g}_{i,t}},\theta^{(k)}_*\r\rangle>0.
\]
This implies although $(d_0,\Pi)$ is feasible for our true system, its stationary state-action probabilities $\{\theta_*^{(k)}\}_{k=1}^K$ can be \textit{infeasible} with respect to the imaginary constraint set \eqref{stat-constraint}, and all our analysis so far fails to cover such randomized stationary policies.

To resolve this issue, we have to ``enlarge'' the imaginary constraint set \eqref{stat-constraint} so as to cover all state-action probabilities $\{\theta_*^{(k)}\}_{k=1}^K$ arising from any randomized stationary policy $\Pi$ such that $(d_0,\Pi)\in\mathcal{G}$. But a perturbation of constraint set would result in a perturbation of objective in the imaginary system also. Our main goal in this section is to bound such a perturbation and show that the perturbation bound leads to the final $\mathcal{O}(\sqrt{T})$ regret bound.

\subsubsection{A relaxed constraint set}
We begin with a supporting lemma on the uniform mixing time bound over all joint randomized stationary policies. The proof is given in Appendix \ref{proof:section5}.
\begin{lemma}\label{lemma:small-bound}
Consider any randomized stationary policy $\Pi$ in \eqref{main-constraint} with arbitrary starting state distribution $d_0\in\mathcal{S}^{(1)}\times\cdots\times\mathcal{S}^{(K)}$. Let $\mathbf{P}_\Pi$ be the corresponding transition matrix on the product state space. Then, the following holds
\begin{equation}\label{joint-mixing}
\l\| (d_0-d_\Pi)\l(\mathbf{P}_{\Pi}\r)^t \r\|_1 \leq 2e^{(r_1-t)/r_1}, \forall t\in\{0,1,2,\cdots\},
\end{equation}
where $r_1$ is fixed positive constant independent of $\Pi$.
\end{lemma}

The following lemma shows a relaxation of $\mathcal{O}(1/T)$ on the imaginary constraint set \eqref{stat-constraint} is enough to cover all the $\{\theta_*^{(k)}\}_{k=1}^K$ discussed at the beginning of this section. The proof is given in Appendix \ref{proof:section5}.

\begin{lemma}\label{lemma:perturbed}
For any $T\in\{1,2,\cdots\}$ and any randomized stationary policies $\Pi$ in \eqref{main-constraint}, with arbitrary starting state distribution $d_0\in\mathcal{S}^{(1)}\times\cdots\times\mathcal{S}^{(K)}$ and stationary state-action probability $\{\theta_*^{(k)}\}_{k=1}^K$,
\begin{align}
&\sum_{t=0}^{T-1}\l| \expect{\sum_{k=1}^K f_{t}^{(k)}(a^{(k)}_t,s^{(k)}_t) \Big| d_0,\Pi }  -  
\sum_{k=1}^K\l\langle \expect{\mathbf{f}^{(k)}_{t}}, \theta_*^{(k)} \r\rangle  \r| \leq  C_1K\Psi  \label{diff-1}\\
& \sum_{t=0}^{T-1}\l| \expect{\sum_{k=1}^K g_{i,t}^{(k)}(a^{(k)}_t,s^{(k)}_t) \Big| d_0,\Pi }  -  
\sum_{k=1}^K\l\langle \expect{\mathbf{g}^{(k)}_{i,t}}, \theta_*^{(k)} \r\rangle  \r| \leq C_1K\Psi \label{diff-2}
\end{align}
where $C_1$ is an absolute constant. In particular, $\{\theta_*^{(k)}\}_{k=1}^K$ is contained in the following 
relaxed constraint set 
\[
\overline{\mathcal{G}}^+:=\l\{ \theta^{(k)}\in\Theta^{(k)},~k=1,2,\cdots,K:~\sum_{k=1}^K\l\langle \expect{\mathbf{g}^{(k)}_{i,t}}, \theta^{(k)} \r\rangle
\r.
\l.\leq\frac{C_1K\Psi}{T}
,i=1,2,\cdots,m  \r\}.
\]
\end{lemma}

\subsubsection{Best stationary performance over the relaxed constraint set}
Recall that the best stationary performance in hindsight over all randomized stationary policies in the constraint set $\overline{\mathcal{G}}$ can be obtained as the minimum achieved by the following linear program.

\begin{align}
\min&~~\frac1T\sum_{t=0}^{T-1}\sum_{k=1}^K\l\langle \expect{\mathbf{f}_t^{(k)}}, \theta^{(k)}  \r\rangle   \label{ori-lp1}\\
s.t.&~~ \sum_{k=1}^K\l\langle \expect{\mathbf{g}_{i,t}^{(k)}}, \theta^{(k)}  \r\rangle\leq 0,~~i=1,2,\cdots,m.  \label{ori-lp2}
\end{align}

On the other hand, if we consider all the randomized stationary policies contained in the original constraint set \eqref{main-constraint}, then, By Lemma \ref{lemma:perturbed}, the relaxed constraint set $\overline{\mathcal{G}}$ contains all such policies and 
the best stationary performance over this relaxed set comes from the minimum achieved by the following perturbed linear program:
\begin{align}
\min&~~\frac1T\sum_{t=0}^{T-1}\sum_{k=1}^K\l\langle \expect{\mathbf{f}_t^{(k)}}, \theta^{(k)}  \r\rangle  \label{relax-lp1}\\
s.t.&~~ \sum_{k=1}^K\l\langle \expect{\mathbf{g}_{i,t}^{(k)}}, \theta^{(k)}  \r\rangle\leq \frac{C_1K\Psi}{T},~~i=1,2,\cdots,m.  \label{relax-lp2}
\end{align}

We aim to show that the minimum achieved by \eqref{relax-lp1}-\eqref{relax-lp2} is not far away from that of \eqref{ori-lp1}-\eqref{ori-lp2}. In general, such a conclusion is not true due to the unboundedness of Lagrange multipliers in constrained optimization. However, since Slater's condition holds in our case, the perturbation can be bounded via the following well-known Farkas' lemma (\cite{bertsekas2009convex}):

\begin{lemma}[Farkas' Lemma]\label{lemma:Farkas}
Consider a convex program with objective $f(x)$ and constraint function $g_i(x),~i=1,2,\cdots,m$:
\begin{align}
\min&~~f(x), \label{cp:1}   \\
s.t.&~~g_i(x)\leq b_i,~~i=1,2,\cdots,m,\\
&~~x\in\mathcal{X},  \label{cp:3}
\end{align}
for some convex set $\mathcal{X}\subseteq\mathbb{R}^n$. Let $x^*$ be one of the solutions to the above convex program. Suppose there exists $\widetilde{x}\in\mathcal{X}$ such that $g_i\l(\widetilde{x}\r)<0,~\forall i\in\{1,2,\cdots,m\}$. Then, there exists a separation hyperplane parametrized by $(1,\mu_1,\mu_2,\cdots,\mu_m)$ such that $\mu_i\geq0$ and
\[
f(x) + \sum_{i=1}^m\mu_ig_i(x)\geq f(x^*) + \sum_{i=1}^m\mu_ib_i,~~\forall x\in\mathcal{X}.
\]
\end{lemma}
The parameter $\mu=(\mu_1,\mu_2,\cdots,\mu_m)$ is usually referred to as a Lagrange multiplier.
From the geometric perspective, Farkas' Lemma states that if Slater's condition holds, then, there exists a non-vertical separation hyperplane supported at $\Big(f(x^*),b_1,\cdots,b_m\Big)$ and contains the set 
$\l\{\Big(f(x),g_1(x),\cdots,g_m(x)\Big),~x\in\mathcal{X}\r\}$ on one side. Thus, in order to bound the perturbation of objective with respect to the perturbation of constraint level, we need to bound the slope of the supporting hyperplane from above, which boils down to controlling the magnitude of the Lagrange multiplier.
This is summarized in the following lemma:

\begin{lemma}[Lemma 1 of \cite{nedic2009approximate}]\label{lemma:bound-lagrange}
Consider the convex program \eqref{cp:1}-\eqref{cp:3}, and define the Lagrange dual function 
$$q(\mu)=\inf_{x\in\mathcal{X}}\l\{ f(x) + \sum_{i=1}^m\mu_i(g_i(x)-b_i) \r\}.$$ 
Suppose there exists $\widetilde{x}\in\mathcal{X}$ such that \xcolor{ $g_i\l(\widetilde{x}\r)-b_i\leq-\eta,~\forall i\in\{1,2,\cdots,m\}$} for some positive constant $\eta>0$. Then, the level set 
$\mathcal{V}_{\bar\mu}=\l\{ \mu_1,\mu_2,\cdots,\mu_m\geq0,~ q(\mu)\geq q(\bar\mu)\r\}$ is bounded for any nonnegative $\bar\mu$. Furthermore, we have
$$\max_{\mu\in\mathcal{V}_{\bar\mu}}  \|\mu\|_2\leq\frac{1}{\min_{1\leq i \leq m}\l\{-g_i(\widetilde{x})+b_i\r\}}\l(  f(\widetilde{x})-q(\bar\mu)   \r).$$
\end{lemma}

The technical importance of these two lemmas in the current context is contained in the following corollary.
\begin{corollary}\label{coro:comp}
Let $\l\{\theta^{(k)}_*\r\}_{k=1}^K$ and $\l\{\overline\theta^{(k)}_*\r\}_{k=1}^K$ be solutions to \eqref{ori-lp1}-\eqref{ori-lp2} and \eqref{relax-lp1}-\eqref{relax-lp2}, respectively. Then, the following holds
\[
\frac1T\sum_{t=0}^{T-1}\sum_{k=1}^K\l\langle \expect{\mathbf{f}_t^{(k)}}, \overline\theta_*^{(k)}  \r\rangle
\geq \frac1T\sum_{t=0}^{T-1}\sum_{k=1}^K\l\langle \expect{\mathbf{f}^{(k)}}, \theta_*^{(k)}  \r\rangle - \frac{C_1K^2\sqrt{m}\Psi^2}{\eta T}
\]
where $\eta$ is the constant defined in Assumption \ref{assumption:slater}.
\end{corollary}
\begin{proof}[Proof of Corollary \ref{coro:comp}]
Take 
\begin{align*}
&f\l( \theta^{(1)},\cdots,\theta^{(K)} \r) = \frac1T\sum_{t=0}^{T-1}\sum_{k=1}^K\l\langle \expect{\mathbf{f}^{(k)}}, \theta^{(k)}  \r\rangle,  \\
&g_i\l( \theta^{(1)},\cdots,\theta^{(K)} \r) = \sum_{k=1}^K\l\langle \expect{\mathbf{g}_{i,t}^{(k)}}, \theta^{(k)}  \r\rangle,\\
&\mathcal{X} = \Theta^{(1)}\times\Theta^{(2)}\times\cdots\times\Theta^{(K)},
\end{align*}
and $b_i=0$
in Farkas' Lemma and we have the following display
\[
\frac1T\sum_{t=0}^{T-1}\sum_{k=1}^K\l\langle \expect{\mathbf{f}^{(k)}}, \theta^{(k)}  \r\rangle
+ \sum_{i=1}^m\mu_i\sum_{k=1}^K\l\langle \expect{\mathbf{g}_{i,t}^{(k)}}, \theta^{(k)}  \r\rangle 
\geq \frac1T\sum_{t=0}^{T-1}\sum_{k=1}^K\l\langle \expect{\mathbf{f}^{(k)}}, \theta_*^{(k)}  \r\rangle,
\]
for any $\l( \theta^{(1)},\cdots,\theta^{(K)} \r)\in\mathcal{X}$ and some $\mu_1,\mu_2,\cdots,\mu_m\geq0$. In particular, substituting $\l( \overline\theta^{(1)}_*,\cdots,\overline\theta^{(K)}_* \r)$ into the above display gives
\begin{align}
\frac1T\sum_{t=0}^{T-1}\sum_{k=1}^K\l\langle \expect{\mathbf{f}^{(k)}}, \overline\theta^{(k)}_*  \r\rangle  
&\geq \frac1T\sum_{t=0}^{T-1}\sum_{k=1}^K\l\langle \expect{\mathbf{f}^{(k)}}, \theta_*^{(k)}  \r\rangle
-  \sum_{i=1}^m\mu_i\sum_{k=1}^K\l\langle \expect{\mathbf{g}_{i,t}^{(k)}}, \overline\theta_*^{(k)}  \r\rangle   \nonumber\\
&\geq \frac1T\sum_{t=0}^{T-1}\sum_{k=1}^K\l\langle \expect{\mathbf{f}^{(k)}}, \theta_*^{(k)}  \r\rangle
-  \frac{C_1K\Psi}{T}\sum_{i=1}^m\mu_i,  \label{inter-bound}
\end{align}
where the final inequality follows from the fact that $\l( \overline\theta^{(1)}_*,\cdots,\overline\theta^{(K)}_* \r)$ satisfies the relaxed constraint 
$\sum_{k=1}^K\l\langle \expect{\mathbf{g}_{i,t}^{(k)}}, \overline\theta_*^{(k)}  \r\rangle\leq  \frac{C_1K\Psi}{T}$ and $\mu_i\geq0,~\forall i\in\{1,2,\cdots,m\}$. Now we need to bound the magnitude of Lagrange multiplier $\l(\mu_1,\cdots,\mu_m\r)$. Note that in our scenario, 
\[
\Big|f\l( \theta^{(1)},\cdots,\theta^{(K)} \r) \Big|=\l| \frac1T\sum_{t=0}^{T-1}\sum_{k=1}^K\l\langle \expect{\mathbf{f}^{(k)}}, \theta^{(k)}  \r\rangle  \r|\leq \Psi K,
\]
and the Lagrange multiplier $\mu$ is the solution to the maximization problem 
$$\max_{\mu_i\geq0,i\in\{1,2,\cdots,m\}}q(\mu),$$ 
where $q(\mu)$ is the dual function defined in Lemma \ref{lemma:bound-lagrange}.
thus, it must be in any super level set $\mathcal{V}_{\bar\mu}=\l\{ \mu_1,\mu_2,\cdots,\mu_m\geq0,~ q(\mu)\geq q(\bar\mu)\r\}$.
In particular, taking $\bar\mu = 0$ in Lemma \ref{lemma:bound-lagrange} and using Slater's condition \eqref{slater-2}, we have
there exists $\widetilde{\theta}^{(1)},\cdots,\widetilde{\theta}^{(K)}$ such that
\begin{multline*}
\sum_{i=1}^m\mu_i \leq \sqrt{m} \|\mu\|_2\leq 
\frac{\sqrt{m}}{\eta}\l(f\l(\widetilde{\theta}^{(1)},\cdots,\widetilde{\theta}^{(K)}\r) \r. 
\l.-\inf_{\l(\theta^{(1)},\cdots,\theta^{(K)}\r)\in\mathcal{X}}f\l( \theta^{(1)},\cdots,\theta^{(K)} \r)\r)
\leq \frac{2\sqrt{m}\Psi K}{\eta},
\end{multline*}
where the final inequality follows from the deterministic bound of $|f(\theta^{(1)},\cdots,\theta^{(K)})|$ by $\Psi K$.
Substituting this bound into \eqref{inter-bound} gives the desired result.
\end{proof}

As a simple consequence of the above corollary, we have our final bound on the regret and constraint violation regarding any $(d_0,\Pi)\in\mathcal{G}$.
\begin{theorem}\label{thm:final-regret}
Let $\mathscr{P}$ be the sequence of randomized stationary policies resulting from the proposed algorithm with $V=\sqrt{T}$ and $\alpha=T$. Let $d_0$ be the starting state of the proposed algorithm.
For any randomized stationary policy $\Pi$ starting from the state $d_0$ such that 
$(d_0,\Pi)\in\mathcal{G}$, we have
\begin{align*}
&F_T(d_0,\mathscr{P}) - F_T(d_0,\Pi)\leq \mathcal{O}\l(m^{3/2}K^2\sum_{k=1}^K\l|\mathcal{A}^{(k)}\r|\l|\mathcal{S}^{(k)}\r|\cdot\sqrt{T}\r),\\
&G_{i,T}(d_0,\mathscr{P})\leq \mathcal{O}\l(m^{3/2}K^2\sum_{k=1}^K\l|\mathcal{A}^{(k)}\r|\l|\mathcal{S}^{(k)}\r|\cdot\sqrt{T}\r),~i=1,2,\cdots,m.
\end{align*}
\end{theorem}

\begin{proof}
Let $\Pi_*$ be the randomized stationary policy corresponding to the solution $\{\theta_*^{(k)}\}_{k=1}^K$ to \eqref{ori-lp1}-\eqref{ori-lp2} and let $\Pi$ be any randomized stationary policy such that $(d_0,\Pi)\in\mathcal{G}$. Since $G_{i,T}(d_{\Pi_*},\Pi_*) =  \sum_{t=0}^{T-1}\sum_{k=1}^K\l\langle\expect{\mathbf{g}_{i,t}},\theta^{(k)}_*\r\rangle\leq0$, it follows 
$(d_{\Pi_*},\Pi_*)\in\mathcal{G}$.
By Theorem \ref{thm:final-1}, we know that
\[
F_T(d_0,\mathscr{P}) - F_T(d_{\Pi_*},\Pi_*)\leq \mathcal{O}\l(m^{3/2}K^2\sum_{k=1}^K\l|\mathcal{A}^{(k)}\r|\l|\mathcal{S}^{(k)}\r|\cdot\sqrt{T}\r),
\]
and $G_{i,T}(d_0,\mathscr{P})$ satisfies the bound in the statement.
It is then enough to bound $F_T(d_{\Pi_*},\Pi_*) - F_T(d_{0},\Pi)$. We split it in to two terms:
\[
F_T(d_{\Pi_*},\Pi_*) - F_T(d_{0},\Pi)\leq \underbrace{F_T(d_{\Pi_*},\Pi_*) - F_T(d_{\Pi},\Pi)}_{\text{(I)}}   
+ \underbrace{F_T(d_{\Pi},\Pi) - F_T(d_{0},\Pi)}_{\text{(II)}}.
\]
By \eqref{diff-1} in Lemma \ref{lemma:perturbed}, the term (II) is bounded by $C_1K\Psi$. It remains to bound the first term. Since $(d_0,\Pi)\in\mathcal{G}$, by Lemma \ref{lemma:perturbed}, the corresponding state-action probabilities $\{\theta^{(k)}\}_{k=1}^K$ of $\Pi$ satisfies $\sum_{k=1}^K\l\langle\expect{\mathbf{g}_{i,t}},\theta^{(k)}\r\rangle\leq C_1K\Psi/T$ and $\{\theta^{(k)}\}_{k=1}^K$ is feasible for \eqref{relax-lp1}-\eqref{relax-lp2}. Since $\{\overline{\theta}^{(k)}_*\}_{k=1}^K$ is the solution to \eqref{relax-lp1}-\eqref{relax-lp2}, we must have
\[ F_T(d_{\Pi},\Pi) = \sum_{t=0}^{T-1}\sum_{k=1}^K\l\langle \expect{\mathbf{f}_t^{(k)}}, \theta^{(k)}  \r\rangle 
\geq \sum_{t=0}^{T-1}\sum_{k=1}^K\l\langle \expect{\mathbf{f}_t^{(k)}}, \overline\theta_*^{(k)}  \r\rangle \]
On the other hand, by Corollary \ref{coro:comp},
\begin{align*}
\sum_{t=0}^{T-1}\sum_{k=1}^K\l\langle \expect{\mathbf{f}_t^{(k)}}, \overline\theta_*^{(k)}  \r\rangle
\geq& \sum_{t=0}^{T-1}\sum_{k=1}^K\l\langle \expect{\mathbf{f}^{(k)}}, \theta_*^{(k)}  \r\rangle - \frac{C_1K^2\sqrt{m}\Psi^2}{\eta}
= F_T(d_{\Pi_*},\Pi_*) -  \frac{C_1K^2\sqrt{m}\Psi^2}{\eta}.
\end{align*}
Combining the above two displays gives $\text{(I)}\leq \frac{C_1K^2\sqrt{m}\Psi^2}{\eta}$ and the proof is finished.
\end{proof}

\section{Additional lemmas and proofs}

\subsection{Missing proofs in Section \ref{sec:prelim-thm}}\label{proof:section2}
We prove Lemma \ref{lemma:mixing1} and \ref{lemma:prod-chain} in this section.

\begin{proof}[Proof of Lemma \ref{lemma:mixing1}]
For simplicity of notations, we drop the dependencies on $k$ throughout this proof.
We first show that for any $r\geq \widehat{r}$, where  $\widehat{r}$ is specified in Assumption \ref{assumption-1},
$
\mathbf{P}_{\pi_1}\mathbf{P}_{\pi_2}\cdots \mathbf{P}_{\pi_r}
$
is a strictly positive stochastic matrix.

Since the MDP is finite state with a finite action set, the set of all pure policies (Definition \ref{def:pp}) is finite. Let $\mathbf{P}_1,~\mathbf{P}_2,\cdots,~\mathbf{P}_N$ be probability transition matrices corresponding to these pure policies. Consider any sequence of randomized stationary policies $\pi_1,\cdots,\pi_r$. Then, it follows 
their transition matrices can be expressed as convex combinations of pure policies, i.e.
\[
\mathbf{P}_{\pi_1} = \sum_{i=1}^N\alpha^{(1)}_i\mathbf{P}_i,~~
\mathbf{P}_{\pi_2} = \sum_{i=1}^N\alpha^{(2)}_i\mathbf{P}_i,
~~\cdots,
\mathbf{P}_{\pi_r} = \sum_{i=1}^N\alpha^{(r)}_i\mathbf{P}_i,
\]
where $\sum_{i=1}^N\alpha^{(j)}_i = 1,~\forall j\in\{1,2,\cdots,r\}$ and $\alpha^{(j)}_i\geq0$.
Thus, we have the following display
\begin{align}
\mathbf{P}_{\pi_1}\mathbf{P}_{\pi_2}\cdots \mathbf{P}_{\pi_r}
=&\l( \sum_{i=1}^N\alpha^{(1)}_i\mathbf{P}_i\r)
\l( \sum_{i=1}^N\alpha^{(2)}_i\mathbf{P}_i\r)\cdots\l( \sum_{i=1}^N\alpha^{(r)}_i\mathbf{P}_i\r) \nonumber\\
=&\sum_{(i_1,\cdots,i_r)\in\mathcal{G}_r}
\alpha_{i_1}^{(1)}\cdots \alpha_{i_r}^{(r)}
\cdot     \mathbf{P}_{i_1}\mathbf{P}_{i_2}\cdots\mathbf{P}_{i_r},  \label{convex-comb}
\end{align}
where $\mathcal{G}_r$ ranges over all $N^r$ configurations.

Since $\l(\sum_{i=1}^N\alpha^{(1)}_i\r)\cdots\l(\sum_{i=1}^N\alpha^{(r)}_i\r) = 1$, it follows \eqref{convex-comb} is a convex combination of all possible sequences
$ \mathbf{P}_{i_1}\mathbf{P}_{i_2}\cdots\mathbf{P}_{i_r}$. By assumption \ref{assumption-1}, we have  $\mathbf{P}_{i_1}\mathbf{P}_{i_2}\cdots\mathbf{P}_{i_r}$ is strictly positive for any $(i_1,\cdots,i_r)\in\mathcal{G}_r$, and there exists a universal lower bound $\delta>0$ of all entries of $\mathbf{P}_{i_1}\mathbf{P}_{i_2}\cdots \mathbf{P}_{i_r}$ ranging over all configurations in $(i_1,\cdots,i_r)\in\mathcal{G}_r$.
This implies $\mathbf{P}_{\pi_1}\mathbf{P}_{\pi_2}\cdots \mathbf{P}_{\pi_r}$ is also strictly positive with the same lower bound $\delta>0$ for any sequences of randomized stationary policies $\pi_1,\cdots,\pi_r$.

Now, we proceed to prove the mixing bound. Choose $r=\widehat{r}$ and  we can
decompose any $\mathbf{P}_{\pi_1}\mathbf{P}_{\pi_2}\cdots \mathbf{P}_{\pi_r}$ as follows:
$$
\mathbf{P}_{\pi_1}\cdots \mathbf{P}_{\pi_r} = \delta\mathbf{\Pi} + (1-\delta)\mathbf{Q},
$$
where $\mathbf{\Pi}$ has each entry equal to $1/\l|\mathcal{S}\r|$ (recall that $\l|\mathcal{S}\r|$ is the number of states which equals the size of the matrix) and $\mathbf{Q}$ depends on $\pi_1,\cdots,\pi_r$. Then, $\mathbf{Q}$ is also a stochastic matrix (nonnegative and row sum up to 1) because both $\mathbf{P}_{\pi_1}\cdots \mathbf{P}_{\pi_r}$ and 
 $\mathbf{\Pi}$ are stochastic matrices. Thus, for any two distribution vectors $d_1$ and $d_2$, we have
\[
 \l(d_1-d_2\r)\mathbf{P}_{\pi_1}\cdots \mathbf{P}_{\pi_r} 
 =  \delta\l(d_1-d_2\r)\mathbf{\Pi} + (1-\delta)\l(d_1-d_2\r)\mathbf{Q}
 =(1-\delta)\l(d_1-d_2\r)\mathbf{Q},
\]
 where we use the fact that for distribution vectors
 $$\l(d_1-d_2\r)\mathbf{\Pi} = \frac{1}{\l| \mathcal{S} \r|}\mathbf{1} - \frac{1}{\l| \mathcal{S}\r|}\mathbf{1} = 0.$$
Since $\mathbf{Q}$ is a stochastic matrix, it is non-expansive on $\ell_1$-norm, namely, for any vector $x$, $\|x\mathbf{Q}\|_1\leq \|x\|_1$. To see this, simply compute
\begin{equation}\label{non-expansive}
\|x\mathbf{Q}\|_1=\sum_{j=1}^{\l| \mathcal{S} \r|}\l| \sum_{i=1}^{\l| \mathcal{S} \r|}x_iQ_{ij} \r|
\leq\sum_{j=1}^{\l| \mathcal{S}\r|}  \sum_{i=1}^{\l| \mathcal{S}\r|}   \l|x_iQ_{ij} \r| 
=\sum_{j=1}^{\l| \mathcal{S}\r|}  \sum_{i=1}^{\l| \mathcal{S}\r|}   \l|x_i \r|Q_{ij} = \sum_{i=1}^{\l| \mathcal{S}\r|}   \l|x_i \r| = \|x\|_1.
\end{equation}
Overall, we obtain,
\[
 \l\|\l(d_1-d_2\r)\mathbf{P}_{\pi_1}\cdots \mathbf{P}_{\pi_r} \r\|_1=  (1-\delta)\l\|\l(d_1-d_2\r)\mathbf{Q}\r\|_1
 \leq (1- \delta)\l\|d_1-d_2\r\|_1.
\]
We can then take $\tau = -\frac{1}{\log\l(1 - \delta\r)}$ to finish the proof. 
\end{proof}

\begin{proof}[Proof of Lemma \ref{lemma:prod-chain}]
Since the probability transition matrix of any randomized stationary policy is a convex combination of those of pure policies,
it is enough to show that the product MDP is irreducible and aperiodic under any joint pure policy.
For simplicity, let $\mathbf{s}_t = \l(s^{(1)},\cdots,s^{(K)}\r)$ and $\mathbf{a}_t = \l(a^{(1)},\cdots,a^{(K)}\r)$. Consider any joint pure policy $\Pi$ which select a fixed joint action 
$\mathbf{a}\in\mathcal{A}^{(1)}\times\cdots\times\mathcal{A}^{(K)}$ given a joint state $\mathbf{s}\in\mathcal{S}^{(1)}\times\cdots\times\mathcal{S}^{(K)}$, with probability 1.
By Assumption \ref{assumption:indep-trans}, we have
\begin{align}
&Pr\l(  s^{(1)}_{t+1},\cdots,s^{(K)}_{t+1}\l| s^{(1)}_{t},\cdots,s^{(K)}_{t},a^{(1)}_t,\cdots, a^{(K)}_t  \r.\r) \nonumber\\
=&Pr\l(  s^{(1)}_{t+1}\l| s^{(1)}_{t},\cdots,s^{(K)}_{t},a^{(1)}_t,\cdots, a^{(K)}_t, s^{(2)}_{t+1},\cdots,s^{(K)}_{t+1}  \r.\r) \nonumber\\
 &\cdot Pr\l(  s^{(2)}_{t+1},\cdots,s^{(K)}_{t+1}\l| s^{(1)}_{t},\cdots,s^{(K)}_{t},a^{(1)}_t,\cdots, a^{(K)}_t  \r.\r)  \nonumber\\
=&Pr\l(  s^{(1)}_{t+1}\l| s^{(1)}_{t},a^{(1)}_t  \r.\r)Pr\l(  s^{(2)}_{t+1},\cdots,s^{(K)}_{t+1}\l| s^{(1)}_{t},\cdots,s^{(K)}_{t},a^{(1)}_t,\cdots, a^{(K)}_t  \r.\r) \nonumber\\
=&\cdots=  \prod_{k=1}^{K-1}Pr\l(  s^{(k)}_{t+1}\l| s^{(k)}_{t},a^{(k)}_t  \r.\r) 
\cdot Pr\l(  s^{(K)}_{t+1}\l| s^{(1)}_{t},\cdots,s^{(K)}_{t},a^{(1)}_t,\cdots, a^{(K)}_t  \r.\r)\nonumber\\
=& \prod_{k=1}^{K}Pr\l(  s^{(k)}_{t+1}\l| s^{(k)}_{t},a^{(k)}_t  \r.\r), \label{iter-expectation}
\end{align}
where the second equality follows from the independence relation in Assumption \ref{assumption:indep-trans}. Thus, we obtain the equality,
\[
Pr(\mathbf{s}_{t+1}=\mathbf{s}' \big| \mathbf{s}_t=\mathbf{s},\mathbf{a}_t=\mathbf{a}) 
= \prod_{k=1}^{K}Pr\l(  s^{(k)}_{t+1} = \tilde{s}^{(k)}\l| s^{(k)}_{t} = s^{(k)},a^{(k)}_t = a^{(k)}  \r.\r),
\]
Then, the one step transition probability between any two states
$\mathbf{s},\tilde{\mathbf{s}}\in\mathcal{S}^{(1)}\times\cdots\times\mathcal{S}^{(K)}$
can be computed as
\begin{align*}
Pr(\mathbf{s}_{t+1}=\tilde{\mathbf{s}} \big| \mathbf{s}_t = \mathbf{s}) 
=& \sum_{\mathbf{a}}Pr(\mathbf{s}_{t+1}=\tilde{\mathbf{s}} \big| \mathbf{s}_t=\mathbf{s},\mathbf{a}_t=\mathbf{a})
\cdot Pr(\mathbf{a}_t=\mathbf{a} \big|  \mathbf{s}_t=\mathbf{s})\\
=& \sum_{\mathbf{a}}\prod_{k=1}^{K}Pr\l(  s^{(k)}_{t+1} = \tilde{s}^{(k)}\l| s^{(k)}_{t} = s^{(k)},a^{(k)}_t = a^{(k)}  \r.\r)
\cdot Pr(\mathbf{a}_t=\mathbf{a} \big|  \mathbf{s}_t=\mathbf{s})\\
=&\prod_{k=1}^{K} P_{a^{(k)}(\mathbf{s})}\l(s^{(k)},\tilde{s}^{(k)}\r),
\end{align*}
where we can remove the summation on $\mathbf a$ due to the fact that $\mathbf{a}_t$ is a pure policy. The notation $a^{(k)}(\mathbf{s})$ denotes a fixed mapping from product state space $\mathcal{S}^{(1)}\times\cdots\times\mathcal{S}^{(K)}$ to an individual action space 
$\mathcal{A}^{(k)}$ resulting from the pure policy, and
$P_{a^{(k)}(\mathbf{s})}\l(s^{(k)},\tilde{s}^{(k)}\r)$ is the Markov transition probability from state $s^{(k)}$ to $\tilde{s}^{(k)}$ under the action $a^{(k)}(\mathbf{s})$. One can then further compute the $r$ ($r\geq2$) step transition probability from between any two states
$\mathbf{s},\tilde{\mathbf{s}}\in\mathcal{S}^{(1)}\times\cdots\times\mathcal{S}^{(K)}$ as
\begin{align}
Pr(\mathbf{s}_{t+r}=\tilde{\mathbf{s}} \big| \mathbf{s}_t = \mathbf{s}) 
=&\sum_{\mathbf{s}_{t+r-1}}\cdots\sum_{\mathbf{s}_{t+1}}\prod_{k=1}^{K} P_{a^{(k)}(\mathbf{s})}\l(s^{(k)},s_{t+1}^{(k)}\r)\cdot
\prod_{k=1}^{K} P_{a^{(k)}(\mathbf{s}_{t+1})}\l(s_{t+1}^{(k)},s_{t+2}^{(k)}\r) \nonumber\\
&\cdots\prod_{k=1}^{K} P_{a^{(k)}(\mathbf{s}_{t+r-1})}\l(s_{t+r-1}^{(k)},\tilde{s}^{(k)}\r)     \nonumber\\
=&\sum_{\mathbf{s}_{t+r-1}}\cdots\sum_{\mathbf{s}_{t+1}}\prod_{k=1}^{K} P_{a^{(k)}(\mathbf{s})}\l(s^{(k)},s_{t+1}^{(k)}\r) 
\cdot P_{a^{(k)}(\mathbf{s}_{t+1})}\l(s_{t+1}^{(k)},s_{t+2}^{(k)}\r)   \nonumber\\
&\cdots P_{a^{(k)}(\mathbf{s}_{t+r-1})}\l(s_{t+r-1}^{(k)},\tilde{s}^{(k)}\r). \label{lump-sum}
\end{align}

For any $k\in\{1,2,\cdots,K\}$, the term 
$$P_{a^{(k)}(\mathbf{s})}\l(s^{(k)},s_{t+1}^{(k)}\r)\cdot
P_{a^{(k)}(\mathbf{s}_{t+1})}\l(s_{t+1}^{(k)},s_{t+2}^{(k)}\r)\cdots P_{a^{(k)}(\mathbf{s}_{t+r-1})}\l(s_{t+r-1}^{(k)},\tilde{s}^{(k)}\r)$$
denotes the probability of moving from $s^{(k)}$ to $\tilde{s}^{(k)}$ along a certain path under a certain sequence of fixed decisions 
$a^{(k)}(\mathbf{s}),~a^{(k)}(\mathbf{s}_{t+1})$, $\cdots,~a^{(k)}(\mathbf{s}_{t+r-1})$. 
Let 
$$\mathbf{s}^{(k)} = \l( s_{t+1}^{(k)},s_{t+2}^{(k)},\cdots,s_{t+r-1}^{(k)} \r)\in\mathcal{S}^{(k)}\times\cdots\times\mathcal{S}^{(k)},~k\in\{1,2,\cdots,K\}$$ be the state path of k-th MDP.
One can then change the order of summation in \eqref{lump-sum} and sum over state paths of each MDP as follows:

\[
\eqref{lump-sum} = \sum_{\mathbf{s}^{(K)}}\cdots\sum_{\mathbf{s}^{(1)}} 
\prod_{k=1}^{K} P_{a^{(k)}(\mathbf{s})}\l(s^{(k)},s_{t+1}^{(k)}\r)\cdot
P_{a^{(k)}(\mathbf{s}_{t+1})}\l(s_{t+1}^{(k)},s_{t+2}^{(k)}\r) 
\cdots P_{a^{(k)}(\mathbf{s}_{t+r-1})}\l(s_{t+r-1}^{(k)},\tilde{s}^{(k)}\r)
\]
We would like to exchange the order of the product and the sums so that we can take the path sum over each individual MDP respectively. However, the problem is that the transition probabilities are coupled through the actions. 
The idea to proceed is to first apply a ``hard'' decoupling by taking the infimum of transition probabilities of each MDP over all pure policies, and use Assumption \ref{assumption-1}, to bound the transition probability from below uniformly. We have
\begin{align*}
\eqref{lump-sum}\geq 
&\inf_{\mathbf{s}^{(1)}}\sum_{\mathbf{s}^{(K)}}\cdots\sum_{\mathbf{s}^{(2)}}\prod_{k=2}^{K}P_{a^{(k)}(\mathbf{s})}\l(s^{(k)},s_{t+1}^{(k)}\r)\cdots P_{a^{(k)}(\mathbf{s}_{t+r-1})}\l(s_{t+r-1}^{(k)},\tilde{s}^{(k)}\r)\\
&\cdot\inf_{\mathbf{s}^{(j)},~j\neq1}\sum_{\mathbf{s}^{(1)}} P_{a^{(1)}(\mathbf{s})}\l(s^{(1)},s_{t+1}^{(1)}\r)\cdots P_{a^{(1)}(\mathbf{s}_{t+r-1})}\l(s_{t+r-1}^{(1)},\tilde{s}^{(1)}\r)\\
\geq&\inf_{\mathbf{s}^{(1)}}\sum_{\mathbf{s}^{(K)}}\cdots\sum_{\mathbf{s}^{(2)}}\prod_{k=2}^{K}P_{a^{(k)}(\mathbf{s})}\l(s^{(k)},s_{t+1}^{(k)}\r)\cdots P_{a^{(k)}(\mathbf{s}_{t+r-1})}\l(s_{t+r-1}^{(k)},\tilde{s}^{(k)}\r)\\
&\cdot\inf_{\pi_1^{(1)},\cdots,\pi_r^{(1)}}\sum_{\mathbf{s}^{(1)}} P_{\pi_1^{(1)}}\l(s^{(1)},s_{t+1}^{(1)}\r)\cdots P_{\pi_r^{(1)}}\l(s_{t+r-1}^{(1)},\tilde{s}^{(1)}\r),
\end{align*}
where $\pi_1^{(1)},\cdots,\pi_r^{(1)}$ range over all pure policies, and the second inequality follows from the fact that \textit{fix any path of other MDPs}  (i.e. $\mathbf{s}^{(j)},~j\neq1$), the term 
$$\sum_{\mathbf{s}^{(1)}} P_{a^{(1)}(\mathbf{s})}\l(s^{(1)},s_{t+1}^{(1)}\r)\cdots P_{a^{(1)}(\mathbf{s}_{t+r-1})}\l(s_{t+r-1}^{(k)},\tilde{s}^{(1)}\r)$$
is the probability of reaching $\tilde{s}^{(1)}$ from $s^{(1)}$ in $r$ steps using a sequence of actions
$a^{(1)}(\mathbf{s}^{(1)}),\cdots,a^{(1)}(\mathbf{s}^{(1)}_{t+r-1})$, where each action is a deterministic function of the previous state at the 1-st MDP only. Thus, it dominates the infimum over all sequences of pure policies $\pi_1^{(1)},\cdots,\pi_r^{(1)}$ on this MDP. Similarly, we can decouple the rest of the sums and obtain the follow display:
\begin{align*}
\eqref{lump-sum}\geq&
\prod_{k=1}^K\inf_{\pi_1^{(k)},\cdots,\pi_r^{(k)}}\sum_{\mathbf{s}^{(k)}} P_{\pi_1^{(k)}}\l(s^{(k)},s_{t+1}^{(k)}\r)\cdots P_{\pi_r^{(k)}}\l(s_{t+r-1}^{(k)},\tilde{s}^{(k)}\r)\\
=&\prod_{k=1}^K\inf_{\pi_1^{(k)},\cdots,\pi_r^{(k)}}P_{\pi_1^{(k)},\cdots,\pi_r^{(k)}}\l(s^{(k)}, \tilde{s}^{(k)}\r),
\end{align*}
where $P_{\pi_1^{(k)},\cdots,\pi_r^{(k)}}\l(s^{(k)}, \tilde{s}^{(k)}\r)$ denotes the $\l(s^{(k)}, \tilde{s}^{(k)}\r)$-th entry of the product matrix 
$\mathbf{P}_{\pi_1^{(k)}}^{(k)}\cdots \mathbf{P}_{\pi_r^{(k)}}^{(k)}$.
Now, by Assumption \ref{assumption-1}, there exists a large enough integer $\widehat{r}$ such that $\mathbf{P}_{\pi_1^{(k)}}^{(k)}\cdots \mathbf{P}_{\pi_r^{(k)}}^{(k)}$ is a strictly positive matrix for any sequence of $r\geq \widehat{r}$ randomized stationary policy. As a consequence, the above probability is strictly positive and \eqref{lump-sum} is also strictly positive.

This implies, if we choose $\tilde{\mathbf{s}} = \mathbf{s}$, then, starting from any arbitrary product state 
$\mathbf{s}\in\mathcal{S}^{(1)}\times\cdots\times\mathcal{S}^{(K)}$, there is a positive probability of returning to this state after $r$ steps for all $r\geq\widehat{r}$, which gives the aperiodicity. Similarly, there is a positive probability of reaching any other composite state after $r$ steps for all $r\geq\widehat{r}$, which gives the irreducibility. This implies the product state MDP is irreducible and aperiodic under any joint pure policy, and thus, any joint randomized stationary policy.

For the second part of the claim, we consider any randomized stationary policy $\Pi$ and the corresponding joint transition probability matrix $\mathbf{P}_\Pi$, there exists a stationary state-action probability vector $\Phi(\mathbf{a},\mathbf{s}),~\mathbf{a}\in\mathcal{A}^{(1)}\times\cdots\times\mathcal{A}^{(K)},~\mathbf{s}\in\mathcal{S}^{(1)}\times\cdots\times\mathcal{S}^{(K)}$, such that
\begin{equation}\label{inter-1}
\sum_{\mathbf{a}}\Phi(\mathbf{a},\tilde{\mathbf{s}}) = \sum_{\mathbf{s}}\sum_{\mathbf{a}}\Phi(\mathbf{a},\mathbf{s})P_\mathbf{a}(\mathbf{s},\tilde{\mathbf{s}}),
~\forall \tilde{\mathbf{s}}\in \mathcal{S}^{(1)}\times\cdots\times\mathcal{S}^{(K)}.
\end{equation}
Then, the state-action probability of the k-th MDP is $\theta^{(k)}(a^{(k)},\tilde{s}^{(k)}) =\sum_{\tilde{s}^{(j)},a^{(j)},~j\neq k}\Phi(\mathbf{a},\tilde{\mathbf{s}})$. Thus,
\begin{align*}
\sum_{a^{(k)}}\theta^{(k)}(a^{(k)},\tilde{s}^{(k)})  
=&\sum_{\tilde{s}^{(j)},~j\neq k}\sum_{\mathbf{a}}\Phi(\mathbf{a},\tilde{\mathbf{s}})
= \sum_{\mathbf{s}}\sum_{\mathbf{a}}\Phi(\mathbf{a},\mathbf{s})\sum_{\tilde{s}^{(j)},~j\neq k}P_\mathbf{a}(\mathbf{s},\tilde{\mathbf{s}})\\
=& \sum_{\mathbf{s}}\sum_{\mathbf{a}}\Phi(\mathbf{a},\mathbf{s})\cdot Pr\l(\tilde{s}^{(k)}|\mathbf{a},\mathbf{s}\r)
= \sum_{\mathbf{s}}\sum_{\mathbf{a}}\Phi(\mathbf{a},\mathbf{s})\cdot Pr\l(\tilde{s}^{(k)}| a^{(k)}, s^{(k)}\r)\\
=&\sum_{a^{(k)}}\sum_{s^{(k)}}\theta^{(k)}(a^{(k)},\tilde{s}^{(k)})\cdot Pr\l(\tilde{s}^{(k)}| a^{(k)}, s^{(k)}\r)\\
=&\sum_{a^{(k)}}\sum_{s^{(k)}}\theta^{(k)}(a^{(k)},\tilde{s}^{(k)})\cdot P_{a^{(k)}}\l( s^{(k)},\tilde{s}^{(k)}\r)
\end{align*}
where the third from the last inequality follows from Assumption \ref{assumption:indep-trans}. This finishes the proof.
\end{proof}

\subsection{Missing proofs in Section \ref{sec:stat-analysis}}\label{proof:section4}
\begin{proof}[Proof of Lemma \ref{lemma:queue-drift-bound}]
Consider the state-action probabilities $\{\tilde\theta^{(k)}\}_{k=1}^K$ which achieves the Slater's condition in \eqref{slater-2}. First of all, note that $Q_i(t)\in\mathcal{F}_{t-1},~\forall t\geq1$. 
Then, using the assumption that $\{\mathbf{g}_{i,t-1}^{(k)}\}_{k=1}^K$ is i.i.d. and independent of all system information up to $t-1$, we have 
\begin{equation}\label{neg-drift}
\expect{Q_i(t-1)\sum_{k=1}^K\l\langle \mathbf{g}_{i,t-1}^{(k)}, \tilde\theta \r\rangle \Big|~\mathcal{F}_{t-1}} 
=\expect{\sum_{k=1}^K\l\langle \mathbf{g}_{i,t-1}^{(k)}, \tilde\theta \r\rangle}Q_i(t-1)\leq -\eta Q_i(t-1).
\end{equation}
Now, by the drift-plus-penalty bound \eqref{dpp-bound}, with $\theta^{(k)}=\tilde\theta^{(k)}$,
\begin{align*}
\Delta(t)\leq& - V\sum_{k=1}^{K}\l\langle \mathbf{f}_{t-1}^{(k)}, \theta_t^{(k)} - \theta_{t-1}^{(k)} \r\rangle - \alpha\sum_{k=1}^K\|\theta^{(k)}_t-\theta^{(k)}_{t-1}\|_2^2  
+ \frac32mK^2\Psi^2
+ V\sum_{k=1}^K\l\langle \mathbf{f}_{t-1}^{(k)}, \tilde\theta^{(k)} - \theta_{t-1}^{(k)} \r\rangle \\
 &+ \sum_{i=1}^mQ_i(t-1)\sum_{k=1}^K
\l\langle\mathbf{g}_{i,t-1}^{(k)}, \tilde\theta^{(k)}\r\rangle + \alpha\sum_{k=1}^K\|\tilde\theta^{(k)} - \theta_{t-1}^{(k)}\|_2^2 
-\alpha\sum_{k=1}^K\|  \tilde\theta^{(k)}-\theta_t^{(k)}\|_2^2\\
\leq& 4VK\Psi  +  \frac32mK^2\Psi^2  +  \sum_{i=1}^mQ_i(t-1)\sum_{k=1}^K
\l\langle\mathbf{g}_{i,t-1}^{(k)}, \tilde\theta^{(k)}\r\rangle   \\
&+ \alpha\sum_{k=1}^K\|\tilde\theta^{(k)} - \theta_{t-1}^{(k)}\|_2^2
-\alpha\sum_{k=1}^K\|  \tilde\theta^{(k)}-\theta_t^{(k)}\|_2^2
\end{align*}
where the second inequality follows from Holder's inequality that 
$$\l|\l\langle \mathbf{f}_{t-1}^{(k)}, \theta_t^{(k)} - \theta_{t-1}^{(k)} \r\rangle\r|\leq \|\mathbf{f}_{t-1}^{(k)}\|_{\infty}\l\| \theta_t^{(k)} - \theta_{t-1}^{(k)}\r\|_1\leq 2\Psi.$$
Summing up the drift from $t$ to $t+t_0-1$ and taking a conditional expectation $\expect{\cdot|\mathcal{F}_{t-1}}$ give
\begin{align*}
&\expect{\|\mathbf{Q}(t+t_0)\|_2^2-\|\mathbf{Q}(t)\|_2^2 \Big| \mathcal{F}_{t-1}}  \\
\leq& 8VK\Psi  +  3mK^2\Psi^2 
 + 2\sum_{i=1}^m\expect{\sum_{\tau=t}^{t+t_0-1}Q_i(\tau-1)\sum_{k=1}^K
\l\langle\mathbf{g}_{i,\tau-1}^{(k)}, \tilde\theta^{(k)}\r\rangle \Big| \mathcal{F}_{t-1}}\\
&+2\alpha\expect{\sum_{k=1}^K\l(\|\tilde\theta^{(k)} - \theta_{t-1}^{(k)}\|_2^2  -  \|  \tilde\theta^{(k)}-\theta_{t+t_0}^{(k)}\|_2^2\r) \Big|\mathcal{F}_{t-1}}\\
\leq& 8VK\Psi  +  3mK^2\Psi^2 + 4K\alpha  
+ 2\sum_{i=1}^m\expect{\sum_{\tau=t}^{t+t_0-1}Q_i(\tau-1)\sum_{k=1}^K 
\l\langle\mathbf{g}_{i,\tau-1}^{(k)}, \tilde\theta^{(k)}\r\rangle \Big| \mathcal{F}_{t-1}}.
\end{align*}
Using the tower property of conditional expectations (further taking conditional expectations 
$\expect{\cdot\Big|\mathcal{F}_{t+t_0-1}\cdots\Big|\mathcal{F}_t}$ inside the conditional expectation) and the bound \eqref{neg-drift}, we have 
\begin{align*}
&\expect{\sum_{\tau=t}^{t+t_0-1}Q_i(\tau-1)\sum_{k=1}^K
\l\langle\mathbf{g}_{i,\tau-1}^{(k)}, \tilde\theta^{(k)}\r\rangle \Big| \mathcal{F}_{t-1}} \\
&\leq-\eta\expect{\sum_{\tau=t}^{t+t_0-1}Q_i(\tau-1)\Big| \mathcal{F}_{t-1}}\\
&\leq -\eta t_0 Q_i(t-1) + \frac{t_0(t_0-1)}{2}\Psi\leq -\eta t_0 Q_i(t) + \frac{t_0(t_0-1)}{2}\Psi + \eta t_0K\Psi,
\end{align*}
where the last inequality follows from the queue updating rule \eqref{Q-update} that 
$$|Q_i(t-1)-Q_i(t)|\leq \l| \sum_{k=1}^K\l\langle \mathbf{g}_{i,t-2}^{(k)},\theta_{t-1}^{(k)} \r\rangle \r|
 \leq K\|\mathbf{g}_{i,t-2}^{(k)}\|_\infty \|\theta_{t-1}^{(k)}\|_1 \leq K\Psi.$$
Thus, we have 
\begin{multline*}
\expect{\|\mathbf{Q}(t+t_0)\|_2^2-\|\mathbf{Q}(t)\|_2^2 \Big| \mathcal{F}_{t-1}} \\
\leq  8VK\Psi  +  3mK^2\Psi^2 + 4K\alpha + t_0(t_0-1)m\Psi 
 + 2mK\Psi\eta t_0 - 2 \eta t_0 \sum_{i=1}^mQ_i(t)\\
 \leq8VK\Psi  +  3mK^2\Psi^2 + 4K\alpha + t_0(t_0-1)m\Psi 
 + 2mK\Psi\eta t_0 - 2 \eta t_0 \|Q_i(t)\|_2.
\end{multline*}
Suppose $\|Q_i(t)\|_2\geq\frac{8VK\Psi  +  3mK^2\Psi^2 + 4K\alpha + t_0(t_0-1)m\Psi 
 + 2mK\Psi\eta t_0 + \eta^2 t_0^2}{\eta t_0}$, then, it follows,
 \[
 \expect{\|\mathbf{Q}(t+t_0)\|_2^2-\|\mathbf{Q}(t)\|_2^2 \Big| \mathcal{F}_{t-1}} \leq -  \eta t_0 \|Q_i(t)\|_2  ,
 \]
which implies 
\[
 \expect{\|\mathbf{Q}(t+t_0)\|_2^2 \Big| \mathcal{F}_{t-1}} \leq \l(\|Q_i(t)\|_2-\frac{\eta t_0}{2}\r)^2
\]
Since $\|Q_i(t)\|_2\geq \frac{\eta t_0}{2}$, taking square root from both sides using Jensen' inequality gives
\[
 \expect{\|\mathbf{Q}(t+t_0)\|_2 \Big| \mathcal{F}_{t-1}}\leq \|Q_i(t)\|_2 - \frac{\eta t_0}{2}.
\]
On the other hand, we always have
\begin{multline*}
\Big|\|\mathbf{Q}(t+1)\|_2-\|\mathbf{Q}(t)\|_2\Big|  
=\l|\sqrt{\sum_{i=1}^m\max\l\{Q_i(t)+\sum_{k=1}^{K}\l\langle\mathbf{g}_{i,t-1}^{(k)},\theta_{t}^{(k)} \r\rangle,0\r\}^2}
- \sqrt{\sum_{i=1}^mQ_i(t)^2}\r|\\
\leq \l(\sum_{i=1}^m\l(\sum_{k=1}^{K}\l\langle\mathbf{g}_{i,t-1}^{(k)},\theta_{t}^{(k)} \r\rangle\r)^2\r)^{1/2}\leq\sqrt{m}K\Psi.
\end{multline*}
Overall, we finish the proof.
\end{proof}

\subsection{Missing proofs in Section \ref{sec:perturb}}\label{proof:section5}
\begin{proof}[Proof of Lemma \ref{lemma:small-bound}]
Consider any joint randomized stationary policy $\Pi$ and a starting state probability $d_0$ on the product state space $\mathcal{S}^{(1)}\times\mathcal{S}^{(2)}\times\cdots\times\mathcal{S}^{(K)}$. Let $\mathbf{P}_\Pi$ be the corresponding transition matrix on the product state space. Let $d_t$ be the state distribution at time $t$ under $\Pi$ and $d_\Pi$ be the stationary state distribution. 
By Lemma \ref{lemma:prod-chain}, we know that this product state MDP is irreducible and aperiodic (ergodic) under any randomized stationary policy. In particular, it is ergodic under any pure policy. Since there are only finitely many pure policies, let $\mathbf{P}_{\Pi_1},\cdots,\mathbf{P}_{\Pi_N}$ be probability transition matrices corresponding to these pure policies. By Proposition 1.7 of \cite{LevinPeresWilmer2006}
, for any $\Pi_i,~i\in\{1,2,\cdots,N\}$, there exists integer $\tau_i>0$ such that 
$\l(\mathbf{P}_{\Pi_i}\r)^t$ is strictly positive for any $t\geq\tau_i$. Let 
$$\tau_1=\max_i\tau_i,$$ 
then, it follows $\l(\mathbf{P}_{\Pi_i}\r)^{\tau_1}$ is strictly positive uniformly for all $\Pi_i$'s.
Let $\delta>0$ be the least entry of $\l(\mathbf{P}_{\Pi_i}\r)^{\tau_1}$ over all $\Pi_i$'s.
Following from the fact that the probability transition matrix $\mathbf{P}_{\Pi}$ is a convex combination of those of pure policies, i.e. $\mathbf{P}_{\Pi}=\sum_{i=1}^N\alpha_i \mathbf{P}_{\Pi_i},~\alpha_i\geq0,~\sum_{i=1}^N\alpha_i=1$, we have $\l(\mathbf{P}_{\Pi}\r)^{\tau_1}$ is also strictly positive. To see this, note that
\[
\l(\mathbf{P}_{\Pi}\r)^{\tau_1} = \l(\sum_{i=1}^N\alpha_i \mathbf{P}_{\Pi_i}\r)^{\tau_1}\geq \sum_{i=1}^N\alpha_i^{\tau_1} \l(\mathbf{P}_{\Pi_i}\r)^{\tau_1}>0,
\]
where the inequality is taken to be entry-wise. Furthermore, the least entry of $\l(\mathbf{P}_{\Pi}\r)^{\tau_1}$ is lower bounded by $\delta/N^{\tau_1-1}$ uniformly over all joint randomized stationary policies $\Pi$, which follows from the fact that the least entry of $\frac{1}{N}\l(\mathbf{P}_{\Pi}\r)^{\tau_1}$ is bounded as
\[
\frac{1}{N}\sum_{i=1}^N\alpha_i^{\tau_1}\delta \geq \l(\frac{1}{N}\sum_{i=1}^N\alpha_i\r)^{\tau_1}\delta= \frac{\delta}{N^{\tau_1}}.
\]
The rest is a standard bookkeeping argument following from the Markov chain mixing time theory (Theorem 4.9 of \cite{LevinPeresWilmer2006}). Let $\mathbf{D}_\Pi$ be a matrix of the same size as $\mathbf{P}_{\Pi}$ and each row equal to the stationary distribution $d_\Pi$. Let 
$\varepsilon = \delta/N^{\tau_1-1}$. We claim that for any 
integer $n>0$, and any $\Pi$,
\begin{equation}\label{small-fact-1}
\mathbf{P}_{\Pi}^{\tau_1n} = (1-(1-\varepsilon)^{n})\mathbf{D}_\Pi + (1-\varepsilon)^{n}\mathbf{Q}^n,
\end{equation}
for some stochastic matrix $\mathbf{Q}$. We use induction to prove this claim. First of all, for $n=1$, from the fact that $\l(\mathbf{P}_{\Pi}\r)^{\tau_1}$ is a positive matrix and the least entry is uniformly lower bounded by $\varepsilon$ over all policies $\Pi$, we can write $\l(\mathbf{P}_{\Pi}\r)^{\tau_1}$ as
\[
 \l(\mathbf{P}_{\Pi}\r)^{\tau_1} = \varepsilon\mathbf{D}_\Pi + (1-\varepsilon)\mathbf{Q},
\]
for some stochastic matrix $\mathbf{Q}$, where we use the fact that $\varepsilon\in(0,1]$. Suppose \eqref{small-fact-1} holds for $n=1,2,\cdots,\ell$, we show that it also holds for $n=\ell+1$. Using the fact that $\mathbf{D}_\Pi\mathbf{P}_\Pi = \mathbf{D}_\Pi$ and $\mathbf{Q}\mathbf{D}_\Pi = \mathbf{D}_\Pi$ for any stochastic matrix $\mathbf{Q}$,  we can write out $\mathbf{P}_{\Pi}^{\tau_1(\ell+1)}$:
\begin{align*}
\mathbf{P}_{\Pi}^{\tau_1(\ell+1)} =& \mathbf{P}_{\Pi}^{\tau_1\ell}\mathbf{P}_{\Pi}^{\tau_1} = \l( \l(1 - (1-\varepsilon)^\ell\r)\mathbf{D}_{\Pi} + (1-\varepsilon)^\ell Q^{\ell} \r)\mathbf{P}_{\Pi}^{\tau_1}\\
=&  \l(1 - (1-\varepsilon)^\ell\r)\mathbf{D}_{\Pi}\mathbf{P}_{\Pi}^{\tau_1} + (1-\varepsilon)^\ell \mathbf{Q}^{\ell}\mathbf{P}_{\Pi}^{\tau_1}\\
=&  \l(1 - (1-\varepsilon)^\ell\r)\mathbf{D}_{\Pi} + (1-\varepsilon)^\ell \mathbf{Q}^{\ell}(\varepsilon\mathbf{D}_\Pi + (1-\varepsilon)\mathbf{Q})\\
=&  \l(1 - (1-\varepsilon)^\ell\r)\mathbf{D}_{\Pi} + (1-\varepsilon)^\ell \mathbf{Q}^{\ell}((1-(1-\varepsilon))\mathbf{D}_\Pi + (1-\varepsilon)\mathbf{Q})\\
= & (1-(1-\varepsilon)^{\ell+1})\mathbf{D}_\Pi + (1-\varepsilon)^{\ell+1}\mathbf{Q}^{\ell+1}.
\end{align*}
Thus, \eqref{small-fact-1} holds. For any integer $t>0$, we write $t=\tau_1n+j$ for some integer $j\in[0,\tau_1)$ and $n\geq0$. Then,
\begin{equation*}
\l(\mathbf{P}_{\Pi}\r)^{t}-\mathbf{D}_\Pi  = \l(\mathbf{P}_{\Pi}\r)^{t}-\mathbf{D}_\Pi = (1-\varepsilon)^n\l(Q^n\mathbf{P}_{\Pi}^j-\mathbf{D}_\Pi\r).
\end{equation*} 
Let $\mathbf{P}_{\Pi}^{t}(i,\cdot)$ be the $i$-th row of $\mathbf{P}_{\Pi}^{t}$, then, we obtain
\[
\max_{i}\|  \mathbf{P}_{\Pi}^{t}(i,\cdot)  -  d_\Pi  \|_1\leq 2(1-\varepsilon)^n, 
\]
where we use the fact that the $\ell_1$-norm of the row difference is bounded by 2. Finally, for any starting state distribution $d_0$, we have
\begin{multline*}
\l\|d_0\mathbf{P}_{\Pi}^t - d_{\Pi}  \r\|_1 =  \l\| \sum_{i} d_0(i)\l( \mathbf{P}_{\Pi}^t(i,\cdot) - d_\Pi \r) \r\|_1\\
=  \sum_{i} d_0(i)\l\| \mathbf{P}_{\Pi}^t(i,\cdot) - d_\Pi \r\|_1\leq \max_{i}\|  \mathbf{P}_{\Pi}^{t}(i,\cdot)  -  d_\Pi  \|_1\leq 2(1-\varepsilon)^n.
\end{multline*}
Take $r_1 = \log\frac{1}{1-\varepsilon}$
finishes the proof.
\end{proof}

\begin{proof}[Proof of Lemma \ref{lemma:perturbed}]
Let $v_t\in\mathcal{S}^{(1)}\times\cdots\times\mathcal{S}^{(K)}$ be the joint state distribution at time $t$ under policy $\Pi$.
Using the fact that $\Pi$ is a fixed policy independent of  $\mathbf{g}^{(k)}_{i,t}$ and Assumption \ref{assumption:indep-trans} that the probability transition is also independent of function path given any state and action,
 the function $\mathbf{g}^{(k)}_{i,t}$ and state-action pair 
$(a^{(k)}_t,s^{(k)}_t)$ are mutually independent.
Thus, for any $t\in\{0,1,2,\cdots,T-1\}$
\begin{multline*}
 \expect{\sum_{k=1}^K g_{i,t}^{(k)}(a^{(k)}_t,s^{(k)}_t) \Big| d_0,\Pi }
 =\sum_{\mathbf{s}\in\mathcal{S}^{(1)}\times\cdots\times\mathcal{S}^{(K)}}
 \sum_{\mathbf{a}\in\mathcal{A}^{(1)}\times\cdots\times\mathcal{A}^{(K)}}v_t(\mathbf{s})\Pi(\mathbf{a}|\mathbf{s})\sum_{k=1}^K\expect{g_{i,t}^{(k)}(a^{(k)},s^{(k)})},
\end{multline*}
where $\mathbf{s}=[s^{(1)},\cdots,s^{(K)}]$ and $\mathbf{a}=[a^{(1)},\cdots,a^{(K)}]$ and the latter expectation is taken with respect to $\mathbf{g}_{i,t}^{(k)}$ (i.e. the random variable $w_{t}$).
On the other hand, by Lemma \ref{lemma:prod-chain}, we know that for any randomized stationary policy $\Pi$, the corresponding stationary state-action probability can be expressed as 
$\{\theta_*^{(k)}\}_{k=1}^K$ with $\theta_*^{(k)}\in\Theta^{(k)}$. Thus,
\begin{multline*}
\sum_{k=1}^K\l\langle \expect{\mathbf{g}^{(k)}_{i,t}}, \theta^{(k)} \r\rangle 
= \sum_{\mathbf{s}\in\mathcal{S}^{(1)}\times\cdots\times\mathcal{S}^{(K)}}\sum_{\mathbf{a}\in\mathcal{A}^{(1)}\times\cdots\times\mathcal{A}^{(K)}}d_\Pi(\mathbf{s})\Pi(\mathbf{a}|\mathbf{s})\sum_{k=1}^K\expect{g_{i,t}^{(k)}(a^{(k)},s^{(k)})}.
\end{multline*}
Hence,
we can control the difference:
\begin{align*}
&   \sum_{t=0}^{T-1}\l| \expect{\sum_{k=1}^K g_{i,t}^{(k)}(a^{(k)}_t,s^{(k)}_t) \Big| d_0,\Pi }  -  
\sum_{k=1}^K\l\langle \expect{\mathbf{g}^{(k)}_{i,t}}, \theta_*^{(k)} \r\rangle  \r|   \\
\leq& \sum_{t=0}^{T-1}\l| \sum_{\mathbf{s}\in\mathcal{S}^{(1)}\times\cdots\times\mathcal{S}^{(K)}} \sum_{\mathbf{a}\in\mathcal{A}^{(1)}\times\cdots\times\mathcal{A}^{(K)}} 
\l( v_t(\mathbf{s}) - d_\Pi(\mathbf{s}) \r)\Pi(\mathbf{a}|\mathbf{s})
 \r| K\Psi    \\
\leq& K\Psi\sum_{t=0}^{T-1}\|v_t-d_\Pi\|_1\leq    2K\Psi\sum_{t=0}^{T-1}e^{(r_1-t)/r_1}
\leq 2eK\Psi\int_0^{T-1}e^{-t/r_1}dt
=2er_1K\Psi,
\end{align*}
where the third inequality follows from Lemma \ref{lemma:small-bound}.
Taking $C_1 = 2er_1$ finishes the proof of \eqref{diff-2} and \eqref{diff-1} can be proved in a similar way.

In particular, we have for any randomized stationary policy $\Pi$ that satisfies the constraint \eqref{main-constraint}, we have
\begin{align*}
  T\cdot\sum_{k=1}^K\l\langle \expect{\mathbf{g}^{(k)}_{i,t}}, \theta_*^{(k)} \r\rangle
\leq&\sum_{t=0}^{T-1}\l| \expect{\sum_{k=1}^K g_{i,t}^{(k)}(a^{(k)}_t,s^{(k)}_t) \Big| d_0,\Pi }  -  
\sum_{k=1}^K\l\langle \expect{\mathbf{g}^{(k)}_{i,t}}, \theta_*^{(k)} \r\rangle  \r|  \\
&+  \sum_{t=0}^{T-1}\expect{\sum_{k=1}^K g_{i,t}^{(k)}(a^{(k)}_t,s^{(k)}_t) \Big| d_0,\Pi }
\leq2er_1K\Psi + 0 = 2er_1K\Psi,
\end{align*}
finishing the proof.
\end{proof}


\bibliographystyle{alphaabbr}
\bibliography{asyn-theory,constrained-online}

\end{document}